\documentclass{article}
\usepackage[top=0.8in, bottom=0.9in, left=1in, right=1in]{geometry}
\usepackage{titlesec}
\usepackage{silence}
\usepackage{url}
\usepackage{upgreek}
\usepackage{amsfonts,amsmath,amssymb,amsthm}           
\usepackage{scalerel}
\usepackage[backref,colorlinks=true,linkcolor=blue]{hyperref}
\usepackage{footmisc}
\usepackage{verbatim}
\usepackage{calrsfs}
\usepackage{dsfont}
\usepackage{xspace}
\usepackage[overload]{empheq}

\usepackage{chngcntr}
\usepackage{etoolbox}
\usepackage{totcount}
\usepackage{multirow}
\usepackage{stmaryrd}
\usepackage{tikz}
\usepackage{wrapfig}
\usetikzlibrary{decorations.markings}
\usetikzlibrary{arrows}

\titleformat{\paragraph}
{\normalfont\normalsize\bfseries}{\theparagraph}{1em}{}
\titlespacing*{\paragraph}
{0pt}{3.25ex plus 1ex minus .2ex}{1.5ex plus .2ex}
\makeatletter
\renewcommand*{\eqref}[1]{%
  \hyperref[{#1}]{\textup{\tagform@{\ref*{#1}}}}%
}
\newtheorem{theorem}{Theorem}
\newtheorem{corollary}[theorem]{Corollary}
\newtheorem{proposition}[theorem]{Proposition}
\newtheorem{remark}[theorem]{Remark}
\newtheorem{definition}[theorem]{Definition}
\newtheorem{example}[theorem]{Example}
\newtheorem{lemma}{Lemma}
\numberwithin{lemma}{theorem}

\counterwithin*{equation}{section}
\counterwithin*{equation}{subsection}
\counterwithin*{theorem}{section}
\counterwithin*{theorem}{subsection}
\newtotcounter{eqtest}
\renewcommand\theequation{\ifnumgreater{\value{subsection}}{0}{\thesubsection.}{\thesection.}\arabic{equation}}%
\renewcommand\theeqtest{\ifnumgreater{\value{subsection}}{0}{\thesubsection.}{\thesection.}\total{eqtest}}%

\newcounter{testfoot}
\newcounter{testfoot2}
\newtotcounter{testfoot7}
\newtotcounter{testfoot3}
\newtotcounter{testfoot4}
\newtotcounter{testfoot5}
\newtotcounter{testfoot8}
\newtotcounter{testfoot6}
\newcommand{\codefootnote}{\arabic{testfoot}}
\newcommand{\codefootnotesec}{\arabic{testfoot2}}
\newcommand{\codefootnotetri}{\total{testfoot3}}

\newcommand{\ie}{\textit{i.e.}}
\newcommand{\eg}{\textit{e.g.}}

\def\sos{SOS\xspace}
\DeclareMathOperator*{\sprod}{\scalerel*{\cdot}{\bigodot}}
\def\nnabla{\nabla}
\def\R{\mathbb R}
\def\P{\mathbb P}
\def\N{\mathbb N}
\def\C{\mathbb{C}}
\def\u{\mathbf u}
\def\s{\mathbf s}
\def\a{\mathbf a}
\def\b{\mathbf b}
\def\c{\mathbf c}
\def\cc{\overline{\mathbf c}}
\def\d{\mathbf d}
\def\dd{{\overline{\mathbf d}}}
\def\ddd{\overline{d}}
\def\pp{{\overline{p}}}
\let\originalv\v
\def\v{\mathbf v}
\def\w{\mathbf w}
\def\e{\mathbf e}
\def\ee{\overline{\mathbf e}}
\def\ees{\overline{e}}
\def\x{\mathbf x}
\def\xx{\overline{\x}}
\def\xxx{\overline{x}}
\def\z{\mathbf z}
\def\y{\mathbf y}
\def\yy{\overline{\mathbf y}}
\def\r{\mathbf r}
\def\CC{\mathcal{C}}
\def\L{\mathcal{L}}

\def\MM{\overline{M}}
\DeclareMathAlphabet{\pazocal}{OMS}{zplm}{m}{n}
\def\CCC{\pazocal{C}}
\def\CCCC{\overline{C}}
\def\PP{\mathcal{P}}
\def\NN{\pazocal{N}}
\def\K{\pazocal{K}}
\def\AA{\pazocal{A}}
\def\AAA{\overline{\pazocal{A}}}

\def\A{\overline{A}}
\def\W{\overline{W}}
\def\B{\overline{B}}
\def\Y{\overline{Y}}
\def\YY{\pazocal{Y}}
\def\zeros{{\mathbf{0}}}
\def\ones{{\scalebox{1.2}{$\mathds{1}$}}}

\def\sx{{\scalebox{1.2}{$\x$}}}
\def\sX{{\scalebox{1.2}{$X$}}}
\def\LL{{\scalebox{1.4}{$\L$}}}
\def\segal{{\scalebox{0.55}{=}}}
\def\mumu{{\boldsymbol{\mu}}}
\def\mumumu{\overline{\boldsymbol{\mu}}}
\def\mumumus{\overline{{\mu}}}
\def\betabeta{{\boldsymbol{\beta}}}
\def\alphaalpha{{\boldsymbol{\alpha}}}
\def\lambdalambda{{\boldsymbol{\lambda}}}
\def\inter{\texttt{interior}}
\def\closure{\texttt{closure}}

\newcommand\vandenbergheWrong{ (see Section 3.1.4 of the book ``Convex
Optimization'' by Stephen Boyd and Lieven Vandenberghe, p.~71, Cambridge University
Press, 2004) }

\tikzstyle{every picture}+=[baseline]                 
\tikzstyle{every picture}+=[remember picture]         
\title{Demystifying the characterization of SDP matrices in
continuous and combinatorial optimization}

\author{Daniel Porumbel}
\date{June, 2022}
\begin{document}
\maketitle
\begin{center}
{\bf Argument}

\end{center}

{\small \noindent
This manuscript was written because I found no other introduction to SDP
programming  that targets
the same audience.
This work is intended to be accessible to anybody who does not hate
maths, who knows what a derivative is and accepts (or has a proof of) 
results like $\det(AB)=\det(A)\det(B)$. If you know this, I think 
you can understand most of this text without buying other 
books; my goal is not to remind/enumerate a list of results but 
to (try to) make the reader examine (the proofs of) these results to get full insight into them.

A first difference compared to other existing introductions to SDP
is that this work 
comes out of a mind
that was itself struggling to understand.
{\it This may seem to be only a weakness, but, 
paradoxically, it is both a weakness and a strength. }
First, 
I tried to minimize the distance between the author and the readers
as much as possible, even up to hoping to achieve a kind of
limited mutual empathy (to the extent that can be possible).
This enabled me avoid a quite common pitfall:
many long-acknowledged experts tend to forget the
difficulties of beginners, even up to kind of trying
to overpower their readers. 
I've seen it but I can't explain it better.
Other
experts try to make all proofs as short as possible and to dismiss as
unimportant certain key results they have seen thousands of time in their career.
I also avoided this, even if I did shorten a few proofs when I revised
this manuscript two years after it was first written. However, I also kept
certain proofs that seem longer than necessary because I feel they offer more
insight; an important goal is to capture 
the ``spirit'' of each proven result instead of reducing it to a flow
of dry formulae.

The very first key step towards mastering SDP programming is
to get full insight into
the eigen-decomposition of real symmetric matrices.
It is enough to see the way many other introductions to SDP programming
address this eigen-decomposition 
to understand 
that their target audience is different from mine.
They often list the eigen-decomposition without proof, while I give two proofs
to really familiarize the reader with it.%
\footnote{
In
``Handbook of Semidefinite Programming
Theory, Algorithms, and Applications''
by H.~Wolkowicz, R.~Saigal and L.~Vandenberghe,
the eigen-decomposition (called spectral theorem) is listed with no proof in 
Chapter 2 ``Convex Analysis
on Symmetric Matrices''. 
The introduction of
the ``Handbook on Semidefinite, Conic and Polynomial Optimization'' by
M.~Anjost and J.B.~Lasserre  refers the reader
to the (700 pages long) book
``Matrix analysis'' by Horn and Johnson.
As a side remark, the introductions of both these handbooks are rather short (14
or resp.~22 pages)
and they mainly remind/enumerate different key results pointing to other books
for proofs.
In 
``Semidefinite Programming for Combinatorial Optimization'',
by
C.~Helmberg,
the eigen-decomposition  is presented in an appendix and redirects
the reader to the same ``Matrix analysis'' book.
The slides of the course ``Programmation linéaire et optimisation combinatoire'' of 
Fr\'ed\'eric~Roupin 
for the 
``Master Parisien de Recherche Opérationnelle'' 
(\url{lipn.univ-paris13.fr/~roupin/docs/MPROSDPRoupin2018-partie1.pdf})
provide many results from my manuscript but no 
proof is given.
The MIT course ``Introduction to Semidefinite Programming '' by R. Freund does
not even provide the SDP definition or the eigen-decomposition.
The  book ``Convex Optimization'' by
S.~Boyd and L.~Vandenberghe 
starts using SDP
matrices from the beginning (\eg, to define ellipsoids in Section 2.2.2) without defining the concept of SDP matrix, not even
in appendix. 
The argument could extend to other non-trivial concepts that
are taken as pre-requisite in above works.
For instance, the above ``Convex Optimization'' book introduces  the square root of an
SDP matrix (in five lines in Appendix A.5.2), without showing the uniqueness -- the proof takes
half a page in Appendix~\ref{appSquareRoot} of this manuscript.
}

If you can say
``Ok, I still vaguely remember the eigen-decomposition (and other key SDP
properties) from my (under-)graduate
studies some $n\geq 5$ years ago; I don't need a proof",
then you do not belong to my target audience. I am highly
skeptical that such approach can lead to anything but superficial learning.
Anyhow, my brain functions in the most opposite manner.
\textit{I like to learn by checking all proofs by myself and
I can't stand taking things for granted}.
The only unproven facts from this manuscript are the fundamental theorem of algebra
and two results from Section~\ref{secsoshierarchy}.
But I do provide 
complete proofs, for instance, for \textit{the Cholesky decomposition of SDP matrices,
the strong duality theorem for linear conic programming (including SDP
programming), six  equivalent
formulations of the Lov\'asz theta number, the copositive formulation of the
maximum stable, a few convexification results for quadratic programs and many
others}.
I tried to prove everything by myself, so that certain proofs are original although
this introduction was not intended to be research work;
of course I got help multiple times from the internet, articles and books, the
references being indicated as footnote citations. 

The most essential building blocks are presented in the first part.
One should really master
this first part before jumping to the second one; the essentials from
the first part may even be generally useful for reading other SDP work. In fact,  
the goal of this manuscript is to give you all the tools needed to move to the next level
 and carry out research work.


}
\noindent\line(1,0){100}
\setcounter{tocdepth}{2}
\tableofcontents

\newpage

\addtocontents{toc}{{\bf PART 1 THE ESSENTIAL BUILDING BLOCKS} \par}
\begin{center}
\line(1,0){400}

\Large
PART 1: THE ESSENTIAL BUILDING BLOCKS

\line(1,0){400}

\end{center}

\section{Characterization of semidefinite positive (SDP) matrices}

\subsection{Real symmetric matrices, eigenvalues and the
eigendecomposition\label{sec11}}
This light introduction aims at familiarizing the reader with the main concepts of real symmetric matrices,
eigenvalues and the eigenvalue decomposition. 
Experts on this topic can skip to Section~\ref{secEquivDefs} or even to further
sections. 
Absolute beginners should first consult 
Appendix~\ref{appBasic} to get the notion
of matrix rank, 
(sub-)space dimension,
(principal) minor, or to recall how
$\det(A)=0 \iff \exists \x$ s.~t.~$A\x=\zeros$.
To familiarize with such introductory concepts it is also useful to solve a few
exercises, but the only exercise I propose is to ask the reader to prove by
himself all theorems whose proof does not exceed half a page.

Given (real symmetric) matrix $A$, we say $\lambda$ is a (real) \textit{eigenvalue} of $A$ if
there exists \textit{eigenvector} $\v$ such that $A\v=\lambda \v$. 
Notice that by multiplying the eigenvector with a constant we obtain another eigenvector. An
eigenvector is called \textit{unitary} if it has a norm of 1, \ie, $|\v|^2=\v^\top\v=\sum_{i=1}^n v_i^2=1$.
We say $\v$ is an eigenvector if it satisfies 
$\lambda I_n \v - A\v=\zeros$, equivalent to $(\lambda I_n -A)\v=\zeros$, which
also means
$\det(\lambda I_n - A)=0$. 

We can also define the eigenvalues as the roots of the \textit{characteristic
polynomial} $\det(xI_n-A)$.
Indeed, if $\det(\lambda_1 I_n-A)=0$, there exists a (real) eigenvector $\v_1$ such that 
$(\lambda_1 I_n -A)\v_1=0$, equivalent to $A\v_1=\lambda_1\v_1$. The fact that $\v_1$ needs
to exist is formally proven in Prop~\ref{lindep};
this Prop~\ref{lindep} proves the statement in the 
more
general
context
of
complex matrices 
because we should not yet take for granted that root $\lambda_1$ is
not complex, \ie, 
$\lambda_1 I_n -A$ could have complex entries in principle.
However, it \textit{is} possible to show that a real symmetric matrix has only real eigenvalues
and real proper eigenvectors. The proof is not really completely obvious, it takes
a third of a page and it is given in appendix (Prop.~\ref{propEigenReal}).
By developing the characteristic polynomial, one can also prove that the determinant is the
product of the eigenvalues (Prop.~\ref{propdetiseigenprod}).

The characteristic polynomial has degree $n$, and so, it has $n$ roots (eigenvalues), 
but some of them can have multiplicities greater than 1, \ie, some eigenvalues can 
appear more than once. 
However, each eigenvalue is associated to at least one eigenvector. An eigenvalue with
multiplicity greater than 1 can have more than one eigenvector. The following result
is called the \textit{eigendecomposition} of the real symmetric matrix $A$.

\begin{align}
&\hspace{1.26em}\text{unitary eigenvectors of $A$}
\notag
\\
&\hspace{2.7em}
\begin{matrix}
\downarrow&
~~~\downarrow&
\hspace{0.4em}\dots&
~\downarrow
\end{matrix}
\notag
\\
A&=
\begin{bmatrix}
        v_{11}     &        v_{12}   &  \dots & v_{1, n}                 \\
        v_{21}     &        v_{22}   &  \dots & v_{2 ,n}                 \\
        \vdots     &        \vdots   &  \ddots& \vdots                    \\
        v_{n  ,1}  &        v_{n,2}&  \dots &   v_{n  ,n  }               \\
\end{bmatrix}
\begin{bmatrix}
   \lambda_1^{\tikz\coordinate(lambda1);}      &         0       &  \dots &  0                        \\
        0          &      \lambda_2^{\tikz\coordinate(lambda2);} &  \dots &  0                        \\
        \vdots     &        \vdots   &  \ddots& \vdots                    \\
        0          &        0        &  \dots & \lambda_n^{^{\tikz\coordinate(lambdan);}}                \\
\end{bmatrix}
\begin{bmatrix}
        v_{11}     &        v_{21}   &  \dots & v_{n,1}                 \\
        v_{12}     &        v_{22}   &  \dots & v_{n,2}                 \\
        \vdots     &        \vdots   &  \ddots& \vdots                    \\
        v_{1,n  }  &        v_{2,n  }&  \dots & v_{n,n}               \\
\end{bmatrix}
\begin{matrix}
\leftarrow \\
\leftarrow \\
\vdots\\
\leftarrow \\
\end{matrix}
\notag \\
&=\Big[\v_1~\v_2~\dots \v_n\Big] \texttt{diag}(\lambda_1,~\lambda_2,~\dots
\lambda_n)\begin{bmatrix}\v_1^\top \\ \v_2^\top \\ \vdots \\
\v_n^\top\end{bmatrix}= U\Lambda U^\top
\label{firstEqDecomp}
\\
&=
\sum_{i=1}^n \lambda_i \v_i\v_i^\top ,\notag
\end{align}%
\begin{tikzpicture}[overlay]
\node(ignore) at (14.6,4.7) {\rotatebox{90}{transposed}};
\node(ignore2) at (15,4.7) {\rotatebox{90}{eigenvectors}};
\node(start) at (8.37,6.61) {\scalebox{1}{eigenvalues of $A$}};
\draw[darkgray, decoration={markings,mark=at position 1 with {\arrow[thick]{>}}}, postaction={decorate} ]
                                                    (start)  to  (lambda1.north);
\draw[darkgray, decoration={markings,mark=at position 1 with {\arrow[thick]{>}}}, postaction={decorate} ]
                                                    (start)  to  (lambda2.north);
\draw[darkgray, decoration={markings,mark=at position 1 with {\arrow[thick]{>}}}, postaction={decorate} ]
                                                    (start)  to  (lambdan.east);
\end{tikzpicture}%
where $\v_1,\v_2,\dots \v_n$ are the unitary (column) eigenvectors of $A$ associated to
(some repeated) eigenvalues $\lambda_1,~\lambda_2,\dots \lambda_n$
and $\Lambda=\texttt{diag}(\lambda_1,~\lambda_2~\dots
,\lambda_n)$ is the diagonal matrix with $\Lambda_{ii}=\lambda_i~\forall i\in[1..n]$.
The eigenvectors are unitary and orthogonal, meaning that $\v_i^\top \v_j=0,~\forall
i,j\in[1..n],~i\neq j$ and $\v_i^\top \v_i=1~\forall i\in[1..n]$. This directly
leads to
\begin{equation}\label{eqOrthoNormal}
[\v_1~\v_2~\dots \v_n]^\top [\v_1~\v_2~\dots \v_n]=I_n
\textnormal{,~equivalent to~}
[\v_1~\v_2~\dots \v_n][\v_1~\v_2~\dots \v_n]^\top=I_n,
\end{equation}
using the very well known property $XY=I\implies YX=I$
(Prop.~\ref{propInverseComutes}).

Perhaps a simpler way to visualise the form of an SDP matrix 
consists of rewriting 
$A=\sum_{i=1}^n \lambda_i \v_i\v_i^\top$ with $i$ as a
superscript, always with $\lambda_i\geq 0$:

$$A=\sum_{i=1}^n \lambda^i \v^i{\v^i}^\top =
\sum_{i=1}^n \lambda^i
\begin{bmatrix}
        v^i_{1}\cdot v^i_{1} &        v^i_{1}\cdot v^i_{2} &  \dots & v^i_{1}\cdot v^i_{n}                 \\
        v^i_{2}\cdot v^i_{1} &        v^i_{2}\cdot v^i_{2} &  \dots & v^i_{2}\cdot v^i_{n}                 \\
        \vdots         &        \vdots   &  \ddots& \vdots                    \\
        v^i_{n}\cdot v^i_{1} &        v^i_{n}\cdot v^i_{2}   &  \dots & v^i_{n}\cdot v^i_{n}                 \\
\end{bmatrix}
$$

\subsubsection*{The simple case of distinct eigenvalues}

It is important to familiarize with this decomposition. For this,  let us first examine a proof that works
for symmetric matrices $A$ with distinct eigenvalues (all with multiplicity one). 
We will first show that the eigenvectors of symmetric $A$ are orthogonal.
Let $\v_1,~\v_2,\dots \v_n$ be the unitary eigenvectors of resp.~$\lambda_1,~\lambda_2,\dots\lambda_n$.
For any $i,j\in[1..n],~i\neq j$,
we can write 
$\v_i^\top A\v_j=\v_i^\top\lambda_j\v_j=\lambda_j\v_i^\top \v_j$ and also
$\v_i^\top A\v_j=\lambda_i\v_i^\top \v_j$ based on $\v_i^\top A={(\v_i^\top
A)^\top}^\top=(A^\top \v_i)^\top =(A\v_i)^\top=\lambda_i\v_i^\top$.
This leads to $\lambda_j\v_i^\top \v_j=\lambda_i\v_i^\top \v_j$, and, using $\lambda_i\neq \lambda_j$, 
we obtain $\v_i^\top \v_j=0~\forall i,j\in[1..n],~i\neq j$. 

\begin{sloppypar}
We now construct the eigen-decomposition.
Using $\lambda_i \v_i=A\v_i~\forall i\in[1..n]$, 
we obtain
$A[\v_1~\v_2~\dots \v_n]=[\lambda_1\v_1~\lambda_2\v_2~\dots \lambda_n\v_n]$,
equivalent to 
$A [\v_1~\v_2~\dots \v_n] = [\v_1~\v_2~\dots \v_n] \texttt{diag}(\lambda_1,~\lambda_2,\dots
\lambda_n) $. We now multiply both sides at right by
$[\v_1~\v_2~\dots \v_n]^\top$ which is equal to $[\v_1~\v_2~\dots \v_n]^{-1}$
by virtue
\eqref{eqOrthoNormal}; we thus obtain
$A = [\v_1~\v_2~\dots \v_n] \texttt{diag}(\lambda_1,~\lambda_2,\dots
\lambda_n)[\v_1~\v_2~\dots \v_n]^\top$,
which is exactly \eqref{firstEqDecomp}.
\end{sloppypar}

\subsubsection*{Generalizing the proof for the case of
eigenvalues with different multiplicities}

We now prove the eigen-decomposition in the general case by extending the above
construction. We need two steps

\begin{description}
\leftskip-0em
    \item [Step.~1]
We first show that the above construction for eigenvalues with
unitary multiplicities can be readily generalized to the case where
there is a different 
eigenvector for each of
the $k_i$ repetitions of each
root $\lambda_i$ of the characteristic
polynomial
(\ie, for any eigenvalue $\lambda_i$).
In this case, 
there are
$k_i$ linearly independent eigenvectors associated to $\lambda_i$; we say that
the eigenspace dimension $k'_i$ of $\lambda_i$ (the geometric multiplicity
$k'_i$)
is equal to the algebraic multiplicity $k_i$.
This step is relatively straightforward and it is proved in full detail (and slightly generalized) 
in Prop.~\ref{propGeoAlgeMultResultsEigendecomp}.

   \item [Step.~2]

We need to show that the eigenspace of $\lambda_i$ has dimension $k_i$.
To show this, we use the notion of similar matrices: $X$ and $Y$ are similar if there exists $U$
such that $X=UYU^{-1}$;
we say that $Y$ is the
representation of $X$
in the basis composed by the columns of $U$.
Two simple properties (Prop.~\ref{simMatSameEigenVals}
and~\ref{propSameGeoMultForSimMat}) show that:
(i) two similar matrices have the same
characteristic polynomial  so that any
common eigenvalue $\lambda_i$ is repeated $k_i$ times 
and
(ii) similar matrices have the same eigenspace dimension $k'_i$ for any common eigenvalue
$\lambda_i$.
Assuming $k'_i<k_i$ for symmetric $A\in\R^{n\times n}$, 
we can construct $U\in\R^{n\times n}$ by putting $k'_i$ orthogonal unitary eigenvectors of
$\lambda_i$ on the first $k'_i$ columns and by filling the other columns
(introducing $n-k'_i$ unitary vectors perpendicular on the first $k'_i$ columns) such that $U^\top U=I_n$.
We obtain
$
U^\top A U=
\left[
\begin{smallmatrix}
\lambda_i I_{k'_i}           & \zeros     \\
\zeros &      D
\end{smallmatrix}
\right].
$
Since similar matrices have the same characteristic polynomial and $k'_i<k$, $\lambda_i$ needs to be
a root of the characteristic polynomial of $D$, and so, $\lambda_i$ needs to have an eigenvector in
$D$. This means that the eigenspace of $\lambda_i$ in
$\left[
\begin{smallmatrix}
\lambda_i I_{{k'_i}}           & \zeros     \\
\zeros &      D
\end{smallmatrix}
\right]
$
has dimension at least $k'_i+1$, contradicting the fact that similar matrices have the same
eigenspace dimension for $\lambda_i$. As such, the assumption $k'_i<k_i$ was false and we
obtain $k'_i=k_i$, \ie, $\lambda_i$ 
has the same algebraic and geometric multiplicity $k_i$.
For the skeptical, 
this Step 2
 is proven in full detail in Prop.~\ref{propGeoAlgeEqual}.
\end{description}

\subsection{\label{secEquivDefs}Equivalent SDP definitions}
\begin{definition} \label{def1}A symmetric matrix $A\in \R^{n \times n}$ is positive
semidefinite (SDP) if the following holds for any vector $\x \in \R^n$ 
\begin{equation}\label{sdpbase}
\hspace{-4em}
\x^\top A \x = \x \sprod A\x = A \sprod \x \x^\top =
\texttt{trace}(X \x\x^\top)=
\sum_{i=1}^n\sum_{j=1}^n
a_{ij}x_ix_j\geq 0,\hspace{-4em}
\end{equation}
where $\sprod$ stands for the scalar product.  If the above inequality is always
strict, the matrix is positive definite. If the inequality is reversed, the
matrix is negative semidefinite.
\end{definition}

\begin{proposition}\label{propEigenSDP}A symmetric matrix $A\in \R^{n \times n}$ is 
positive semidefinite (resp.~definite) if and only if all eigenvalues $\lambda_i$
verify $\lambda_i \geq 0$ (resp.~$>0$).
\end{proposition}
\begin{proof}
~\\
$\Longrightarrow$~\\
We consider $A$ is positive semidefinite (resp.~definite).
Assume there exist an eigenvalue $\lambda<0$ (resp $\lambda\leq 0$) associated
to eigenvector $\v$. We have $\v^\top A\v=\v^\top \lambda \v=\lambda
(v^2_1+v^2_2+\dots v^2_n)$. If $\lambda<0$ (resp $\lambda\leq 0$), then
$\v^\top A\v<0$ (resp
$\v^\top A\v\leq 0$). Thus, $A$ is not positive semidefinite (resp. definite).
This is a contradiction, and so, the assumption $\lambda<0$ (resp $\lambda\leq
0$) was false. We need to have $\lambda\geq 0$ (resp $\lambda>0$).\\
$\Longleftarrow$~\\
Without loss of generality, we suppose
$$\lambda_1\leq \lambda_2\leq \dots \lambda_n$$
We consider $\lambda_1$ satisfies $\lambda_1\geq 0$ (resp $\lambda_1>0$). We
consider the eigenvalue decomposition of symmetric matrix $A$, as constructed in
\eqref{firstEqDecomp} -- see also \eqref{eigendecompnice} of Proposition
\ref{propEigenDecomp}.

\begin{equation}\label{eigendecomp}
A=\sum_{i=1}^n \lambda_i \v_i\v^\top_i,
\end{equation}
where $\v_1,~\v_2,\dots,\v_n$ are the unitary orthonormal eigenvectors of $A$.
We consider the following minimization problem over all unitary vectors $\x\in
\R^n$  and we will show it is non-negative (resp.~strictly positive):
\begin{equation}\label{eq13}
\min_{|\x|=1} \x \sprod A\x
\end{equation}

The above formula $\x \sprod A\x$ with unitary $\x$ is actually a
particular case of the Rayleigh ratio/quotient usually written under the
form $R(A,\x)=\frac{\x \sprod A\x}{\x\sprod\x}$.

\begin{lemma} \label{lemmaMinRayleigh} The minimum value of \eqref{eq13} above is the smallest eigenvalue
$\lambda_1$. This minimum is attained by the unit eigenvector $\v_1$ of
$\lambda_1$. This lemma actually holds for any real symmetric matrix $A$.
This also implies that the maximum eigenvalue can be determined by calculating the minimum eigenvalue of $-A$, 
\ie, 
$\lambda_n=-\lambda_{min}(-A)=-\min_{|\x|=1} -\x \sprod A\x=
\max_{|\x|=1} \x \sprod A\x$.
\end{lemma}
\begin{proof}
We can write $\x$ in basis $\v_1,~\v_2,\dots \v_n$. This is always possible
because the equation $[\v_1~\v_2~\dots \v_n]\a=\x$ always has the 
solution $\a=[\v_1~\v_2~\dots \v_n]^\top\x$ simply using 
\eqref{eqOrthoNormal}. We can write

\begin{equation}\label{developx}
\x=a_1\v_1+a_2\v_2+\dots a_n\v_n
\end{equation}
Since $|\x|=1$, we have 
$\sum_{i=1}^n a^2_i\v_i\sprod\v_i
+2\sum_{i=1}^n\sum_{j=1,j\neq i}^n a_ia_j\v_i\sprod \v_j = 1$. Since $\v_i\sprod
\v_j=0~\forall i\neq j$ and $\v_i\sprod \v_i=1$, we obtain $1=\sum_{i=1}^n
a^2_i$.
We now replace this in \eqref{eq13} and we obtain 
\begin{align}
\min_{|\x|=1} \x \sprod A\x &=
\min_{\sum_{i=1}^n a^2_i = 1} 
    \left(\sum_{i=1}^n a_i\v_i\right)\sprod
    A
    \left(\sum_{i=1}^n a_i\v_i\right)\notag\\
    &=
\min_{\sum_{i=1}^n a^2_i = 1} 
    \left(\sum_{i=1}^n a_i\v_i\right)\sprod
    \left(\sum_{i=1}^n a_i\lambda_i\v_i\right)\notag\\
    &=
    \min_{\sum_{i=1}^n a^2_i = 1} 
    \sum_{i=1}^n \lambda_i a^2_i \notag\\
    &\geq 
    \min_{\sum_{i=1}^n a^2_i = 1} 
    \sum_{i=1}^n \lambda_1 a^2_i \notag \\
    &=\lambda_1,
\end{align}
The inequality is not strict. Using $\x=\v_1$ we obtain $\v_1\sprod
A\v_1=\lambda_1\v_1\sprod \v_1=\lambda_1$. The proof has not used the fact that
$\lambda_1\geq 0$, and so, it can be applied for any symmetric matrix $A$.
\end{proof}

We recall that any
$\y\in \R^n$ can be written as $\y=\alpha \x$ such that $|\x|=1$ (technically
$\alpha=|\y|=\sqrt{\sum_{i=1}^ny^2_i}$). We obtain $\y^\top A \y=\alpha^2
\x^\top A \x=\alpha^2\x\sprod A\x\geq \alpha^2\lambda_1$. 
Observe this minimum $\alpha^2\lambda_1$ can always be attained by
$\y=\alpha\v_1$.
If $\lambda_1 \geq 0$
(resp.~$\lambda_1>0$), the matrix $A$ is positive semidefinite (resp.~definite).
\end{proof}

\begin{proposition}\label{propCongruent} 
We say that matrices $S$ and $T$ are {\tt congruent} if there is a non-singular $Q$
such that $T=Q^\top SQ$. Two congruent matrices have the same SDP status, \ie,
$S$ is semidefinite positive if and only if $T$ is semidefinite positive
and
$S$ is definite positive if and only if $T$ is definite positive:
$$S\succeq \zeros \iff T\succeq \zeros \textnormal{ and }
S\succ \zeros \iff T\succ \zeros.$$
\end{proposition}
\begin{proof}
~\\
\noindent $\Longrightarrow$\\
If $S$ is SDP, then $\x^\top T \x=\x^\top Q^\top S Q \x = (Q \x)^\top S (Q
\x)\geq 0\implies T\succeq\zeros$. Using the fact that $Q$ is non-singular, 
we also obtain that $Q\x$ is zero only when $\x$ is zero, and so, 
$S\succ \zeros \implies (Q \x)^\top S (Q \x) >0\forall \x\neq \zeros~\implies
T\succ \zeros$.

\noindent $\Longleftarrow$\\
The converse proof is identical, because we can
write  $S=\left(Q^{-1}\right)^\top T Q^{-1}$ and apply the 
above two lines argument on swapped $S$ and
$T$.
We simply used $(Q^\top)^{-1}=(Q^{-1})^\top$, which 
is equivalent to 
$Q^\top (Q^{-1})^\top=I_n$, which
follows from
transposing $Q^{-1}Q=I_n$.
%
\end{proof}

We now use the above result\footnote{The same result appears in Section 6.9.1~(p.~91) of
the lecture notes of Maur\' icio de Oliveira, available on-line as of 2019
at 
\url{http://maecourses.ucsd.edu/~mdeolive/mae280b/lecture/lecture6.pdf}.}
to introduce a (very practical) remark on how certain well-known elementary row/column operations preserve the
SDP status.
\begin{proposition}\label{propRowColOpers}
It is well known that
the operations below preserve the determinant; we now prove that they also
preserve the SDP status: (a) the
initial matrix $A$ is SDP if only if the transformed matrix $A'$ is SDP and (b)
 $A$ is positive definite if and only if the $A'$ is positive definite.
Finally, $A$ and $A'$ are also congruent.
\begin{itemize}
\leftskip-1em
    \item[(i)] add row $i$ multiplied by $z_{ji}\in \R$ to row $j$ and then column
$i$ multiplied by $z_{ji}$ to column $j$
    \item[(ii)] perform a sequence of row operations as above and then the
corresponding (transposed) column operations
    \item[(iii)] permute the rows of $A$ and then permute in the same manner (apply the
same permutation)
on the columns of $A$
\end{itemize}
\end{proposition}
\begin{proof}
The row operation from (i) amounts to performing 
$\left(I_n+Z_{j\swarrow i}\right)A,$ where $Z_{j\swarrow i}$ is a matrix that
contains only one non-zero element: put $z_{ji}$ at row $j$ and column $i$.
The column operation from (i) amounts to multiplying
at right by
$\left(I_n+Z_{j\swarrow i}\right)^\top
= I_n+Z_{i\swarrow j}$.
The final matrix resulting from
operation (i) is:
$$A'=\left(I_n+Z_{j\swarrow i}\right)A\left(I_n+Z_{j\swarrow i}\right)^\top$$
As such, $A$ and $A'$ become congruent (notice $\det\left(I_n+Z_{j\swarrow
i}\right)=1$), finishing the proof by
Prop~\ref{propCongruent} above.

\noindent The operation (ii) simply consists of applying (i) several times, 
leading to the following congruent matrices:
$$A'=
\left(I_n+Z_{j_{1}\swarrow i_{1}}\right)
\left(I_n+Z_{j_{2}\swarrow i_{2}}\right)
\dots
\left(I_n+Z_{j_{p}\swarrow i_{p}}\right)
A
\left(I_n+Z_{j_{p}\swarrow i_{p}}\right)^\top
\dots
\left(I_n+Z_{j_{2}\swarrow i_{2}}\right)^\top
\left(I_n+Z_{j_{1}\swarrow i_{1}}\right)^\top
$$

\noindent The operation (iii) does not change the SDP status because
the permutation (reordering)
that transforms $A$ into $A'$ can be applied to transform $\x$ into
some $\x'$ so that 
$A \sprod \x \x^\top
=A' \sprod \x' \x'^\top$, and this operation can also be reversed.
\end{proof}

\begin{proposition}\label{propMinorsNonneg} A symmetric matrix $A\in\R^{n\times n}$ is positive
semidefinite if and only if all principal minors (recall Def~\ref{defMinor}) are non-negative. 
This is equivalent to stating
$A\succeq\zeros \iff det([A]_J)\geq
0~\forall J\subseteq [1..n]$,
where the operator $[\cdot]_J$ represents the principal minor obtained by selecting
rows $J$ and columns $J$.

We will later see that $A$ is positive \texttt{definite} if and only if all
leading principal minors are strictly positive (Sylvester criterion, see
Prop.~\ref{propSylvester}), which is also equivalent to 
$det([A]_J)> 0~\forall J\subseteq [1..n]$ using
Prop.~\ref{propNullMinorNullDet}.
\end{proposition}
\begin{proof}
~\\
$\Longrightarrow$\\
Take any subset of indices $J\subseteq [1..n]$ and
consider any vector $\overline{\x}_J\in\R^n$ that contains non-zero elements
only on positions $J$. Let $\x_J$ be the $|J|$-dimensional vector obtained
by extracting/keeping only the positions $J$ of $\overline{\x}_J$. Using the SDP
definition~\eqref{sdpbase}, the following needs to hold:
\begin{equation}\label{reduction1}A \sprod \overline{\x}_J \overline{\x}_J^\top\geq
0\end{equation}
Since $\overline{\x}_J \overline{\x}_J^\top$ contains non-zero elements only on
lines $J$ and columns $J$, we can re-write above formula as:
\begin{equation}\label{reduction2}[A]_J \sprod {\x}_J {\x}_J^\top\geq
0\end{equation}
Since this holds for any $\x_J$, the principal minor $[A]_J$ is SDP. This means
that the eigenvalues of $[A]_J$ are non-negative (Prop.~\ref{propEigenSDP}), and
so, the determinant of $[A]_J$ is non-negative because it is equal to the
product of its eigenvalues (Prop.~\ref{propdetiseigenprod}).\\
~\\
$\Longleftarrow$\\
Let $r$ be the rank of $A$. 
Based on Prop.~\ref{propPrincipalMinorTheorem}, $A$ has at least a non-zero {\it principal} minor of order $r$.
We can reorder the rows and columns
of $A$ 
to move this principal minor 
in the upper-left corner; we
obtain a (permuted) matrix
$A'=
\left[\begin{smallmatrix}
\AAA & B^\top \\
B    &           C 
\end{smallmatrix}\right]$,
where $\AAA$ is non singular; $A$, $A'$ and $\AAA$ have rang $r$.
Prop~\ref{propRowColOpers}.(iii), 
certifies that $A'$ has the same SDP status as $A$.
Since $A'$ has rang $r$, the bottom $n-r$ rows 
(\ie $\left[\begin{smallmatrix}
B    &           C 
\end{smallmatrix}\right]$), 
can be written as a linear combination of the first $r$ rows
(\ie, $\left[\begin{smallmatrix}
\AAA & B^\top \\
\end{smallmatrix}\right]$).
We can  subtract this linear combination of the first $r$ rows from the last $n-r$ rows 
to cancel them (make them zero). After performing the transposed
operation on the columns, we obtain a matrix 
$A''=
\left[\begin{smallmatrix}
\AAA & \zeros \\
\zeros    &           \zeros
\end{smallmatrix}\right]$ that has the same SDP status as $A'$ and $A$,
using
Prop.~\ref{propRowColOpers}.(ii). 
To prove $A,A',A''\succeq \zeros$, it is enough to
solve the following (sub-)problem:

\begin{equation} \tag{*}
\textnormal{Any {\it non-singular} symmetric~}
\AAA\in\R^{r\times r}
\textnormal{ that satisfies }
\det([\AAA]_J)\geq 0~\forall J\subseteq [1..r]
\textnormal{ is SDP}
\end{equation}

First, notice
$\AAA_{r,r}>0$ because
$\AAA_{r,r}=0$ would lead to
$\AAA_{i,r}=0~\forall i\in[1..r-1]$ (otherwise the
$2\times 2$ minor of $\AAA$ selecting rows/columns $i$ and $r$
would be negative) and if the last column of $\AAA$ is zero then
$\det(\AAA)=0$, contradiction. 
We can now subtract
the last row $r$
from each other row $i\in[1..r-1]$ in such a manner
(\ie, after multiplying it by $\frac{\AAA_{i,r}}{\AAA_{r,r}}$)
to cancel all elements on the last column above $\AAA_{r,r}$.
We then apply the same row operations but transposed (generating column
operations). This leads to
a matrix
$\left[\begin{smallmatrix}
\AAA_{r-1} & \zeros \\
\zeros    &  \AAA_{r,r}
\end{smallmatrix}\right]$
with the same
SDP status as $\AAA$ by virtue of 
Prop.~\ref{propRowColOpers}.(ii).
Also, any principal minor $\AAA_J$ of $\AAA_{r-1}$ is non-singular because
it corresponds to a principal minor 
$\left[\begin{smallmatrix}
\AAA_J & \zeros \\
\zeros    &  \AAA_{r,r}
\end{smallmatrix}\right]$ whose determinant is $\det(\AAA_J)\cdot \AAA_{r,r}$
where $\AAA_{r,r}>0$.
Thus, to prove (*), it is enough to prove
$\AAA_{r-1}\succeq \zeros$. But since $\AAA_{r-1}$ satisfies
all conditions of (*), 
we have actually obtained the same problem (*) 
on a dimension reduced by one. We can repeat this
until we remain with a matrix of size 1 and this proves $\AAA\succeq 0$, enough
to certify $A\succeq 0$.
\end{proof}

\subsection{Schur complements, the self-duality of the SDP cone and related properties}

\begin{proposition}\label{propSchurPart}(Schur complements particular case) \label{propschur1}
The $(n+1)\times (n+1)$ matrix 
$\bigl[\begin{smallmatrix}
        1     &    \b^\top        \\
       \b     &       C
\end{smallmatrix} \bigr]$ 
is SDP if and only if $C-\b\b^\top$ is SDP.
\end{proposition}
\begin{proof}
We will give two proofs. The first one produces a congruent matrix
using 
row/column operations. The second one is actually a formalization
of
the first, but it uses a ``magical'' short decomposition.\\

\noindent {\it Proof 1)} Let us subtract the first row of 
$\bigl[\begin{smallmatrix}
        1     &    \b^\top        \\
       \b     &       C
\end{smallmatrix} \bigr]$ from all other rows $i+1$ ($\forall i\in[1..n]$) premultiplying
$[1~\b^\top]$ with $b_i$. We then perform the transposed operation on the resulting matrix. The two
operations lead to 

\begin{equation}\label{eqRowColOp}
\begin{bmatrix}
        1     &    \b^\top        \\
       \b     &       C
\end{bmatrix}
\to 
\begin{bmatrix}
        1     &    \b^\top        \\
 \zeros_n     &   C-\b\b^\top
\end{bmatrix}
\to 
\begin{bmatrix}
        1     &    \zeros_n^\top     \\
 \zeros_n     &   C-\b\b^\top
\end{bmatrix}
\end{equation}

\noindent Using Prop~\ref{propRowColOpers}, the above two operations together lead to a
congruent matrix (at right) with the same SDP status as the initial
one (at left). The second matrix is SDP if and only if $C-\b\b^\top$ is SDP
(the ``$\Rightarrow$'' implication follows from performing a scalar product with
any $\x\x^\top$ with $x_0=0$ and the ``$\Leftarrow$'' implication follows
from the fact that the sum of two SDP matrices is SDP).

\noindent {\it Proof 2)} We formalize \eqref{eqRowColOp} using matrix multiplications. The first
transformation can be realized by:
\begin{equation}\label{eqfirsttran}
\begin{bmatrix}
        1     &    \b^\top        \\
 \zeros_n     &   C-\b\b^\top
\end{bmatrix}
=
\begin{bmatrix}
        1     &    \zeros_n^\top     \\
      -\b     &       I_n
\end{bmatrix}
\begin{bmatrix}
        1     &    \b^\top     \\
 \b           &   C
\end{bmatrix}
\end{equation}
and the second one by:
\begin{equation}\label{eqfirsttran2}
\begin{bmatrix}
        1     &    \zeros_n^\top     \\
 \zeros_n     &   C-\b\b^\top
\end{bmatrix}
=
\begin{bmatrix}
        1     &    \b^\top        \\
 \zeros_n     &   C-\b\b^\top
\end{bmatrix}
\begin{bmatrix}
        1     &    -\b^\top     \\
      \zeros_n&       I_n
\end{bmatrix}
\end{equation}
Combining~\eqref{eqfirsttran}-\eqref{eqfirsttran2}, we obtain:
\begin{equation*}
\begin{bmatrix}
        1     &    \zeros_n^\top     \\
 \zeros_n     &   C-\b\b^\top
\end{bmatrix}
=
\begin{bmatrix}
        1     &    \zeros_n^\top     \\
      -\b     &       I_n
\end{bmatrix}
\begin{bmatrix}
        1     &    \b^\top     \\
 \b           &   C
\end{bmatrix}
\begin{bmatrix}
        1     &    -\b^\top     \\
      \zeros_n&       I_n
\end{bmatrix}
\end{equation*}
We obtain again that
$\left[\begin{smallmatrix}
        1     &    \zeros_n^\top     \\
 \zeros_n     &   C-\b\b^\top
\end{smallmatrix} \right]$ 
is congruent to 
$\left[\begin{smallmatrix}
        1     &    \b^\top        \\
       \b     &       C
\end{smallmatrix} \right]$, because
$\det\left[\begin{smallmatrix}
       1     &    \zeros_n^\top     \\
     -\b     &       I_n
\end{smallmatrix} \right]=1$.
This finishes the proof by virtue of
Prop~\ref{propCongruent}.
\end{proof}

\begin{proposition}\label{propSchurGen}(Schur complements general case) 
Given positive definite $A\in\R^{m\times m}$, the $(n+m)\times (n+m)$ matrix 
$\bigl[\begin{smallmatrix}
        A     &    B^\top        \\
       B      &       C
\end{smallmatrix} \bigr]$ 
is SDP if and only if $C-B A^{-1}B^\top$ is SDP.
\end{proposition}
\begin{proof}

As in previous Prop.~\ref{propschur1}, we want to subtract from the (bottom) $n$
rows that cover $B\in\R^{n\times m}$ a
linear combination of the top $m$ rows so as to cancel (make zero) all terms of
$B$. We look for a matrix $X\in \R^{n\times m}$ such that $XA=-B$;
incidentally, $X$ and $B$ have the same size because the multiplication
with a square matrix conserves the size (of $A$).
By this operation, each row $i$ of $X$ generates a linear
combination of the rows of $A$ that equals the negative of row $i$ of $B$. 
The transpose of this operation is applied on columns to cancel $B^\top$.
It is clear
that $X=-BA^{-1}$. This explains the bottom-left term of 
matrix $U$ below.
\begin{equation*}
\begin{bmatrix}
        A                &    \zeros_{m\times n}     \\
 \zeros_{n\times m}      &  C-B A^{-1}B^\top
\end{bmatrix}
=
\underbrace{
\begin{bmatrix}
        I_m   &    \zeros_{m\times n}     \\
-B A^{-1}     &       I_n
\end{bmatrix}
}_
{U}
\begin{bmatrix}
        A     &    B^\top        \\
       B      &       C
\end{bmatrix}
\underbrace{
\begin{bmatrix}
        I_m        &    -A^{-1}B^\top           \\
\zeros_{n\times m} &       I_n
\end{bmatrix}
}_{U^\top}
\end{equation*}
\begin{tikzpicture}[overlay]
\node(ignore) at (14.6,0.72) {\rotatebox{-30}{\color{gray}$=(A^{-1})^\top B^\top$}};
\end{tikzpicture}

Notice that $U^\top$ is written after applying the simplification  
$(A^{-1})^\top=A^{-1}$ (which follows from $I=(A^{-1}A)^\top=A^\top \left(A^{-1}\right)^\top \implies
\left(A^{-1}\right)^\top=\left(A^\top\right)^{-1}=A^{-1}$).
Since $\det(U)=1$,
we obtain 
that
$\bigl[\begin{smallmatrix}
        A                &    \zeros_{m\times n}     \\
 \zeros_{n\times m}      &  C-B A^{-1}B^\top
\end{smallmatrix} \bigr]$ 
and
$\bigl[\begin{smallmatrix}
        A     &    B^\top        \\
       B      &       C
\end{smallmatrix} \bigr]$
are congruent.
We finish with Prop~\ref{propCongruent} as for
Prop~\ref{propSchurPart} above.
\end{proof}
\begin{proposition} \label{propSDPSelfDual}A symmetric matrix $A\in\R^{n\times n}$ is SDP if and only if $A\sprod B\geq 0$ for 
any SDP matrix $B$. We say that the cone $S_n^+$ of SDP matrices is self-dual.
\end{proposition}
\begin{proof}
We apply the eigen-decomposition \eqref{firstEqDecomp} on $A$ and $B$:
:
\begin{equation} \label{eqEigenDecomTwice}
A=\sum_{i=1}^n \lambda_i \v_i\v^\top_i,~~~
B=\sum_{i=1}^n \lambda'_i \u_i\u^\top_i,
\end{equation}
\begin{lemma} \label{lemmaprod}Given vectors $\u,\v\in\R^n$, we have 
$$(\u \u^\top)\sprod (\v\v^\top)=(\u\sprod\v)^2$$
\end{lemma}
\begin{proof}
$$(\u \u^\top)\sprod (\v\v^\top) = \v^\top (\u\u^\top ) \v 
= (\v^\top\u)(\u ^\top \v )
= (\u^\top \v)^\top(\u ^\top \v )=(\u\sprod\v)^2$$
\end{proof}

\noindent$\Longrightarrow$~\\
If $A$ is SDP, then $\lambda_i\geq 0\forall i\in[1..n]$. 
Using substitution 
\eqref{eqEigenDecomTwice}, we obtain
$$A\sprod B=\sum_{i=1}^n\sum_{j=1}^n \lambda_i\lambda'_j (\v_i\v^\top_i)\sprod (\u_j\u^\top_j)
           =\sum_{i=1}^n\sum_{j=1}^n \lambda_i\lambda'_j (\v_i\sprod\u_j)^2,$$
where we used Lemma \ref{lemmaprod} for the last equality. This shows that $A\sprod B\geq 0$.

\noindent$\Longleftarrow$~\\
Let us take some $i\in[1..n]$ and consider $B=\v_i^\top \v_i$, where recall that $\v_i$ is a unit
eigenvector of $A$. Since $A\sprod B\geq 0$, we deduce
$A\sprod \v_i^\top \v_i\geq 0$, or
$\v_i^\top A \v_i\geq 0$, equivalent to
$\v_i^\top \lambda_i \v_i\geq 0$. This means that $\lambda_i\geq 0$. All eigenvalues of $A$ are
non-negative, and so, $A$ is SDP.
\end{proof}

\begin{proposition}\label{propABprod} If $A,B$ are SDP, then $A\sprod B\geq 0$ and 
$A \sprod B = 0 \Longleftrightarrow AB=\zeros_{n\times n}$.
\end{proposition}
\begin{proof}
We apply the eigen-decomposition \eqref{firstEqDecomp} 
listing the only terms with non-zero eigenvalues:
\begin{equation} \label{eqEigenDecomTwiceBis}
A=\sum_{i=1}^r \lambda_i \v_i\v^\top_i,~~~~
B=\sum_{i=1}^{r'} \lambda'_i \u_i\u^\top_i
\end{equation}
where $r$ and $r'$ are the ranks of $A$ and resp.~B
(the number of non-zero eigenvalues, see Prop~\ref{propeigenmult}).
We now use Lemma \ref{lemmaprod} to calculate
$$A\sprod B=\sum_{i=1}^n\sum_{j=1}^r \lambda_i\lambda'_j (\v_i\v^\top_i)\sprod (\u_j\u^\top_j)
           =\sum_{i=1}^n\sum_{j=1}^{r'} \lambda_i\lambda'_j
(\v_i\sprod\u_j)^2\geq 0$$

If $A\sprod B=0$, then all terms
$\v_i\sprod\u_j$ with $i\in[1..r]$ and $j\in[1..r']$ need to be zero (recall we only use strictly
positive eigenvalues).
Now observe
$$AB=\sum_{i=1}^r\sum_{j=1}^{r'} \lambda_i\lambda_j (\v_i\v^\top_i)(\u_j\u^\top_j)
=\sum_{i=1}^r\sum_{j=1}^{r'} \lambda_i\lambda_j \v_i(\v^\top_i\u_j)\u^\top_j
=\sum_{i=1}^r\sum_{j=1}^{r'} \lambda_i\lambda_j \v_i\cdot 0\cdot \u^\top_j=\zeros_{n\times n}$$

We still need to show the converse: $AB=\zeros_{n\times n }\Longrightarrow A\sprod B=0$. Taking any
$k\in[1..n]$, the
$k^{th}$ diagonal element of $AB$ is $0=\sum_{\ell = 1}^n a_{k\ell}b_{k\ell}$ (we used the symmetry of
$B$). Summing up for all $k$ we obtain $0=\sum_{k=1}^n \sum_{\ell = 1}^n a_{k\ell}b_{k\ell} =
A\sprod B$.
\end{proof}

\subsection{Three easy ways to generate (semi-)definite positive matrices}
There are at least three easy ways of generating semidefinite (or definite)
positive matrices. 
\begin{enumerate}
\item Generate $A\in \R^{n\times n}$ by taking $A=V^\top V$ for any 
$V\in\R^{p\times n}$. It is easy to verify $\x^\top V^\top V\x=(V\x)^\top
(V\x)=|V\x|^2\geq 0~\forall \x\in\R^n$. If $V$ has rank $n$, then $V\x$ is
non-zero for any non-zero $\x$, and so, $|V\x|^2>0~\forall \x\in
\R^n-\{\zeros\}$, meaning that $A=V^\top V$ is positive definite. 
If $S\succ\zeros$ and $V$ has rank $n$, we also have $V^\top S V\succ \zeros.$
As a side
remark, notice $rank(A)=rank(V)$ using Prop.~\ref{propranktransprod}
(based on the rank--nullity theorem)
\item Take a diagonally dominant
matrix such that $A_{ii}\geq
r_i=\sum\limits_{\substack{j=1\\j\neq i}}^n \left|A_{ij}\right|~\forall i\in[1..n]$.
Any eigenvalue $\lambda$ of such matrices verify 
$\left|\lambda-A_{ii}\right|\leq r_i$ for some $i\in[1..n]$ by virtue of the
(relatively easy to prove) Gershgorin circle
Theorem~\ref{thGershgorin}.%
\footnote{I found this approach at 
page 4 of the 
Habilitation thesis (\textit{Habilitationsschrift}) 
of
Christoph Helmberg 
``Semidefinite Programming for Combinatorial Optimization'',
Technical University of Berlin (\textit{%
Technische Universit\" at Berlin}),
The Zuse Institute Berlin (\textit{Konrad-Zuse-Zentrum f\" ur
Informationstechnik Berlin}),
ZIB-report ZR-00-34,
available on-line as of 2017 at
\url{http://opus4.kobv.de/opus4-zib/files/602/ZR-00-34.pdf}.
}
If $A_{ii}\geq r_i~\forall i\in[1..n]$, we need to
have $\lambda\geq 0$. The matrix $A+\varepsilon I_n$ is {\it positive definite} for any
$\varepsilon>0$.
\item Take $A=|X|I_n + X$, where $|X|=\sqrt{X\sprod X}$. This follows from the fact
that the minimum eigenvalue of $X$ is greater than or equal to $-|X|$, by virtue of Prop.~\ref{propFrobenius} below.
\end{enumerate}
\begin{proposition}\label{propFrobenius} Given symmetric $X\in \R^{n\times n}$,
the Frobenius norm $|X|=\sqrt{\sum\limits_{i,j=1}^n X_{ij}^2}=\sqrt{X\sprod X}$ is equal to 
$\sqrt{\lambda_1^2+\lambda_2^2+\dots \lambda_n^2}$, where 
$\lambda_1,~\lambda_2,\dots \lambda_n$ are the eigenvalues of $X$. This means that the maximum
eigenvalue of $X$ is at most $|X|$ and the minimum eigenvalue is at least $-|X|$.
\end{proposition}
\begin{proof}Standard calculations can confirm $|X|=\sqrt{X\sprod
X}=\sqrt{\texttt{trace}(XX)}$;
and more generally we have $\texttt{trace}{(XY)}=X\sprod Y$.
We apply the 
eigendecomposition 
\eqref{firstEqDecomp}
to write symmetric $X$ in the form $X = U 
\texttt{diag}(\lambda_1,~\lambda_2,~\dots \lambda_n)
 U^\top $, where $U^{\top}=U^{-1}$. We obtain $XX=
U
\texttt{diag}(\lambda^2_1,~\lambda^2_2,~\dots \lambda^2_n)
U^\top$. This means that the eigenvalues of $XX$
are $\lambda^2_1,~\lambda^2_2,~\dots \lambda^2_n$, see also
Prop~\ref{simMatSameEigenVals}. Since the trace is the sum of the eigenvalues
(see Prop.~\ref{propdetiseigenprod}), we obtain 
\begin{equation}\label{eqFrobenius}|X|=\sqrt{\texttt{trace}(XX)}
=\sqrt{\lambda^2_1 + \lambda^2_2 + \dots +\lambda^2_n}.
\end{equation} 
Is is clear $X$ can have no eigenvalue strictly larger than $|X|$ or strictly lower than $-|X|$, because this would violate
\eqref{eqFrobenius}.%
\footnote{I first found this result in Section 10.1 of the lecture notes of Robert Freund
``Introduction to Semidefinite Programming (SDP)'', available on-line as of 2017 at
\url{https://ocw.mit.edu/courses/electrical-engineering-and-computer-science/6-251j-introduction-to-mathematical-programming-fall-2009/readings/MIT6_251JF09_SDP.pdf}.
}
\end{proof}
\subsection{Positive definite matrices: unique Cholesky factorization and
Sylvester criterion}

\begin{proposition}\label{propCholeskyDP} (Cholesky factorization of positive definite matrices) A 
real symmetric 
matrix $A$ is positive definite if and only if it can be factorized as:
\begin{equation*}
A=RR^\top=
\begin{bmatrix}
r_{11}   &  0     &   0  & \dots &   0   \\
r_{21}   & r_{22} &   0  & \dots &   0   \\
r_{31}   & r_{32} &r_{33}& \dots &   0   \\
\vdots   & \vdots &\vdots&\ddots & \vdots\\
r_{n1}   & r_{n2} &r_{n3}& \dots &r_{nn}  \\
\end{bmatrix}
\begin{bmatrix}
r_{11}   &r_{21}  & r_{31}& \dots&r_{n1}  \\
0        &r_{22}  & r_{32}& \dots&r_{n2}  \\
0        & 0      & r_{33}& \dots&r_{n3}  \\
\vdots   & \vdots &\vdots&\ddots & \vdots \\
0        &  0     &  0    &\dots &r_{nn}  \\
\end{bmatrix},
\end{equation*}
where the diagonal terms are strictly positive. The factorization is unique.

\noindent \texttt{Practical hint:} It could be useful to interpret $A=RR^\top$
in the sense that $A_{ij}$ is the product of rows $i$ and $j$ of $R$;
only the first $i$ (resp.~$j$) components of row $i$ (resp.~$j$) 
are non-zero.
\end{proposition}
\begin{proof}~\\
$\Longleftarrow$\\
Take any non-zero $\x\in\R^n-\{\zeros_n\}$ and observe that $\x^\top A \x = \x^\top R R^\top \x = 
(R^\top \x)^\top (R^\top \x)$. Writing $\y=R^\top \x$, this value is equal to $\sum_{i=1}^n
y^2_i\geq 0$. This inequality is strict because the only $\y\in \R^n$ such that 
$\sum_{i=1}^n y^2_i= 0$ is $\y=\zeros_n$ and because $\y=R^\top \x$ can not be $\zeros_n$
for any non-zero $\x$ (since $det(R^\top)=r_{11}r_{22}\dots r_{nn}>0$).\\
$\Longrightarrow$\\
We proceed by induction. The implication is obviously true for $n=1$. We suppose that 
there exists a unique factorization:
$$[A]_{n-1}=
[R]_{n-1}
[R]_{n-1}^\top 
=
\begin{bmatrix}
r_{11}   &  0      & \dots &   0       \\
r_{21}   & r_{22}  & \dots &   0       \\
\vdots   & \vdots  &\ddots & \vdots    \\
r_{n-1,1}&r_{n-1,2}& \dots &r_{n-1,n-1}\\
\end{bmatrix}
\begin{bmatrix}
r_{11}   &r_{21}  & \dots&r_{n-1,1}  \\
0        &r_{22}  & \dots&r_{n-1,2}  \\
\vdots   & \vdots &\ddots & \vdots \\
0        &  0     &\dots &r_{n-1,n-1}  \\
\end{bmatrix},
$$
where $r_{11},~r_{22},~\dots r_{n-1,n-1}>0$. We will prove that this decomposition can be extended
to a $n\times n$ decomposition for matrix $A$. 
The values $r_{1n},~r_{2n},\dots r_{n-1,n}$ are set to zero by definition to preserve the decomposition of $[A]_{n-1}$.
We can exactly determine $r_{n,1},~r_{n,2},\dots r_{n,n}$ 
using the following calculations:
    \begin{enumerate}
        \item[(a)] $\displaystyle r_{n1}=\frac{a_{n1}}{r_{11}}$, based on $a_{n1}=\r_{n,\times}\sprod \r_{1,\times}$, where
$\r_{i,\times}$ is the $i^\text{th}$ line of $R$.
        \item[(b)] $\displaystyle r_{n2}=\frac{a_{n2}-r_{21}r_{n1}}{r_{22}}$, based on $a_{n2}=\r_{n,\times}\sprod
\r_{2,\times}$.
        \item[(c)] $\displaystyle
            r_{ni}=\frac{a_{ni}-r_{i1}r_{n1}-r_{i2}r_{n2}-\dots-r_{i,i-1}r_{n,i-1}}{r_{ii}}$ for any
$i\leq n-1$, based on $a_{ni}=\r_{n,\times}\sprod \r_{i,\times}$.
        \item[(d)] $r_{nn}=\sqrt{a_{nn}-\sum_i^{n-1}{r_{ni}^2}}$, so
that the value of $r_{nn}$ that makes the factorisation $A=RR^\top$ work can be potentially
non-real, \eg, we could have $r_{nn}=i$ so that $i^2=-1$.
    \end{enumerate}

We only still need to show $r_{nn}$ is real. 
Since $[R]_{n-1}$ is non-singular by the induction hypothesis, there exists
$\x\in\R^{n-1}$ so that $\x^\top [R]_{n-1}=[r_{n,1}~r_{n,2},\dots r_{n,n-1}]$.
By bordering $\x$ with a $n^{th}$ component of value -1,
we obtain
$[\x^\top\,~-1] R=[\zeros_{n-1}\,~-r_{nn}]$. This means that
$
[\x^\top\,~-1] A 
\left[\begin{smallmatrix}
\x \\
- 1
\end{smallmatrix}\right]
=
[\x^\top\,~-1] RR^\top
\left[\begin{smallmatrix}
\x \\
- 1
\end{smallmatrix}\right]
=
[\zeros_{n-1}\,~-r_{nn}]
[\zeros_{n-1}\,~-r_{nn}]^\top
=r_{nn}^2$.
We thus have $r_{nn}^2>0$ because $A\succ \zeros$; recalling how
$r_{nn}$ was determined at point (d) above, it is clear that $r_{nn}$ can not be
an imaginary number, \ie, we can only have $r_{nn}\in\R_+$.
\end{proof}

\begin{proposition}\label{propSylvester}(Sylvester criterion) A symmetric matrix $A\in \R^{n \times
n}$ is positive definite if and only all leading principal minors are strictly
positive. By symmetrically permuting the rows and columns, this is equivalent to
the fact that any nested sequence of principal minors contains only strictly
positive minors. Principal minors $A''$ and $A'$ 
corresponding to rows/columns $J''$ and resp.~$J'$ are nested if and only if
$J''\subsetneq J'$ and $|J''|+1=|J'|$.
\end{proposition}

\noindent This theorem can be tackled from many different angles. The interested reader 
may be able to find a proof by himself (in less than an hour) if all material presented 
until here (including Appendix~\ref{appBasic}) has been acquired.  We present below three proofs, so as to gain an insight from every
possible angle.

\noindent\textit{Proof 1}
~\\
We proceed by induction. Both implications are obviously true for $n=1$. 
We now show how to move from $n-1$ to $n$.
We can use the fact that the
$(n-1)\times (n-1)$ leading principal minor $[A]_{n-1}$ has a non-null determinant. Since $\det
[A]_{n-1}\neq 0$, there exists $\x\in\R^{n-1}$ such that $\x^\top [A]_{n-1}$ is equal
to the first $n-1$ positions of the last row of $A$. After subtracting this linear 
combination of the first $n-1$ rows of $A$ from the last row followed by the
transposed operation on columns, we obtain a matrix
of the following form:
$$A'=
\begin{bmatrix}
                  &       &       &  0  \\
      \multicolumn{3}{r}{\smash{\raisebox{-0.5\normalbaselineskip}{\scalebox{1.5}{$[A]_{n-1}$}}}} & 0 \\
                  &       &       &      \vdots \\
0                 &  0    &\dots  &  a'_{n,n}  \\
\end{bmatrix}.
$$

We finish by applying Prop~\ref{propRowColOpers}: the above subtraction does not
change the SDP status or the determinant, and so, it is enough the prove the 
Sylvester criterion for a matrix of the form of $A'$ above. And this is obvious.
For the direct implication, simply check that
$A'\succ\zeros\implies a'_{n,n}>0\implies \det(A')=\det([A]_{n-1})\cdot
a'_{n,n}>0$. For the converse, we use
that $\det(A')>0\implies a'_{n,n}>0$. This proves $A'\succ\zeros$ because $A'\cdot \x\x^\top
=[x_1~x_2\dots x_{n-1}]
[A]_{n-1}
[x_1~x_2\dots x_{n-1}]^\top+a'_{n,n}x_n^2$, which
is strictly positive for any $\x\neq \zeros$ because $[A]_{n-1}\succ \zeros$ and
$a'_{n,n}>0$.\qed

\noindent\textit{Proof 2}
~\\
$\Longrightarrow$\\
Using the above Cholesky factorization of positive definite matrices
(Prop.~\ref{propCholeskyDP}), $A$ can be written $A=RR^\top$, where $R$ is a
lower triangular matrix with strictly positive diagonal elements
$r_{11},~r_{22},\dots r_{kk}>0$. One can simply
verify that $[A]_k=[R]_k[R]^\top_k$, where the operator $[\cdot]_k$ represents
the leading principal minor of size $k$. We obtain that 
$det([A]_k)=det([R]_k)det([R]^\top_k)=\left(r_{11}r_{22}\dots
r_{kk}\right)^2>0$, by virtue of $r_{11},r_{22},\dots r_{kk}>0$.
~\\
$\Longleftarrow$\\
We proceed by induction. The implication is obviously true for $n=1$. Suppose it
is true for $n-1$, so that the
$(n-1)\times (n-1)$ leading principal minor $[A]_{n-1}$ is positive definite. 

We need to show that $A$ is positive definite as well. Assume the contrary: $A$
has an eigenvalue $\lambda_u\leq 0$ with unit eigenvector $u$. Since $det(A)>0$ is the
product of the eigenvalues of $A$, $\lambda_u$ can not be zero, and so, $\lambda_u<0$.
Using again the fact that $det(A)$ is the product of the eigenvectors, we obtain that $A$ needs to
have (at least) another negative eigenvalue $\lambda_v<0$ with unit eigenvector $\v$.

There exist $a_u,a_v\in \R$ (not both 0) such that $a_uu_n+a_vv_n=0$, \ie, if
$u_n=0$ take $a_u=1,~a_v=0$ and if $u_n\neq 0$ take
$a_u=\frac{-v_n}{u_n},~a_v=1$. We define $\x=a_u\u + a_v\v$, so that $x_n=0$. We compute
\begin{align}
\x^\top A \x 
    = \x\sprod A\x
    &= \left(a_u\u + a_v\v\right) \sprod A \left(a_u\u + a_v\v\right)\notag \\
    &= \left(a_u\u + a_v\v\right) \sprod \left(\lambda_u a_u \u +\lambda_v a_v\v\right)\notag \\
    &= \lambda_u a^2_u\u\sprod \u + \lambda_v a^2_v\v\sprod \v  \notag\\
    &= \lambda_u a^2_u + \lambda_v a^2_v  \notag\\
    &<0\notag
\end{align}
The last inequality follows from $\lambda_u,\lambda_v<0$ and from the fact that $a_u$ and 
$a_v$ are not both zero. We now develop $\x^\top A\x=A\sprod \x\x^\top$
and notice that matrix $\x\x^\top$ has non-zero elements only on the
first $n-1$ rows and columns. 
As such, $A\sprod \x\x^\top<0$ simplifies to $[A]_{n-1}\sprod [x_1~x_2\dots
x_{n-1}] [x_1~x_2\dots x_{n-1}]^\top<0$, which
violates the induction hypothesis that $[A]_{n-1}$ is positive definite. We obtained a contradiction 
on the existence of an eigenvalue $\lambda_u\leq 0$. This means that all eigenvalues of $A$ are
strictly positive. By simply applying Prop.~\ref{propEigenSDP}, we obtain that $A$ is positive
definite. \qed
~\\
\noindent\textit{Proof 3}
~\\
$\Longrightarrow$\\
We can use the proof of the ``$\Longrightarrow$'' implication of
Prop.~\ref{propMinorsNonneg}, but the inequalities
\eqref{reduction1}-\eqref{reduction2} become strict. This means that every
principal minor of $A$ is positive definite, and, using
Prop.~\ref{propEigenSDP}, the principal minor has only strictly positive
eigenvalues. The determinant of the minor is strictly positive, as it is the
product of the eigenvalues (Prop.~\ref{propdetiseigenprod}).
~\\
$\Longleftarrow$
~\\
We proceed by induction. The statement is true for $n=1$. Suppose it
is true for $n-1$. The $(n-1)\times(n-1)$ leading principal minor is
positive definite, and so, its minimum eigenvalue
$\lambda_{\text{min}}([A]_{n-1})$ is strictly 
positive. We then apply lemma below to show that
$A$ has at least $n-1$
strictly positive eigenvalues, \ie, $0< \lambda_{\text{min}}([A]_{n-1}) \leq \lambda_2\leq \lambda_3\leq \dots \lambda_n$.
This ensures that $\lambda_1\lambda_2\dots\lambda_n = \det(A) > 0 $
(recall Prop.~\ref{propdetiseigenprod}) can only hold because $\lambda_1>0$,
so that actually all eigenvalues need to be positive which proves $A\succ\zeros$ (via
Prop~\ref{propEigenSDP}). 

\begin{lemma}\label{lemmaIntertwinedMinorEigenvals} Consider symmetric matrix $A$ 
with eigenvalues $\lambda_1\leq \lambda_2\leq \dots \leq \lambda_n$.
Any principal minor $A'$ of order
$n-1$  with eigenvalues 
$\lambda'_1\leq \lambda'_2\leq \dots \leq \lambda'_{n-1}$ verifies
$\lambda'_1\leq \lambda_2$.%
\end{lemma}
\begin{proof}
We consider $A'$ is obtained from $A$ by removing row $i$ and column $i$. Denote
by $\u_1$ and $\u_2$ the unit orthogonal eigenvectors of $A$ 
corresponding to
$\lambda_1$ and resp.~$\lambda_2$, recall \eqref{firstEqDecomp}.
One can surely find $\alpha_1,~\alpha_2\in \R$ (not both zero) such that vector 
$\u=\alpha_1\u_1+\alpha_2\u_2$ satisfies $u_i=0$. Furthermore, using an
appropriate scaling, the values of $\alpha_1$ and $\alpha_2$ can be chosen such
that $\alpha_1^2+\alpha_2^2=1$. Notice that $|\u|^2=
(\alpha_1\u_1+\alpha_2\u_2)\sprod
(\alpha_1\u_1+\alpha_2\u_2)=\alpha_1^2 \u_1\sprod \u_1+
                            \alpha_2^2 \u_2\sprod \u_2+
                            2\alpha_1\alpha_2 \u_1\sprod \u_2=
                            \alpha_1^2 |\u_1|^2+
                            \alpha_2^2 |\u_2|^2=
                            \alpha_1^2+\alpha_2^2=1$, where we used that $\u_1$
and $\u_2$ are unitary orthogonal eigenvectors.

Let us calculate $\u\sprod A\u$ (a particular form of the Rayleigh ratio
$R(A,\u)=\frac{\u \sprod A\u}{\u\sprod\u}$). We have
$\u\sprod A\u=\u\sprod \left(\alpha_1\lambda_1\u_1+\alpha_2\lambda_2\u_2\right)=
               \left(\alpha_1\u_1+\alpha_2\u_2\right)\sprod 
            \left(\alpha_1\lambda_1\u_1+\alpha_2\lambda_2\u_2\right)=
             \alpha^2_1\lambda_1 + \alpha^2_2\lambda_2\leq 
                \lambda_2$, where we used again that $\u_1$
and $\u_2$ are unitary orthogonal eigenvectors, followed by $\lambda_1\leq
\lambda_2$. Let
$\u'$ be  
$\u$ without component $i$ and we obtain
$\u\sprod A\u= 
A\sprod \u\u^\top=
A'\sprod \u'\u'^\top=
\u'\sprod A'\u'$, 
where we used the fact that row and column $i$ of $\u\u^\top$ are zero
based on $u_i=0$.
We obtained that the unitary $\u'$ yields $\u'\sprod A'\u'\leq \lambda_2$. Using
Lemma~\ref{lemmaMinRayleigh}, the smallest eigenvalue of $A'$ is less than or
equal to $\u'\sprod A'\u'\leq \lambda_2$, \ie, $\lambda'_1\leq \lambda_2$.
\end{proof}
~\\
The following lemma could be generally useful, although it is not used for other proofs
in this document.
\begin{lemma}\label{lemmaIntertwinedMinorEigenvals2} Consider symmetric matrix $A$ 
with eigenvalues $\lambda_1\leq \lambda_2\leq \dots \leq \lambda_n$.
Any principal minor $A'$ of order
$n-1$  with eigenvalues 
$\lambda'_1\leq \lambda'_2\leq \dots \leq \lambda'_{n-1}$ satisfies:%
\footnote{It is possible to prove that
$$\lambda_1\leq \lambda'_1\leq\lambda_2\leq \lambda'_2\leq\dots
\lambda'_{n-2}\leq \lambda_{n-1}\leq \lambda'_{n-1}\leq \lambda_n$$
The proof lies outside the scope of this document, 
but it could use the following argument by contradiction. Assume 
$\lambda'_1,~\lambda'_2,\dots \lambda'_k<\lambda_k$ and we can derive a contradiction.
All vectors of the subspace $S'$ generated by the first $k$
eigenvectors of $A'$ 
have a Rayleigh ratio $\leq \lambda'_k$.
All vectors of the subspace $S$ generated by the last $n-k+1$
eigenvectors of $A$
have a Rayleigh ratio $\geq \lambda_k$.
But the
subspaces $S$ and $S'$ need to have an intersection of at least dimension 1,
because $dim(S)+dim(S')=n-k+1+k = n+1$. We obtained a contradiction: the Rayleigh ratio over this
intersection needs to be both $\geq \lambda_k$ and $\leq
\lambda'_{k}$. An analogous reversed argument could be used to show that
$\lambda'_k\leq \lambda_{k+1}$.
Different proofs can be found by searching key words ``Interlacing eigenvalues''
on the internet, the one from
the David Williamson's course
(\url{http://people.orie.cornell.edu/dpw/orie6334/lecture4.pdf}) is the most
related to the above ideas.
Another proof can be found at least in
[Ikramov H., Recueil de probl\` emes d'alg\' ebre lineaire, publisher 
MIR - Moscou, 1977], exercise 7.4.35.}%
\begin{subequations}
\begin{align}
\lambda_1&\leq \lambda'_1\leq \lambda_2\label{eqIntertwined1}\\
\lambda_{n-1}&\leq \lambda'_{n-1}\leq \lambda_n\label{eqIntertwined2}
\end{align}
\end{subequations}
\end{lemma}
\begin{proof}

We have already proved in Lemma \ref{lemmaIntertwinedMinorEigenvals} that $\lambda'_1\leq
\lambda_2$. To show that $\lambda_1\leq \lambda'_1$, let us consider
the unitary eigenvector $\v'$ of $\lambda'_1$ in $A'$ and notice
$\v' \sprod A'\v'= \lambda'_1$. If $A'$ is obtained from $A$ by removing row $i$ and column $i$, we
can say $\v'$ is obtained by removing position $i$ from an $\v\in\R^n$ such
that $v_i=0$. We can write $\lambda'_1=\v' \sprod A'\v'=
\v \sprod A\v$. We recall Lemma~\ref{lemmaMinRayleigh} that states
that $\lambda_1$ is the minimizer of the 
Rayleigh ratio, and so, we obtain:
$\lambda_1= \min\limits_{|\x|=1} \x \sprod A\x \leq \v \sprod A\v=\lambda'_1$.

To prove \eqref{eqIntertwined2} it is enough to apply \eqref{eqIntertwined1}
on matrices $-A$ and $-A'$.
\end{proof}

\subsection{Cholesky decomposition of \textit{semidefinite} positive matrices}
We provide two proofs. The first one is much shorter, but the
factorization arises somehow out of the blue. The second one is longer
and a bit more general, discussing along the way several 
properties that
are generally useful (Prop.~\ref{propNullMinorNullDet} and
\ref{propRankRectangular}), leading the reader to a deeper insight into the factorization.

\subsubsection{A short proof using the square root and the QR decompositions}
\paragraph{\label{secsquareroot}Square root factorisations of SDP matrices}

Given matrix $A\succeq\zeros$, we are looking for matrices $K$ such 
that $KK=A$. We apply the 
eigendecomposition 
\eqref{firstEqDecomp} and write $A = U \Lambda U^\top $, where 
$\Lambda=\texttt{diag}(\lambda_1, \lambda_2, \dots \lambda_n)$ contains the non-negative
eigenvalues of $A$. 
Take $D=\texttt{diag}\left( \sqrt{\lambda_1},  \sqrt{\lambda_2},\ldots \sqrt{\lambda_n}\right)$
and consider $K= U D U^\top$.
One can easily check
that this $K$ is a square root of $A$:
 $KK={U D U^\top} {U D U^\top}= UDD U^\top=
U\Lambda U^\top = A$, where we used $U^\top U=I_n$ from \eqref{eqOrthoNormal}.
This $K$ is called the
principal square root of $A$ and it is SDP, because it is
similar and congruent to $\texttt{diag}
\left( \sqrt{\lambda_1},  \sqrt{\lambda_2},\ldots \sqrt{\lambda_n}\right)$.
It is the only SDP square root of A (Appendix~\ref{appSquareRoot}).

\begin{remark}\label{remarkSymmetricSquareRoot}
There are \texttt{multiple symmetric} matrices $K\in\R^{n \times n}$ such that $A=KK$ that can be found
by taking $K=UDU^\top$ for any
\texttt{symmetric}
$D$ such that $DD=\Lambda$. Examples of such symmetric matrices $D$ can be found by taking 
$D=\texttt{diag}\left(\pm \sqrt{\lambda_1}, \pm \sqrt{\lambda_2},\ldots \pm\sqrt{\lambda_n}\right)$.
\end{remark}

Non-symmetric matrices $K$ such that $A=KK$ do exist; in fact, they can be
infinitely many, because there can be infinitely many matrices $D\in\R^{n\times n}$ such that
$DD=\texttt{diag}(\lambda_1, \lambda_2, \dots \lambda_n)$. To see this, notice that 
$I_2=
\left[\begin{smallmatrix}
\frac st & \frac rt\\
\frac rt & -\frac st
\end{smallmatrix}\right]
\left[\begin{smallmatrix}
\frac st & \frac rt\\
\frac rt & -\frac st
\end{smallmatrix}\right]$
for infinitely many Pythagorean triples $(r,s,t)$ such that $r^2+s^2=t^2$. If we have
$\lambda_1=\lambda_2$, the above $2\times 2 $ construction can be extended to a $n\times n$ matrix in which
the leading principal minor of size $2\times 2$ is given by this $2\times 2$ construction.

As a side remark, we can prove that $rank(A)=rank(K)$ if $K$ is symmetric by
applying Prop.~\ref{propranktransprod} on
$A=K^\top K$. This is no longer true with non-symmetric
matrices, \eg, $\zeros = 
\left[\begin{smallmatrix} 1 & -1 \\ 1&-1 \end{smallmatrix}\right]
\left[\begin{smallmatrix} 1 & -1 \\ 1&-1 \end{smallmatrix}\right]
$ has rank zero while the square root factor has rank 1.

\paragraph{The Cholesky factorization proof}

\begin{proposition} \label{propCholeskyViaQR}(Cholesky factorization of positive \textit{semidefinite} matrices) 
Any real SDP matrix $A$ can be factorized as
$A=LL^\top$
where $L$ is a lower triangular matrix with non-negative diagonal elements.
This proposition is slightly weaker than Prop~\ref{propCholeskySDP}
where we will show (using a longer proof)
that there always exists a factorization in which the value of $L_{nn}$ depends
on the ranks
of $A$ and of the leading principal minor of size $n-1$ of $A$.
\end{proposition}
\begin{proof}
~\\
\noindent \begin{minipage}{0.6\linewidth}
We know there exists a symmetric $K\in\R^{n\times n}$ such that $A=KK$
using Remark~\ref{remarkSymmetricSquareRoot} above. We now apply the QR
decomposition on $K$
and write $K=QR$ where $Q\in\R^{n\times p}$ satisfies $Q^\top Q=I_p$ and $R\in \R^{p\times n}$ is upper 
triangular for some $p\leq n$---see the proof in Prop.~\ref{propQR}. 
 We can develop:
\end{minipage}
~
\scalebox{0.8}{
\begin{minipage}{0.48\linewidth}
\vspace{-1.5em}
$$n\left\{
\begin{bmatrix}
~ &~ &~ & ~ & ~ \\
~ &~ &~ & ~ & ~ \\

\multicolumn{5}{c}{\smash{\raisebox{-0.1\normalbaselineskip}{\scalebox{1.8}{$K$}}}}\\
~ &~ &~ & ~ & ~ \\
~ &~ &~ & ~ & ~ \\
\end{bmatrix}
\right.
=
\underbrace{
\begin{bmatrix}
~ & ~ & ~ \\
~ & ~ & ~ \\
\multicolumn{3}{c}{\smash{\raisebox{-0.1\normalbaselineskip}{\scalebox{1.8}{$Q$}}}}\\
~ & ~ & ~ \\
~ & ~ & ~ \\
\end{bmatrix}}_{p}
\begin{bmatrix}
{\tikz\coordinate(lambda1);} ~ &~ &~ & ~ & ~ \\
\multicolumn{5}{c}{\smash{\raisebox{-0.1\normalbaselineskip}{\scalebox{1.8}{$R$}}}}\\
{\tikz\coordinate(lambda0);} ~
&\hspace{-0.8em}\raisebox{0.3\normalbaselineskip}{\scalebox{1.3}{$\zeros$}}
&
\hspace{0.4em}
{\tikz\coordinate(lambda2);} 
\hspace{0.4em}
& ~ & ~ \\
\end{bmatrix}
$$

\begin{tikzpicture}[overlay]
\draw[-,darkgray] (lambda0.south west) to (lambda1.south west);
\draw[-,darkgray] (lambda0.south west) to (lambda2.south west);
\draw[-,darkgray] (lambda1.south west) to (lambda2.south west);
\end{tikzpicture}
\vspace{-2em}
\end{minipage}
}

\vspace{0.5em}

$$A=KK=K^\top K=(QR)^\top QR = R^\top Q^\top Q R = R^\top I_p R =R^\top R,$$
which is very close to a Cholesky decomposition, because $R^\top$ is lower triangular.
However, $R^\top$ has only $p\leq n$ columns, but we can extended to a $n\times n$ matrix
by adding $n-p$ zero columns. This way, $R^\top$ is transformed into a lower triangular
\textit{square} matrix ${L}$ such that
${L}~{L}^\top = R^\top R=A$. The last point to address is the fact that the
$QR$ decomposition from Prop.~\ref{propQR} does not state that the diagonal elements of $R$ are non-negative. There might
exist multiple $i\in[1..p]$ such that ${L}_{ii}=R_{ii}<0$. We can overcome
this with a simple trick. The product $LL^\top$ does not change if
we negate all columns $i$ of $L$ satisfying ${L}_{ii}<0$, because
$A_{jj'}$ is the dot product of rows $j$ and $j'$ of $L$ which
does not change if both rows $j$ and $j'$ negate some column(s) $i$.
This leads to a factorisation in which the factors  have a non-negative diagonal. \footnote{I first found this proof in an answer of user \textit{loup blanc}
on the on-line forum
\url{https://math.stackexchange.com/questions/1331451/how-to-prove-cholesky-decomposition-for-positive-semidefinite-matrices}.
It also appear in Corollary 7.2.9 of the book ``Matrix Analysis'' by Roger Horn and Charles Johnson,
second edition, Cambridge University Press, 2013.}
\end{proof}

\subsubsection{A longer Cholesky proof providing more insight into the
properties of SDP matrices}

\begin{proposition} \label{propNullMinorNullDet}
If SDP matrix $A\in\R^{n\times n}$ has some null principal minor (\ie, $\exists J\subset[1..n]$ such that
$det([A]_J)=0$), then $A\nsucc\zeros$ and $det(A)=0$.
\end{proposition}
\noindent \textit{Proof 1.}
One of the eigenvalues of $[A]_J$ need to be zero so that
$[A]_J\nsucc\zeros$ and $[A]_J\sprod \x_J\x_J^\top=0$ for some
vector $\x_J$ with $|J|$ components. Construct $\x\in\R^n$ by
keeping the values of $\x_J$ on the positions $J$ of $\x$ and by
filling the rest with zeros. It is not hard to check that
$A\sprod \x\x^\top = [A]_J\sprod \x_J\x_J^\top=0$, so that
then $A\nsucc\zeros$ and $det(A)$ is zero as the product
of the eigenvalues.\qed

\noindent \textit{Proof 2.} Re-order the rows and columns of $A$ so that the minor 
$[A]_J$ becomes a leading principal minor; the Sylvester criterion is
thus violated, which shows $A\nsucc\zeros$\qed.
%
%

\begin{proposition}\label{propRankRectangular}
If we are given an SDP matrix $A\in\R^{n\times n}$ written under the form
$$A=
\begin{bmatrix}
                  &       &       &  b_{1}  \\
      \multicolumn{3}{r}{\smash{\raisebox{-0.5\normalbaselineskip}{\scalebox{1.5}{$[A]_{n-1}$}}}} & b_{2} \\
                  &       &       &      \vdots \\
b_{1}             &  b_2  &\dots  &  b_n  \\
\end{bmatrix},
$$
then $[b_1~b_2~\dots b_{n-1}]$ can be written as a linear combination of the rows of $[A]_{n-1}$.
This combination uses only rows $J$ where $J\subset[1..n-1]$ such that $[A]_J$ (matrix obtained by
selecting rows $J$ and columns $J$) is a non-null principal minor of maximum order (the rank of $[A]_{n-1}$).
\end{proposition}
\noindent We provide two proofs. The first one is much shorter, but the second one provides more insight into
the arrangement of the matrices.

\vspace{0.8em}
\noindent \textit{Proof 1.} Assume $[b_1~b_2~\dots b_{n-1}]$ does not belong to the row image 
(set of linear combinations of the rows) of $[A]_{n-1}$. Using the 
rank-nullity Theorem \ref{thranknullity}, we have $rank([A]_{n-1})+nullity([A]_{n-1})=n-1$.
The dimension of the image $\texttt{img}([A]_{n-1})$ plus the dimension of $\texttt{null}([A]_{n-1})$ 
is equal to $n-1$. Since the two spaces are perpendicular (any $\x_0\in \texttt{null}([A]_{n-1})$
satisfies $[A]^i_{n-1}\x_0=0$ for any row $i$ of $[A]_{n-1}$), the sums between elements of $\texttt{null}([A]_{n-1})$ 
and elements of $\texttt{img}([A]_{n-1})$ cover the whole (transposed) $\R^{n-1}$.
As such, we can write $[b_1~b_2~\dots b_{n-1}]=\b^\top_\texttt{img} + \b^\top_0$, where
$\b^\top_\texttt{img}\in\texttt{img}([A]_{n-1})$ and $\b_0 \in\texttt{null}([A]_{n-1})$
with $\b_0\neq \zeros$.
We can now calculate:
\begin{align*}
[-t\b^\top_0~1]A\begin{bmatrix}-t\b_0 \\ 1 \end{bmatrix} 
              & = t^2\b^\top_0 [A]_{n-1} \b_0 
                  - 2 t (\b^\top_\texttt{img} + \b^\top_0)\b_0 + b_n\\
              & = 0 - 2t \b_0^\top\b_0 + b_n = b_n - 2t |\b_0|^2 \to -\infty
\end{align*}
We obtained a contradiction, the assumption $[b_1~b_2~\dots b_{n-1}]\notin \texttt{img}([A]_{n-1})$
was false. \qed

~\\

\noindent \textit{Proof 2.}
If $[A]_{n-1}$ is non-singular, the conclusion is obvious: $\x^\top [A]_{n-1} = [b_1~b_2~\dots b_{n-1}]$
has solution $\x^\top =  [b_1~b_2~\dots b_{n-1}][A]_{n-1}^{-1}$.

We hereafter consider $[A]_{n-1}$ has rank $r<n-1$. 
Based on Prop.~\ref{propPrincipalMinorTheorem}, $A$ has at least a non-zero {\it principal} minor of order $r$.
Without loss of generality, we permute the rows and columns of $A$ until this non-zero
minor is positioned in the upper-left corner -- this does not change the
determinant or the SDP status (Prop~\ref{propRowColOpers}). 
Let $[A]_r$ be this leading principal minor. Consider the solution $\x$ of the system $\x^\top [A]_r=\b^\top_r$,
where $\b_r$ is $\b$ reduced to positions $[1..r]$. This solution
exists and it has value $\x^\top=\b_r^\top [A]_r^{-1}$.

We will show that $[b_1~b_2~\dots b_{n-1}]$ can be written as a linear
combination of the rows $[1..r]$
of $[A]_{n-1}$, the coefficients of this combination being $\x^\top$. 
Take any $i\in [r+1..n-1]$.
Let us consider the minor
obtained by selecting rows and columns $[1..r]\cup\{i,n\}$ of $A$, see left matrix
below. We subtract from the last row the linear combination of the first $r$ rows defined by $\x$, so as
to cancel the first $r$ positions of the last row, followed by the transposed
operation on columns. We obtain the right matrix below.

$$
\begin{bmatrix}
         &         &       &           & a_{i,1}&b_{1}\\
      \multicolumn{4}{c}{\smash{\raisebox{-1\normalbaselineskip}{\scalebox{1.5}{$[A]_r$}}}} 
                                       & a_{i,2}& b_{2}\\
         &         &       &           & \vdots   & \vdots \\
         &         &       &           & a_{i,r}& b_{r}\\
a_{i,1}& a_{i,2}& \dots & a_{i,r}& a_{i,i}  & b_i    \\
b_{1}  & b_{2}  & \dots & b_{r}  & b_i      & b_n    \\
\end{bmatrix}
\xrightarrow[\textnormal{that do not change the determinant}]{\textnormal{row and column operations}}
\begin{bmatrix}
         &         &       &           & a_{i,1}  &   0    \\
      \multicolumn{4}{c}{\smash{\raisebox{-1\normalbaselineskip}{\scalebox{1.5}{$[A]_r$}}}} 
                                       & a_{i,2}  &   0    \\
         &         &       &           & \vdots   & \vdots \\
         &         &       &           & a_{i,r}  &    0   \\
a_{i,}   & a_{i,2} & \dots & a_{i,r}   & a_{i,i}  & \widehat{b_i}    \\
    0    &    0    & \dots &    0      & \widehat{b_i}     & \widehat{b_n}    \\
\end{bmatrix},
$$
where $\widehat{b_i}=b_i - \sum_{j \in[1..r]} x_j a_{ij}$.
The determinant of this right matrix can be calculated (see the Leibniz formula or the Laplace formula for
determinants) as
follows: 
$$\widehat{b_n}\cdot \det\left([A]_{[1..r]\cup\{i\}}\right)-\widehat{b_i}^2\cdot
\det\left([A]_r\right),$$
where $[A]_{[1..r]\cup\{i\}}$ is the $(r+1)\times(r+1)$ upper left minor of above matrix. This
$(r+1)\times (r+1)$ minor
needs to be null because $[A]_{n-1}$ has rang $r$ and
$[1..r]\cup\{i\}\subseteq[1..n-1]$. On
the other hand we have $det([A]_r)>0$, so that 
the above determinant simplifies to 
$-\widehat{b_i}^2\cdot \det\left([A]_r\right)$. But this is the determinant of a
matrix obtained from a minor of $A$ after performing linear operations with rows
and columns of $A$. Since $A$ is SDP, this determinant needs to be non-negative,
and so, we need to have $\widehat{b_i}=0$, \ie, $b_i - \sum_{j \in[1..r]} x_j a_{ij}=0$,
meaning that $b_i$ is a linear combination (defined by $\x$) of the first $r$ rows of column $i$
(recall $a_{ij}=a_{ji}$).
Recall $i$ was chosen at random from $[r+1..n-1]$,  so that 
actually all elements $b_i$ of $[b_1~b_2~\dots b_{n-1}]$ can be written as a linear
combination (defined by $\x$) of the first $r$ rows.
\qed 

\begin{corollary}\label{corolZeros} 
If $A\succeq \zeros$ and $A_{jj}=0$ for some $j\in[1..n]$, then the row and column $j$ contain only zeros.
\end{corollary}
\begin{corollary}
If $A\succeq \zeros$ and $a_{11}=a_{12}=a_{21}=a_{22}$, then $a_{i1}=a_{i2}~\forall i \in[1..n]$.
\end{corollary}

\begin{proposition} \label{propCholeskySDP}(Cholesky factorization of positive \textit{semidefinite} matrices) 
A real symmetric matrix $A$ is positive semidefinite if and only if it can be factorized as:
\begin{equation}
A=RR^\top=
\begin{bmatrix}
r_{11}   &  0     &   0  & \dots &   0   \\
r_{21}   & r_{22} &   0  & \dots &   0   \\
r_{31}   & r_{32} &r_{33}& \dots &   0   \\
\vdots   & \vdots &\vdots&\ddots & \vdots\\
r_{n1}   & r_{n2} &r_{n3}& \dots &r_{nn}  \\
\end{bmatrix}
\begin{bmatrix}
r_{11}   &r_{21}  & r_{31}& \dots&r_{n1}  \\
0        &r_{22}  & r_{32}& \dots&r_{n2}  \\
0        & 0      & r_{33}& \dots&r_{n3}  \\
\vdots   & \vdots &\vdots&\ddots & \vdots \\
0        &  0     &  0    &\dots &r_{nn}  \\
\end{bmatrix},
\end{equation}
where the diagonal terms are non-negative. 

The factorization is not always unique. There exists a factorization that has $r_{nn}=0$ only if $rank(A)=rank([A]_{n-1})$, where
$[A]_{n-1}$ is the leading principal minor of size $(n-1)\times (n-1)$.
\end{proposition}
\begin{proof}~\\
$\Longleftarrow$\\
Take any non-zero $\x\in\R^n-\{\zeros_n\}$ and observe that $\x^\top A \x = \x^\top R R^\top \x =
(R^\top \x)^\top (R^\top \x)\geq 0$. This is enough to prove that $A$ is SDP.\\
$\Longrightarrow$\\
We proceed by induction. The implication is obviously true for $n=1$. We suppose that 
there exists a factorization:
$$[A]_{n-1}=
[R]_{n-1}[R]_{n-1}^\top
=
\begin{bmatrix}
r_{11}   &  0      & \dots &   0       \\
r_{21}   & r_{22}  & \dots &   0       \\
\vdots   & \vdots  &\ddots & \vdots    \\
r_{n-1,1}&r_{n-1,2}& \dots &r_{n-1,n-1}\\
\end{bmatrix}
\begin{bmatrix}
r_{11}   &r_{21}  & \dots&r_{n-1,1}  \\
0        &r_{22}  & \dots&r_{n-1,2}  \\
\vdots   & \vdots &\ddots & \vdots \\
0        &  0     &\dots &r_{n-1,n-1}  \\
\end{bmatrix},
$$
where $r_{11},~r_{22},~\dots r_{n-1,n-1}\geq 0$. We will prove that this decomposition can be extended
to a $n\times n$ decomposition for matrix $A$. The values $r_{1n},~r_{2n},\dots
r_{n-1,n}$ (the elements above the diagonal on the last column of $R$)
are set to zero by definition.

The main difficulty is to determine the last row
$\r^\top=\left[\begin{smallmatrix}\r^\top_{n-1} & r_n\end{smallmatrix}\right]$
 of $R$.
We write
$$
A=
\begin{bmatrix}
      \multicolumn{3}{c}{\smash{\raisebox{-0.5\normalbaselineskip}{\scalebox{1.1}{$[A]_{n-1}$}}}} & \multirow{2}{*}{\scalebox{1.2}{$\b_{n-1}$}}  \\
                  &       &                                                                       &   \\
      \multicolumn{3}{c}{\b_{n-1}^\top}                                                       & b_n  \\
\end{bmatrix}
$$

Using Prop.~\ref{propRankRectangular}, there exists $\x\in\R^{n-1}$ such that
$\b_{n-1}^\top = \x^\top [A]_{n-1}$. This leads to
$$
\begin{bmatrix}
      \multicolumn{3}{c}{\smash{\raisebox{-0.5\normalbaselineskip}{\scalebox{1.1}{$[A]_{n-1}$}}}}\\
                  &       &       \\
      \multicolumn{3}{c}{\b^\top_{n-1}}  \\
\end{bmatrix}
=
\begin{bmatrix}
      \multicolumn{3}{c}{\smash{\raisebox{-0.5\normalbaselineskip}{\scalebox{1.1}{$I_{n-1}$}}}}\\
                  &       &       \\
      \multicolumn{3}{c}{\x^\top}  \\
\end{bmatrix}
\scalebox{1.2}{$[A]_{n-1}$}
=
\begin{bmatrix}
      \multicolumn{3}{c}{\smash{\raisebox{-0.5\normalbaselineskip}{\scalebox{1.1}{$I_{n-1}$}}}}\\
                  &       &       \\
      \multicolumn{3}{c}{\x^\top}  \\
\end{bmatrix}
\scalebox{1.1}{$[R]_{n-1}$}
\scalebox{1.1}{$[R]_{n-1}^\top$}
=
\begin{bmatrix}
      \multicolumn{3}{c}{\smash{\raisebox{-0.5\normalbaselineskip}{\scalebox{1.1}{$[R]_{n-1}$}}}}\\
                  &       &       \\
      \multicolumn{3}{c}{\x^\top [R]_{n-1}}  \\
\end{bmatrix}
\scalebox{1.1}{$[R]_{n-1}^\top$}.
$$
Using the notational shorthand 
$
\r_{n-1}^\top
=
\x^\top [R]_{n-1}
$, we can write:
$$
\begin{bmatrix}
      \multicolumn{3}{c}{\smash{\raisebox{-0.5\normalbaselineskip}{\scalebox{1.1}{$[A]_{n-1}$}}}}\\
                  &       &       \\
      \multicolumn{3}{c}{\b^\top_{n-1}}  \\
\end{bmatrix}
=
\begin{bmatrix}
      \multicolumn{3}{c}{\smash{\raisebox{-0.5\normalbaselineskip}{\scalebox{1.1}{$[R]_{n-1}$}}}}\\
                  &       &       \\
      \multicolumn{3}{c}{\r_{n-1}^\top}  \\
\end{bmatrix}
\scalebox{1.1}{$[R]_{n-1}^\top$}
.
$$
In fact, we can already fix the first $n-1$ positions of the last row
(of $R$) to $\r^\top=\left[\begin{smallmatrix}\r^\top_{n-1} & r_n\end{smallmatrix}\right]$.
Then, let us extend the above equality to work with $n\times n$ matrices, by adding: (i) a column with
zeros to the left factor of the product and (ii) a row with zeros and a column $\r_{n-1}$ to the right factor.
We obtain
$$
\begin{bmatrix}
      \multicolumn{3}{c}{\smash{\raisebox{-0.5\normalbaselineskip}{\scalebox{1.1}{$[R]_{n-1}$}}}} & \multirow{2}{*}{\scalebox{1}[1.4]{$\zeros$}}  \\
                  &       &                                                                       &   \\
      \multicolumn{3}{c}{\r_{n-1}^\top}                                                       & 0  \\
\end{bmatrix}
\begin{bmatrix}
      \multicolumn{3}{c}{\smash{\raisebox{-0.5\normalbaselineskip}{\scalebox{1.1}{$[R]_{n-1}$}}}} & \multirow{2}{*}{\scalebox{1}[1.4]{$\zeros$}}  \\
                  &       &                                                                       &   \\
      \multicolumn{3}{c}{\r_{n-1}^\top}                                                       & 0  \\
\end{bmatrix}^\top
=
\begin{bmatrix}
      \multicolumn{3}{c}{\smash{\raisebox{-0.5\normalbaselineskip}{\scalebox{1.1}{$[A]_{n-1}$}}}} & \multirow{2}{*}{\scalebox{1}[1.2]{$\b_{n-1}$}}  \\
                  &       &                                                                       &   \\
      \multicolumn{3}{c}{\b_{n-1}^\top}                                                       & z  \\
\end{bmatrix} = A_z,
$$
where $z=\r_{n-1}^\top \r_{n-1}$.
The matrix at right has the same rank as the matrices at left (use Prop.~\ref{propranktransprod}), that
is precisely the rank of $[R]_{n-1}$ which is equal to the rank of $[A]_{n-1}$ (by
applying again Prop.~\ref{propranktransprod} on 
$[A]_{n-1}=[R]_{n-1}R_{n-1}^\top$). In short, $A_z$ has the same rank as $[A]_{n-1}$.

Let us now study the minor 
$\left[
\begin{smallmatrix}
\overline{A} & \overline{\b}\\
\overline{\b}^\top& z
\end{smallmatrix}
\right]$ 
of $A_z$
obtained by bordering the largest order non-zero minor $\overline{A}$ of $[A]_{n-1}$ with the last column
and row of $A_z$ (keeping only the positions that constitute the minor).
The determinant of this new minor of $A_z$
has to be zero because $A_z$ has the same rank as $[A]_{n-1}$.
We now compare $A$ to $A_z$; let us
write the bottom-right term $b_n$ of $A$ in the form $b_n=z+\alpha$. 
By replacing $z$ with $z+\alpha$, the above minor evolves to 
$
\left[
\begin{smallmatrix}
\overline{A} & \overline{\b}\\
\overline{\b}^\top& z+\alpha
\end{smallmatrix}
\right]
$ which is a matrix of determinant $\alpha\det\left(\overline{A}\right)$.
Since $A$ is SDP and $\overline{A}$ is a non-zero minor of $A$, this determinant has to be non-negative, and so, 
we have $\alpha\geq 0$.
This enables us to finish the construction by setting
$$
\begin{bmatrix}
      \multicolumn{3}{c}{\smash{\raisebox{-0.5\normalbaselineskip}{\scalebox{1.1}{$[R]_{n-1}$}}}} & \multirow{2}{*}{\scalebox{1}[1.4]{$\zeros$}}  \\
                  &       &                                                                       &   \\
      \multicolumn{3}{c}{\r_{n-1}^\top}                                                       & \sqrt{\alpha}  \\
\end{bmatrix}
\begin{bmatrix}
      \multicolumn{3}{c}{\smash{\raisebox{-0.5\normalbaselineskip}{\scalebox{1.1}{$[R]_{n-1}$}}}} & \multirow{2}{*}{\scalebox{1}[1.4]{$\zeros$}}  \\
                  &       &                                                                       &   \\
      \multicolumn{3}{c}{\r_{n-1}^\top}                                                       & \sqrt{\alpha}  \\
\end{bmatrix}^\top
=
\begin{bmatrix}
      \multicolumn{3}{c}{\smash{\raisebox{-0.5\normalbaselineskip}{\scalebox{1.1}{$[A]_{n-1}$}}}} & \multirow{2}{*}{\scalebox{1}[1.2]{$\b_{n-1}$}}  \\
                  &       &                                                                       &   \\
      \multicolumn{3}{c}{\b_{n-1}^\top}                                                       & z+\alpha  \\
\end{bmatrix} = A,
$$

Finally, if 
\texttt{rank}$(A) = $ \texttt{rank}$([A]_{n-1})$
then we can say 
\texttt{rank}$(A)=$
\texttt{rank}$([A]_{n-1})=$
\texttt{rank}$([R]_{n-1})=$
\texttt{rank}$(R)$, and so, 
we need to have $r_{nn}=0$ because any
$r_{nn}>0$ would make
$\alpha=r_{nn}^2>0$, leading to
\texttt{rank}$(R)>$
\texttt{rank}$([R]_{n-1})$.
On the other hand, if 
\texttt{rank}$(A) > $ \texttt{rank}$([A]_{n-1})$,
then we have $r_{nn}=\sqrt{\alpha}> 0$.
This justifies the last statement of
the proposition.
\end{proof}

\begin{corollary} If a SDP matrix $A$ is not positive definite, the Cholesky factorization may or
may no be unique.
\end{corollary} 
\begin{proof} The following factorization is unique:
$$
\begin{bmatrix}
1  &  1 \\
1  &  1 \\
\end{bmatrix}
=
\begin{bmatrix}
1  &  0 \\
1  &  0 \\
\end{bmatrix}
\begin{bmatrix}
1  &  1 \\
0  &  0 \\
\end{bmatrix}
$$
because it requires $1=r^2_1$ and $1=r_1r_2+r_2^2$, imposing $r_1=1$ and $r_2=0$.

The following factorization is not unique:
$$
\begin{bmatrix}
1  &  1 & 2 \\
1  &  1 & 2 \\
2  &  2 & 13 \\
\end{bmatrix}
=
\begin{bmatrix}
1  &  0 & 0 \\
1  &  0 & 0 \\
2  &  0 & 3 \\
\end{bmatrix}
\begin{bmatrix}
1  &  1  & 2 \\
0  &  0  & 0 \\
0  &  0  & 3 \\
\end{bmatrix}
$$
This is the decomposition that the proof of the above theorem would
construct by taking $\x^\top=\left[\begin{smallmatrix}2&0\end{smallmatrix}\right]$ on the
induction basis of the 
previous $2\times 2$ decomposition of
$[A]_{n-1}=\left[\begin{smallmatrix} 1 & 1\\ 1 & 1 \end{smallmatrix}\right]$.
Precisely, this proof calculates
$\left[\begin{smallmatrix}r_{31}&r_{32}\end{smallmatrix}\right]
=\x^\top [R]_{n-1}=
\left[\begin{smallmatrix}2&0\end{smallmatrix}\right]
\left[\begin{smallmatrix} 1 & 0\\ 1 & 0 \end{smallmatrix}\right]=
\left[\begin{smallmatrix} 2 & 0    \end{smallmatrix}\right]$.
 However, more generally, the
coefficients $r_{32}$ and $r_{33}$ are active only in the equation of $a_{33}$, \ie,
$13=2^2+r_{32}^2+r_{33}^2$. We could take $r_{32}=1$ and $r_{33}=\sqrt{8}$. Also, we could take
$r_{32}=3$ and $r_{33}=0$, leading to a different factorization:
$$
\begin{bmatrix}
1  &  1 & 2 \\
1  &  1 & 2 \\
2  &  2 & 13 \\
\end{bmatrix}
=
\begin{bmatrix}
1  &  0 & 0 \\
1  &  0 & 0 \\
2  &  3 & 0 \\
\end{bmatrix}
\begin{bmatrix}
1  &  1  & 2 \\
0  &  0  & 3 \\
0  &  0  & 0 \\
\end{bmatrix}
$$
Unlike the previous factorization constructed by the proof of the above theorem, this
decomposition has a null bottom-right diagonal term, although $rank(A)>rank([A]_2)$. However, notice
that $rank(R)=rank([R]_{n-1})+1$, where $[R]_{n-1}$ is the leading principal minor of size $2\times
2$. Since $rank(RR^\top)=rank(R)$ (use Prop.~\ref{propranktransprod}), the rank of $A$ is equal to
the rank of $R$ and the rank of $[A]_{n-1}$ equals the rank of $[R]_{n-1}$.
\end{proof}

\subsection{Any $A\succeq ${\bf 0} has infinitely many factorizations $A=VV^\top$ related by
rotations and reflections}
{\def\RR{\pazocal{R}}
\begin{corollary} \label{corSdpToVectors}Any SDP matrix $A\in\R^{n\times n}$
can be factorized in a form $A=VV^\top$.
Generating such $V$ is simply equivalent to 
taking $n$ vectors $\v_1^\top$, $\v_2^\top$, $\dots$ $\v_n^\top$ (the rows of $V$) 
such that $A_{ij}=\v_i^\top \v_j=\v_i\sprod \v_j~\forall i,j\in[1..n]$.
There are infinitely many such matrices $V$ 
(or vectors $\v_1$, $\v_2$, $\dots$ $\v_n$) for a fixed $A$.
We can say that any SDP matrix $A$ can be constructed by choosing a set of $n$ vectors of $\R^n$.
\end{corollary}
\begin{proof}
For a given $A\succeq 0$, we have actually already presented three ways of computing a factor $V$ such
that $A=VV^\top$. We recall these three ways at points (1)-(3) below. At point (4), 
we show that from any such factor $V$ we can further generate infinitely many
other factors.
\begin{itemize}
\item[(1)] Use the above Cholesky decomposition of
SDP matrices to write
$A=RR^\top$ and take $V=R$ as needed. 
\item[(2)] Use the eigendecomposition
\eqref{firstEqDecomp} to write $A = U \Lambda U^\top $, where 
$\Lambda=\texttt{diag}(\lambda_1, \lambda_2, \dots \lambda_n)$. Since $\lambda_i\geq
0~\forall i\in[1..n]$, we can 
define real matrix $\sqrt{\Lambda}=\texttt{diag}(\sqrt{\lambda_1},
\sqrt{\lambda_2}, \dots \sqrt{\lambda_n})$ and
write $A=U\sqrt{\Lambda}\sqrt{\Lambda}U^\top=
(U\sqrt{\Lambda})(U\sqrt{\Lambda})^\top=VV^\top$ with $V=U\sqrt{\Lambda}$.
\item[(3)] Use one of the multiple square root decompositions $A=KK$ with symmetric $K$ from 
Remark \ref{remarkSymmetricSquareRoot}. This gives 
$V=K$ and $V^\top =K$. 
\item[(4)] 
We can generate infinitely many more decomposition from any $V$ determined
as above.
For this, let us consider any unitary orthogonal matrix ${\RR}$, \ie, a matrix
such that $\RR^\top \RR=I_n$. There are infinitely-many such matrices $\RR$, 
each one of them representing a composition of rotation and reflection operators.\footnote{%
Using
${\RR}~{\RR}^\top =I_n$,
notice $\x^\top \y=\x^\top
{\RR}~{\RR}^\top \y=(\x^\top {\RR})(\y^\top
{\RR})^\top$, \ie, the operator that maps $\x^\top \to \x^\top {\RR}$ preserves
the angles, and so, it needs to be a composition of rotations and reflections.}
Now
check that $V_{\text{new}}V_{\text{new}}^\top=(V{\RR})(V{\RR})^\top=V{\RR}{\RR}^\top V^\top=VV^\top=A$.
\end{itemize}
\end{proof}

As a side remark, the factorizations mentioned at above points (3) and (4) are related.
In fact, Remark
\ref{remarkSymmetricSquareRoot} used at (3) constructs the factorisation 
by applying a particular case of the technique used at (4) on $A=VV^\top=U\sqrt{\Lambda}\sqrt{\Lambda}U^\top$.
More exactly, recall that symmetric $K$ from Remark \ref{remarkSymmetricSquareRoot} 
was generated by setting $K=UDU^\top$ for any
\texttt{symmetric}
$D$ such that $DD=\Lambda$.
Such a $D$ could be found by taking $D=\sqrt{\Lambda}\RR$, where
$\RR=\texttt{diag}(e_1,~e_2,\dots,e_n)$ with $e_i=\pm 1~\forall i\in[1..n]$;
this $\RR$ can actually be seen as a composition of reflexion operators like matrix $\RR$ at point (4) above.
We can write 
$A=VV^\top=U\sqrt{\Lambda}\sqrt{\Lambda}U^\top=U\sqrt{\Lambda}\RR \RR^\top\sqrt{\Lambda}U^\top
=UDD^\top U^\top
=UDDU^\top=UDU^\top UDU^\top=KK$.

\begin{remark} Given two factorizations $A=VV^\top = UU^\top$,
it is always possible to write
$V=U\RR$, where 
$\RR$ satisfies $\RR\RR^\top=I_n$, \ie, 
$\RR$ represents a composition of rotation and reflection operators.
The rows of $V$ can thus be obtained from the rows of $U$ by applying
a composition of rotations and reflections.
\end{remark}
\begin{proof}

{
\def\lf{\lfloor}
\def\rf{\rfloor}

Let $r$ be the rank of $V$, $U$ and $A$ (recall Prop~\ref{propranktransprod})
and $J$ a set of rows such that $\lf V \rf _J$ has rank $|J|=r$, where
the operator $\lf \cdot \rf _J$ represents the given matrix reduced to rows $J$.
Any row $\v_i$ of $V$ outside $J$ (\ie, such that $i\in[1..n]\setminus J$) can be
written as linear combination of rows $J$ using coefficients $\x\in\R^r$: 
$\v_i=\x^\top \lf V \rf _J$

Let's examine row $i$ of the product $A=VV^\top=UU^\top$ for any fixed
$i\in[1..n]\setminus J$.
Replacing above $\v_i=\x^\top \lf V\rf_J$ in $\a_i=\v_iV^\top$, we obtain
$\a_i=\x^\top \lf V\rf_J V^\top$. But now notice that $\lf V\rf_J V^\top$ actually
represents the rows $J$ of matrix $A$; we can thus replace $|V|_JV^\top = \lf A\rf_J$ and obtain 
$\a_i=\x^\top \lf A\rf _J$. 

We will now prove that $\u_i=\x^\top \lf U \rf_J$;
we introduce notational shortcuts $\u_i=\overline{\u_i}+\z = \x^\top \lf U \rf_J+\z$ and we will show $\z=\zeros$.
Let us first calculate $\overline{\u_i}U^\top= \x^\top \lf U \rf_J U^\top = \x \lf A\rf_J=\a_i$. Using
$A=UU^\top$, we also have $\u_iU^\top = \a_i$ which means
that $\left(\u_i - \overline{\u_i}\right)U^\top= \z U^\top = \zeros^\top$.
Taking row $i$, we obtain $\z\u_i= 0$, or $\z\left(\overline{\u_i}+\z\right)^\top=0$. We now
simply replace $\overline{\u_i}=\x^\top \lf U \rf_J $ and obtain
$\z\left(\x^\top \lf U \rf_J +\z\right)^\top=0
\implies
\z \lf U \rf_J^\top \x +\z\z^\top=0$. 
But $\z \lf U \rf_J^\top$ represents the columns $J$ of
$\z U^\top = \zeros^\top$, \ie, $\z \lf U \rf_J^\top=\zeros^\top$ is
a row vector (of length $|J|=r$) that only contains zeros. This leads to $\z\z^\top =0$ and this shows
that $\z=0$, leading to 
$\u_i=\x^\top \lf U \rf_J$. We actually obtained that all rows $[1..n]-J$ of
the equality $A=VV^\top=UU^\top$ represent merely linear combinations of the rows $J$
of this equality.

This equality is thus a (linear combination) consequence of $\lf A\rf _J=\lf V \rf_J V^\top = \lf U \rf_J U^\top$,
where 
$\lf V \rf_J $ and $ \lf U \rf_J $ have full rank $|J|=r$.
We now replace all rows $[1..n]-J$ of $V$ with $n-r$ unit orthogonal vectors rows that are 
also perpendicular to $\lf V\rf_J^\top$ and obtain matrix $\overline{V}$; this is possible because these unit orthogonal vectors 
are simply a basis for the null space of $\lf V\rf_J^\top$ which has dimension $n-r$ by virtue
of the rank-nullity theorem. 
One can check that the product $\overline{V}~\overline{V}^\top$ is a matrix
$\overline{A}$ such that $\overline{a}_{j,j'}=a_{j,j'}~\forall j,j'\in J$,
$\overline{a}_{i,i}=1~\forall i\in[1..n]-J$ and $\overline{a}_{i,j}=\overline{a}_{j,i}=0~\forall
i\in[1..n]-J,j\in J$.
We perform a similar operation on $U$ and obtain matrix $\overline{U}$ (filling rows $[1..n]-J$
with a different basis than for $V$) and one can check that similarly we have 
$$
\overline{U}\,\overline{U}^\top=
\overline{V}\,\overline{V}^\top=
\overline{A}$$

We can now use that that
$\overline{U},
\overline{V}$ and
$\overline{A}$ are non-singular and invertible. This ensures that there exists $\RR$ such that
$\overline{V}=\overline{U}\RR$. We obtain
$\overline{U}\,\overline{U}^\top=
\overline{U}\RR
\RR^\top
\overline{U}^\top$. Multiplying at left with $\overline{U}^{-1}$
and at right with ${\overline{U}^\top}^{-1}$, we obtain that
$\RR \RR^\top = I_n$. Now recall that 
$\overline{V}=\overline{U}\RR$ contains rows $J$ inherited from
$V$ and resp.~$U$, so that we also have 
$\lf V\rf_J=\lf U \rf_J \RR$. This can easily be extended to
$V=U\RR$ because each missing row (\ie, each $i\in[1..n]-J$) is a linear
combination of rows $J$ (\ie, 
$\v_i=\x\lf V\rf_J=\x\lf U \rf_J\RR=\u_i\RR$).
}
\end{proof}
}


\subsection{\label{secSDPHessian}Convex functions have an SDP Hessian assuming
the Hessian is symmetric}
Notice proposition below requires the Hessian matrix to be symmetric. 
This condition was omitted from certain texts
\vandenbergheWrong
but
we address it in our work.
Convex functions with asymmetric non-SDP Hessians do exist, see
Example~\ref{exNonSymHessian} in Appendix~\ref{appRelated}.  For such
cases, the convexity condition should 
actually 
evolve from
$\nnabla^2 f(\y)\succeq \zeros~\forall \y\in\R^n$ (\ie, SDP Hessian)
to
$\nnabla^2 f(\y) + \nnabla^2 f(\y)^\top \succeq \zeros~\forall \y\in\R^n$.
\begin{proposition}\label{propHessian}
A twice differentiable function $f:\R^n \to \R$ with a symmetric Hessian
matrix $\nnabla^2 f(\y)=H_{\y}$ defined by terms $h^\y_{ji}=\dfrac{\partial ^2f}{\partial x_j\partial
x_i}\big(\y\big)$ for all $\y \in \R^n$ and $i,j\in[1..n]$ is convex
if and only if $H_{\y}\succeq\zeros~\forall \y\in\R^n$,
\ie, if and only if the Hessian is SDP in all points.
\end{proposition}
\begin{proof}
Let us consider any $\y\in \R^n$. We take any direction $\v\in \R^n$ and define $g:\R\to
\R$ via
$$g(t)=f(\y+t\v)$$
Using the chain rule to the gradient,%
\footnote{
If you are unfamiliar with gradients, consider that
$\nnabla f(\y)$ is the hyperplane tangent to the function
graph at $\y$. Using notational shortcut
$\nnabla f(\y)=\left[\begin{smallmatrix} \nnabla_1&\nnabla_2&\dots \nnabla_n\end{smallmatrix}\right]$,
the function value (of this hyperplane) increases by $\epsilon\nnabla_i $ when one performs a step of length $\epsilon$ along direction $x_i$.
Let us further study this hyperplane: its value increases by $\nnabla_i$ when one performs a unit step (of length 1) along direction
$x_i$ from any starting point. What happens if one moves along some other direction $\v$? Answer: this is equivalent to advancing
a step of $v_1$ along $\x_1$, followed by a step of $v_2$ along $\x_2$, etc, leading to a total increase of
$\nnabla_1 v_1+\nnabla_2 v_2+\dots \nnabla_n v_n=
\left[\begin{smallmatrix} \nnabla_1&\nnabla_2&\dots \nnabla_n\end{smallmatrix}\right]\v$.
As a side remark, we can also see 
$\left[\begin{smallmatrix} \nnabla_1&\nnabla_2&\dots \nnabla_n\end{smallmatrix}\right]$ as a gradient direction (vector).
The increase (of the hyperplane) when one moves along some direction $\v$ is given by the scalar product
between $\v$ and this gradient vector, equivalent to the projection of $\v$ along this
gradient. Among all vectors of the unit sphere, the 
vector/direction that makes the hyperplane value change the most
(in absolute value)
is the one that is 
collinear to the gradient direction.
}
we obtain
\begin{equation}\label{firstder}
g'(t) = \underbrace{\nnabla f\left(\y+t\v\right)}_{\textnormal{row vector}} \v=\sum_{i=1}^n\frac{\partial f}{\partial x_i}(\y+t\v)v_i
\end{equation}
We derivate again in $t$ to obtain $g''(t)$. For this, we observe that the derivative in $t$ of any term
$\frac{\partial f}{\partial x_i}(\y+t\v)$ can be calculated as in \eqref{firstder}
using the chain rule, \ie, $\left(\frac{\partial f}{\partial
x_i}(\y+t\v)\right)'=
\nnabla \frac{\partial f}{\partial x_i}(\y+t\v)\sprod \v=
\sum_{j=1}^n \frac{\partial^2 f}{\partial x_j\partial
x_i}(\y+t\v)v_j$. Summing up over all $i\in[1..n]$, we obtain:
\begin{align}
g''(t)&=\sum_{i=1}^n\sum_{j=1}^n \frac{\partial^2 f}{\partial x_j\partial
x_i}(\y+t\v)v_jv_i
=
H_{\y+t\v}\sprod \left(\v\v^\top\right)\label{eqgderivpos}\\
\implies g''(0)&=
H_{\y}\sprod \left(\v\v^\top\right)
=
\v^\top H_{\y} \v\notag
\end{align}
If $f$ is convex, then $g$ is convex for any direction $\v\in\R^n$. As such, 
$g''(0)=\v^\top H_{\y} \v\geq 0~\forall \v\in\R^n$, which means that $H_{\y}$ is SDP for any chosen $\y\in\R^n$.

Conversely, if
$H_{\y}$ is SDP for any $\y\in\R^n$, then $g$ is convex for any $\v\in \R^n$
because $g''(t)\geq 0~\forall t\in\R$
by virtue of \eqref{eqgderivpos}. The convexity definition ensures
that $\alpha f(\y_1)+\beta f(\y_2)\geq f(\alpha \y_1+\beta \y_2)$
for all
$\alpha,\beta \geq 0 \text{ such that }\alpha+\beta=1$, \ie, 
the line joining $f(\y_1)$ and $f(\y_2)$ is above
or equal to
the function
value evaluated at any point $\alpha \y_1+\beta \y_2$ on the segment $[\y_1,\y_2]$. This can be 
instantiated as follows:
\begin{equation}\label{eqgconvex}
{\beta}f(\y - \alpha \v) 
+{\alpha}f(\y+\beta\v)\geq
f(\y)~~~\forall
\y,\v\in \R^n,\alpha,\beta \geq 0 \text{ such that }\alpha+\beta=1
\end{equation}
To prove $f$ is convex, we need to show that
for any $\y_1,~\y_2\in\R^n$, $\alpha\geq0$ and $\beta=1-\alpha$
we have
$\beta f(\y_1)+\alpha f(\y_2)\geq f(\beta\y_1+\alpha \y_2)$.
But this reduces to \eqref{eqgconvex}:
if we fix $\v=\y_2-\y_1$ and replace $\beta\y_1+\alpha\y_2=\y$,
one can check $\y_1=\y-\alpha \v$ and $\y_2=\y+\beta \v$,
\ie, we obtain \eqref{eqgconvex} for $\y$ and $\v$ defined above.
\end{proof}

\section{Primal-Dual SDP programs and optimization considerations}
\subsection{Primal and dual SDP programs\label{sec21}}
\subsubsection{Main duality}
\begin{proposition}\label{propMainDuality}The dual of a primal SDP program is an SDP program and
the following degeneracy-related properties hold: 
\begin{itemize}
\item[(a)] If the primal is unbounded, the dual is infeasible. 
\item[(b)] If the primal is infeasible, the dual
can be unbounded, infeasible or non-degenerate.
\end{itemize}
We will later see that that if the primal is feasible and bounded (non-degenerate), there might be a duality gap
with regards to the optimal value of the dual, and, even more, the dual can even be infeasible.
\end{proposition}

\noindent {\it Proof.} Let us introduce the first SDP program in variables $\x\in\R^n$,
using matrices of an arbitrary order.
\begin{subequations}
\label{sdp}
\begin{align}[left ={(SDP)  \empheqlbrace}]
\min~~&\sum_{i=1}^n c_ix_i    \label{sdp1}\\
s.t~~&\sum_{i=1}^n A_i x_i \succeq B \label{sdp2}\\
    &\x\in \R^n             \label{sdp3}
\end{align}
\end{subequations}
The inequalities \eqref{sdp2} are often called {\it linear matrix inequalities}.
Let us now relax them 
(or penalize their potential violation)
using Lagrangian multipliers $Y\succeq \zeros$
to obtain the Lagrangian function:
$$\L(\x,Y)=\sum_{i=1}^n c_ix_i - Y\sprod (\sum A_i x_i-B)$$
Observe that if $\x$ satisfies \eqref{sdp2}, we get awarded in the above Lagrangian because the
Lagrangian term subtracts the product of two SDP matrices (non-negative by virtue of Prop.~\ref{propABprod}).
As in linear programming, we use the convention that 
$OPT(SDP)=\infty$ if \eqref{sdp2} has no feasible solution (\ie, we have an infeasible program)
and  $OPT(SDP)=-\infty$ if \eqref{sdp1}-\eqref{sdp3} can be indefinitely small (\ie, we have an unbounded program).
However, in all cases,
$\min\limits_{\x\in\R^n} \L(\x,Y)$ is a relaxation of (SDP) and we can write:
\begin{equation}\label{eqLagrok}\min_{\x\in\R^n} \L(\x,Y)\leq OPT(SDP),~\forall Y\succeq \zeros\end{equation}

We now develop the expression of the Lagrangian:
$$\min_{\x\in\R^n}\L(\x,Y)=\min_{\x\in\R^n}Y\sprod B + \sum_{i=1}^n (c_i-Y\sprod A_i) x_i $$
If there is a single $i\in[1..n]$ such that $c_i-Y\sprod A_i\neq 0$, the above minimum is $-\infty$
(unbounded), by using an appropriate value of $x_i$.
To have a bounded $\min_{\x\in\R^n}\L(\x,Y)$, the matrix $Y$ needs to satisfy $c_i-Y\sprod A_i= 0$
for all $i\in[1..n]$. Notice that if we actually consider
a non-negative variable
$x_i\geq 0$ 
for some $i\in[1..n]$, the condition
becomes $c_i-Y\sprod A_i\geq 0$ for such $i$. We are interested in finding:
$$\max_{Y\succeq\zeros}\min_{\x\in\R^n}\L(\x,Y)$$
that can be written:
\begin{subequations}
\label{dsdp}
\begin{align}[left ={(DSDP)  \empheqlbrace}]
\max~&B\sprod Y    \label{dsdp1}\\
s.t ~&A_i\sprod Y = c_i{\raisebox{0.9em}{\tikz\coordinate(equalityConstr);}}~\forall i\in[1..n]\label{dsdp2}\\
    ~&Y\succeq \zeros             \label{dsdp3}
\end{align}
\end{subequations}

\begin{tikzpicture}[overlay]
\node(textconstr) at (13.1,2.7) {\scalebox{0.85}{Or {$A_i\sprod Y \leq c_i$} if $x_i\geq 0$ is added to (SDP)}};
\draw[darkgray, decoration={markings,mark=at position 1 with {\arrow[thick]{>}}}, postaction={decorate} ]
                                                    (textconstr)  to  (equalityConstr.north);
\end{tikzpicture}

\noindent Based on \eqref{eqLagrok}, we obtain:
\begin{equation}\label{eqWeakDual}
OPT(DSDP)\leq OPT(SDP).
\end{equation}
\textbf{The case of degenerate programs} 
is addressed below, proving points (a) and (b) of the conclusion.
\begin{itemize}
\item[(a)] In above \eqref{eqWeakDual},
we can say that that $OPT(DSDP)=-\infty$ if $(DSDP)$ is not feasible, 
because $\min\limits_{\x\in\R^n}\L(\x,Y)=-\infty~\forall Y\succeq \zeros$ in this case.
It is clear that if $(DSDP)$ is feasible, then $(SDP)$ can not be unbounded from below.
Thus, if $OPT(SDP)=-\infty$ (unbounded from below), then $(DSDP)$ needs to be infeasible.
\item[(b)]
If (SDP) is infeasible,
we can infer nothing about the dual, \ie, 
$(DSDP)$ can be unbounded, infeasible or non-degenerate.
If you are familiar with linear programming, then it is well-known
that the dual of an infeasible LP is infeasible or unbounded.
By writing such LPs in an SDP form, one can obtain the 
examples (i) and (i) below.
The most difficult is to find a primal SDP whose dual
is non-degenerate. Such programs can be found by exploiting 
a phenomenon of ``clenching'' in the development of 
$\sum_{i=1}^n A_i x_i - B \succeq \zeros$ or $Y\succeq \zeros$,
\eg, we can use
Corollary~\ref{corolZeros} to
force certain rows or columns of an SDP matrix to be zero, 
pushing it to a certain form (or to infeasibility).
For instance, the SDP matrix in example (iii) below need to
have $y=0$: the zero at position (2,2) forces
the second row and the second column to contain only zeros
by virtue of Corollary~\ref{corolZeros}. This makes this program infeasible.
On the dual side, we have $Y_{11}=0$ because $x$ has a coefficient of zero in the primal
objective function and $Y_{33}=1$ because $Y_{12}+Y_{21}+Y_{33}=1$, which leads to
a unique dual feasible solution.
\hfill \qed
\end{itemize}

\noindent \hspace{-1em}\begin{tabular}{@{}c@{\hspace{-0.3em}}|@{}c@{}c@{\hspace{-0.4em}}|@{}c}
\scalebox{0.9}{
$
\begin{array}{l}
\min~x   \\
s.t.~\begin{bmatrix} x&0\\0&-x \end{bmatrix}\succeq \begin{bmatrix} 1 & 0 \\0&1\end{bmatrix}
\end{array}
$
}
&
\scalebox{0.9}{
$
\begin{array}{l}
\min~x   \\
s.t.~\begin{bmatrix} x-y\hspace{-1em}&0\\0&-x+y \end{bmatrix}\succeq \begin{bmatrix} 1 & 0 \\0&1\end{bmatrix}
\end{array}
$
}
\hspace{-1.2em}
&
\scalebox{0.9}{
$\begin{array}{l}
\min~-y  \\
s.t.~\begin{bmatrix} x&0&0\\0&-x&0\\0&0&y \end{bmatrix}\succeq \begin{bmatrix}1&0&0\\0&1&0\\0&0&0\end{bmatrix}
\end{array}$
}
&\scalebox{0.9}{
\begin{minipage}{0.25\linewidth}
$
\begin{array}{l}
\min~y   \\
s.t.~\begin{bmatrix} x&y&0\\y&0&0\\0&0&y\end{bmatrix}\succeq\begin{bmatrix}0&0&0\\0&0&0\\0&0&1\end{bmatrix}
\end{array}
$
\end{minipage}
}
\\
&&&\\
(i) unbounded dual &\multicolumn{2}{c|}{(ii) infeasible duals}& (iii) non-degenerate dual
\end{tabular}

~\\

\begin{proposition}\label{propEqPrimDualForms}
Program (DSDP) from \eqref{dsdp1}-\eqref{dsdp3} can be written in the primal form
of (SDP) from \eqref{sdp1}-\eqref{sdp3}.
\end{proposition}
\begin{proof}
We first solve the system $A_i\sprod Y = c_i~\forall i\in[1..n]$. 
If this system has no solution, then the given (DSDP) program is infeasible;
in this case, any
infeasible (SDP) can be considered
(by convention) equivalent to the given infeasible (DSDP) -- assuming 
 that both programs have the same optimization (min/max) direction.

{\it The non-degenerate case} is the essential one: we consider from
now on that the system $A_i\sprod Y = c_i~\forall i\in[1..n]$
has at least a feasible solution $-B'$. The set of all solutions
is given by 
\begin{equation}\label{eqbase}Y=-B'+\sum_{j=1}^k A'_jx'_j,\end{equation}
where $A'_1,~A'_2,\dots A'_k$ are a basis (maximum set of independent vectors) of the null space of 
$\{A_i:~i\in[1..n]\}$ (see the null space definition in
\eqref{eqDefNull}).%
\footnote{Assuming $A_1,~A_2,\dots A_n\in\R^{m \times m}$ are linearly independent, we have
$k=\frac{m(m+1)}2-n$.  We used the rank-nullity Theorem~\ref{thranknullity} with
the fact the set of symmetric matrices of size $m\times m$ has dimension
$\frac{m(m+1)}2$, using the symmetry constraints.}
The matrices $A'_j$ with $j\in[1..k]$ and
$A_i$ with $i\in[1..n]$ satisfy:
\begin{equation}\label{eqFirstDual} 
A_i\sprod A'_j=0,~\forall i\in[1..n],~j\in[1..k]\text{ and }-A_i\sprod B'=c_i,~\forall i\in[1..n].
\end{equation} 
The space spanned by (the linear combinations of) $A_i$ ($\forall i \in [1..n]$)
and $A'_j$ ($\forall j\in [1..k]$) need to cover the whole space of symmetric matrices of the size
of $Y$, by virtue of the rank-nullity Theorem \ref{thranknullity} that could be applied on the
vectorized versions of these matrices.
One can confirm that any feasible $Y$ can be expressed in the
form~\eqref{eqbase}: just notice that $Y+B'$ 
belongs to the null space of 
$\{A_i:~i\in[1..n]\}$, and so, 
it can be expressed as a linear combination of the
basis $A'_1,~A'_2,\dots A'_k$.

Replacing \eqref{eqbase} in \eqref{dsdp1}-\eqref{dsdp3}, we
obtain:
\begin{align}
\max-B\sprod B'&+\sum_{j=1}^k(B\sprod A'_j)x'_j    \notag\\
s.t.~&\sum_{j=1}^k  A'_jx'_j \succeq B' \label{eqDualPrimalized}\\
    &\x'\in \R^k             \notag,
\end{align}
which is an SDP program in the primal form \eqref{sdp1}-\eqref{sdp3}.

\end{proof}

\begin{proposition} \label{propMainDualityWithNonNegatives} If we have $x_j\geq 0$ for certain variables $J\subseteq [1..n]$ of
\eqref{sdp1}-\eqref{sdp3}, we can still write an equivalent program without explicit
non-negative variables by incorporating the non-negativities in new rows and columns of \eqref{sdp2}. The dual can be written
in the form \eqref{dsdp1}-\eqref{dsdp3}, but there is an equivalent dual
in which equalities \eqref{dsdp2} become
$A_j\sprod Y \leq c_j$ for all $j\in J$.
\end{proposition}
\begin{proof}
We define matrices $A'_i$ ($\forall i\in [1..n]$) and $B'$ by bordering 
$A_i$ and resp.~$B$ with  $|J|$ rows and columns that contain only zeros
except at the following new positions: $A'_j$ contains $1$ at a position $(m_j,m_j)$ that
correspond to the bordering row and column associated to $j\in J$. We can drop $x_j\geq 0$ but
write an equivalent program \eqref{sdp1}-\eqref{sdp3} with matrices $A'_i$ ($\forall i\in [1..n]$) and
$B'$.

The dual of this program has the form \eqref{dsdp1}-\eqref{dsdp3} and it contains constraints
of the form $A'_i\sprod Y'=c_i~\forall i\in[1..n]$. For $i\notin J$, $A'_i\sprod Y'=c_i$ is equivalent
to $A_i\sprod Y=c_i$. For $j\in J$, $A'_j\sprod Y'=c_j$ becomes $Y'_{m_j,m_j}+A_j\sprod Y=c_j$, which is
equivalent to $A_i\sprod Y\leq c_i$, because $Y'_{m_j,m_j}\geq 0 $ does not play a role elsewhere in the
program. The objective function $B'\sprod Y'$ is equivalent to $B\sprod
Y$. The initial dual with matrices $A'_i$ ($\forall i\in [1..n]$) and $B'$ can
be equivalently written with matrices $A_i$ ($\forall i\in [1..n]$) and $B$ by 
using inequality constraints $A_j\sprod Y\leq c_j~\forall j\in J$. Finally, notice
that $Y'\succeq \zeros\implies Y\succeq \zeros$ because $Y$ is a principal minor of $Y$ (use
the definition from Prop.~\ref{propMinorsNonneg}).
\end{proof}

\begin{proposition}\label{propAggregated}(the case of multiple constraints) Suppose one needs to impose multiple constraints 
in \eqref{sdp1}-\eqref{sdp3}:
\begin{equation}\label{eqBlocIndiv}\sum_{i=1}^n A_i^j x_i^j \succeq B^j\end{equation}
for $j\in[1..n']$. Notices that the involved matrices can have order 1 for some $j\in[1..n']$, \ie,
\eqref{eqBlocIndiv} for such $j$ is a linear constraint.
However, all these constraints can be incorporated in a unique constraint of the form
\eqref{sdp2} expressed with aggregated block-diagonal matrices with $n'$ blocks. An aggregated dual
can be expressed
in the canonical form~\eqref{dsdp1}-\eqref{dsdp2} using aggregated block-diagonal matrices. This
aggregated dual
is equivalent to a dual in which we have $n'$ dual matrix variables.
\end{proposition}
\begin{proof} We define aggregated block-diagonal matrices $A'_i$ ($\forall i\in[1..n]$) and $B'$ with 
$n'$ blocks defined by $A_i^j$ and resp.~$B^j$ for all $j\in [1..n']$:
$$B'=
\begin{bmatrix}
B^1    &  \zeros   &  \ldots    & \zeros \\
\zeros &  B^2      &  \ldots    & \zeros \\
\vdots & \vdots    &  \ddots    & \vdots \\
\zeros & \zeros    &  \ldots    &  B^{n'}
\end{bmatrix}
\text{ and }
A'_i=
\begin{bmatrix}
A_i^1  &  \zeros   &  \ldots    & \zeros \\
\zeros &A_i^2      &  \ldots    & \zeros \\
\vdots & \vdots    &  \ddots    & \vdots \\
\zeros & \zeros    &  \ldots    &  A_i^{n'}
\end{bmatrix}
\forall i\in[1..n]$$
Constraints
\eqref{eqBlocIndiv} are equivalent to a unique constraint in aggregated block-diagonal matrices
$$\sum_{i=1}^n A'_i\succeq B'.$$
We obtain an aggregated program \eqref{sdp1}-\eqref{sdp3} expressed with
block-diagonal matrices $A'_i$
($\forall i\in[1..n]$) and $B'$.

We can now construct the dual of this aggregated primal program. 
We obtain an aggregated dual program of the form \eqref{dsdp1}-\eqref{dsdp3} expressed
in aggregated variables $Y'\succeq 0$. The dual objective function is:
\begin{equation}\label{eqDualBlocks}B'\sprod Y'=\sum_{j=1}^{n'} B^j\sprod Y^j,\end{equation}
where $Y^j$ is obtained from $Y'$ by keeping only the rows and columns that 
correspond to block $B^j$ inside $B'$ ($\forall j\in[1..n']$). The dual constraints are
$A'_i\sprod Y'=c_i$ ($\forall i \in[1..n]$) and they can also be written as:
\begin{equation}\label{eqDualBlocksConstraints}\sum_{j=1}^{n'} A_i^j\sprod Y^j=c_i~\forall i\in[1..n].\end{equation}
Notice that the variables $y'_{i',j'}$ outside the $n'$ diagonal blocks
of $Y'$ play no role in the constraints or in the objective function of the dual.
Also, $Y'\succeq \zeros$ implies $Y^j\succeq \zeros~\forall j\in[1..n']$ because all $Y^j$ are
principal minors of $Y$ (use the definition from Prop.~\ref{propMinorsNonneg}). The aggregated dual 
of the form \eqref{dsdp1}-\eqref{dsdp2} in variables $Y'$ is equivalent to a dual in variables
$Y^j$ (with $j\in[1..n]$) using objective \eqref{eqDualBlocks} and constraints \eqref{eqDualBlocksConstraints}.
\end{proof}
\subsubsection{The dual of the dual is the initial program}
The remaining of Section~\ref{sec21}
is devoted to a few properties that may seem 
a bit boring and easy to trust, because they only ask 
to verify certain equivalences between the dual and the
primal forms. However, the exercise of verifying these
properties may offer a good insight into the different
ways of expressing the same SDP program and into the 
different ways of understanding its space of feasible solutions.

We now provide two results on the dual of the dual. The first one uses a new type of Lagrangian duality,
while the second one only uses the first duality from Prop.~\ref{propMainDuality}.
\begin{proposition}\label{propTwiceDualize} By dualizing twice the $(SDP)$ from 
\eqref{sdp1}-\eqref{sdp3} we obtain the initial $(SDP)$ and the following properties
hold
\begin{itemize}
\item[(a)]
If the $(DSDP)$ from \eqref{dsdp1}-\eqref{dsdp3} is unbounded, its dual $(SDP)$
is infeasible. 
\item[(b)]
If $(DSDP)$ is infeasible, $(SDP)$ can be unbounded, infeasible or non-degenerate.
\end{itemize}
\end{proposition}
\noindent {\it Proof.} 
Let us calculate the Lagrangian dual of (DSDP) from \eqref{dsdp1}-\eqref{dsdp3} and verify that we 
obtain the (SDP) from
\eqref{sdp1}-\eqref{sdp3}. We relax constraints \eqref{dsdp2} using coefficients $\x'\in\R^n$:
\begin{equation}
\L'(Y,\x')=B\sprod Y + \sum_{i=1}^n (c_i-A_i\sprod Y) x'_i.
\label{eq2111}
\end{equation}

For any $Y$ that satisfies \eqref{dsdp2}, the value of above Lagrangian is $B\sprod Y$, \ie, the
objective value \eqref{dsdp1} of $Y$ in $(DSDP)$. 
If \textit{no} $Y\succeq \zeros$ satisfies \eqref{dsdp2}, we say $OPT(DSDP)=-\infty$,
adopting a similar convention 
as in linear programming or in Prop.~\ref{propMainDuality}.
We also say 
and  $OPT(DSDP)=\infty$ if \eqref{dsdp1}-\eqref{dsdp3} can be indefinitely large (\ie, we have an unbounded program).
However, in all cases, we can state:
\begin{equation}\label{eqLagrok2}
\max_{Y\succeq 0}\L'(Y,\x')\geq OPT(DSDP),~\forall \x'\in\R^n
\end{equation}
We now re-write the above Lagrangian:
$$\max_{Y\succeq 0}\L'(Y,\x') = \max_{Y\succeq 0}\left(B-\sum_{i=1}^n A_ix'_i\right)\sprod Y + \sum_{i=1}^n c_ix'_i$$
We will show that this expression can only be bounded if $B-\sum\limits_{i=1}^n A_ix'_i\preceq \zeros$.
Notice that $Y=\zeros$ leads to $\left(B-\sum\limits_{i=1}^n A_ix'_i\right)\sprod Y=0$.
We need values $\x'$ such that $\left(\sum\limits_{i=1}^n A_ix'_i-B\right)\sprod Y\geq 0~\forall Y\succeq
\zeros$. Using Prop \ref{propSDPSelfDual}, this can only hold if $\sum\limits_{i=1}^n A_ix'_i-B\succeq
\zeros$. For such $\x'$, we have 
$\max\limits_{Y\succeq 0}\L'(Y,\x') = \sum_{i=1}^n c_ix'_i$. We can write:
\begin{subequations}
\begin{align}[left ={\hspace{-5em}\displaystyle\min_{\x'\in\R^n}\max_{Y\succeq 0}\L'(Y,\x')=\empheqlbrace}]
\min~~&\sum_{i=1}^n c_ix'_i    \notag \\
s.t ~~&\sum A_i x'_i \succeq B \notag \\
    &\x'\in \R^n,            \notag 
\end{align}
\end{subequations}
which is exactly the $(SDP)$ from \eqref{sdp1}-\eqref{sdp3}. Based on \eqref{eqLagrok2}, we discover
\eqref{eqWeakDual} again:
\begin{equation}\label{eqWeakDual2}
OPT(DSDP)\leq OPT(SDP).
\end{equation}
We now address points (a) and (b) of the conclusion.
\begin{itemize}
\item[(a)]
It is clear that if $(SDP)$ is feasible,
then $(DSDP)$ can not be unbounded. This means that if
$(DSDP)$ unbounded, then $(SDP)$ is infeasible.
\item[(b)]
If $(DSDP)$ is infeasible, we can infer nothing about
$(SDP)$, \ie, $(SDP)$ could be unbounded, infeasible,
or non-degenerate. One can find examples of unbounded or
infeasible duals by generalizing the linear programming
examples. 
For instance, if
if $(DSDP)=\max\{y:y=-1,~y\geq 0\}$, then
$(SDP) = \min\{-x:~x\geq 1\}$ is unbounded.
A pair of infeasible primal-dual programs can
simply be taken from example (ii)  provided at the end of the proof of Prop~\ref{propMainDuality}.
To find an example of an infeasible (DSDP) with a feasible (SDP), it is
enough to take example (iii) at the end of the proof of Prop.~\ref{propMainDuality}
and to change the right-hand side and the objective function of the (SDP).

If $(DSDP)=\max\left\{
           \left[\begin{smallmatrix}1&0&0\\0&0&0\\0&0&0\end{smallmatrix}\right]\sprod Y:~
           \left[\begin{smallmatrix}1&0&0\\0&0&0\\0&0&0\end{smallmatrix}\right]\sprod Y=0,~
           \left[\begin{smallmatrix}0&1&0\\1&0&0\\0&0&1\end{smallmatrix}\right]\sprod Y=-1,
          \right\}$, then
we obtain
$(SDP)=
\min~\left\{-x_2:~
\left[\begin{smallmatrix}
        x_1   &        x_2 &    0    \\
        x_2   &        0   &    0      \\
        0   &        0   &    x_2    
\end{smallmatrix}\right]
\succeq 
\left[\begin{smallmatrix}1&0&0\\0&0&0\\0&0&0\end{smallmatrix}\right]
\right\}$
that has solution $x_2=0$ (apply Corollary~\ref{corolZeros} on
the fact that the middle element is zero) and $x_1\geq 1$ of objective
value zero. \hfill \qed
\end{itemize}

\begin{proposition}\label{propDualPrimalFormDualizedIntoSmtngEqToPrimal}
Assuming the dual $(DSDP)$ from \eqref{dsdp1}-\eqref{dsdp3} is feasible, 
we can write it into the primal form \eqref{eqDualPrimalized}
as described by Prop.~\ref{propEqPrimDualForms}.
If we apply the first duality from Prop.~\ref{propMainDuality} on this primal form, we 
obtain a dual that is equivalent to the primal (SDP) from \eqref{sdp1}-\eqref{sdp3}.
\end{proposition}
\begin{proof} The proof relies on a few arguments from the proof of Prop.~\ref{propEqPrimDualForms}.
First, recall that the feasible dual program \eqref{dsdp1}-\eqref{dsdp3} can be written in the primal form~\eqref{eqDualPrimalized}.
We re-write \eqref{eqDualPrimalized} as follows:
\begin{equation*}
\begin{aligned}
-B\sprod B' - \\
\\[1.8em]
\\
\end{aligned}
\left(
\begin{aligned}
\min&\sum_{i=1}^k(-B\sprod A'_i)x'_i    \notag\\
s.t &\sum_{i=1}^k  A'_ix'_i \succeq B' \\
    &\x'\in \R^k             \notag.
\end{aligned}
\right).
\end{equation*}
Recall that $A'_1,~A'_2,\dots A'_k$ and $B'$ arise from
solving $A_i\sprod Y=c_i$, \ie, any solution
$Y$ can be written in the form $Y=-B'+\sum_{i=1}^k  A'_ix'_i$, where
\begin{equation}
A'_1,\dots, A'_k\textnormal{ are a basis of the null space of }
A_1,\dots, A_n \textnormal{ and }- B'\sprod A_i=c_i~\forall i\in[1..n]
\tag{*}
\end{equation}
We now apply the first duality from Prop.~\ref{propMainDuality} on the above program in (SDP) form and obtain:
\begin{equation*}
\begin{aligned}
-B\sprod B' - \\
\\[0.4em]
\\
\end{aligned}
\left(
\begin{aligned}
\max~&B'\sprod Y    \\
s.t ~&A'_i\sprod Y = -B\sprod A'_i~\forall i\in[1..k]\\
    ~&Y\succeq \zeros.             
\end{aligned}
\right)
\end{equation*}
This dual could be infeasible even if the corresponding primal is feasible (we have already
presented such examples, see also Prop.~\ref{propDualCanBeInfeasible}).
However, 
the system of linear equations $A'_i\sprod Y = -B\sprod A'_i~\forall i\in[1..k]$ has at least
the feasible solution $Y=-B$. 
The existence of this solution is sufficient to place us in the non-degenerate case 
of the proof of
Prop.~\ref{propEqPrimDualForms},
which leads to re-writing the above program in the primal form.
For this, we consider the origin $-B$ and the basis $A_1,~A_2,\dots, A_n$ that is orthogonal to 
all $A'_i,~\forall i\in[1..k]$. Recall (*) that
$A'_1,~A'_2,\dots, A'_k$ are a basis for the null space
of 
$A_1,~A_2,\dots A_n$. As such,
the space 
generated by 
(the linear combinations of) $A_1,~A_2,\dots A_n$ and $A'_1,~A'_2,\dots A'_k$ cover
the whole space of symmetric matrices.
Any $Y\succeq \zeros$ can be written in the form $\sum\limits_{i=1}^n x_iA_i -B$. The above program can thus be
re-written as:
\begin{align*}
-B\sprod B'-\max~&\left(-B\sprod B' +\sum_{i=1}^n(B'\sprod A_i)x_i\right)\\
s.t.~&\sum_{i=1}^n  A_ix_i \succeq B \\
    &\x\in \R^n             \notag.
\end{align*}
Recall now $B'\sprod A_i=-c_i,~\forall i\in[1..n]$ from (*). By replacing this in above
program and simplifying $-B\sprod B'$, we obtain:
\begin{align*}
-\max~~&\left(\sum_{i=1}^n-c_ix_i\right)\\
s.t.~~&\sum_{i=1}^n  A_ix_i \succeq B \\
    &\x\in \R^n             \notag,
\end{align*}
which is exactly \eqref{sdp1}-\eqref{sdp3}.
\end{proof}
\subsubsection{\label{secPrimalToDual}From the primal form to the dual form}

The transformation from the primal form to the dual form is more difficult and it is not always
possible. 
\begin{proposition}\label{propPrimalToDual}
The (SDP) program in the primal form \eqref{sdp1}-\eqref{sdp3}
can be written in the dual form \eqref{dsdp1}-\eqref{dsdp3}, 
provided that
the matrices $A_1,~A_2,\dots A_n$ from \eqref{sdp2} are linearly independent.
\end{proposition}
\begin{proof}
We write $Y=\sum_{i=1}^n  A_ix_i - B $ and notice that \eqref{sdp2} stipulates
$Y\succeq \zeros$. We need to 
show that the set of such $Y$ can be expressed as the solution set
of a system of equations
$A'_i\sprod Y = c'_i~\forall i\in[1..k]$, \ie, of the form ~\eqref{dsdp2}. 
However, the most difficult task is to 
write objective function~\eqref{sdp1} 
in the form $B'\sprod Y$ as in~\eqref{dsdp1}; for this, we need to express
$x_1,~x_2,\dots x_n$ as linear combinations of variables $Y$. We will see that
$x_1,~x_2,\dots x_n$ can  be exactly determined from $Y$ when $A_1,~A_2,\dots
A_n$ are linearly independent.

Consider the following program in the form \eqref{sdp1}-\eqref{sdp3} in which 
$A_1$ and $A_2$ are \textit{not} linearly independent, since
$A_1=A_2=
\left[\begin{smallmatrix}
        0     &        1   \\
        1     &        0   \\
\end{smallmatrix} \right]$.

\begin{align*}
\min~& 2x_1+3x_2 \\
s.t.~&
\begin{bmatrix}
        1     &     x1+x2   \\
    x_1+x_2   &        1   \\
\end{bmatrix}\succeq \zeros\\
&x_1,~x_2\in\R
\end{align*}
If we try to write $Y=
\left[\begin{smallmatrix}
        1     &     x1+x2   \\
    x_1+x_2   &        1   \\
\end{smallmatrix} \right]$, then the set of feasible symmetric matrices $Y$ is the set of solutions of the system
$Y\sprod
\left[\begin{smallmatrix}
        1     &        0   \\
        0     &        0   \\
\end{smallmatrix} \right]=1$ 
and
$Y\sprod
\left[\begin{smallmatrix}
        0     &        0   \\
        0     &        1   \\
\end{smallmatrix} \right]=1$. However, it is impossible to express the objective
function value in terms of variables $Y$. What is the objective value of 
$Y=
\left[\begin{smallmatrix}
        1     &        0   \\
        0     &        1   \\
\end{smallmatrix} \right]$ ? It can actually be any real number or even
$-\infty$, depending on the choice of $x_1$ and $x_2$ such that $0=x_1+x_2$.

We hereafter assume that $A_1,~A_2,\dots A_n$ are linearly independent and we
provide an algorithm for finding a system of equations whose solutions $Y$ can all be
written in the form $Y=\sum_{i=1}^n A_ix_i-B$. Let us work with ``vectorized''
versions of the matrices: the notation $\overline{X}$ represents a column
vector containing the diagonal and the upper triangular elements of symmetric matrix $X$.
The size of this vector is $m=\frac{q(q+1)}{2}$, where $q$ is the order of $X$. We have $Y=\sum_{i=1}^n A_ix_i-B\Longleftrightarrow\Y=\sum_{i=1}^n
\A_ix_i-\B$. Using aggregate matrix $\A\in \R^{m\times n}$ such that 
$\A=[\A_1~\A_2\dots \A_n]$, we can write:
\begin{equation}\label{eqvectorizeeq}
\Y + \B = \A\x
\end{equation}

Since $A_1,~A_2,\dots A_n$ are linearly independent, $\A$ has rank $n$; this
also ensures $m\geq n$. Without loss of generality, we can consider that the
first $n$ rows of $\A$ are linearly independent and form a non-null minor
$[\A]_n$. The equations corresponding to the first $n$ rows of
\eqref{eqvectorizeeq} can be written
$
[\Y]_n + [\B]_n = [\A]_n\x
$. As such, we can deduce $\x=[\A]_n^{-1}([\Y]_n + [\B]_n )$. This ensures that
the objective value $\sum_{i=1}^n c_ix_i$ from \eqref{sdp1} can be written as a
linear combination of the $[\Y]_n$ values, \ie, in the form $B'\sprod Y$ as in
\eqref{dsdp1}. We now replace $\x$ in \eqref{eqvectorizeeq} and obtain:
\begin{equation}\label{eqvectorized2}
\Y + \B = \A [\A]_n^{-1}([\Y]_n + [\B]_n ).
\end{equation}

The first $n$ rows of above formula are redundant. The remaining $k=m-n$ rows
actually represent a system of linear equations: notice that each element of
the left-hand $\Y$ is expressed as a linear combination of the free variables
$[\Y]_n$ (plus a fixed coefficient).
Now check than any such linear equation in variables $Y$ can be
rewritten in the form $A'_i\sprod Y = c'_i$ (for any $i\in[1..k]$).

We have just shown that any solution $\Y$ of \eqref{eqvectorizeeq}
satisfies a system of equations of the form $A'_i\sprod Y = c'_i$ (with $i\in[1..k]$).
We now prove the converse: any solution of this system can be written in
the form \eqref{eqvectorizeeq} for a certain $\x\in\R^n$; more exactly, this
$\x$ is the unique solution of $[\Y]_n = [\A]_n\x- [\B]_n $.
To show this, it is enough to see that any solution of the above system satisfies \eqref{eqvectorized2}
by construction and that,
after replacing $[\Y]_n = [\A]_n\x- [\B]_n $, \eqref{eqvectorized2} becomes
\begin{align*}
\Y + \B &= \A [\A]_n^{-1}([\Y]_n + [\B]_n )  \\
&= \A [\A]_n^{-1}([\A]_n\x- [\B]_n + [\B]_n )  \\
&= \A [\A]_n^{-1}([\A]_n\x)  \\
&= \A\x 
\end{align*}
By ``devectorizing''
$\Y,~\B,~\A_1,~\A_2,\dots \A_n$ into symmetric matrices
$Y,~B,~A_1,~A_2,\dots A_n$ (put the corresponding elements on the diagonal and
resp.~on the symmetric matrix positions), we obtain $Y=\sum_{i=1}^n A_ix_i -B$.
\end{proof}
\begin{example}\label{exFromDualToPrimal} We apply the above proof of Prop.~\ref{propPrimalToDual} on an example showing
how to rewrite a program \eqref{sdp1}-\eqref{sdp3} in the form of
\eqref{dsdp1}-\eqref{dsdp3}. Consider 
\begin{align}
\min~& x_1+2x_2+x_3+2x_4 \notag \\
s.t.~&
\begin{bmatrix}
        0     &     x_1+x_3   &    x_4  \\
    x_1+x_3   &    x_2+x_3+ 1 &    x_4 +2x_2 \\
       x_4    &       x_4+2x_2 &  x_1+x_2+2\\
\end{bmatrix}\succeq \zeros \label{eqInitialProg}\\
&x_1,~x_2,~x_3,~x_4\in\R \notag
\end{align}
We write a formula corresponding to \eqref{eqvectorizeeq} but without
``vectorization'':
\begin{equation}\label{eqxy}
\begin{bmatrix}
    y_{11}     &    y_{12}   &   y_{13}  \\
    y_{12}     &    y_{22}-1 &   y_{23}  \\
    y_{13}     &    y_{23}   &   y_{33}-2\\
\end{bmatrix}=
\begin{bmatrix}
        0     &     x_1+x_3   &    x_4  \\
    x_1+x_3   &    x_2+x_3    &    x_4 +2x_2 \\
       x_4    &       x_4+2x_2 &  x_1+x_2  \\
\end{bmatrix}
\end{equation}
As in paragraph below \eqref{eqvectorizeeq}, we will express the four variables
$x_1,~x_2,~x_3$ and $x_4$ in terms of four variables $Y$. 
We can not choose $Y_{11}$, because the corresponding row of \eqref{eqvectorizeeq}
has only zeros in the matrix $\A$. 
Let us choose:
$y_{12} = x_1 + x_3$,
$y_{22} -1 = x_2 + x_3$,
$y_{33} -2 = x_1 + x_2$ and 
$y_{13}    = x_4$. We obtain
\begin{equation}\label{eqxfromy}
x_1 = \frac{y_{12}+y_{33} - y_{22} - 1}2,~
x_2 = \frac {y_{22}+y_{33} - y_{12}- 3}2,~
x_3 = \frac {y_{12}+y_{22} -y_{33} + 1}2\text{ and }x_4=y_{13}.
\end{equation}
Our objective function can be written:
\begin{equation}\label{eqObj} 
\frac{y_{12}+y_{33} - y_{22} - 1}2+
2\frac {y_{22}+y_{33} - y_{12}- 3}2+
 \frac {y_{12}+y_{22} -y_{33} + 1}2+
2y_{13}=
y_{22}+
y_{33}+
2y_{13} - \frac 32.
\end{equation}
We now replace the values of $x_1,~x_2,~x_3,~x_4$ in the remaining equations of
\eqref{eqxy}. We have $y_{11}=0$ and $y_{23}=y_{13}+y_{22}+y_{33} - y_{12}- 3$,
or $3=  y_{13}+y_{22}+y_{33}-y_{12}-y_{23}$. Combining these two equations with
\eqref{eqObj}, we obtain the program:
\begin{align*}
\min 
      \Big[\begin{smallmatrix}
         0     &  0    &    1   \\
         0     &  1    &    0   \\
         1     &  0    &    1   \\
       \end{smallmatrix}\Big]
       &\sprod Y-\frac 32\\
      \Big[\begin{smallmatrix}
         1     &  0    &    0   \\
         0     &  0    &    0   \\
         0     &  0    &    0   \\
       \end{smallmatrix}\Big]
       &\sprod Y=0\\
      \Big[\begin{smallmatrix}
         0     &  -\frac12    &     \frac 12  \\
        -\frac 12     &  1    &    -\frac12   \\
         \frac 12    &  -\frac 12    &    1   \\
       \end{smallmatrix}\Big]
       &\sprod Y=3\\
       &Y\succeq \zeros
\end{align*}

One can check that any feasible solution $Y$ of above program can be written 
as in \eqref{eqInitialProg}. For this, it is enough to determine
variables $x_1,~x_2,~x_3,~x_4$ from $Y$ using \eqref{eqxfromy} and check that
all positions of $Y$ and the objective function 
from above program can be written using variables $\x$
which leads to
\eqref{eqInitialProg}. For instance, to check that $y_{23}$ can be written as
$x_4+2x_2$ with $x_2$ and $x_4$ determined by \eqref{eqxfromy}, we write:
$x_4+2x_2=y_{13}+{y_{22}+y_{33} - y_{12}- 3}=y_{23}$, where we used the second
constraint on $Y$ for the last equality.
\end{example}

We used several times in this section the transformation from the dual form into
the primal form. All these transformations rely on solving the system 
$A_i\sprod Y = c_i$, with $i \in [1..n]$. You might have noticed there are several ways to
solve this system. Just to give an example, consider that $y_1+y_2=10$
can be solved in two manners: (a) take $y_1=x_1$ as a free variable and all solutions
have the form $[y_1~y_2]=[x_1~\,x_1-10]$
or (b) take solution $[5~5]$ and all solutions take the form
$[5~5]+x_1[-1~1]$. We can say that the approach (a) was used in the Example~\ref{exFromDualToPrimal} above,
while the approach (b) was used in Prop~\ref{propEqPrimDualForms}. They are both related, 
because switching from (a) to (b) is equivalent to performing a change of variable in
the resulting primal, \eg, take \eqref{sdp1}-\eqref{sdp3} and use variables $x'_i=x_i+7$ instead
of $x_i~\forall i\in[1..n]$. Such a change leads to a different primal form,
in which the 
objective $\sum_{i=1}^n c_i x_i$ evolves to 
$\sum_{i=1}^n c_i x'_i-7n$.

\subsection{Relations between the primal optimum and the dual optimum}
\begin{proposition}(Complementary Slackness) If $\overline{\x}$ and $\overline{Y}$ are the optimum primal and
resp.~dual solutions of SDP and resp. DSDP (in \eqref{sdp1}-\eqref{sdp3}
and~\eqref{dsdp1}-\eqref{dsdp3} resp.), the duality gap can be written:
\begin{equation}\label{eqdualgap}
OPT(SDP)-OPT(DSDP) = \overline{Y}\sprod\left(\sum_{i=1}^n A_i\overline{x_i} -B\right)
\end{equation}
There is no strict complementarity as in linear programming, \ie, the matrices
of the above product might share an eigenvector whose eigenvalue is zero
in both matrices for any optimal $\xx$ and $\overline{Y}$.
\end{proposition}
\begin{proof} It is enough to develop 
\begin{align*}
OPT(SDP)-OPT(DSDP) &=\sum_{i=1}^n c_i\overline{x_i} - B\sprod \overline{Y}\\
                   &=\sum_{i=1}^n \left(A_i\sprod \overline{Y}\right) \overline{x_i} - B\sprod \overline{Y}\\
                   &=\overline{Y}\sprod\left(\sum_{i=1}^n A_i\overline{x_i} -B\right)
\end{align*}
When strong duality holds (\ie, when $OPT(SDP)-OPT(DSDP)=0$), we can observe the following using the
eigen-decomposition \eqref{firstEqDecomp}:
\begin{itemize}
    \item[(a)] the eigenvectors of $\left(\sum_{i=1}^n A_i\overline{x_i} -B\right)$ with
non-zero eigenvalues belong to the space generated by the eigenvectors of $\overline{Y}$ with
eigenvalue $0$.
    \item[(b)] the eigenvectors of $\overline{Y}$ with
non-zero eigenvalues belong to the space generated by the eigenvectors of $\left(\sum_{i=1}^n
A_i\overline{x_i} -B\right)$ with eigenvalue $0$.
\end{itemize}
In intuitive terms, we can say that any eigenvector of $\left(\sum_{i=1}^n A_i\overline{x_i}
-B\right)$ with a non-zero eigenvalue can be seen as an eigenvector of $\overline{Y}$
with eigenvalue 0 and vice-versa. Notice we did {\it not} claim that 
$rank\left(\sum_{i=1}^n A_i\overline{x_i} -B\right) =nullity(\overline{Y})$. Take for instance the
matrices 
$
\left[\begin{smallmatrix}
        1     &        0     &      0      \\
        0     &        0     &      0      \\
        0     &        0     &      0      \\
\end{smallmatrix} \right]$ 
and 
$
\left[\begin{smallmatrix}
        0     &        0     &      0      \\
        0     &        0     &      0      \\
        0     &        0     &      1      \\
\end{smallmatrix} \right]$. They satisfy above conditions (a) and (b), but the first one has rank 1 and
the second one has nullity 2. However, if there exist multiple primal and
dual optimal solutions, do some of them satisfy
$rank\left(\sum_{i=1}^n A_i\overline{x_i} -B\right) =nullity(\overline{Y})$?

This is called the strict complementarity property: any
eigenvector of
$\left(\sum_{i=1}^n A_i\overline{x_i}
-B\right)$ with a zero eigenvalue is an eigenvector of $\overline{Y}$ with
non-zero eigenvalue. In linear programming, a theorem of Goldman and Tucker%
\footnote{
I first learned of this theorem from the article 
``Semidefinite programming'' by M. Overton and H. Wolkowicz,
a foreword for a special issue on SDP of
 Mathematical Programming (Volume 77, Issue 1, 1997).
}
states that there always exist primal-dual solutions $\x$ and $\y$ that are
\textit{strictly} complementary, \ie, 
every
zero of $A\x-\b$ corresponds to non-zero of $\y$.
In SDP programming, 
this strict complementarity property no longer
holds. Matrices
$ \left(\sum_{i=1}^n A_i\overline{x_i}-B\right)$ and $\overline{Y}$
can share an eigenvector whose eigenvalue is zero in both matrices.
Consider
the following primal-dual programs, both expressed in the form of
\eqref{sdp1}-\eqref{sdp3}; the dual was transformed as described in
Prop.~\ref{propEqPrimDualForms} (see below).

~\\

\begin{minipage}{0.5\textwidth}
\begin{align*}
\min~&  x_3\\
s.t.~&
\begin{bmatrix}
        x_1   &        x_2 &    0      &     0    \\
        x_2   &        0   &    0      &     0    \\
        0     &        0   &    x_2    &     0    \\
        0     &        0   &    0      &   x_3-2  \\
\end{bmatrix}\succeq \zeros\\
&x_1,~x_2,~x_3\in\R
\end{align*}
\end{minipage}
\begin{minipage}{0.5\textwidth}
\begin{align*}
\max~&  2y_{44}\\
s.t.~& Y=
\begin{bmatrix}
        0     &        0   &    0      &     0    \\
        0     &     a^2+\Delta^2&    0      &     a    \\
        0     &        0   &    0      &     0    \\
        0     &        a   &    0      &     1  \\
\end{bmatrix}\\
&Y\succeq 0
\end{align*}
\end{minipage}\\
~\\
~\\
The optimum primal solution $\xx$ satisfies $\overline x_2=0$ (row 2 and column 2 need to be
zero, because position (2,2) is zero) and $\overline x_3=2$. The matrix
$\sum\limits_{i=1}^n A_i\overline{x_i} -B$ is 
$
\left[\begin{smallmatrix}
\overline{x}_{11}&     0     &      0     &   0    \\
        0     &        0     &      0     &   0    \\
        0     &        0     &      0     &   0    \\
        0     &        0     &      0     &   0    \\
\end{smallmatrix} \right]$ which has at maximum rank 1.


On the dual side, 
any feasible $Y$ satisfies $y_{11}=0$ because the coefficient 
of $x_1$ is zero in the primal objective function. This forces row 1 and column 1 of $Y$ to have only zeros. The dual
constraint corresponding to $x_2$ is $2y_{12}+y_{33}=0$; since $y_{12}=0$, we
have $y_{33}=0$. The dual constraint corresponding to $x_3$
imposes $y_{44}=1$.
There is no constraint on $y_{24}=a$; $y_{22}$ needs to be greater than or equal
to $a^2$ so as to have a non-negative principal minor corresponding to
rows/columns 2 and 4; we can write
$y_{22}=a^2+\Delta^2$. The nullity of any feasible $Y$ is at least 2. Both
matrices have eigenvector 
$
\left[\begin{smallmatrix}
        0     &
        0        &
        1        &
        0        \\
\end{smallmatrix} \right]^\top$ with eigenvalue 0.
\end{proof}
\begin{proposition}(Non-zero duality gap) If $\overline{\x}$ and $\overline{Y}$ are the optimum primal and
resp.~dual solutions of SDP and resp. DSDP (in \eqref{sdp1}-\eqref{sdp3}
and~\eqref{dsdp1}-\eqref{dsdp3} resp.), the duality gap 
$
OPT(SDP)-OPT(DSDP)
$ is not necessarily 0.
\end{proposition}
\begin{proof}  Consider the following primal-dual programs,
where the dual is actually written in the primal form, as obtained
after applying the transformation from
Prop.~\ref{propEqPrimDualForms}.

~\\

\begin{minipage}{0.5\textwidth}
\begin{align*}
\min~&  x_3\\
s.t.~&
\begin{bmatrix}
        x_1   &        x_2 &    0      &     0    \\
        x_2   &        0   &    0      &     0    \\
        0     &        0   &    x_2    &  1-x_3   \\
        0     &        0   &  1-x_3    &   x_3    \\
\end{bmatrix}\succeq \zeros\\
&x_1,~x_2,~x_3\in\R
\end{align*}
\end{minipage}
\begin{minipage}{0.5\textwidth}
\begin{align*}
\max~& -2y_{34}\\
s.t.~& Y=
\begin{bmatrix}
        0     &        0   &    0      &     0    \\
        0     &     a^2+\Delta^2&    0      &     a    \\
        0     &        0   &    0      &     0    \\
        0     &        a   &    0      &     1  \\
\end{bmatrix}\\
&Y\succeq 0
\end{align*}
\end{minipage}\\
~\\
~\\
The optimum primal solution satisfies $\overline x_2=0$ (row 2 and column 2 need to be
zero using Corollary~\ref{corolZeros}, because position (2,2) is zero) and $\overline x_3=1$ (row 3 and column 3
need to be zero because position (3,3) is $x_2=0$).
The primal optimum solution is 1.

\noindent On the dual side, $y_{11}$ needs to be zero because the objective coefficient of
$x_1$ is zero. As such, all elements on row 1 and column 1 of $Y$ need to be
zero. The constraint corresponding to $x_2$ stipulates $2y_{12}+y_{33}=0$. Since
$y_{12}=0$, we need to have $y_{33}=0$. This means that the row 3 and column 3
contain only zeros, and so, $y_{34}=0$, \ie, the {\it dual optimum objective value is 0}.
Finally, let us fill the remaining elements of the optimal $Y$ and check its
feasibility. The dual constraint corresponding to $x_3$ imposes $y_4=1$.
There is no constraint on $y_{24}=a$; $y_{22}$ needs to be greater than or equal
to $a^2$ so as to have a non-negative principal minor corresponding to
rows/columns 2 and 4. 
\end{proof}

\begin{proposition}\label{propDualCanBeInfeasible} The dual 
DSDP~\eqref{dsdp1}-\eqref{dsdp3}
of a feasible SDP \eqref{sdp1}-\eqref{sdp3} is
not necessarily feasible.
\end{proposition}
\begin{proof}
We have actually already shown an example that shows this at point
(b) in the proof of Prop~\ref{propTwiceDualize}. Let us modify a bit
this example for the sake of diversity, by adding a variable $x_3$.
Consider the following primal program that has at least the feasible solution
$x_1=x_2=x_3=0$.
\begin{align*}
\min~&  x_3-x_2\\
s.t.~&
\begin{bmatrix}
        x_1   &        x_2 &    x_3    \\
        x_2   &        0   &    0      \\
        x_3   &        0   &    x_2    \\
\end{bmatrix}\succeq \zeros\\
&x_1,~x_2,~x_3\in\R
\end{align*}
The dual of this program is infeasible. Since $x_1$ has a null objective
function coefficient, $y_{11}$ needs to be zero. This means (Corollary~\ref{corolZeros}) that all elements on
row 1 and column 1 of $Y$ need to be zero. We thus obtain $y_{13}=0$. On the
other hand, the dual constraint corresponding to $x_3$ stipulates that
$y_{13}=\frac 12$, which is a contradiction. Finally, notice $-x_2$ is not necessary
in the primal objective value.
\end{proof}

\subsection{\label{secstrongdualSDP}Strong duality}

Several results from this section (including the final proof of the strong
duality) are taken from a course of Anupam Gupta, also using arguments
from the lecture notes of L\' aszl\' o Lov\' asz.%
\footnote{As of 2017, they are available, respectively at
\url{http://www.cs.cmu.edu/afs/cs.cmu.edu/academic/class/15859-f11/www/notes/lecture12.pdf}
and
\url{http://www.ime.usp.br/~fmario/sdp/lovasz.pdf}.\setcounter{testfoot}{\value{footnote}}\label{testfootpage}}
\subsubsection{Basic facts on the cone of SDP matrices and the cone of definite positive matrices}
\begin{proposition} \label{propSDPcones}The SDP matrices of size $n\times n $ form a 
closed convex cone $S_n^+$ (see also Def.~\ref{defConeConvex}). The set of
positive \textbf{definite} matrices form an open
cone $\inter(S_n^+)$, see also the interior Definition~\ref{defConeInter}. The closure of $\inter(S_n^+)$ 
is $S_n^+$.
\end{proposition}
\begin{proof} 
Most statements follow by applying the definitions.
First, it is easy to prove that any $X\succeq\zeros$ for which there is a non-zero 
$\v\in\R^n$ such that $\v^\top X\v=0$ does not belong to the interior of
$S_n^+$. We only need to show there is \textit{no} open ball centered at $X$ which is
completely contained in $S_n^+$ (see also the interior
Definition~\ref{defConeInter}). This follows from the fact that $S_n^+$
contains no element from the set $\{X-\varepsilon I:~\varepsilon>0\}$ because
$\v^\top (X-\varepsilon I)\v = 0 -\varepsilon |\v|^2<0$,
where $|\v|$ is the 2-norm $\sqrt{v_1^2+v_2^2+\dots+v_n^2}$ of $\v$.

On the other hand, a matrix $X\succ \zeros$ does belong to the interior $\inter(S_n^+)$ as it does
contain an open ball centered at $X$. 
It is not hard to show there exists a sufficiently small $\varepsilon>0$ such that 
$X+\varepsilon Y$ remains positive definite
for any symmetric matrix of bounded 2-norm. 
Take any $\v\in R^n$ such that $|\v|=1$.
Let $\lambda_1>0$ be the minimum eigenvalue of $X$ so that 
$\v^\top X \v\geq \lambda_1$
using Lemma~\ref{lemmaMinRayleigh}.
We now develop 
$\v^\top \left (X+\varepsilon Y\right) \v=
\v^\top X \v
+
\varepsilon
\v^\top Y \v
\geq \lambda_1+ \varepsilon \v^\top Y \v$. 
Since both $\v$
and $Y$ have bounded norm, $\v^\top Y \v$ is bounded; this 
way, there exists a sufficiently small $\varepsilon>0$ that makes
$\varepsilon \v^\top Y \v$ ridiculously small compared to
$\lambda_1$. This shows that 
$\v^\top \left (X+\varepsilon Y\right) \v>0~$ for any  $\v$ of norm 1, equivalent to $X+\varepsilon Y\succ \zeros$.


It is easy to verify that both 
$S_n^+$ and $\inter(S_n^+)$
are convex cones.
If $X\succeq \zeros$ (resp.~$X\succ \zeros$), for any $\alpha>0$ we have 
$\alpha X\succeq \zeros$ (resp.~$\alpha X\succ \zeros$) because $\v^\top (\alpha
X)\v=\alpha \v^\top X \v\geq 0$ (resp.~$>0$) for any $\v\in\R^n-\{\zeros\}$.
One can also confirm the convexity: for any $\alpha\in(0,1)$ and $X,Y\succeq
\zeros$ (resp.~$X,Y\succ\zeros$), we have that 
$Z_\alpha=\alpha X+ (1-\alpha) Y$ verifies $Z_\alpha\succeq\zeros$
(resp.~$Z_\alpha\succ \zeros$), because
$\v^\top Z_\alpha\v=\alpha\v^\top X\v+(1-\alpha)\v^\top Y\v\geq 0$ (resp.~$>0$)
for any $\v\in\R^n-\{\zeros\}$.

To prove that $S_n^+$ is closed, we need to show $S_n^+$ contains all its limit
points.  Assume the contrary: there is a sequence $\{X_i\}$ with
$X_i\succeq\zeros~\forall i\in\N^*$ such that $\lim\limits_{i\to\infty}X_i=Z\nsucceq \zeros$.
This means there exists a non-zero rank 1 matrix $V=\v\v^\top$ such that
$Z\sprod V=-a<0$. As such, $\lim\limits_{i\to\infty}X_i\sprod V=Z\sprod V=-a<0$.
The convergence definition states that for any $\varepsilon>0$ there exists an
$m\in \N$ such that $X_i\sprod V\in [-a-\varepsilon,-a+\varepsilon]$ for any $i\geq
m$. Taking an $\varepsilon<a$, we obtain that $X_m\sprod V<0$ which means
$X_m\nsucceq\zeros$, contradiction. All limit points
$\lim\limits_{i\to\infty}X_i$ have to belong to $S_n^+$.

We still need to prove that the closure of $\inter(S_n^+)$ is $S_n^+$. Since all
limit points of $S_n^+$ belong to $S_n^+$ (above paragraph), all limit points of 
$\inter(S_n^+)\subset S_n^+$ have to belong to $S_n^+$ as well. We only need to
prove that any $X\succeq\zeros$ is a limit point of a sequence $\{X_i\}$ such
that $X_i\succ\zeros~\forall i\in \N^*$. It is enough to take $X_i=X+\frac
1iI_n$ and one can check that $X_i=X+\frac 1iI_n\succ\zeros$ because  $\v^\top (X+\frac 1iI_n)\v=
\v^\top X\v +\frac 1i\v^\top I\v\geq \frac 1i|\v|^2>0
$ for all $\v\in\R^n-\{\zeros\}$.
The convergence to $X$ is easy to prove
using
$\displaystyle\lim_{i\to\infty} \frac 1iI_n=\zeros$.

\end{proof}

\subsubsection{The proof of the strong duality}

We need the following proposition.

\begin{proposition}\label{propNoDPsolution} Let $F(\x)=\sum\limits_{i=1}^n x_iA_i-B$ for any 
$\x\in\R^n$, where all matrices $B$ and $A_i$ (with $i\in[1..n]$) are symmetric.
$$F(\x)\nsucc\zeros~\forall \x\in\R^n \iff$$
\begin{equation}\label{eqhyperplane}
\exists Y\succeq\zeros,~Y\neq \zeros\text{ such that }
A_i\sprod
Y=0~\forall i\in[1..n] \textnormal{ and }
F(\x)\sprod Y = -B\sprod Y \leq 0
,~\forall \x\in\R^n.
\end{equation}
In other words, if the sub-space generated by
$A_1,~A_2,\dots A_n$ with basis $-B$ does not touch
$\inter(S_n^+)$, then this sub-space belongs
a hyperplane $\{X:~X\sprod Y=-B\sprod Y,~X\text{
symmetric}\}$. 
\end{proposition}
\begin{proof}~\\
$\Longleftarrow$\\
If $F(\x)\sprod Y\leq 0$, $F(\x)$ can not be positive definite because any
$Z\succ\zeros$ verifies $Z\sprod Y>0$ for non-zero $Y\succeq\zeros$. To check this, use the eigenvalue
decomposition (\ie, \eqref{eigendecompnice} of Proposition
\ref{propEigenDecomp}) to write $Y$ as a sum of non-zero rank 1 matrices
of the form $\v\v^\top$. We have $Z\sprod \v\v^\top>0$ by 
Definition~\ref{def1}.

\noindent $\Longrightarrow$~\\
The interior $\inter(S_m^+)$ of the SDP cone 
does not intersect the image of $F$. 
We can apply the hyperplane separation Theorem~\ref{thHyperGeneral}: there exists a non-zero symmetric  
$Y\in \R^{m\times m}$ and a $c\in \R$ such that 
\begin{equation}\label{eqStrictIneq}
F(\x)\sprod Y \leq c\leq X\sprod Y
~\forall \x\in\R^n,~\forall X\in \inter(S_m^+)
\end{equation}

It is clear that we can not have $c>0$ because $X\sprod Y$ can be arbitrarily
close to $0$ by choosing $X=\varepsilon I_m$ for an arbitrarily small $\varepsilon>0$. We now
prove $X\sprod Y\geq 0 ~\forall X\in \inter(S_m^+)$. Let us assume the contrary:
$\exists X\in \inter(S_m^+)$ such that $X\sprod Y=c'<0$. 
By the cone property of $\inter(S_m^+)$, we have $tX\in \inter(S_m^+)~\forall t>0$. The value 
$(tX)\sprod Y=tc'$ can be arbitrarily low by choosing an arbitrarily large
$t$, and so, $(tX)\sprod Y$ can be easily less than $c$, contradiction.
This means that $X\sprod Y\geq 0$ for all $X\in \inter(S_m^+)$.

Based on \eqref{eqStrictIneq} and on the fact that $c\leq 0$, we obtain:
\begin{equation*}
F(\x)\sprod Y
\leq 0\leq 
X\sprod Y
~\forall \x\in\R^n,~\forall X\in \inter(S_m^+)
\end{equation*}
We prove that $Y\succeq \zeros$. For this, we show that 
$\overline{X}\sprod Y\geq 0~\forall \overline{X}\in S_m^+$.
Assume the contrary: there is some $\overline{X}\in S_m^+$ such that
$\overline{X}\sprod Y<0$. 
For any $\varepsilon>0$, we have 
$\overline{X}+\varepsilon I_m
\in
\inter(S_m^+)$. 
For a small enough $\varepsilon$, 
$(\overline{X}+\varepsilon I_m)\sprod Y$ remains strictly negative, which
contradicts
$\overline{X}+\varepsilon I_m
\in
\inter(S_m^+)$. 
We obtain $\overline{X}\sprod Y\geq 0~\forall 
\overline{X}\succeq \zeros$.
Using the fact that the SDP cone is self-dual (Prop~\ref{propSDPSelfDual}), we
obtain $Y\succeq \zeros$. We have just found an SDP matrix $Y$ such that
\begin{equation}\label{eqStrictIneqBisBis}
F(\x)\sprod Y\leq 0~\forall \x\in\R^n.
\end{equation}

We still need to show that $F(\x)$ is constant for every $\x\in\R$, which
is equivalent to $A_i\sprod Y=0~\forall i\in[1..n]$. This is not difficult.
If there is a single $\x\in\R^n$ such that $F(\x)=F(\zeros)+\Delta$ with
$\Delta \neq 0$, notice $F(t\x)$ can become arbitrarily large using an appropriate
value of $t$ (\ie, use $t\to\infty$ for $\Delta>0$ or $t\to-\infty$ otherwise),
which is impossible. We thus need to have:


$$F(\x)\sprod Y = F(\zeros)\sprod Y = -B\sprod Y\leq 0~\forall \x\in\R^n,$$
where we used \eqref{eqStrictIneqBisBis} for the last inequality.
\end{proof}

For the reader's convenience, we repeat the definitions of the SDP program~\eqref{sdp1}-\eqref{sdp3}
and resp.~of its dual~\eqref{dsdp1}-\eqref{dsdp3}:
\begin{equation}\label{sdpBis}
(SDP)~~\min\left\{\sum_{i=1}^n c_ix_i:~ \sum_{i=1}^n A_i x_i \succeq B,~\x\in \R^n\right\},
\end{equation}
and
\begin{equation}\label{dsdpBis}
(DSDP)~~
\max\left\{B\sprod Y:~A_i\sprod Y = c_i~\forall i\in[1..n],~Y\succeq \zeros\right\}.
\end{equation}

\begin{theorem}\label{thStrongDual1} If the primal $(SDP)$ from \eqref{sdpBis} is bounded and has a strictly
feasible solution (Slater's interiority condition), then the primal and the dual optimal values are the same and the dual
$(DSDP)$ from
\eqref{dsdpBis} reaches this optimum value. Recall (Prop.~\ref{propMainDuality}) 
that if $(SDP)$ is unbounded, then $(DSDP)$ is infeasible.
\end{theorem}
\begin{proof} Let $p$ be the optimal primal value. The system $\sum\limits_{i=1}^n c_ix_i<p$ and
$\sum_{i=1}^n A_i x_i \succeq B$ has no solution. We define
$$A'_i=
\begin{bmatrix}
-c_i       & \zeros_n^\top    \\
\zeros_n  &  A_i     \\
\end{bmatrix}~
\forall i\in [1..n]
\text{ and }
B'=
\begin{bmatrix}
-p       & \zeros_n^\top \\
\zeros_n & B\\
\end{bmatrix}
$$
and observe that $\sum\limits_{i=1}^n A'_i x_i - B'\nsucc \zeros$ $\forall \x\in\R^n$
(we can \textit{not} say $\sum\limits_{i=1}^n A'_i x_i - B'\nsucceq \zeros$, as
the optimal solution $\x$ can cancel the top-left term of the expression). We can thus apply 
Prop.~\ref{propNoDPsolution} (implication ``$\Longrightarrow$'') and conclude there is some non-zero $Y'\succeq \zeros$ such that
$A'_i\sprod Y'=0~\forall i\in[1..n]$ and $-B'\sprod Y'\leq 0$. Writing 
$Y'=\left[\begin{smallmatrix}
t      & \ldots \\
\vdots &    Y\\
\end{smallmatrix}\right]$, 
we obtain:
\begin{equation}\label{eqeg1} tc_i=A_i\sprod Y,~\forall i\in[1..n]\end{equation}
and 
\begin{equation}\label{eqeg2}-B\sprod Y \leq -tp.\end{equation}

We now prove $t>0$ by contradiction.
Supposing $t=0$, we obtain $A_i\sprod Y=0~\forall i\in[1..n]$ and $-B\sprod Y\leq 0$. Applying again
Prop.~\ref{propNoDPsolution} (implication ``$\Longleftarrow$''), we conclude $\sum\limits_{i=1}^n A_i x_i - B\nsucc\zeros~\forall \x\in
\R^n$, which contradicts the fact the primal~\eqref{sdpBis} is strictly feasible. We need to have $t>0$.

Taking SDP matrix $\overline{Y}=\frac 1t Y$, \eqref{eqeg1}-\eqref{eqeg2} become: $c_i=A_i\sprod
\overline{Y}~\forall i\in[1..n]$ and $B\sprod\overline{Y}\geq p$. In other words, 
$\overline{Y}$ is a feasible solution in the dual~\eqref{dsdpBis} and it has an objective value
$B\sprod\overline{Y}\geq p$. Using the weak duality \eqref{eqWeakDual}, $B\sprod\overline{Y}\leq p$,
and so, $B\sprod\overline{Y}=p$, \ie, the dual achieves the optimum primal value.
\end{proof}

\begin{theorem}\label{thStrongDual2} If the dual $(DSDP)$ from 
\eqref{dsdpBis} is bounded and has a strictly feasible solution,
then the primal and the dual
optimal values are the same and the primal $(SDP)$ from \eqref{sdpBis} reaches this optimum value.
Recall (Prop.~\ref{propTwiceDualize}) that if $(DSDP)$ is unbounded, then $(SDP)$ is infeasible.
\end{theorem}
\begin{proof}
Apply Theorem~\ref{thStrongDual1} and Prop.~\ref{propDualPrimalFormDualizedIntoSmtngEqToPrimal}.
The main idea is to write the dual $(DSDP)$ from \eqref{dsdpBis} in the primal form (this is
possible using Prop.~\ref{propEqPrimDualForms}). Theorem~\ref{thStrongDual1} states that 
the dual of this primal form reaches the optimum solution. But the dual of this primal form is
exactly equivalent to the primal $(SDP)$ from \eqref{sdpBis} by virtue of Prop.~\ref{propDualPrimalFormDualizedIntoSmtngEqToPrimal}.
\end{proof}

\begin{theorem}\label{thStrongDual3} If both $(SDP)$ and $(DSDP)$ are strictly
feasible, then $OPT(SDP)=OPT(DSDP)$ and this value is reached by both programs.
\end{theorem}
\begin{proof}
Using Prop.~\ref{propMainDuality}, if $(DSDP)$ has a feasible solution, then
$(SDP)$ is bounded.
Using Prop.~\ref{propTwiceDualize}, if $(SDP)$ has a feasible solution,
$(DSDP)$ is bounded. We can now apply Theorems~\ref{thStrongDual1} and
\ref{thStrongDual2} to obtain the desired result.
\end{proof}

\subsubsection{Further properties on the intersection of the SDP cone with a sub-space}
\begin{proposition}
Let $F(\x)=\sum\limits_{i=1}^n x_iA_i-B$. If the image of $F$ (\ie, the space spanned
by $A_i$ with $i\in[1..n]$ and basis $-B$) does not intersect the SDP cone $S_m^+$, then
there exists
$ Y\succeq\zeros \text{ such that }F(\x)\sprod Y = -B\sprod Y < 0~\forall \x\in \R^n$.
\end{proposition}
\begin{proof}~\\
\noindent $\Longleftarrow$\\
We know that $Y\succeq \zeros$ 
satisfies 
$F(\x)\sprod Y<0$
for any $\x\in\R^n$. Since all $X\succeq \zeros$ verify $X\sprod Y\geq 0$, $F(\x)$ can
not belong to the SDP cone.

\noindent $\Longrightarrow$\\
There exists a sufficiently small $\varepsilon$ such that (the image of)
$F_\varepsilon(\x)=\sum_{i=1}^n x_iA_i-B+\varepsilon I_m$ still does not
intersect the SDP cone. This means that this image does not intersect
$\inter(S^+_m)$ either and
we can apply~Prop~\ref{propNoDPsolution}
on $F_\varepsilon$. There exists a non-zero $Y\succeq \zeros$ such that
$F_\varepsilon(\x)\sprod Y = \left(-B+\varepsilon I_m\right)\sprod Y \leq 0~\forall \x\in \R^n$. It is now enough to
check that for any $\x\in\R^n$ the following value is constant:
$F(\x) \sprod Y=F_\varepsilon(\x)\sprod Y - \varepsilon I_m \sprod Y<0$;
we used
$I_m \sprod Y>0$ which follows from $\texttt{trace}(Y)>0$ (or apply Prop.~\ref{propABprod}).
\end{proof}

\begin{proposition}
Let $F(\x)=\sum\limits_{i=1}^n x_iA_i-B$
such that $\exists \x_0\in\R^n$ such that $F(\x_0)=\zeros$.
If the image of $F$  intersects the SDP cone $S_m^+$
only in the origin $\zeros$, 
then
there exists
$Y\succ\zeros \text{ such that }F(\x)\sprod Y = 0~\forall \x\in \R^n$.
\end{proposition}
\begin{proof}~\\
\noindent $\Longleftarrow$\\
We know that $Y\succ\zeros$ 
satisfies $F(\x)\sprod Y=0$
for any $\x\in\R^n$.
It is not hard to prove that all non-zero $X\in S^+_m$ satisfy $X\sprod Y>0$,
because Prop.~\ref{propABprod} states that $X\sprod Y=0 \implies
XY=\zeros$, which leads to $X=\zeros Y^{-1}=\zeros$ (because $Y\succ\zeros$ is non-singular).
Thus $F(\x)$ can
not cover any non-zero SDP matrix.

\noindent $\Longrightarrow$\\
Let $S^*_m=\{X\in S^+_m:~\texttt{trace}(X)=1\}$. It is not hard
to check that $S^*_m$ is convex, closed and bounded. 
The image $\texttt{img}(F)$ of $F$ is a closed convex set.

We show that 
$S^*_m-\texttt{img}(F)=\{X^a-X^b:~X^a\in S^*_m,~X^b\in
\texttt{img}(F)\}$ is closed. We take any convergent sequence 
$\{X^a_i-X^b_i\}$ and we will show the limit point belongs to
$S^*_m-\texttt{img}(F)$. 
It is not hard to see that any $X^a\in S^*_m$ is bounded in the sense that it satisfies
$X^a_{ij}\leq 1~\forall i,j\in[1..m]$, because the 2$\times 2$ minor corresponding to 
rows/columns $i$ and $j$ has to be non-negative.
The sequence $\{X^b_i\}$ (with $i\to \infty$) needs to be bounded, because otherwise
$\{X^a_i-X^b_i\}$ would be unbounded, and so, non-convergent. Using
the Bolzano-Weierstrass theorem (Theorem~\ref{thBolzanoWeierstrass}), the
bounded sequence $\{X^a_i\}$ has a convergent sub-sequence 
$\{X^a_{n_i}\}$. Using the Bolzano-Weierstrass theorem again,
the sub-sequence $\{X^b_{n_i}\}$ contains a convergent sub-sub-sequence 
$\{X^b_{m_i}\}$, with $\{m_i\}\subseteq \{n_i\}$. Since $S^*_m$ and $\texttt{img}(F)$ are closed,
$\lim\limits_{i\to \infty}X^a_{m_i}=X^a\in S^*_m$ and
$\lim\limits_{i\to \infty}X^b_{m_i}=X^b\in\texttt{img}(F)$, and so,
$\lim\limits_{i\to \infty}X^a_{m_i}-X^b_{m_i}=X^a-X^b\in
S^*_m-\texttt{img}(F)$, \ie, $S^*_m-\texttt{img}(F)$
contains all its limit points.

Since $\zeros\notin S^*_m-\texttt{img}(F)$, the simple separation
Theorem~\ref{thSimpleSep} states there is an $Y$ such that $(X^a-X^b)\sprod
Y>0,~\forall X^a\in S^*_m,~X^b\in \texttt{img}(F)$. This is
equivalent to 
\begin{equation}\label{eqtmp}
X^a\sprod Y>X^b\sprod Y,~\forall X^a\in S^*_m,~X^b\in
\texttt{img}(F).
\end{equation}
Since $\zeros\in\texttt{img}(F)$, we have $X^a\sprod Y>0~\forall X^a\in
S^*_m$. This is equivalent to 
$X^a\sprod Y>0~\forall X\in
{S^+_m}\setminus\{\zeros\}$, and so, $Y\sprod (\v\v^\top)>0~\forall \v\in\R^n-\{\zeros\}$, \ie, 
$Y\succ\zeros$.

We still need to show that $F(\x)\sprod Y=0~\forall \x\in\R^n$. Assume there
exists $\x\in\R^n$ such that $F(\x)\sprod Y=\Delta \neq 0$. 
Recall from hypothesis that there exists $\x_0\in\R^n$ such that $F(\x_0)=\zeros$.
We can write $\x=\x_0+\x_1$ and we obtain $F(\x_0+\x_1)=d$.
For any $\alpha\in\R$, we
have $F(\x_0+\alpha\x_1)\sprod Y=\alpha \Delta$, and so, $F(\x_0+\alpha\x_1)\sprod Y$ can be arbitrarily large,
violating \eqref{eqtmp}.
\end{proof}

\subsection{The difficulty of exactly solving (SDP) and algorithmic comments}

We first notice that $\max\left\{y_{12}:~
    \left[\begin{smallmatrix}
        1 &  y_{12}\\
      y_{21}& 2
    \end{smallmatrix}\right]\succeq \zeros\right\}$ is $\sqrt{2}$. This means it is
rather unlikely that a purely numerical algorithm based on binary encodings can exactly optimize any SDP program.
It is however possible to find the optimum value of any SDP program
up to any specified additive error (precision)
using
ellipsoid, interior point, spectral or conic bundle methods.

We now show that exactly solving SDP is at least as hard as the square-root sum problem
whose exact complexity is still an open problem. Consider the following feasibility program in variables
$x_2,~x_3,\dots x_n$.%
\footnote{See the article ``Semidefinite programming and combinatorial optimization'' by 
Michel Goemans, in
the 
\textit{International Congress of Mathematicians}, Volume III, Documenta Mathematica, Extra Vol.~ICM III, 1998, 657-666, 
page 2. The article is available on-line as of 2017 at 
\url{https://www.math.uni-bielefeld.de/documenta/xvol-icm/17/Goemans.MAN.ps.gz}.
}
$$
\begin{bmatrix}
1                   & k-(x_2+x_3+\dots +x_n)&\\
k-(x_2+x_3+\dots + x_n)&  a_1                   \\
                    &                     & 1 & x_2 \\
                    &                     & x_2 & a_2 \\
                    &                     &     &     & 1  &  x_3  \\
                    &                     &     &     & x_3  &  a_3  \\
                    &                     &     &     &      &      & \ddots \\
                    &                     &     &     &      &      &         &  1   &   x_n  \\
                    &                     &     &     &      &      &         &  x_n &   a_n \\
\end{bmatrix}
\succeq
\zeros.$$
This program is feasible if and only if $k\leq \sum_{i=1}^n \sqrt {a_i}$. This is the 
square-root sum problem: given $k$ and $a_1,~a_2,\dots a_n$, decide if $k\leq 
\sum_{i=1}^n \sqrt{a_i}$. It is still an open question if this problem is polynomial or not,
using at least the on-line forum \url{http://www.openproblemgarden.org/op/complexity_of_square_root_sum}. 
However, the same link reads that the SDP optimum can be determined (approximated) with 
any specified additive accuracy (error) $\varepsilon$ 
 using the interior point method or the ellipsoid algorithm in time polynomial in the size of the instance and $\log 1/\varepsilon$.

To show the difficulty of exactly solving SDP programs, we consider the following minimization problem.
\begin{align*}
\min &~x   \\
& 
\begin{bmatrix}
x  &  1   &   3  \\
1  & x+2  &  0   \\
3  &  0   & x+1
\end{bmatrix}\succeq\zeros
\end{align*}
The determinant of the whole matrix is $f(x)=x^3+3x^2-8x-19$. 
We first present an intuitive figure (graph)
\begin{wrapfigure}{r}{0.43\textwidth}
\vspace{-1.7em}
  \begin{center}
    \includegraphics[width=0.42\textwidth]{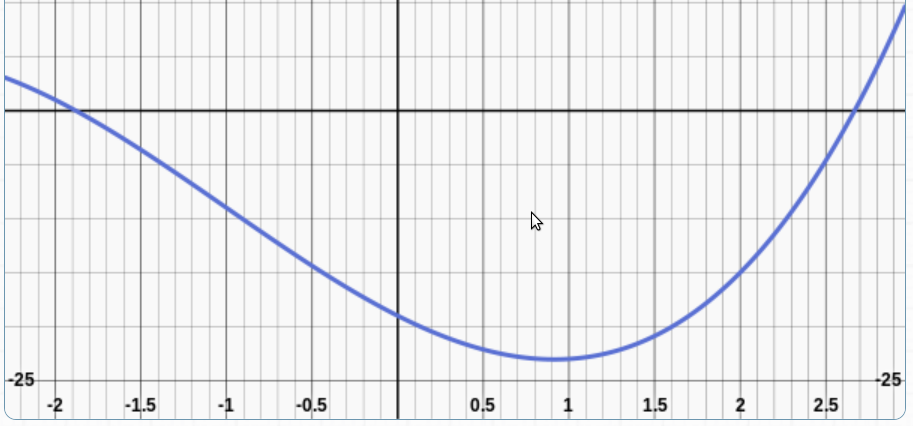}
  \end{center}
\begin{tikzpicture}[overlay]
\node at (6.7,3)     {\scalebox{0.95}{$x^*$}};
\node at (3.25,3.02) {\scalebox{0.75}{$0$  }};
\node at (3.35,1.69) {\scalebox{0.75}{$-19$}};
\end{tikzpicture}
\vspace{-2.7em}
\caption{The graph of $f(x)=x^3+3x^2-8x-19$}
\vspace{-0.2em}
\end{wrapfigure}
of $f$ on the right%
\footnote{Obtained using 
\url{http://derivative-calculator.net}.} 
 and then we will
formally determine the values of $x\geq 0$ such that $f(x)\geq 0$.
We are not interested in values $x<0$ because the leading principal minor
of size 1 would be $x<0$. The optimum of above program is clearly greater than
0.

By calculating the second derivative 
$f''(x)=6x+6$, it is clear that $f$ is strictly convex over $[0,\infty)$.
Since $f'(0)=-8$, the function 
first decreases when going from $0$ to $\infty$
and, after a point, it increases.
This means $f$ has a unique root $x^*\in[0,\infty)$ and the optimum of above program is
at least $x^*$. Before determining this root, notice the figure intuitively shows that
$x^*>2.6$. We can also formally check that $f(2.6)=-1.944$, and, by convexity, we
surely have $x^*>2.6$. 
Other principal minors do not pose any problem because they are surely non-negative for
any $x^*>2.6$. 
The most problematic one is the one associated to
rows and columns 1 and 3. Since $2.6*3.6=9.36>9$, this principal minor is
positive for $x=2.6$ and it remains so by increasing $x$.

This means that the optimum value of above program is equal to the largest root
$x^*$ of the cubic equation $f(x)=0$.  To determine this root, let us write
$f(x)=(x+1)^3-11(x+1)-9$. Using Cardano's formula,%
\footnote{See the wikipedia article 
\url{https://en.wikipedia.org/wiki/Cubic_function}.}                                    
we obtain:
$$x^*=
\sqrt[3]{\frac 92+\sqrt{\frac{3137}{108}}i}
+
\sqrt[3]{\frac 92-\sqrt{\frac{3137}{108}}i}
-1 
~\approx~
2.6679286.
$$
It is highly unlikely that an SDP algorithm could exactly determine such values in
polynomial time,
solving cubic equations like $f(x)=0$ or possibly other
higher degree equations.
 On the other
hand, the numerical methods mentioned above can determine the optimum value up
to any specified additive error.

\section{Interesting SDP programs}
\subsection{An SDP program does not always reach its min (inf) or max (sup) value}
The value
$$\min\left\{t:~
\begin{bmatrix}
t & 1 \\
1 &t' 
\end{bmatrix}
\succeq \zeros \right\}$$
is zero but {\it no} feasible solution reaches this optimum objective value of zero. However, we could use ``$\min$'' to actually mean ``$\inf$''. Since the program is bounded and strictly feasible 
($\eg$, for $t=t'=2$, we have $\left[\begin{smallmatrix}2&1\\1&2\end{smallmatrix}\right]\succ \zeros$)
we can use the strong duality (Theorem~\ref{thStrongDual1}) to conclude that the dual does 
achieve the optimum $0$. One can easily check that the dual has only one feasible point 
$\left[\begin{smallmatrix}1&0\\0&0\end{smallmatrix}\right]$ with objective value 0.

An example of 
a program 
in the dual
form~\eqref{dsdp1}-\eqref{dsdp3}
can be found by re-formulating the above primal program using the method from Section~\ref{secPrimalToDual};
aiming at a maximization form, one may find:
$$\max\left\{-Y_{11}:~\begin{bmatrix}0&1\\1&0\end{bmatrix}\sprod Y=2,~Y\succeq \zeros\right\}.$$

\subsection{The lowest and greatest eigenvalue using the SDP duality}

Consider symmetric matrix $X\in\R^{m\times m}$ with 
eigenvalues $\lambda_1\leq \lambda_2\leq \dots \leq \lambda_n$.
The following program reports the smallest eigenvalue.
\begin{align*}[left ={\lambda_1=  \empheqlbrace}]
\max~~&t \\
s.t~~&X-tI_m\succeq \zeros\\
    &t\in \R             
\end{align*}

Observe this program is bounded and strictly feasible (take $t\to-\infty$), and so, the strong duality
Theorem~\ref{thStrongDual1} states that the dual reaches the optimum value $\lambda_1$.
By dualizing after transforming ``$\max t$'' into ``$-\min -t$'', we obtain 
an objective ``$-\max -X\sprod Y$'', that is equivalent to ``$\min X \sprod Y$''.
\begin{align*}[left ={\lambda_1=  \empheqlbrace}]
\min~~&X\sprod Y \\
s.t~~&\texttt{trace}(Y) = 1\\
    &Y\succeq \zeros
\end{align*}
We have obtained a more general version of Lemma~\ref{lemmaMinRayleigh} in which the rank of $Y$ is
not necessarily 1 (as in Lemma~\ref{lemmaMinRayleigh}). Furthermore, the lowest eigenvalue 
of $-X$ is $-\lambda_n$. The first above program for $-X$ can be written, after replacing $t$ with
$-t$, as
$-\lambda_n = \max\big\{-t:~tI_m-X\succeq \zeros,~t\in \R\big\}$. As such, $\lambda_n$ can be
obtained by the following primal-dual programs. 

~\\

\noindent
\refstepcounter{equation}
  \begin{minipage}{0.5\textwidth}
  \begin{equation}\label{eqmaxeigenprimal}
  \tag{\theequation a}
    \lambda_n=\left\{
      \begin{array}{ll}
                \min~~&t \\
                s.t~~&tI_m-X\succeq \zeros\\
                      &t\in \R             
      \end{array}
    \right.
  \end{equation}
  \end{minipage}
  \begin{minipage}{0.5\textwidth}     
  \begin{equation}\label{eqmaxeigendual} 
  \tag{\theequation b}
        \lambda_n=\left\{
            \begin{array}{ll}
                \max~~&X\sprod Y \\
                s.t~~&\texttt{trace}(Y) = 1\\
                    &Y\succeq \zeros
            \end{array}
        \right.
  \end{equation}
  \end{minipage}

~\\

\subsection{Change of variable in SDP programs}

Let us focus on the dual SDP $(DSDP)$ form from \eqref{dsdp1}-\eqref{dsdp3}; recall
its constraints 
have the form $A_i\sprod Y = c_i~\forall i\in[1..n]$,
where $Y \succeq \zeros$ are the variables.
We will now change variables $Y$ into variables $Z=QYQ^\top$ for a non-singular matrix $Q$. 
Recalling Prop.~\ref{propCongruent},
such $Z$ and $Y$ are \textit{congruent} and they satisfy
$Z\succeq \zeros \iff Y \succeq \zeros$.
Notice that if
$Y=\y\y^\top$, then the variable change maps $\y\y^\top$ to $Q\y\y^\top Q^\top=(Q\y)(Q\y)^\top$, \ie,
we can also say vector $\y$ is mapped to $\z=Q\y$. 

One could try to express the dual program $(DSDP)$ in variables $Z$ by simply
replacing $Y$ with
$Q^{-1}Z {Q^{-1}}^\top$. 
But developing $A_i\sprod Y=A_i\sprod (Q^{-1}Z {Q^{-1}}^\top)$ by brute force
could be quite painful and lead to a mess.
We can obtain a more elegant reformulation using the lemma below.
\begin{lemma} Given symmetric matrices $A,B\in\R^{m\times m}$ and non-singular matrix $R$, the following
holds:
$$
A\sprod B = \left({R^{-1}}^\top A R^{-1}\right)\sprod \left(R B R^\top\right)
$$
This is a consequence of a more general result. Given symmetric matrices $A\in R^{n\times n}$ and $B\in R^{m
\times m}$ and non-symmetric $R\in \R^{n \times m}$, we
have:

$$R^\top A R \sprod B = A\sprod R B R^\top  $$
\end{lemma}
\begin{proof}
We only prove the second more general result, because the first one
simply follows from it by substituting 
$A={R^{-1}}^\top A R^{-1}$.

\begin{align*}
A \sprod \left(R B R^\top\right)
&=
\texttt{trace}\left(A  R B
R^\top\right)&\\
&=
\texttt{trace}\left(R^\top A R B
\right)& \textnormal{we applied~}\texttt{trace}(XY) =\texttt{trace}(YX)
 \\
&=R^\top A R \sprod B
\end{align*}

\end{proof}
Based on this lemma, we have $B\sprod Y=\left({Q^{-1}}^\top B Q^{-1}\right)\sprod \left(QYQ^\top\right)
=\left({Q^{-1}}^\top Y Q^{-1}\right)\sprod Z$. By applying the same calculations on $A_i\sprod X~\forall i\in[1..n]$, we 
obtain that the dual $(DSDP)$ program is equivalent to:

\begin{align*}[left ={(DSDP_Z)  \empheqlbrace}]
\max~&\left({Q^{-1}}^\top B Q^{-1}\right)\sprod Z    \\
s.t ~&\left({Q^{-1}}^\top A_i Q^{-1}\right)\sprod Z = c_i~\forall i\in[1..n] \hspace{-6em} \\
    ~&Z\succeq \zeros             
        \tag{recall $Z\succeq\zeros \iff Y\succeq \zeros$ since $Y$ and $Z$ are congruent}
\end{align*}

\subsection{\label{sec34}Convex quadratic programming is a particular case
of SDP programming}
We consider a convex quadratic program with $n$ variables $x_1,x_2,\dots x_n\in\R$:
\begin{subequations}
\begin{align}
\min~~& \x^\top A_0 \x + \b_0^\top \x \label{cqp1}\\
s.t.~~& \x^\top A_i \x + \b_i^\top \x\leq  c_i~\forall i\in[1..p]\label{cqp2}
\end{align}
\end{subequations}
Particularizing Prop.~\ref{propHessian}, a quadratic function $\x^\top A_i\x$ ($\forall i \in
[0..p]$) is convex if and only if the Hessian $2A_i$ is SDP, \ie, we need to have
$A_i\succeq\zeros~\forall i\in[0..p]$.

\subsubsection{\label{secxx}Reformulation using the (Cholesky) factorization of SDP matrices}

Each $A_i$ (with $i\in[0..p]$) can be factorized as $A_i=R_iR_i^\top$
using any of the presented decompositions of SDP matrices (\eg,
Cholesky, eigenvalue or square root, see Corollary \ref{corSdpToVectors}).
 Let us move the objective function into the constraint set by introducing 
a real variable $c_0$ that has to be minimized such that $\x^\top A_0 \x + \b_0^\top \x\leq c_0$.
After simple algebraic manipulations, 
the above program \eqref{cqp1}-\eqref{cqp2} can be written as:
\begin{align*}
\min~~& c_0 \\
s.t.~~& c_i - \b_i^\top \x- \x^\top R_i R_i^\top \x \geq  0~\forall i\in[0..p]
\end{align*}
Using the Schur complements Property~\ref{propSchurGen}, the above program is further equivalent to program below. Indeed,
by applying the Schur complement, one can easily check that a feasible solution of above program is feasible in program below
 and vice-versa.
\begin{align}
\min~~& c_0 \notag\\
s.t.~~& 
 \begin{bmatrix}
 I_n        &            R_i^\top \x  \\
 \x^\top R_i& c_i - \b_i^\top \x
 \end{bmatrix}\succeq \zeros~~~
 \forall i\in[0..p],\notag
\end{align}
which is a program in variables $x_1,~x_2,\dots x_n$ and $c_0$. This is an SDP program
with $p+1$ constraints (that could be expressed in an aggregated form \eqref{sdp1}-\eqref{sdp3}
with a unique constraint, see also Prop.~\ref{propAggregated}).

\subsubsection{Reformulation by relaxing $\x\x^\top$ into $X\succeq \x\x^\top$}
We will show that \eqref{cqp1}-\eqref{cqp2} is equivalent to the following SDP program:
\begin{subequations}
\begin{align}
\min~~& A_0 \sprod X + \b_0^\top \x \label{eqReforma} \\
s.t.~~& A_i \sprod X + \b_i^\top \x\leq  c_i~\forall i\in[1..p]\label{eqReformb}\\
      & \begin{bmatrix} 1 &\x^\top \\ \x & X \end{bmatrix}\succeq \zeros \label{eqReformc}
\end{align}
\end{subequations}
First, it is clear that any feasible solution $\x$ of
\eqref{cqp1}-\eqref{cqp2}
can be associated
to a feasible solution 
$X=\x\x^\top$
of \eqref{eqReforma}-\eqref{eqReformc} that has the same objective value.
Conversely, it is possible to show that a feasible solution \eqref{eqReforma}-\eqref{eqReformc} 
corresponds to feasible solution of 
\eqref{cqp1}-\eqref{cqp2}, but we still need a few lines.
Any $X$ and $\x$ that satisfy \eqref{eqReformc} also satisfy 
$X\succeq \x\x^\top$ 
by virtue of the Schur complements Prop.~\ref{propSchurPart}.
Thus, $X$ can be written in the form $X=\x\x^\top + S$ for some $S\succeq\zeros$.
Since $A_i\succeq \zeros~\forall i\in[0..p]$, each product $A_i\sprod S$ is non-negative.
This ensures the fact that $\x\x^\top$ satisfy all constraints \eqref{cqp2} and,
similarly, the objective value \eqref{eqReforma} satisfies
$A_0\sprod \left(\x\x^\top+S\right)+\b_0^\top \x
\geq 
A_0\sprod \x\x^\top+\b_0^\top \x$. In other words, $\x$ is a feasible solution of
\eqref{cqp1}-\eqref{cqp2} and its objective value is no worse than 
that of $X$ and $\x$ in \eqref{eqReforma}.

The above \eqref{eqReforma}-\eqref{eqReformc} is actually an SDP program
that can be easily written in the standard form \eqref{sdp1}-\eqref{sdp3}.
Both \eqref{eqReformb} 
and \eqref{eqReformc} are linear matrix inequalities like \eqref{sdp2} in 
variables $x_1,x_2,\dots,x_n$ and $X_{11}, X_{12},\dots X_{nn}$; \eqref{eqReformb} uses ``$1\times 1$ matrices''
and \eqref{eqReformc} uses $(n+1)\times (n+1)$ matrices.
We thus have two inequalities
of the form \eqref{sdp2} 
that can be aggregated into a unique constraint \eqref{sdp2} that uses block-diagonal matrices (Prop.~\ref{propAggregated}). 
As a final side remark, similar programs may even have constraints
like 
$
\left[\begin{smallmatrix}
 -y          &\frac 12 \b_0^\top \\
\frac 12 \b_0&  A_0
\end{smallmatrix}
\right]
\sprod 
\left[\begin{smallmatrix} 1 &\x^\top \\ \x & X \end{smallmatrix}\right]\leq 0$ 
that contain variables in both factors, because of variable $y$ in the left-hand term. However, after developing
the scalar product, the variable $y$ simply arises as a term in a final sum
as in \eqref{sdp2}.

\subsubsection{\label{secqpunconstrained}
Unconstrained quadratic programming reduces to SDP programming}

We consider an \textit{unconstrained} quadratic program :
\begin{align}
\min~~& \x^\top A \x + \b^\top \x \label{cqpunconstrained}
\end{align}
If $A$ is not SDP, we can show that this program is unbounded from below. For
this, let us 
take any eigenvector $\v$ such that $A\v=-\alpha \v$ with $\alpha>0$. Consider
the function $f(t)= (t\v^\top) A (t\v) + \b^\top \v=-\alpha t^2 \v^\top \v +
t\b^\top \v$. This is clearly a strictly concave function in $t$ and it goes to $-\infty$
when $t$ goes to $\infty$. 

To solve \eqref{cqpunconstrained} by SDP programming,
one first has to check if $A$ is SDP or not. The necessity of this step makes
the approach from Section~\ref{secxx} unusable here, because we here need a method
that handles at the same time both the case
$A\nsucceq \zeros$
and
$A\succeq \zeros$.

Let us now show how to address the case $A\succeq \zeros$ by SDP programming. We first need the following result 
which should not be taken from granted (although it is omitted from certain lecture notes).

\begin{proposition}\label{propConvexReachesOpt} If an unconstrained quadratic
program  $\x^\top A \x + \b^\top \x$ (\ie, a polynomial of degree 2) is bounded from below, there is 
a solution $\xx$ that does reach the minimum value. The polynomial is convex ($A\succeq \zeros$) and the
 gradient in $\xx$ is $\zeros$.

~\\
\noindent This does \textbf{not hold for polynomials of any
degree or for convex functions}.  Indeed, $p(x_1,x_2)=(1-x_1x_2)^2+x_1^2$ does not reach its
minimum (infimum) $0=\lim\limits_{m\to \infty}p(\frac 1m,m)$.
Function $h(x)=e^x$ is
infinitely differentiable, convex and bounded from below, but there is no $x$ with $h'(x)=0$ that does reach the minimum value.
On the other hand, 
it is actually possible to prove than
a convex polynomial (of any
degree) always reaches its minimum value, but the proof of such result is outside the scope of this document.%
\footnote{%
For a proof, see the response of R.~Israel on my question asked on
the on-line forum
\url{https://math.stackexchange.com/questions/2341163/does-a-convex-polynomial-always-reaches-its-minimum-value/}.
The example $(1-x_1x_2)^2+x_1^2$ is taken from a response of J.P. McCarthy on 
\url{https://math.stackexchange.com/questions/279497/polynomial-px-y-with-inf-mathbbr2-p-0-but-without-any-point-where}.}

\end{proposition}
\begin{proof}
We showed above that the polynomial is bounded only if $A\succeq \zeros$.


Let us show that if there is \textit{no} $\xx\in\R^n$ that cancels the gradient,
the program is \textit{unbounded}. The gradient can be written $\nnabla 
\left(\x^\top A \x + \b^\top \x\right)=2A\x+\b$, and so,
there is \textit{no} $\xx\in\R^n$ such that $2A\xx=-\b$. This means that $-\b$
(or $\b$) does not belong to the image 
(set of linear combinations of the columns)
of $2A$. Using a similar (transposed)
argument as in the first paragraph of the first proof of
Prop.~\ref{propRankRectangular}, $-\b$ needs to have the form
$-\b=-\b_{\texttt{img}}-\b_0$, where 
$-\b_{\texttt{img}}\in {\texttt{img}}(2A)$
and $-\b_0\in{\texttt{null}}(2A)$ with $\b_0\neq \zeros$, 
where 
${\texttt{img}}(2A)$ is the image of $2A$ and 
${\texttt{null}}(2A)$ is the null space of $A$ (see the null space definition in
\eqref{eqDefNull}).

Taking $\x=t\b_0$, 
we can now define $h(t)=(t\b_0)^\top A (t\b_0)+(\b_{\texttt{img}}+\b_0)^\top
(t\b_0)=0+t\b_0^\top \b_0=t|\b_0|^2$, where we used $\b_0\in\texttt{null}(A)\implies A\b_0=\zeros$. Since $\b_0\neq\zeros$, we have
$|\b_0|\neq 0$, and so, $h(t)$ is a non-zero linear function that is clearly
unbounded from below. Since $\x=t\b_0$ is a perfectly feasible solution of
our unconstrained program, this program is unbounded from below. We proved that 
if there is \textit{no} stationary point $\xx\in\R^n$ that cancels the gradient, the program is
\textit{unbounded}. This means that a bounded convex quadratic program needs to have some stationary point 
$\xx\in\R^n$ such that
$\nnabla
\left(\x^\top A \x + \b^\top \x\right)_{\xx}=2A\xx+\b=\zeros$. For a convex function, the
stationary point needs to be the global minimum, and so, $\xx$ reaches the minimum
value of $\x^\top A \x +\b^\top \x$.
\end{proof}

We hereafter consider \eqref{cqpunconstrained} is convex ($A\succeq \zeros$).

\begin{proposition}\label{propCqpunconstrained}The optimum value of
unconstrained convex quadratic program~\eqref{cqpunconstrained}
is equal to:
\begin{subequations}
\begin{align}
t^*=\max~~& t  \label{eqtmptmp0} \\
s.t.~ &
\begin{bmatrix}
    -t            &   \frac 12 \b^\top    \\
\frac 12 \b       &              A
\end{bmatrix}
\succeq \zeros \label{eqtmptmp},
\end{align}
\end{subequations}
where we use the convention $t^*=-\infty$ if \eqref{eqtmptmp} is infeasible for any $t\in\R$,
this case being equivalent to the fact that \eqref{cqpunconstrained} is unbounded from below.
\end{proposition}

\begin{proof}
It is enough to prove that: (i) if \eqref{cqpunconstrained} is unbounded, then
\eqref{eqtmptmp0}-\eqref{eqtmptmp} is infeasible and (ii) if \eqref{cqpunconstrained}
is bounded then its optimum value is equal to $t^*$ in \eqref{eqtmptmp0}-\eqref{eqtmptmp}.

Based on Prop~\ref{propConvexReachesOpt}, 
the above case (i) can only arise
if either $A$ is not SDP so that \eqref{eqtmptmp} is clearly infeasible,
or if 
\eqref{cqpunconstrained} 
has no stationary point. If there is no stationary point, then there is
no $\xx$
such that
$2A\xx+\b=\zeros$. This means that $\frac 12 \b^\top$ can not be written
as a linear combination of (the rows of) $A$.
Using Prop.~\ref{propRankRectangular}, we obtain that
\eqref{eqtmptmp} is infeasible

In the non-degenerate case (ii), Prop.~\ref{propConvexReachesOpt}
guarantees that 
the quadratic function has a stationary point $\xx$ so that
$2A\xx=-\b$. This stationary point minimizes
\eqref{cqpunconstrained} so that the optimum value is
$\xx^\top A \xx + \b^\top \xx
=- \xx^\top \frac 12 \b + \b^\top \xx=
\frac 12 \xx^\top \b$. We will prove that $t^*=
\frac 12 \xx^\top \b$.

Using Prop.~\ref{propRowColOpers}, the SDP status of 
$
\left[
\begin{smallmatrix}
    -t            &   \frac 12 \b^\top    \\
\frac 12 \b       &              A
\end{smallmatrix}
\right]
$ does not change if we add to the first row a linear combination
of the other rows. If we add 
$
\xx^\top \left[
\begin{smallmatrix}
\frac 12 \b       &              A
\end{smallmatrix}
\right]
$ to the first row followed by the same (transposed) operation
on the columns, we obtain the matrix
$
\left[
\begin{smallmatrix}
    -t  +\frac 12 \xx^\top \b         &   \zeros \\
\zeros       &              A
\end{smallmatrix}
\right]
$. This latter matrix is SDP only if 
$-t  +\frac 12 \xx^\top \b\geq 0$, which means that the optimum
$t^*$ value is $t^*=\frac 12 \xx^\top \b$, which is also the optimum
value of the unconstrained quadratic function.
\end{proof}

Finally, it is interesting that the linear matrix inequality \eqref{eqtmptmp}
is equivalent to
the following system of $2^n$ linear inequalities:
$\det\left(
\begin{smallmatrix}
    -t            &   \frac 12 \b_J^\top    \\
\frac 12 \b_J       &              A_J
\end{smallmatrix}\right)\geq 0~\forall J\subseteq [1..n]$, where $\b_J$ and resp.~$A_J$ are obtained
from $\b$ and resp.~$A$ by selecting rows or columns $J$. When $A\succeq \zeros$, \eqref{eqtmptmp0}-\eqref{eqtmptmp} can actually be written as a
linear program.

\subsection{An LP with equality constraints as an SDP program in the dual form}
Consider the Linear Program (LP):
\begin{align*}
\max~&b_1x_1+b_2x_2+\dots b_nx_n\\
s.t.~&c^j_0+c^j_1x_1+c^j_2x_2+\dots c^j_nx_n = 0~\forall j\in[1..m]
\end{align*}
There are two ways of converting this LP into an SDP. The direct method consists of replacing each
equality with two inequalities and of applying Prop.~\ref{propAggregated} on the resulting 
system of inequalities. We obtain an aggregated SDP programs expressed with aggregated diagonal
matrices.

The second method uses the property 
$\left(\c^j{\c^j}^\top\right)\sprod\left( \begin{bmatrix} 1 \\ \x\end{bmatrix}[1~\x^\top]\right)
=\left(\c^j\sprod \begin{bmatrix} 1 \\ \x\end{bmatrix}\right)^2$ proved in Lemma~\ref{lemmaprod},
where $\c^j=[c^j_0~c^j_1~c^j_2\dots c^j_n]^\top~\forall j\in[1..m]$. We can thus write:
\begin{equation}\label{eqprods}\c^j\sprod \begin{bmatrix} 1 \\ \x\end{bmatrix}=0~\iff~
\left(\c^j{\c^j}^\top\right)\sprod
\begin{bmatrix} 1 & \x^\top \\ \x & \x\x^\top \end{bmatrix}
=0.
\end{equation}
Thus, any feasible solution $\x$ of the above LP
can be directly converted to a feasible solution 
of the SDP program below by simply setting
$X=\x\x^\top$ and performing a sum over
all $j\in[1..m]$.

\begin{align*}
\max~&b_1x_1+b_2x_2+\dots b_nx_n\\
s.t.~&
\left(\sum\limits_{j=1}^m \c^j{\c^j}^\top\right)\sprod 
\begin{bmatrix} 1 & \x^\top \\ \x & X \end{bmatrix}=0\\
&
\begin{bmatrix} 1 & \x^\top \\ \x & X \end{bmatrix}\succeq \zeros
\end{align*}

The last point to prove is that any feasible solution 
or
$\left[\begin{smallmatrix} 1 & \x^\top \\ \x & X \end{smallmatrix}\right]$
of the above
SDP program can be associated to a feasible solution $\x$ of the LP.
First, we notice that
$$\left(\sum\limits_{j=1}^m \c^j{\c^j}^\top\right)\sprod 
\begin{bmatrix} 1 & \x^\top \\ \x & X \end{bmatrix}=0
\implies 
\left(\c^j{\c^j}^\top\right)\sprod
\begin{bmatrix} 1 & \x^\top \\ \x & X \end{bmatrix} = 0~\forall j\in[1..m],
$$
because for each $j\in [1..m]$ we have a non-negative 
product of two SDP matrices in the left-hand side sum, \ie, 
$\left(\c^j{\c^j}^\top\right)\sprod
\left[\begin{smallmatrix} 1 & \x^\top \\ \x & X \end{smallmatrix}\right]\geq 0~\forall j\in[1..m]$.
We can now use that $\left[\begin{smallmatrix} 1 & \x^\top \\ \x & X \end{smallmatrix}\right]\succeq \zeros$
is equivalent to $X\succeq \x\x^\top$ so as to obtain
\begin{align*}\left(\c^j{\c^j}^\top\right)\sprod
\begin{bmatrix} 1 & \x^\top \\ \x & X \end{bmatrix}
=0
&\implies
\left(\c^j{\c^j}^\top\right)\sprod
\begin{bmatrix} 0 & \zeros \\ \zeros & X-\x\x^\top \end{bmatrix}
+
\left(\c^j{\c^j}^\top\right)\sprod
\begin{bmatrix} 1 & \x^\top \\ \x & \x\x^\top \end{bmatrix}
=0\\
&\implies
\left(\c^j{\c^j}^\top\right)\sprod
\begin{bmatrix} 1 & \x^\top \\ \x & \x\x^\top \end{bmatrix}=0,
\end{align*}
where we used the fact that all involved matrices are SDP and that their scalar products are
non-negative. 
We can now apply \eqref{eqprods} to obtain the implication below, which confirms $\x$ is feasible in the initial LP.
$$
\left(\c^j{\c^j}^\top\right)\sprod
\begin{bmatrix} 1 & \x^\top \\ \x & \x\x^\top \end{bmatrix}
=0
\implies \c^j\sprod \begin{bmatrix} 1 \\ \x\end{bmatrix}=0
~\forall j\in[1..m].$$

\newpage 
\begin{center}
\line(1,0){400}

\Large
PART 2: MORE ADVANCED SDP PROGRAMMING

\line(1,0){400}

\end{center}

\addtocontents{toc}{~\\ ~\\ \noindent{\bf PART 2 MORE ADVANCED SDP PROGRAMMING} \par}

\section{Six equivalent formulations of the Lov\'asz theta number $\vartheta(G)$}
\subsection{A first SDP formulation of the theta number}
\subsubsection{The primal form $(\vartheta_G)$\label{secThetaPrimal}}

We consider a graph $G=([1..n],E)$.
We introduce the Lov\'asz theta number using the following program
based on SDP matrix $\overline{Z}\in\R^{n\times n}$.
\begin{subequations}
\begin{alignat}{6}[left ={(\vartheta_G) \empheqlbrace}]
\min~~&t                  &                                     &&&                         \label{theta1}\\
s.t.~~& \overline{z}_{ii}&&=t-1                                  &&~~\forall i\in[1..n]     \label{theta2a}\\
      & \overline{z}_{ij}&&=-1                                   &&~~\forall \{i,j\}\notin E\label{theta2b}\\
      & \overline{Z}     &&\succeq\zeros.                        &&                         \label{theta3}
\end{alignat}
\end{subequations}

We use the notational convention $\vartheta(G)=OPT(\vartheta_G)$. 
We consider above $(\vartheta_G)$ as a primal SDP program of the form \eqref{sdp1}-\eqref{sdp3} in which
the variables are $t\in\R^n$ and $\overline{z}_{ij}\in\R^n~\forall \{i,j\}\in E$.
The matrix $B$ from \eqref{sdp1}-\eqref{sdp3} contains ones on the diagonal and on 
all
positions $(i,j)$ corresponding to
$\{i,j\}\notin E$. 
As a side remark, 
using \eqref{eqmaxeigenprimal}, we observe that $(\vartheta_G)$ returns the maximum eigenvalue of 
the matrix $B$ described above.

\begin{theorem}\label{ththetaalpha} ($\vartheta(G)\geq \alpha(G)$) The Lov\'asz theta number
$\vartheta(G)$ is greater than or equal to the maximum stable $\alpha(G)$ of $G$. 
\end{theorem}
\begin{proof} 
Let us consider the maximum stable $J\in[1..n]$ with $|J|=\alpha(G)$. If we restrict $\overline{Z}$
from \eqref{theta1}-\eqref{theta3} to its minor corresponding to rows $J$ and columns $J$, we obtain
$tI_{\alpha(G)}-\ones\succeq \zeros$, where $\ones$ is a matrix in which all
elements are equal to 1. The fact that $t\geq \alpha(G)$ follows
from the next lemma (that needs to hold for $n=\alpha(G)$).
\end{proof}
\begin{lemma} \label{lemmaZMinusOne}The lowest $z\in \R$ for which the $n\times n$ matrix below is SDP is $z=n-1$.
$$
A_n(z)=
\underbrace{
    \begin{bmatrix}
        z    &  -1    &   -1 & \dots &   -1   \\
    -1       &    z   &   -1 & \dots &   -1   \\
    -1       & -1     &  z   & \dots &   -1   \\
    \vdots   & \vdots &\vdots&\ddots & \vdots\\
    -1       & -1     &-1    & \dots &    z   \\
    \end{bmatrix}
}_{n}
\succeq \zeros
$$
\end{lemma}
\noindent We provide three proofs so as to
master 
multiple proof techniques
and explore a variety of SDP tools 
presented in this manuscript.

\vspace{0.5em}
\noindent\textit{Proof 1.} We can write $A_n(z)=(z+1)\cdot I_n-\ones$, where
$\ones$ is a matrix filled with $n\times n$ ones. Using
\eqref{eqmaxeigenprimal}, the minimum value $(z+1)$ such that 
$(z+1) I_n-\ones\succeq \zeros$ is the maximum eigenvalue of $\ones$.
But what are the eigenvalues and eigenvectors of $\ones$? Since
all positions of 
$\ones \v$ must be equal for any eigenvector $\v\in\R^n$, we obtain that 
either (i) the eigenvalue $\lambda_v$ of $\v$ is zero or (ii) all elements of 
$\v$ are equal so that $\lambda_v=n$. The sought value $z+1$ is thus $n$,
leading to $z=n-1$.
\qed.


\vspace{1em}
\noindent\textit{Proof 2.} We apply the Gershgorin circle Theorem
\ref{thGershgorin} to show $A_n(n-1)\succeq \zeros$. The theorem states that any
eigenvalue $\lambda$ of $A_n(n-1)$ needs to satisfy
$\left|\lambda-A_n(n-1)_{ii}\right|\leq
\sum\limits_{\substack{j=1\\j\neq i}}^n |A_n(n-1)_{ij}|$ for some
$i\in[1..n]$. However, for any $i\in[1..n]$ this reduces to
$\left|\lambda -(n-1)\right|\leq n-1$, which means $\lambda\geq 0$, leading to
$A_n(n-1)\succeq\zeros$. 
We show $A_n(z)\nsucceq\zeros$ for any $z<n-1$ by
noticing $[1~1~1\dots 1]A_n(z)[1~1~1\dots 1]^\top <0$ for any $z<n-1$. This proves
that $z=n-1$ is
the lowest $z$ such that $A_n(z)\succeq\zeros$.
\qed

\vspace{1em}
\noindent\textit{Proof 3.} We proceed by induction. The lemma is clearly true for $n=1$. 
For any $n\geq 2$,
we need to have $z>0$ because $z=0$ would make negative any $2\times
2$ minor (use Prop~\ref{propMinorsNonneg}).
We apply the Schur complement from Prop~\ref{propSchurGen}. Using the notations $A$, $B$ and $C$
from Prop~\ref{propSchurGen}, we write $A_n(z)=\left[\begin{smallmatrix} A & B^\top\\ B&C\end{smallmatrix}\right]$
with $A=z$ (which satisfies condition $A\succ \zeros$),
$B^\top=-\ones^\top_{n-1}=\underbrace{[-1~-1~\dots -1]}_{n-1\text{ positions }}$, and $C=A_{n-1}(z)$, \ie, 
$A_{n-1}(z)$ is the $(n-1)\times (n-1)$ bottom-right minor. Using the Schur complement, 
we obtain $A_n(z)\succeq \zeros \iff  C-BA^{-1}B^\top\succeq \zeros
\iff A_{n-1}(z) - \frac 1z \ones_{n-1}\ones^\top_{n-1}\succeq \zeros$. This last
matrix inequality boils down to:
\begin{align*}
\underbrace{
    \begin{bmatrix}
        z-\frac 1z   &   -1-\frac 1z & \dots &   -1-\frac 1z   \\
     -1  -\frac 1z   &  z -\frac 1z  & \dots &   -1-\frac 1z   \\
     \vdots &\vdots&\ddots & \vdots\\
     -1   -\frac 1z  &-1 -\frac 1z   & \dots &    z-\frac 1z   \\
    \end{bmatrix}
}_{n-1}
\succeq \zeros
&
\iff
\underbrace{
    \begin{bmatrix}
    \frac{(z+1)(z-1)}{z} & -\frac{z+1}z  & \dots &  -\frac{z+1}z     \\
     -\frac{z+1}z   &\frac{(z+1)(z-1)}{z}& \dots &  -\frac{z+1}z      \\
     \vdots &\vdots&\ddots & \vdots\\
     -\frac{z+1}z  &-\frac{z+1}z  & \dots &   \frac{(z+1)(z-1)}{z} \\
    \end{bmatrix}
}_{n-1}
\succeq\zeros  \\
&\iff
\underbrace{
    \begin{bmatrix}
    &    z-1 &   -1 & \dots &   -1   \\
    & -1     &  z-1 & \dots &   -1   \\
    & \vdots &\vdots&\ddots & \vdots\\
    & -1     &-1    & \dots &    z-1 \\
    \end{bmatrix}
}_{n-1}
\succeq\zeros 
\end{align*}
We have obtained that $A_n(z)\succeq\zeros\iff A_{n-1}(z-1)\succeq \zeros$. We can use the
induction hypothesis: the lowest $z-1$ such that $A_{n-1}(z-1)\succeq \zeros$ is $z-1=n-2$, 
and so, $z=n-1$ is the lowest $z$ such that $A_{n}(z)\succeq \zeros$.\qed

\begin{theorem}\label{ththetacliquecover} ($\vartheta(G)\leq \chi(\overline{G})$) The Lov\' asz theta
number $\vartheta(G)$ is less than or equal to the clique cover number of $G$.
This clique cover number is the chromatic number of the complementary graph
$\overline{G}$.
\end{theorem}
\begin{proof} 
Consider a partition $(C_1,C_2, \dots C_k)$ of the vertex set
$[1..n]$ such that each $C_i$ (with $i\in[1..k]$) is a clique.
Applying Lemma~\ref{lemmaZMinusOne} above, we have
$$
\underbrace{
    \begin{bmatrix}
        k-1    &  -1    &   -1 & \dots &   -1   \\
    -1       &    k-1   &   -1 & \dots &   -1   \\
    -1       & -1     &  k-1   & \dots &   -1   \\
    \vdots   & \vdots &\vdots&\ddots & \vdots\\
    -1       & -1     &-1    & \dots &    k-1   \\
    \end{bmatrix}
}_{k}
\succeq \zeros
$$
This means (use the Cholesky factorisation from Prop~\ref{propCholeskySDP} or
other decompositions from Corollary~\ref{corSdpToVectors}) there exist
$k$-dimensional vectors $\v_1,\v_2,\dots \v_k$ such that $\v_i\sprod \v_i=k-1$
and $\v_i\sprod \v_j=-1~\forall i,j\in[1..k],i\neq j$. 

Since $\{C_1,C_2, \dots C_k\}$ cover the vertex set $[1..n]$, all vertices
$u\in[1..n]$ can be associated to vector $\v^u=\v_{\ell(u)}$, where $\ell(u)$ is the clique that contains $u$
(we have $u\in C_{\ell(u)}$). We actually associate the same vector to all
vertices of a clique.

We now define the matrix $\overline{Z}$ by setting
$\overline{z}_{ij}=\v^i\sprod \v^j$. This is a Gram matrix that is clearly SDP
(use Prop.~\ref{propranktransprod}). Notice 
that: (i) $\overline{z}_{ii}=k-1~\forall i\in[1..n]$ and (ii) if 
$i,j\in[1..n]$ do not belong to the same clique, then $\overline{z}_{ij}=-1$,
and so, $\overline{z}_{ij}=-1~\forall \{i,j\}\notin E$. Properties 
(i) and (ii) are enough to ensure that $\overline{Z}$ can be written in
the form \eqref{theta2a}-\eqref{theta2b} with $t=k$. We have constructed a feasible solution
of $(\vartheta_G)$ from \eqref{theta1}-\eqref{theta3} with objective value $k$.
Taking $k=\chi(\overline{G})$, we have
$OPT(\vartheta_G)\leq \chi(\overline{G})$.

This proof is inspired from a related result of a lecture note of Anupam
Gupta.%
\footnote{See (11.3) of the document available, as of 2017, at \url{http://www.cs.cmu.edu/afs/cs.cmu.edu/academic/class/15859-f11/www/notes/lecture11.pdf}.}
\end{proof}

Using $\alpha(G)=\omega(\overline{G})$ (where
$\omega$ denotes the maximum clique size), we also obtain the ``sandwich'' property
as a direct consequence of Theorem~\ref{ththetaalpha} and
Theorem~\ref{ththetacliquecover}.
\begin{equation}\label{eqSandwichTheta}
\omega(\overline{G})\leq \vartheta(G)\leq \chi(\overline{G})
\end{equation}

The following corollary of Lemma \ref{lemmaZMinusOne} can be generally useful, but we do not use it in this
document.
\begin{corollary} When $n\geq 2$ unitary vectors are as spread out as much as possible, the dot product of any pair
of them is $\frac{-1}{n-1}$.
\end{corollary}
\begin{proof}
By ``spread out as much as possible'', we want 
$t=\max\{\v_i\sprod \v_j:~i,j\in[1..n],~i\neq j\}$ to be as low as possible.
Lemma~\ref{lemmaZMinusOne} states that $nI_n-\ones\succeq \zeros,$
where $\ones$ is a matrix with all elements equal to 1.
By dividing all terms by $\frac{1}{n-1}$, we obtain a matrix $T\succeq \zeros$ with 1 on the
diagonal and with $\frac {-1}{n-1}$ on all non-diagonal positions.
This means that $t\leq \frac {-1}{n-1}$. This $t$ value is optimal. If we
decrease any non-diagonal term(s) of $T$, we obtain a matrix
$T'\nsucceq\zeros$, since $[1~1\dots 1]T'[1~1\dots 1]^\top<0$.
\end{proof}
\subsubsection{The dual form of $(\vartheta_G)$\label{secThetaDual}}

We now introduce the dual program $(D\vartheta_G)$ of $(\vartheta_G)$ from
\eqref{theta1}-\eqref{theta3}. Let $D\vartheta(G)=OPT(D\vartheta_G)$.

\begin{subequations}
\begin{align}[left ={(D\vartheta_G) \empheqlbrace}]
\max~~&\sum_{i=1}^n Y_{ii} + \sum_{\{i,j\}\notin E}Y_{ij}=
      \ones\sprod Y    \label{dtheta1}\\
s.t.~~&\texttt{trace}(Y) = 1    \label{dtheta2}\\
      &Y_{ij} = 0~~~\forall \{i,j\}\in E  \label{dtheta3} \\
      &Y\succeq \zeros \label{dtheta4},
\end{align}
\end{subequations}
where \eqref{dtheta1} simplifies to $\ones\sprod Y$ because $Y_{ij}=0~\forall
\{i,j\}\in E$ by virtue of \eqref{dtheta3}. Notice that both $(\vartheta_G)$ 
and $(D\vartheta_G)$ are strictly feasible (take a sufficiently large $t$ in 
$(\vartheta_G)$ and $Y=\frac 1n I_n$ in $(D\vartheta_G)$), and so, 
we can apply the strong duality Theorem~\ref{thStrongDual3} to state that 
the optimum value is reached by both programs.

\begin{proposition} The ``sandwich'' property
$\alpha({G})\leq \vartheta(G)\leq \chi(\overline{G})$ 
obtained in \eqref{eqSandwichTheta} in the previous Section~\ref{secThetaPrimal}
can also be proved using the dual $(D\vartheta_G)$.
\end{proposition}

\noindent \textit{Proof of} $\alpha({G})\leq \vartheta(G)$~\\
Take the largest stable $S$ of $G$. Construct matrix $Y^S$ such that
$Y^S_{ij}=\frac1{|S|}~\forall i,j\in S$ and $Y^S_{ij}=0$ otherwise. It is not hard to
check that $Y^S$ is a feasible solution of $(D\vartheta_G)$ in
\eqref{dtheta1}-\eqref{dtheta4} with objective value $|S|$. This is enough
to state $\alpha(G)\leq D\vartheta(G)=\vartheta(G)$. 

\noindent \textit{Proof of} $\vartheta(G)\leq \chi(\overline{G})$\\
Consider a clique cover $\{C_1,~C_2,\dots C_k\}$ of $G$.
Without loss of
generality, we can re-order $[1..n]$ such that all $C_i$ (with $i\in[1..k]$)
represent segments of $[1..n]$, \ie, $C_1$ contains the first $|C_1|$ elements
$\big[1..|C_1|\big]$, 
$C_2=\big[|C_1|+1..|C_1|+|C_2|\big]$, etc. We can decompose $Y$ 
of $(D\vartheta(G))$ from \eqref{dtheta1}-\eqref{dtheta4}
into $k^2$
blocks:
$$
Y=
\begin{bmatrix}
\YY_{11}   &   \YY_{12}  &  \ldots  & \YY_{1k} \\
\YY_{21}   &   \YY_{22}  &  \ldots  & \YY_{2k} \\
\vdots   & \vdots    &  \ddots  & \vdots \\
\YY_{k1}   &   \YY_{k2}  &  \ldots  & \YY_{kk}
\end{bmatrix},
$$
where block $\YY_{ij}$ has size $|C_i|\times |C_j|$ for any $i,j\in[1..k]$.
Notice that $\YY_{ii}$ is diagonal because $C_i$ is a clique ($\forall
i\in[1..k]$).
This means $\sum_{i=1}^k \ones\sprod \YY_{ii}=1$, based on $\texttt{trace}(Y)=1$.
We will show that 
\begin{equation}\label{eqSumY}\ones\sprod Y=\sum_{i=1}^k \sum_{j=1}^k
\ones\sprod \YY_{ij}\leq k.\end{equation}

The principal minor of $Y$ associated to rows and columns $C_i$ and $C_j$ needs to be
SDP. We thus obtain 
$\left[\begin{smallmatrix} \YY_{ii} & \YY_{ij}\\ \YY_{ji} & \YY_{jj}\end{smallmatrix}\right]\succeq \zeros$.
Take $\x^\top=[\underbrace{1~1\dots 1}_{|C_i|}~\underbrace{-1~-1~\dots
-1}_{|C_j|}]$. Using 
$\left[\begin{smallmatrix} \YY_{ii} & \YY_{ij}\\ \YY_{ji} & \YY_{jj}\end{smallmatrix}\right]
\sprod 
\x\x^\top \geq 0$, we obtain:
$$
\ones \sprod \YY_{ij}\leq \frac {\ones \sprod \YY_i+\ones \sprod \YY_j}2,~\forall
i,j\in [1..k]
$$
Notice this holds (with equality) for terms $i=j\in[1..k]$. We can simply
now obtain \eqref{eqSumY} by applying
$$\sum_{i=1}^k \sum_{j=1}^k \ones\sprod \YY_{ij}\leq 
\sum_{i=1}^k \sum_{j=1}^k\frac {\ones \sprod \YY_i+\ones \sprod \YY_j}2 
=\frac{2k \sum_{i=1}^k \ones \sprod \YY_i}2 = k.
$$


\subsection{A second SDP formulation $(\vartheta'_G)$ of the theta number}

We start from the following well-known linear program for the maximum
stable $\alpha(G)$.
\begin{subequations}
\begin{align}[left ={(\alpha_G) \empheqlbrace}]
\max~~&\sum_{i=1}^n y_i                             \notag\\
s.t.~~&y_i+y_j\leq 1~\forall \{i,j\}\in E           \notag\\
      &\y\in\{0,1\}^n                               \notag  
\end{align}
\end{subequations}
We now introduce a first SDP relaxation of above $(\alpha_G)$.
\begin{subequations}
\begin{alignat}{3}[left ={(\vartheta'_G) \empheqlbrace}]
\max~~&\sum_{i=1}^n y_{0i}&&                    \label{thetaprim1}\\
s.t.~~&y_{00} = 1&&                              \label{thetaprim2}\\ 
      &y_{ij} = 0&&\forall \{i,j\}\in E          \label{thetaprim3}\\ 
      &y_{ii} = y_{0i}&&\forall i\in[1..n]       \label{thetaprim4}\\ 
      &Y\succeq \zeros &~&                      \label{thetaprim5}
\end{alignat}
\end{subequations}
It is not hard to check that $(\vartheta'_G)$ is a relaxation of $(\alpha_G)$.
Take an optimal solution $\y$ of $(\alpha_G)$ and construct
$Y=\left[\begin{smallmatrix} 1\\ \y\end{smallmatrix}\right][1~\y^\top]$ that is
feasible for $(\vartheta'_G)$ and has the same objective value as in
$(\alpha_G)$.

\begin{theorem}\label{ththetasecprog} The optimum value $\vartheta'(G)$ of $(\vartheta'_G)$ from
\eqref{thetaprim1}-\eqref{thetaprim5} is equal to 
the Lov\'asz theta number
\ie, to
the optimum value $\vartheta(G)$
of $(\vartheta_G)$ from \eqref{theta1}-\eqref{theta3}.
\end{theorem}
\begin{proof} We first show that $(\vartheta'_G)$ is equivalent to the following
relaxation of $(\alpha_G)$.

\begin{subequations}
\begin{align}[left ={(\widetilde{\vartheta'}_G) \empheqlbrace}]
\max~~&\sum_{i=1}^n 2y_{0i}-y_{ii}              \label{thetawidetildeprim1}\\
s.t.~~&y_{00} = 1                               \label{thetawidetildeprim2}\\ 
      &y_{ij} = 0~\forall \{i,j\}\in E          \label{thetawidetildeprim3}\\ 
      &Y\succeq \zeros                          \label{thetawidetildeprim4}
\end{align}
\end{subequations}
We take an optimal solution $Y$ of $(\widetilde{\vartheta'}_G)$ 
and we will show it satisfies $y_{ii}=y_{0i}~\forall i\in[1..n]$.
Consider any $i\in[1..n]$. If $y_{ii}=0$, then $y_{0i}$ needs to be zero as
well. If $y_{ii}\neq 0$, we can multiply row $i$ and column $i$ of $Y$ by
$\frac{y_{0i}}{y_{ii}}$ to obtain a matrix that is still
feasible for $(\widetilde{\vartheta'}_G)$ and whose objective value is greater than
or equal to that of $Y$. This simply follows from
$2y_{0i} \frac{y_{0i}}{y_{ii}}-y_{ii}\frac{y^2_{0i}}{y^2_{ii}}
\geq 
2y_{0i}-y_{ii}\stackrel{y_{ii}>0}{\iff} y^2_{0i}\geq
2y_{0i}y_{ii}-y^2_{ii}\iff (y_{0i}-y_{ii})^2\geq 0$. 
This proves $OPT(\widetilde{\vartheta'}_G)=OPT(\vartheta'_G)$. We will show that 
$OPT(\widetilde{\vartheta'}_G)=OPT(\vartheta_G)=\vartheta(G)$.

Let us write the dual of $(\widetilde{\vartheta'}_G)$.
\begin{subequations}
\begin{align}[left ={(D\widetilde{\vartheta'}_G) \empheqlbrace}]
\min~~&t                                    \label{dvarthetaprim1}\\
s.t.~~&
        \begin{bmatrix}
        t       & -\ones^\top \\
        -\ones   &      Z
        \end{bmatrix}
        =
        \begin{bmatrix}
            t    &    -1    &  -1     &   -1     &  \ldots  &   - 1    \\
            -1   &     1    &  z_{12} &   z_{13} &  \ldots  &   z_{n1}    \\
            -1   &   z_{21} &    1    &   z_{23} &  \ldots  &   z_{n2}    \\
            -1   &   z_{31} &  z_{32} &     1    &  \ldots  &   z_{n3}    \\
        \vdots   &  \vdots  &  \vdots &  \vdots  &  \ddots  &   \vdots \\
            -1   &   z_{n1} &  z_{n2} &   z_{n3} &  \ldots  &   z_{nn}    \\
        \end{bmatrix}
      \succeq \zeros                         \label{dvarthetaprim2}     \\
        & z_{ij}=0~\forall \{i,j\}\notin E     \label{dvarthetaprim3} 
\end{align}
\end{subequations}
Programs $(\widetilde{\vartheta'}_G)$ in
\eqref{thetawidetildeprim1}-\eqref{thetawidetildeprim4} and $(D\widetilde{\vartheta'}_G)$ 
in \eqref{dvarthetaprim1}-\eqref{dvarthetaprim3} are strictly feasible. For 
$(\widetilde{\vartheta'}_G)$, it is enough to take $Y=I$.
For $(D\widetilde{\vartheta'}_G)$, we have
$\left[
        \begin{smallmatrix}
        n+1       & -\ones^\top \\
        -\ones   &      I_n
        \end{smallmatrix}
\right]\succ \zeros$, as certified by applying the Sylvester criterion
(Prop.~\ref{propSylvester}) in reversed order (from bottom-right to top-left).
The determinant of the whole matrix is 1: if we add the last $n$ rows to 
the first one, we obtain a first row with only one non-zero element of value 1 in
the upper-left corner. As such, using the strong duality
Theorem~\ref{thStrongDual3}, we have
$OPT(\widetilde{\vartheta'}_G)=OPT(D\widetilde{\vartheta'}_G)$ and this value
is effectively reached by both programs.

We now reformulate $(D\widetilde{\vartheta'}_G)$ and let us focus on
\eqref{dvarthetaprim2}. Using the Schur complement Prop.~\ref{propSchurGen}, we
have 
$$
        \begin{bmatrix}
        t       & -\ones^\top \\
        -\ones   &      Z
        \end{bmatrix}\succeq \zeros
\iff
        Z-\frac 1t \ones\succeq \zeros
\iff    \frac{tZ-\ones}t\succeq \zeros
\iff    tZ-\ones \succeq \zeros,
$$
where we used several times $t>0$ (notice $t=0$ would render
$(D\widetilde{\vartheta'}_G)$ infeasible because of the upper-left 2$\times$ 2
minor of the matrix in \eqref{dvarthetaprim2}). Writing $\overline{Z}=tZ-\ones$, we notice
$\overline{z}_{ii}=t-1~\forall i\in[1..n]$
(because $z_{ii}=1$ in \eqref{dvarthetaprim2}) and $\overline{z}_{ij}=-1~\forall
\{i,j\}\notin E$ (because $z_{ij}=0$ in \eqref{dvarthetaprim3}). If $\{i,j\}\in E$, we have $\overline{z}_{ij}=tz_{ij}-1$, \ie,
$\overline{z}_{ij}$ can actually take any real value, independent from the
other terms of the matrix. Replacing this in $(D\widetilde{\vartheta'}_G)$, we
obtain that $(D\widetilde{\vartheta'}_G)$ is equivalent 
to following program which is exactly $(\vartheta_G)$ from \eqref{theta1}-\eqref{theta3}. 
\begin{alignat*}{6}[left ={(\vartheta_G) \empheqlbrace}]
\min~~&t                  &                                     &&&                         \\
s.t.~~& \overline{z}_{ii}&&=t-1                                  &&~~\forall i\in[1..n]     \\
      & \overline{z}_{ij}&&=-1                                   &&~~\forall \{i,j\}\notin E\\
      & \overline{Z}     &&\succeq\zeros.                        &&                         
\end{alignat*}

All programs presented during the proof have the same optimum, \ie, we proved
that
$
OPT(\vartheta'_G)=
OPT(\widetilde{\vartheta'}_G)=
OPT(D\widetilde{\vartheta'}_G)=
OPT(\vartheta_G)=\vartheta(G)$. The arguments from this proof are based on a
reversed version of the proof from Section 3.1. of the PhD thesis of
of Neboj\originalv{s}a Gvozdenovi\'c.%
\footnote{As of 2019, it is available on-line at
\url{http://pure.uva.nl/ws/files/4245957/54393_thesis.pdf}.}%
\end{proof}

\subsection{A formulation $\vartheta''(G)$ of the theta number without SDP matrices}
We associate an unit vector $\u^i$ for each vertex $i\in[1..n]$. We say that
the unit vectors $\{\u^1,\u^2,\dots \u^n\}$ constitute an {\it orthonormal
representation} of 
$G=([1..n],E)$ if and only if $\u^i\sprod\u^j=0~\forall \{i,j\}\notin E,~i\neq j$.
We introduce the following function:
\begin{equation}\label{eqthetaortho1}
\vartheta''(G)=\min\left\{\max_{i\in[1..n]}\frac{1}{(\c\sprod \u^i)^2}:~
                          |\c|=1,~\{\u^1,\u^2,\dots \u^n\} \text{ is an orthonormal representation of }G
                    \right\}
\end{equation}
Given any $\c$ and $\{\u^1,\u^2,\dots \u^n\}$ that yield an optimum in
above~\eqref{eqthetaortho1}, we can apply a rotation $R$ that maps $\c$
into unitary vector $R\c=[1~0~0~0\dots 0]^\top$ and $\{\u^1,\u^2,\dots \u^n\}$
into $\{\u'^1,\u'^2,\dots \u'^n\}$. A rotation is always represented by 
a matrix $R$ that is unitary and orthonormal,%
\footnote{We actually still need to prove there exists a rotation of the space
that maps the unitary vector $\c$ to $\cc=[1~0~0~0\dots
0]^\top\in\R^k$. 
This task would be very simple if $\c$ had the form
$\c=[c_1~c_2~0~0~0\dots 0]^\top$, \ie, it would be enough to
apply a rotation on the (space of) the first two coordinates and keep
all remaining $k-2$ coordinates fixed.
We will construct a basis so as to make the representation
of $\c$ in this basis always have the above form.
Let us take an orthonormal
basis $\v_1,\v_2\in\R^k$ (with $\v_1=\cc$) of the sub-space spanned 
by $\c$ and $\cc$ and complete it to a full
orthonormal basis by adding $\v_3,~\v_4\dots \v_k$. 
Writing $V=[\v_1~\v_2~\v_3\dots \v_k]$, any vector $\v\in\R^k$ can be
represented in the 
new basis as $\v=V\v'$, \ie, 
$\v'=V^{-1}\v$ is the expression of $\v$ in the new basis.
Notice $\c$ and $\cc$ can be written as linear combinations of $\v_1$ and
$\v_2$, and so, their representations $\c'$ and $\cc'$ in the new basis
only use the first two coordinates.
After checking that $\c'$ and $\cc'$ are unitary, 
we can define $\alpha=\arccos \c'\sprod \cc'$.
The matrix $R_2=\left[\begin{smallmatrix} \cos\alpha&sin\alpha\\
-\sin\alpha&\cos \alpha\end{smallmatrix}\right]$ rotates by $\alpha$
any 2-dimensional point, and so, it maps $\cc$ to $\c'$ in the space of the
first
two coordinates (or $\c'$ to $\cc$, case in which we replace $\alpha$ by
$-\alpha$).
We extend $R_2$ to $R_k\in\R^{k\times k}$ by putting an
one on each new diagonal position $(i,i)$ with $i>2$ and zeros on all other
positions $(i,j)$ with $i>2$ or $j>2$. $R_k$ rotates the first two coordinates
of any vector $\v'\in\R^k$ and leaves untouched the rest of values $v'_3,~v'_3\dots v'_k$. 
Let us calculate $VR_kV^\top\c=VR_kV^\top V\c'=VR_k\c'=V\cc'=\cc$.
Since $VR_kV^\top$ is orthonormal, it is a rotation matrix that
performs the desired rotation.
\setcounter{testfoot2}{\value{footnote}}\label{testfootpage2}}
and so, does not change scalar products or norms. Indeed, if $R^\top R=I$, for any two vectors
$\v$ and $\v'$ we have $\v^\top \v'=\v^\top R^\top R \v'=(R\v)^\top (R\v')$,
and so, the mapping $\v\to R\v$ does not change scalar products or norms.
This means that $[1~0~0~0\dots 0]^\top$ and $\{\u'^1,\u'^2,\dots \u'^n\}$ is 
also an optimal solution for~\eqref{eqthetaortho1}. As such, we can simplify
\eqref{eqthetaortho1} to:

\begin{equation}\label{eqthetaortho2}
\vartheta''(G)=\min\left\{\max_{i\in[1..n]}\left(\frac{1}{{u_1^i}}\right)^2:~
                          \{\u^1,\u^2,\dots \u^n\} \text{ is an orthonormal representation of }G
                    \right\}
\end{equation}
\begin{theorem} $\vartheta''(G)=\vartheta(G)$ \end{theorem}
\begin{proof} We will first show $\vartheta(G)\leq \vartheta''(G)$ and then $\vartheta''(G)\leq
\vartheta(G)$.\\
\noindent $\vartheta(G)\leq \vartheta''(G)$\\
\noindent 
We will start from an optimal orthonormal representation of $G$ and we will construct a feasible
solution of 
$(D\widetilde{\vartheta'}_G)$ from \eqref{dvarthetaprim1}-\eqref{dvarthetaprim3} that we proved 
(Theorem~\ref{ththetasecprog}) to reach the optimum $\vartheta(G)$.
Without loss of generality, we can consider $|u^1_1|\leq |u^i_1|~\forall i\in[1..n]$, and
so, \eqref{eqthetaortho2} states that $\vartheta''(G)=\left(\frac 1{u^1_1}\right)^2$. Let us
construct vectors $\v^0,~\v^1,~\v^2,\dots \v^n$ such that $\v^0=[-\frac 1{u^1_1}~0~0~0\dots 0]^\top
$ and for all $i\in[1..n]$ we set $\v^i=\frac{u^1_1}{u^i_1}\u^i$. We construct the Gram matrix $Z'$
of these vectors such that $Z'_{ij}=\v^i\sprod \v^j~\forall i,j\in[0..n]$. We have $Z'\succeq
\zeros$ using Prop.~\ref{propranktransprod}, $Z'_{00}=\left(\frac 1{u^1_1}\right)^2=\vartheta''(G)$, and $Z'_{ij}=0~\forall \{i,j\}\notin E$, using
$\v^i\sprod \v^j=\frac{u^1_1}{u^i_1}\frac{u^1_1}{u^j_1}\u^i\sprod \u^j$ and the fact that
$\u^1,~\u^2,\dots \u^n$ are a orthonormal representation of $G$. We also have 
$Z'_{0i}=-1$, because $v^i_1=u^1_1~\forall i\in[1..n]$.

The diagonal elements are $Z'_{ii}=\left(\frac{u^1_1}{u^i_1}\right)^2 \u^i\sprod \u^i\leq
1~\forall i\in[1..n]$, where we used $|u^1_1|\leq |u^i_1|$ and the fact that 
$\u_1,~\u_2,~\u_3,\dots \u_n$ are unitary. We transform $Z'$ into a matrix $Z$ by only increasing the
diagonal elements up to $Z_{ii}=1~\forall i\in[1..n]$; $Z$ remains SDP as the sum of $Z$
and a positive diagonal matrix. This matrix $Z$ is a feasible solution in
$(D\widetilde{\vartheta'}_G)$ in \eqref{dvarthetaprim1}-\eqref{dvarthetaprim3}
with objective value $Z'_{00}=\left(\frac 1{u^1_1}\right)^2=\vartheta''(G)$.
Since $OPT(D\widetilde{\vartheta'}_G)=\vartheta(G)$ as shown in the proof of
Theorem~\ref{ththetasecprog}, we have $\vartheta(G)\leq \vartheta''(G)$.\\
~\\
\noindent $\vartheta''(G)\leq \vartheta(G)$\\
We start from $(\vartheta_G)$ in \eqref{theta1}-\eqref{theta3} and 
we will construct an orthogonal representation of value $\vartheta(G)$.
Notice
the optimal solution $\overline{Z}$ of \eqref{theta1}-\eqref{theta3} verifies $\overline{Z}_{ii}=\vartheta(G)-1$ and
$\overline{Z}_{ij}=-1~\forall \{i,j\}\notin E$. Using the Cholesky factorisation
from
Prop.~\ref{propCholeskySDP}
(or the eigenvalue or the square root factorization as in Corollary \ref{corSdpToVectors}), 
there exist $\v_1,~\v_2,\dots \v_n$ such that
$\overline{Z}_{ij}=\v_i\sprod \v_j$. We construct the representation 
$\u_i= \frac{1}{\sqrt{\vartheta(G)}}\begin{bmatrix}1\\
\v_i\end{bmatrix}~\forall i\in[1..n]$. We verify the following:
\begin{itemize}
    \item[--] $|\u_i|^2=\frac 1{\vartheta(G)}\left(1+\v_i\sprod \v_i\right)=\frac{\vartheta(G)}{\vartheta(G)}=1~\forall i\in[1..n]$, \ie, 
                the representation is unitary;
    \item[--] for $\{i,j\}\notin E$, we have $\u_i\sprod \u_j=\frac
                                              1{\vartheta(G)}\left(1+\v_i\sprod\v_j\right)=0$, \ie,
                the representation is orthonormal;
    \item[--] the value of the representation in \eqref{eqthetaortho2} is 
                $\left(\frac{1}{\frac 1{\sqrt{\vartheta(G)}}}\right)^2=\vartheta(G)$
because $u^i_1=\frac 1{\sqrt{\vartheta(G)}}~\forall i\in[1..n]$.
\end{itemize}
This shows $\vartheta''(G)\leq \vartheta(G)$, which finishes the proof.

The proof of $\vartheta(G)\leq \vartheta''(G)$ is personal. The proof of $\vartheta''(G)\leq
\vartheta(G)$ is adapted from the proof of ``$\vartheta_1(G,w)\leq \vartheta_2(G,w)$'' from
Section 6 of the survey article ``The sandwich theorem'' of Donald Knuth.%
\footnote{Published in the \textit{Electronic Journal of Combinatorics}
in 1994 (1), available at 
\url{http://www.combinatorics.org/ojs/index.php/eljc/article/view/v1i1a1/pdf}.}
\end{proof}

\subsection{A fourth formulation $\vartheta^\ell(G)$ of the theta number}
We consider an orthonormal representation $\u^1,~\u^2,\dots \u^n$ of $\overline{G}$, \ie, we have
$\u^i\sprod \u^j=0~\forall \{i,j\}\in E$. The \textit{leaning} of this representation is 
$\sum\limits_{i=1}^n (\e_1\sprod \u^i)^2
=
\sum\limits_{i=1}^n (\u^i_1)^2$, where $\e_1=[1~\underbrace{0~0\dots 0}_{n-1}]^\top$. 
One can think of this {\it leaning} term in the sense that any vector is leaning somehow
on (casts its shadow on) the first dimension. We introduce:
\begin{equation}\label{eqthetaortho3}
\vartheta^\ell(G)=\max\left\{\sum_{i=1}^n (\e_1\sprod \u^i)^2:~
                          \{\u^1,\u^2,\dots \u^n\} \text{ is an orthonormal representation of }\overline{G}
                    \right\}
\end{equation}
\begin{proposition}$\vartheta^\ell(G)=\vartheta(G)$
\end{proposition}
\begin{proof} We will show that generating an orthonormal representation of $\overline{G}$ in
\eqref{eqthetaortho3} is equivalent 
to generating a feasible solution of $(\vartheta'_G)$ in \eqref{thetaprim1}-\eqref{thetaprim5} with the same objective
value. We start by noticing that generating a feasible solution $Y\in\R^{(n+1)\times(n+1)}$ of $(\vartheta'_G)$ is equivalent to generating $n+1$
vectors $\c,\overline{\v}^1,\overline{\v}^2,\dots \overline{\v}^n$ such that $Y_{ij}=\overline{\v}^i\sprod \overline{\v}^j~\forall i,j\in[1..n]$ and
$Y_{0i}=Y_{i0}=\c\sprod \overline{\v}^i~\forall i\in[1..n]$. Given constraints
\eqref{thetaprim2}-\eqref{thetaprim4}, these vectors need to satisfy: $\c\sprod \c =Y_{00}=1$,
$\overline{\v}^i\sprod \overline{\v}^j=0~\forall \{i,j\}\in E$, and 
$\c\sprod \overline{\v}^i=\overline{\v}^i\sprod \overline{\v}^i = d^2_i,~\forall i\in[1..n]$.
The objective value of $Y$ in $(\vartheta'_G)$ is $\sum\limits_{i=1}^n d_i^2$.

Since there exists an orthonormal rotation matrix $R$ that maps $\c$ to $R\c=\e_1$ and leaves unchanged scalar products (see
Footnote \codefootnotesec, p.~\pageref{testfootpage2}), generating vectors
$\c,\overline{\v}^1,\overline{\v}^2,\dots \overline{\v}^n$ with above properties
is equivalent to generating vectors
$R\c=\e_1,R\overline{\v}^1={\v}^1,R\overline{\v}^2={\v}^2,\dots R\overline{\v}^n={\v}^n$ with the same pairwise scalar products, \ie, 
such that
${\v}^i\sprod {\v}^j=0~\forall \{i,j\}\in E$ and 
$\e\sprod {\v}^i={\v}^i\sprod {\v}^i = d^2_i,~\forall i\in[1..n]$ because $RR^\top=I$.
Remark $v^i_1=d^2_i$.

Generating vectors $\e_1,\v^1,\v^2,\dots \v^n$ with above properties is equivalent to
generating vectors $\e_1,\u^1,\u^2,\dots \u^n$, where $\u^i=\frac 1{d_i}\v^i~\forall i\in[1..n]$
(if $d_i=0$, we have $\v^i=\zeros$ and we set
$\u^i=\left[\begin{smallmatrix}\v^i\\1\end{smallmatrix}\right]=\left[\begin{smallmatrix}\zeros\\1\end{smallmatrix}\right]$ and extend all
other $\u^j~\forall j\in[1..n]-\{i\}$
with a 0 element).
One can check that $|\u^i|^2=\u^i\sprod \u^i=\frac 1{d^2_i}\v^i\sprod\v^i=1~\forall i\in[1..n]$ and 
$\u^i\sprod\u^j=\frac 1{d_i}\frac 1{d_j}\v^i\sprod \v^j=0~\forall \{i,j\}\in E$.
This means that $\u^1,~\u^2,\dots \u^n$ is an orthonormal representation of $\overline{G}$ such 
that $u^i_1=d_i~\forall i\in[1..n]$. The leaning of $\u^1,~\u^2,\dots \u^n$ in \eqref{eqthetaortho3} is $\sum\limits_{i=1}^n d_i^2$, \ie, 
the same as the objective value of feasible matrix $Y$ in $(\vartheta'_G)$ in \eqref{thetaprim1}-\eqref{thetaprim5}.
This confirms $\vartheta^\ell(G)=OPT (\vartheta'_G)=\vartheta(G)$.
\end{proof}

\subsection{Two formulations of the theta number using maximum eigenvalues}

\subsubsection{A formulation only using the maximum eigenvalue}
Let us introduce:
\begin{equation}
\vartheta^{\lambda_{\max}}(G)=\max\big\{ 
            {\lambda_{\max}}(Z):
                ~Z\succeq \zeros,~Z_{ii}=1~\forall i\in[1..n],~
                 z_{ij}=0~\forall \{i,j\}\in E
            \big\},
\label{thetaeigen}
\end{equation}
where ${\lambda_{\max}}(Z)$ is the maximum eigenvalue of $Z$.

\begin{theorem} $\vartheta^{\lambda_{\max}}(G)=\vartheta(G)$ \label{ththetaeigen}\end{theorem}
\begin{proof} Using Lemma~\ref{lemmaMinRayleigh}, the maximum eigenvalue of 
$Z$ can be computed as ${\lambda_{\max}}(Z)=\max\limits_{|\x|=1} Z\sprod (\x\x^\top)$.
As such, \eqref{thetaeigen} is equivalent to:
\begin{equation}
\vartheta^{\lambda_{\max}}(G)=\max\big\{
            Z\sprod (\x\x^\top):
                ~Z\succeq \zeros,~Z_{ii}=1~\forall i\in[1..n],~
                 Z_{ij}=0~\forall \{i,j\}\in E,~
                 |\x|=1
            \big\}
\label{thetaeigen2}
\end{equation}
We show that any feasible solution $(Z,\x)$ of above program can be mapped to a feasible solution 
$Y$ of $(D\vartheta_G)$ from \eqref{dtheta1}-\eqref{dtheta4}. We showed
 in Section~\ref{secThetaDual} that $\vartheta(G)=OPT(D\vartheta_G)$.
For the reader's convenience, we recall below the definition of this program.
\begin{equation}\label{dthetalocal}
D\vartheta(G)=\max \left\{
                    Y\sprod \ones:~
                    Y\succeq \zeros,~
                    Y_{ij}=0~\forall \{i,j\}\in E,~
                    \texttt{trace}(Y) = 1
                \right\}
\end{equation}
Consider matrix $Y$ obtained by multiplying each row and column $i$ of $Z$ with $x_i$. In
other words, $Y_{ij}=Z_{ij}x_ix_j$. It is not hard to check that $Y$ reaches in \eqref{dthetalocal}
the same objective value $Y\sprod \ones=\sum\limits_{i,j\in [1..n]}Z_{ij}x_ix_j=Z\sprod
(\x\x^\top)$. Furthermore, $Y$ satisfies all constraints of \eqref{dthetalocal}, \ie, 
$\texttt{trace}(Y)=\sum\limits_{i=1}^n x^2_i=|\x|^2=1$.

By reversing the above transformation, it is quite straightforward to prove the converse: any feasible solution $Y$ of \eqref{dthetalocal} can be
mapped to a feasible solution $(Z,\x)$ of \eqref{thetaeigen2}. For this, we take
$x_{i}=\sqrt{Y_{ii}}~\forall i\in [1..n]$. We obtain $Z$ by dividing each row and column $i$
of $Y$ with $x_i$ for all $i$ such that $x_i>0$. More exactly, we obtain
$Z_{ij}=\frac{Y_{ij}}{x_ix_j}~\forall i,j\in[1..n],~x_i>0,x_j>0$. If $x_i=0$, we set 
$Z_{ii}=1$ and $Z_{ij}=0~\forall j\in [1..n]-\{i\}$. It is not hard to check that
$(Z,\x)$ satisfies the constraints of \eqref{thetaeigen2}. For instance, we have
$|\x|^2=\texttt{trace}(Y)=1$ and $Z_{ii}=\frac{Y_{ii}}{\sqrt{Y_{ii}}^2}=1$.
One can also check $Y\sprod \ones=\sum\limits_{i,j\in [1..n]}Y_{ij}=\sum\limits_{i,j\in
[1..n]}Z_{ij}x_ix_j=Z\sprod (\x\x^\top)$,
where we used $Y_{ij}=Z_{ij}x_ix_j$ that holds even if $x_i=0$ or $x_j=0$.

This proves \eqref{thetaeigen2} and \eqref{dthetalocal} are equivalent,
which means $\vartheta^{\lambda_{\max}}(G)=D\vartheta(G)=\vartheta(G)$.
\end{proof}

\subsubsection{A formulation using the maximum and the minimum eigenvalue}

Let us note ${\lambda_{\min}}(X)$ and ${\lambda_{\max}}(X)$ the minimum and resp.~the maximum
eigenvalue of $X$. We introduce
\begin{equation}
\vartheta^{\lambda}(G)=\max\big\{ 
            1 - \frac{{\lambda_{\max}}(X)}{{\lambda_{\min}}(X)}:~
                X_{ii}=0~\forall i\in[1..n],~
                X_{ij}=0~\forall \{i,j\}\in E
            \big\},
\label{thetaeigenbis}
\end{equation}
where we use the convention $\frac{{\lambda_{\max}}(X)}{{\lambda_{\min}}(X)}=0$ if
${{\lambda_{\max}}(X)}={{\lambda_{\min}}(X)}=0$. 
Since the only SDP matrix with zeros on the diagonal is $\zeros$,
any feasible $X\neq\zeros $ in 
above \eqref{thetaeigenbis} verifies $X\nsucceq \zeros$ and $X\npreceq \zeros$, and so,
${\lambda_{\max}}(X)>0>{\lambda_{\min}}(X)$. Any $X\neq \zeros$ yields a value 
$1 - \frac{{\lambda_{\max}}(X)}{{\lambda_{\min}}(X)}>1$ in \eqref{thetaeigenbis}.
If $X=\zeros$ is the optimal solution of \eqref{thetaeigenbis}, then $G$ must be a clique,
so that constraints $                X_{ij}=0~\forall \{i,j\}\in E$ force every non-diagonal
element of $X$ to be zero. In this case, we have
$\alpha(G)=\vartheta(G)=\chi(\overline{G})=\vartheta^{\lambda}(G)=1$.
\begin{theorem}$\vartheta^{\lambda}(G)=\vartheta(G)$\end{theorem}
\begin{proof}
We take the optimal $X$ in \eqref{thetaeigenbis} and we construct a feasible solution $Z$ of
\eqref{thetaeigen} with the same objective value. 
As stated above, 
$X=\zeros$
leads to
 $\vartheta^{\lambda}(G)=\vartheta(G)=1$
and any $X\neq \zeros$
satisfies
${\lambda_{\min}}(X)<0$.
Considering
${\lambda_{\min}}(X)<0$, 
we can construct $X'=\frac{X}{|{\lambda_{\min}}(X)|}$ and notice
$\frac{{\lambda_{\max}}(X)}{{\lambda_{\min}}(X)}=
 \frac{{\lambda_{\max}}(X')}{{\lambda_{\min}}(X')}=
-{{\lambda_{\max}}(X')}$. This means $X'$ is also optimal in \eqref{thetaeigenbis} and
has ${\lambda_{\min}}(X)=-1$, more exactly $X'$ lead \eqref{thetaeigenbis} to objective value
$1+{\lambda_{\max}}(X')$. Consider now $Z=X'+I$ and notice ${\lambda_{\min}}(Z)=0$, and so,
$Z\succeq \zeros$. One can easily check $Z$ is feasible in \eqref{thetaeigen} and has 
objective value ${\lambda_{\max}}(X'+I)={\lambda_{\max}}(X')+1$, \ie, the same as the
objective value of $X$ or $X'$ in \eqref{thetaeigenbis}. This shows $\vartheta^{\lambda}(G)\leq
\vartheta^{\lambda_{\max}}(G)=
\vartheta(G)$.

We still have to prove $
\vartheta^{\lambda_{\max}}(G)\leq \vartheta^{\lambda}(G)$. We now attempt to reverse
the above process. Let $Z$ be the optimal solution 
of \eqref{thetaeigen} 
and
take $X=Z-I$. Notice $X$ is feasible in \eqref{thetaeigenbis} and that
\begin{equation}\label{eqZX}
\vartheta^{\lambda_{\max}}(G)=
{\lambda_{\max}}(Z-I+I)=1+{\lambda_{\max}}(X)=1+\frac{{\lambda_{\max}}(X)}1
\leq 1-\frac{{\lambda_{\max}}(X)}{{\lambda_{\min}}(X)},
\end{equation}
where we used
$0> {\lambda_{\min}}(X)\geq -1$ that we prove now.
In fact, we already left aside 
the particular case ${\lambda_{\min}}(X)=0$ because
that would imply $X\succeq 0$ which would further lead
(by applying $\texttt{diag}(X)=\zeros$) to $X=0$, which means
the optimal $Z$ is $Z=X+I=I$ and the graph is a clique.
We can not have ${\lambda_{\min}}(X)>0$ because
that would imply $X\succ \zeros$, impossible when 
$\texttt{diag}(X)=\zeros$.
We now notice
$Z=X+I\succeq \zeros\implies {\lambda_{\min}}(X+I)\geq 0
\implies {\lambda_{\min}}(X)\geq -1$. This shows
$0> {\lambda_{\min}}(X)\geq -1$, confirming
the last inequality of \eqref{eqZX}. This finishes
the proof because \eqref{eqZX}
simplifies to
$\vartheta^{\lambda_{\max}}(G)
\leq 1-\frac{{\lambda_{\max}}(X)}{{\lambda_{\min}}(X)}
\leq \vartheta^{\lambda}(G)$.
\end{proof}

\subsection{The theta number $\vartheta(G)$ is bounded by the fractional chromatic number
$\chi^*(\overline{G})$ of $\overline{G}$}

\subsubsection{The fractional chromatic number}
Let $\CC$ be the set of cliques of $G$. In the primal-dual programs below we introduce the
most standard formulation of the fractional chromatic number of $\overline{G}$. It is not hard to see that 
the standard chromatic number $\chi(\overline{G})$ is an upper bound for these programs, by 
taking $\lambda_{C_i}=1~\forall i \in [1..\chi(\overline{G})]$,
where $\{C_1,~C_2,\dots C_{\chi(\overline{G})}\}$ is the 
optimal coloring of $\overline{G}$ (clique covering of $G$).

~\\

\noindent
  \begin{minipage}{0.5\textwidth}
  \begin{equation}\notag
    (\hat \chi^*_{\overline{G}})\left\{
      \begin{array}{lll}
        \min&  \sum\limits_{C\in \CC} \lambda_C              &                         \\
        s.t.&  \sum\limits_{\substack{C\in\CC\\ C\ni i}} \lambda_C \geq 1  &  \forall i\in[1..n]     \\
            &   \lambda_C\geq 0                              &  \forall C\in\CC 
      \end{array}
    \right.
  \end{equation}
  \end{minipage}
  \begin{minipage}{0.5\textwidth}     
  \begin{equation}\notag
        (D\hat\chi^*_{\overline{G}})\left\{
            \begin{array}{lll}
                \max & \sum\limits_{i=1}^n x_i          &                 \\
                s.t. & \sum\limits_{i\in C} x_i\leq 1  & \forall C\in \CC \\
                     & x_i\geq 0                        & \forall i\in[1..n]
            \end{array}
        \right.
  \end{equation}
  \end{minipage}

~\\

These primal-dual programs can be modified as follows. Let us drop the non-negativity constraint
$\x\geq \zeros$ in the dual $(D\hat\chi^*_{\overline{G}})$.
We prove by contradiction that the resulting program has only non-negative
optimal solutions.
Assume there is an optimal solution $\x$ such that $x_j<0$ for some $j\in[1..n]$.
We show that by increasing $x_j$ to $0$, $\x$ remains feasible. For this, take
any clique $C\ni j$;
since 
$C-\{j\}$ is naturally a clique and $\x$ is feasible, we have
$\sum_{i\in C-\{j\}}x_i \leq 1$. If we now increase $x_j$ to $0$, the new solution
satisfies 
$\sum_{i\in C}x_i=\sum_{i\in C-\{j\}}x_i\leq 1$.
This new solution is still feasible and has a higher objective value, and so, 
we obtained a contradiction. The assumption $x_j<0$ was false.
The above programs are thus equivalent to the next ones, \ie, 
$OPT(\hat \chi^*_{\overline{G}})=
OPT(D\hat \chi^*_{\overline{G}})=
OPT(\chi^*_{\overline{G}})=
OPT(D\chi^*_{\overline{G}})=
\chi^*(\overline{G})$

\noindent
\refstepcounter{equation}
  \begin{minipage}{0.5\textwidth}
  \begin{equation}\label{chistar1a}
  \tag{\theequation a}
    (\chi^*_{\overline{G}})\left\{
      \begin{array}{lll}
        \min&  \sum\limits_{C\in \CC} \lambda_C              &                         \\
        s.t.&  \sum\limits_{\substack{C\in\CC\\ C\ni i}} \lambda_C =    1  &  \forall i\in[1..n]
\\
            &   \lambda_C\geq 0                              &  \forall C\in\CC 
      \end{array}
    \right.
  \end{equation}
  \end{minipage}
  \begin{minipage}{0.5\textwidth}     
  \begin{equation}\label{chistar1b} 
  \tag{\theequation b}
        (D    \chi^*_{\overline{G}})\left\{
            \begin{array}{lll}
                \max & \sum\limits_{i=1}^n x_i          &                 \\
                s.t. & \sum\limits_{i\in C} x_i\leq 1  & \forall C\in \CC \\
                     & x_i\in \R                        & \forall i\in[1..n]
            \end{array}
        \right.
  \end{equation}
  \end{minipage}

~\\

\subsubsection{A hierarchy of SDP programs sandwiched between $\vartheta(G)$ and $\chi^*(\overline{G})$}

We will use the $(D\widetilde{\vartheta'}_G)$ formulation of $\vartheta(G)$ from
\eqref{dvarthetaprim1}-\eqref{dvarthetaprim3}. For the reader's convenience, we repeat 
the definition of this program, multiplying by $-1$ the first row and column of the
SDP matrix (this does not change its SDP status).
\begin{subequations}
\begin{alignat}{4}[left ={\left(D\widetilde{\vartheta'}_G\right) \empheqlbrace}]
\min~~&t                                    \label{ddvarthetaprim1}\\
s.t.~~&
        \begin{bmatrix}
        t       &  \ones^\top \\
         \ones   &      Z
        \end{bmatrix}
      \succeq \zeros    &&                     \label{ddvarthetaprim2}     \\
        & z_{ii}=1&&\forall i\in [1..n]        \label{ddvarthetaprim3}     \\
        & z_{ij}=0&&\forall \{i,j\}\notin E     \label{ddvarthetaprim4} 
\end{alignat}
\end{subequations}
Notice that $(\chi^*_{\overline{G}})$ in \eqref{chistar1a} is defined using variables indexed by a
subset of the power set $\P([1..n])$ of the vertex set $[1..n]$. This gives us some intuitions
that we might need programs with variables indexed by certain subsets of $[1..n]$. Let us introduce
$\P_r([1..n])=\left\{S\subseteq [1..n]:~|S|\leq r\right\}~\forall r\in[1..n]$. Given a vector
$\y$ indexed by all $I\in \P_{2r}([1..n])$, we can write $\y=(y_I)_{I\in
\P_{2r}([1..n])}$. Let us now
introduce a matrix $M_r(\y)$ with rows and columns indexed by $\P_{r}([1..n])$ such that 
$M_r(\y)_{I,J}=\y_{I\cup J}$. We can also compactly write:
\begin{equation}\label{eqMr} M_r(\y)=(y_{I\cup J})_{I,J\in \P_r([1..n])}.\end{equation}

We now introduce the following program:
\begin{equation}\label{eqpsi}
\psi^r(G)=\big\{       
            \min y_{\emptyset}:~
            M_r(\y)\succeq \zeros,~
            y_{\{i\}}=1~\forall i\in[1..n],~
            y_{\{i,j\}}=0~\forall \{i,j\}\notin E
        \big\},
\end{equation}
where one could notice that $M_r(\y)_{I,I}=M_r(\y)_{\emptyset,I}~\forall I\in \P_r([1..n])$, in
particular $M_r(\y)_{\{i\},\{i\}}=M_r(\y)_{\emptyset,\{i\}}=y_{\{i\}}=1~\forall i \in [1..n]$.
\begin{theorem} $\vartheta(G)=\psi^1(G)\leq \psi^2(G)\leq \dots \leq \psi^{\omega(G)}(G)\leq \chi^*(\overline{G})$.
\end{theorem}
\begin{proof}
We first show that 
\eqref{eqpsi} with $r=1$
is equivalent to 
$\left(D\widetilde{\vartheta'}_G\right)$ in
\eqref{ddvarthetaprim1}-\eqref{ddvarthetaprim4}.
It is actually enough to carefully ``decode'' all notations to see that
the two programs are simply
identical up to a notational
translation.
With this goal, we can replace 
$M_1(\y)_{\emptyset,\emptyset}=y_{\emptyset}$ with $t$,
$M_1(\y)_{\{i\},\{j\}}=y_{\{i,j\}}=0$
with $z_{ij}=0~\forall \{i,j\}\notin E$,
and $M_1(\y)_{\{i\},\{j\}}=y_{\{i,j\}}$ with $z_{ij}~\forall \{i,j\}\in E$;
$M_{\emptyset,\{i\}}=y_{\{i\}}=1$ for $i\in[1..n]$ is simply translated into 
the vectors $\ones$ and $\ones^\top $
that border $Z$ in \eqref{ddvarthetaprim2}.
One can check that we have just mapped one program into another, which guarantees that $\vartheta(G)=OPT\left(D\widetilde{\vartheta'}_G\right)=\psi^1(G)$.

We now show that $\psi^r(G)\leq \psi^{r+1}(G)~\forall r\in[1..n]$. 
Notice it is not possible to 
border a solution 
$M_r(\y)$ of $\psi^{r}(G)$ 
with zeros
to obtain a solution
$M_{r+1}(\y')$ of $\psi^{r+1}(G)$, because 
certain elements of 
$M_{r+1}(\y')$
are inherited from 
$M_{r}(\y')$, \eg, a proper bordering imposes $M_{r+1}(\y')_{\emptyset,[1..r+1]}=M_r(\y)_{\{1\},[1..r]}$.
Take the
solution $\y'$ that
achieves the optimum value $\psi^{r+1}(G) $.
The key is to notice
$M_r(\y')$ is a principal minor of $M_{r+1}(\y')$, and so, $M_r(\y')\succeq \zeros$.
By removing from $\y'$ all indices $S\in \P_{2r+2}[1..n]-\P_{2r}[1..n]$, we obtain a vector $\y$ 
indexed by $\P_{2r}[1..n]$ that generates a feasible solution 
$M_r(\y)$ of $\psi^{r}(G)$.
Indeed, $M_r(\y)\succeq \zeros$ because $M_r(\y)=M_r(\y')$ is a principal minor
of $M_{r+1}(\y')$; all other constraints in \eqref{eqpsi} concern the values
of $\y$ on sets of $\P_2([1..n])$ that are inherited from $\y'$. The objective value is the same,
\ie, $y_{\emptyset}=y'_{\emptyset}$. This is enough to state $\psi^r(G)\leq \psi^{r+1}(G)$, because 
\eqref{eqpsi} has a minimizing objective. The value $\psi^r(G)$ could be even strictly less
that $y_{\emptyset}$ because $\psi^r(G)$ is obtained with a program \eqref{eqpsi} with fewer
constraints than $\psi^{r+1}(G)$, \ie, it involves smaller SDP matrices $M_r(\y)$. In other words, there might
exist some $\yy$ indexed by $\P_{2r}([1..n])$ that achieves a lower value $\psi^r(G)$ and that can not
be extended to some feasible $\yy'$ indexed by $\P_{2i+2}([1..n])$.

We finally show $\psi^{\omega(G)}(G)\leq \chi^*(\overline{G})$. We take the optimum solution
$\lambda$ of $(\chi^*_{\overline{G}})$ from \eqref{chistar1a} and construct a feasible 
solution with the same objective value in \eqref{eqpsi} with
$r=\omega(G)=\P_{\omega(G)}([1..n])$. Take any 
clique $C\in \CC\subseteq \P_{\omega(G)}([1..n])$ and construct $\y^C$ indexed by $\P_{2\omega(G)}([1..n])$ such that

\begin{equation}\label{eqcases}
y^C_S=
  \begin{cases} 
   1         & \text{if } S\subseteq C \\
   0         & \text{if } S\nsubseteq C \\
  \end{cases}
\end{equation}
One only needs to decode notations to see that
$M_{\omega(G)}(\y^C)=\overline{\y^C}~\overline{\y^C}^\top\in\{0,1\}^{P_{\omega(G)}([1..n])\times P_{\omega(G)}([1..n])}$, where
$\overline{\y^C}$ is a reduced version of $\y^C$ that contains only
$\P_{\omega(G)}([1..n])$ elements $\overline{y^C_S}=y^C_S$ with $S\in
\P_{\omega(G)}([1..n])$. Indeed, if $M_{\omega(G)}(\y^C)_{I,J}=0$, then $I\cup J \nsubseteq C$ using
\eqref{eqcases}, and so, we have $I\nsubseteq C$ or $J\nsubseteq C$, which means that $y^C_I=0$ or
$y^C_J=0$. Also, if $M_{\omega(G)}(\y^C)_{I,J}=1$, then we have $I\cup J \subseteq C$ by virtue of
\eqref{eqcases}, which means that $I,J\subseteq C$, and so, $y^C_I=y^C_J=1$. This confirms that
$M_{\omega(G)}(\y^C)=\overline{\y^C}~\overline{\y^C}\succeq \zeros$. 

By applying \eqref{eqMr} on $\y=\sum\limits_{C\in\CC}\lambda_C \y^C$ we obtain
$M_{\omega(G)}\left(\y\right)=M_{\omega(G)}\left(\sum\limits_{C\in\CC}\lambda_C \y^C\right)
=
\sum\limits_{C\in\CC} \lambda_C M_{\omega(G)}\left(\y^C\right)$.
As a sum of PSD matrices $M_{\omega(G)}(\y^C)$ resp.~multiplied by positive scalars $\lambda_C$, 
the matrix $M_{\omega(G)}\left(\y\right)$ is SDP. We now check the two
non-SDP constraints of \eqref{eqpsi}.
First, we have
$y_{\{i\}}=\sum\limits_{C\in\CC}\lambda_C \y^C_{\{i\}}= \sum\limits_{C\in\CC, C\ni i} \lambda_C =    1~\forall i\in[1..n]$, using \eqref{chistar1a}.
Secondly, $y_{\{i,j\}}=0~\forall \{i,j\}\notin E$ also holds because the non-edge $\{i,j\}$
belongs to no clique. We have just constructed a feasible solution $\y$ in \eqref{eqpsi} for
$r=\omega(G)$ with objective value 
$y_{\emptyset}=\sum\limits_{C\in \CC} \lambda_C \y^C_{\emptyset} = \sum\limits_{C\in \CC} \lambda_C=\chi^*(\overline{G})$. 
This is enough to conclude $\psi^{\omega(G)}(G)\leq
\chi^*(\overline{G})$.

Parts of this proof are a simplification of the proof of Theorem 3.1.~from 
the article ``The operator $\Psi $ for the chromatic number of a graph''
of Neboj\originalv{s}a Gvozdenovi\'c and Monique Laurent.%
\footnote{Published in \textit{SIAM Journal on Optimization} in 2008, vol 19(2), pp. 572--591,
available on-line as of 2017 at 
\url{http://homepages.cwi.nl/~monique/files/SIOPTGL1.pdf}.}
\end{proof}

\section{A taste of copositive optimization and sum of squares hierarchies}

\subsection{\label{seccopos}Introducing the completely positive and the copositive cones}
Let us try to produce better relaxations and reformulations 
by replacing the cone $S_n^+$ of SDP matrices with a smaller cone.
For this, we introduce the cone of \textit{completely positive matrices}:

\begin{align}\label{eqcomplpositive1}
C^{n*}&=\left\{X\in S_n:~X=\sum_{i=1}^k \y_i\y_i^\top \text{ with } \y_i\geq \zeros~\forall i\in[1..k]\right\}\\
       &=\texttt{conv}\left\{\y\y^\top:\y\geq \zeros~\right\},\label{eqcomplpositive2}
\end{align}
where $S_n\subsetneq \R^{n\times n}$ is the set of real symmetric matrices
and the operator $\texttt{conv}(...)$ produces all convex combinations of the elements from the set given as argument.
It is clear that any $X$ from \eqref{eqcomplpositive1} can be written as a convex combination of rank-1 matrices
as in \eqref{eqcomplpositive2}.
For this, it is enough to write $X=\sum_{i=1}^k \y_i\y_i^\top=
\sum_{i=1}^k \frac 1k \left(\sqrt{k}\y_i\right)\left(\sqrt{k}\y_i\right)^\top$.
Any $X\in C^{n*}$ 
can be written $X=YY^\top$ with $Y=[\y_1~\y_2~\dots \y_k]$, and so, 
it is clear that $X$ is also SDP using Prop.~\ref{propranktransprod}.
The difference between a completely positive matrix $X$ and an SDP one is that
above $Y$ needs to be non-negative. Any SDP matrix $\widetilde {X}\succeq \zeros$ can be written as $\widetilde{X}=\widetilde{Y}\widetilde{Y}^\top$
using various decompositions (\eg, Cholesky, eigenvalue or square root, see Corollary \ref{corSdpToVectors}),
but $\widetilde{Y}$  is not necessarily non-negative.
For any $n\geq 2$, we have $C^{n*}\subsetneq S_n^+$.

We can define the dual of a cone $C$ (with regards to the scalar product) using
the formula
$C^*=\left\{Y:~X\sprod Y\geq 0~\forall X\in C\right\}$. We already proved in
Prop.~\ref{propSDPSelfDual} that the SDP cone is self-dual, \ie, $(S_n^+)^*=S_n^+$.
Since $C^{n*}$ is smaller than $S_n^*$, its dual cone might be larger than $S_n^+$. Indeed,
we introduce the \textit{cone of copositive matrices}
\begin{equation}\label{eqcopositive}
C^n=\left\{X\in S_n:~X\sprod \y\y^\top\geq 0~\forall \y\geq \zeros\right\}
\end{equation}
such that $C^n=(C^{n*})^*$.

Let $\NN^n\subsetneq S_n$ be the set of non-negative symmetric matrices. 
The following hierarchy of inclusions
\begin{equation}\label{eqHierarchySDP}
C^{n*}\subset S_n^+\cap \NN^n \subset S_n^+ \subset S_n^+ + \NN^n \subset C^n
\end{equation}
 follows from two short arguments.
First, any $X\in C^{n^*}$ satisfies $X\succeq\zeros$ (see above)
and $X\geq\zeros$ (see the definition \eqref{eqcomplpositive1}).
Secondly, any $S\succeq \zeros$ and $N\geq \zeros$ verify
$S\sprod \y\y^\top \geq 0$ and
$N\sprod \y\y^\top \geq 0$ for any $\y\geq \zeros$.

\subsection{Reformulating a homogeneous quadratic program as a copositive problem}

We consider the following problem with a {\it homogeneous objective}
function and \textit{non-negative variables}:

\begin{subequations}
\begin{align}[left ={(QP_+)  \empheqlbrace}]
\min~~&     Q\sprod \x\x^\top        \label{eqqpplus1}\\
s.t~~ &     \a^\top \x = b           \label{eqqpplus2}\\
      &\x\in \R_+^n,                 \label{eqqpplus3}
\end{align}
\end{subequations}
where $\a>0$ is a {\it strictly positive parameter}.
One should keep in mind that  $\x$ is non-negative.  

\subsubsection{Solving $(QP_+)$ is NP-hard}
Not surprisingly, solving this program is NP hard. This 
follows from the fact that it contains the maximum stable problem as a particular 
case,%
\footnote{If we allow the objective
to be non-homogeneous, 
we obtain a particular case of
unconstrained quadratic programming {\it in non-negative variables}.
This problem is NP-hard because it is at least as hard as the bi-partition
problem, \ie, 
one can solve the (bi-)partition problem for elements 
$a_1,a_2,\dots a_n$ by solving 
$\min_{x_i,x'_i\geq 0}\sum_i (x_i + x'_i - 1)^2 + x_ix'_i + \left( \sum_i x_i a_i - \frac 12 \sum_i
a_i\right)^2$. More generally, unconstrained quadratic programming is not
NP-hard because it reduces to SDP programming
(Section~\ref{secqpunconstrained}).
}
as a consequence of the following result.
\begin{proposition}\label{propStableCopositive} Consider a graph $G$ with adjacency matrix $A^G$.
The maximum stable $\alpha(G)$ can be determined by solving the following program,
which is a particular case of $(QP_+)$ from above \eqref{eqqpplus1}-\eqref{eqqpplus3}.
\begin{equation}\label{eqStableViaQP}\frac 1{\alpha(G)}=
\min\left\{\left(A^G+I_n\right)\sprod \x\x^\top:~\e_n^\top \x = 1\right\},
\end{equation}
where $\e_n^\top=\underbrace{[1~1~1\dots 1]}_{\text{n positions}}$.
\end{proposition}
\begin{proof} We use the technique described next. Consider any feasible solution
$\x$ of above program \eqref{eqStableViaQP}.
We say that an edge $\{i,j\}$ 
is \textit{supported} by $\x$ if $x_i,x_j> 0$. We will show that 
if $\x$ has more than zero supported edges, then it be transformed into a no-worse solution $\x'$ 
that has strictly fewer supported edges. 
For this, let us start by
evaluating the contribution of $x_i$ and
$x_j$ to the objective value of edge $\{i,j\}$:
\begin{align}\label{eqval1}
\texttt{val}(x_i,x_j)&= x_i^2+2x_ix_j+x_j^2+2\sum\limits_{k\in[1..n]-\{i,j\}}A^G_{ki}x_kx_i
                                          +2\sum\limits_{k\in[1..n]-\{i,j\}}A^G_{kj}x_kx_j\\
                     &= (x_i+x_j)^2 + z_i x_i + z_j x_j, \label{eqval2}
\end{align}
where the values $z_i$ and $z_j$ are determined by the above sums and do not depend
on $x_i$ or $x_j$. Consider now the function $f:[0,x_i+x_j]\to \R$ defined by
$f(t)=\texttt{val}(t,x_i+x_j-t)$, \ie, we replace $x_i$ with $t$ and
$x_j$ with $x_i+x_j-t$ in above \eqref{eqval1}-\eqref{eqval2}.
This way $f(t)$ can be written
$$f(t)=(x_i+x_j)^2+ z_i t + z_j(x_i+x_j-t)= (z_i-z_j)t + (x_i+x_j)^2 +
z_j(x_i+x_j)$$
The only-non constant term is $(z_i-z_j)t$. We obtain that $f$ is a
linear function that reaches its minimum value at one of
the two bounds of $[0,x_i+x_j]$, \ie, a 
value $\overline{t}$ that minimizes $f$ is either
$\overline{t}=0$ or $\overline{t}=x_i+x_j$. 
We can say
$f\left(\overline{t}\right)=\texttt{val}\left(\overline{t},x_i+x_j-\overline{t}\right)$ is at least as good
as $f(x_i)=\texttt{val}(x_i,x_j)$.

If we define $x'_i=\overline{t},~x'_j=x_i+x_j-\overline{t}$
and $x'_k=x_k~\forall k\in[1..n]-\{i,j\}$, we obtain that $\x'$ is at least 
as good as $\x$. Notice that the edge $\{i,j\}$ is no longer supported in $\x'$
because $x'_i$ or $x'_j$ is equal to zero. Any other supported edge in $\x'$ is 
also supported in $\x$, because $x'_{\ell}>0\implies x_{\ell}>0~\forall \ell\in[1..n]$. 
This means that $\x'$ has at least one supported edge less than $\x$. By repeating
above operation iteratively for all supported edges, we will eventually find
some $\xx'$ that is at least as good as $\x$ and has no supported edge.

This means that the non-zero elements of $\xx'$ generate a stable $S$ of $G$. The objective
value of $\xx'$ can be written $I_{|S|}\sprod \xx~\xx^\top$, where
$\xx$ is a vector containing only the $|S|$ non-zero values of $\xx'$. We
can still say $\e_{|S|}^\top \xx=1$, \ie, the sum of the elements 
of $\xx$ (or $\xx'$) is one.
We now show that all elements of $\xx$ are equal if $\xx'$ is optimal.
Assume the contrary: there is $i,j\in[1..|S|]$ such that $\xxx_i\neq \xxx_j$. The
contribution of $\xxx_i$ and $\xxx_j$ to the objective function is 
$\overline{\texttt{val}} (\xxx_i,\xxx_j)={\xxx_i^2} +  {\xxx_j^2}=
\frac 12 \left(2\xxx_i^2+2 \xxx_j^2\right)
>
\frac 12 \left(\xxx_i^2+\xxx_j^2 + 2\xxx_i\xxx_j\right)
=
2 \left(\frac{\xxx_i+\xxx_j}2\right)^2$.
By replacing $\xxx_i$ and $\xxx_j$ with 
$\left(\frac{\xxx_i+\xxx_j}2\right)$ we obtain a better objective
value, while still respecting the constraint (the sum of the elements
remains the same). This means that $\xx'$ is not optimal, which is a contradiction.
As such, the optimal $\xx'$ has the same value $\frac 1{|S|}$ 
at all positions $i \in S$ and its objective value is $\sum_{i=1}^{|S|}
\frac{1}{|S|^2}=\frac {|S|}{|S|^2}=\frac 1{|S|}$.

We started from an arbitrary solution $\x$ and we constructed
a solution $\xx'$ that is no worse than $\x$ and that has no supported edge;
this means there is always an optimal solution with no supported edges.
We then showed that if $\xx'$ is optimal, it has to contain the value $\frac 1{|S|}$
on all positions of some stable $S$  of $G$; the objective value of this
solution
is $\frac 1{|S|}$.
But this value can only be optimal if 
$S$ is a maximum stable of $G$ with $|S|=\alpha(G)$.
 We obtained that the optimum of the program 
from \eqref{eqStableViaQP} is indeed $\frac 1{\alpha(G)}$.
It is achieved by taking $x_i = \frac{1}{|{S}|}$ for all
$i\in{S}$ and $x_i=0$ if $i\notin S$
for a maximum stable $S$.
\end{proof}

According to the article ``Copositive Optimization'' by 
Immanuel Bomze, Mirjam D{\"u}r and Chung-Piaw Teo,%
\footnote{Published in the \textit{Optima} 89 newsletter in august 2012,
pp.~2-8, available on line at 
\url{http://www.mathopt.org/Optima-Issues/optima89.pdf}.
\setcounter{testfoot4}{\value{footnote}}\label{testfootpage4}%
}
the above result dates back to the 1960s. However, the 
proof is personal.

\subsubsection{\label{secqpplus}The reformulation of $(QP_+)$ as a copositive program}

Let us consider the following program associated to $(QP_+)$ from
\eqref{eqqpplus1}-\eqref{eqqpplus3}.

\begin{subequations}
\begin{align}[left ={C(QP_+)  \empheqlbrace}]
\min~~&     Q\sprod  X                          \label{eqqcpplus1}\\
s.t~~ &   A\sprod X=  (\a\a^\top)\sprod X = b^2 \label{eqqcpplus2}\\
      &X\in C^{n*},                             \label{eqqcpplus3}
\end{align}
\end{subequations}
where recall $A=\a\a^\top$ respects 
$A\geq \zeros$ and 
$A_{ii}>0~\forall i\in [1..n]$
because 
\eqref{eqqpplus1}-\eqref{eqqpplus3} imposes $a_i>0\forall i \in[1..n]$, \ie,
$\a$ is a strictly positive parameter.
Since the above program minimizes a
linear function, its optimum might only be achieved by an extreme point (or along an extreme
ray) of the feasible area defined by \eqref{eqqcpplus2}-\eqref{eqqcpplus3}.

\begin{proposition}\label{propExtremeSolCnstar} We consider a symmetric matrix $A\geq \zeros$ 
with a strictly positive diagonal ($A_{ii}>0~\forall i\in[1..n]$) and some $b\in \R$.
The extreme solutions (vertices) 
of the set below are the rank-1 matrices of the form $X=\u\u^\top$ with $\u\geq
\zeros$. This set has no extreme rays.
\begin{equation}\label{eqCWithConstraint}
[C^{n*}_A]= \left\{X\in C^{n*}:~A\sprod X = b^2\right\}
\end{equation}
\end{proposition}
\begin{proof}We will prove three facts:
\begin{enumerate}
\item [(i)] A rank-1 completely positive matrix $X$ such that $AX=b^2$ is an
extreme solution of $[C^{n*}_A]$.
\item [(ii)] A completely positive matrix $X$ of rank higher than one can not be
an extreme solution of $[C^{n*}_A]$.
\item [(iii)] The set $[C^{n*}_A]$ has \textit{no} extreme ray.
\end{enumerate}

\noindent (i)\\
We first prove that a rank-1 matrix $X=\y\y^\top$ (with $\y\geq\zeros$) 
such that $AX=b^2$ is an
extreme solution of this $[C^{n*}_A]$ set. Assume the contrary: there is symmetric non-zero $M\in\R^{n\times n}$
such that $X-M,X+M\in [C^{n*}_A]$. Based on~\eqref{eqcomplpositive1}, we 
can write $X-M=Y_aY_a^\top$ and $X+M=Y_bY_b^\top$; combining the two, we obtain $X=\frac 12
[Y_aY_b][Y_aY_b]^\top$. But since $rank(X)=rank\left(\y\y^\top\right)=1$, we can apply
Prop.~\ref{propranktransprod} to conclude
$rank(X)=rank\left([Y_aY_b][Y_aY_b]^\top\right)=rank([Y_aY_b])$, and so, 
we have $rank([Y_aY_b])=1$. The columns of $Y_a$ and $Y_b$ need to be multiples
of $\y$. The matrices $X$ and $X-M$ are multiples of $\y\y^\top$, and so, 
$M$ is a multiple of $X$, \ie, $M=tX$ with $t\neq 0$. This contradicts 
$X-M\in [C^{n*}_A]$ because $A\left(X-M\right)=b^2 - tb^2\neq b^2$.
The assumption $X-M,X+M\in [C^{n*}_A]$ for some $M\neq \zeros$ was false, and so, $X$ 
is an extreme solution.

\noindent (ii)\\
We now show that any matrix $X\in C^{n*}$ of rank higher than 1 can not
be an extreme solution of $[C^{n*}_A]$. Assume there exists $k$
non-zero
vectors $\y_1,~\y_2,\dots \y_k\geq \zeros$ such that 
$rank[\y_1~\y_2~\dots \y_k]>1$ (\ie, they are not all multiples of
the same vector) and
$X=\sum_{i=1}^k \y_i\y_i^\top$, recall definition \eqref{eqcomplpositive1}.
Without loss of generality, we suppose $\y_1$ and $\y_2$ are linearly
independent. We can write:
\begin{equation}\label{eqAX}
A\sprod X = \underbrace{A\sprod \y_1\y_1^\top}_{t_1}
          + \underbrace{A\sprod \y_2\y_2^\top}_{t_2}
          + A\sprod \sum_{i=3}^k \y_i\y_i^\top.
\end{equation}
Notice we have $t_1,~t_2>0$, because 
$A\sprod \y_1\y_1^\top\geq \texttt{diag}(A)\sprod \texttt{diag}(\y_1\y_1^\top)>0$ 
follows from the fact that
$A$ and $\y_1$ are non-negative,
the diagonal of $A$
is strictly positive and $\y_1$ is non-null. An analogous argument proves
$t_2>0$. We now introduce the following family of matrices depending on a
parameter $\alpha$.
\begin{equation}\label{eqXalpha}
X_{\alpha} = X + \alpha
                \underbrace{
                    \left(
                        \frac 1{t_1}\y_1\y_1^\top
                        -\frac 1{t_2}\y_2\y_2^\top
                    \right)
                }_{\substack{\neq 0\text { because }\y_1\text{ and} \\ \y_2\text{ are independent}}}
\end{equation}
Notice $X_{\alpha}$ remains completely positive for a sufficiently small
positive or negative $\alpha$ --- more exactly the limits of $\alpha$ are 
$\alpha \in [-t_1,t_2]$.
Let us check if $X_{\alpha}\in [C^{n*}_A]$ by calculating the scalar product
with $A$:
$$
AX_{\alpha}= AX + \alpha
                    \left(
                        \frac 1{t_1} A\sprod \y_1\y_i^\top
                        -\frac 1{t_2}A\sprod \y_2\y_2^\top
                    \right)
            =AX + \alpha\left(\frac {t_1}{t_1}-\frac{t_2}{t_2}\right)
            =AX
$$
We can move from $X$ back and forward  along non-zero
$
                    \left(
                        \frac 1{t_1}\y_1\y_i^\top
                        -\frac 1{t_2}\y_2\y_2^\top
                    \right)
$ 
until we reach (the above) limits of $\alpha$.
Thus, such $X$ can not be an extreme solution of $[C^{n*}_A]$.

\noindent (iii)\\
We finally show (by contradiction) that $[C^{n*}_A]$ can not contain an extreme ray of the form $X+tZ$
with $t>0$ and $Z\neq \zeros_{n\times n}$. 
Assume such $Z$ exists. We can easily notice $Z\geq \zeros$ because $C^{n*}\subset \R_+^n$. 
Since $A\sprod Z$ has to be zero and the diagonal of $A$ is
strictly positive, 
we also obtain $\texttt{diag}(Z)=\zeros$, which means $Z$ is {\it not} SDP
by applying
Corollary~\ref{corolZeros}.
There exists a vector $\u$ such that $Z\sprod \u\u^\top
=-z <0$. 
But now notice that $(X+tZ)\sprod \u\u^\top
=X\sprod \u\u^\top - tz $ can become negative for a sufficiently large $t$, 
which means $X+tZ$ is not SDP. 
This is a contradiction because $C^{n*}\subset S_n^+$ in
\eqref{eqHierarchySDP}.
\end{proof}

The above Prop.~\ref{propExtremeSolCnstar} leads to the fact that the optimal 
solution of 
$(C(QP_+))$ from \eqref{eqqcpplus1}-\eqref{eqqcpplus3} has the form
$X=\y\y^\top$ with $\y\geq \zeros$. This means $\y$ is an optimal
solution of $(QP_+)$ from \eqref{eqqpplus1}-\eqref{eqqpplus3}. Indeed,
if $(QP_+)$ would have a solution $\yy$ of better quality, $\overline{X}=\yy
~\yy^\top$ would also be a solution of better quality than $X$ in $(C(QP_+))$.
This means $(C(QP_+))$ is an exact reformulation of $(QP_+)$.

The difficulty of $(C(QP_+))$ is hidden in the cone constraint. Indeed, checking
membership in $C^{n*}$ is NP-hard. In particular, if we try to factorise $X\in
C^{n*}$ into $X=YY^\top$ using any of the decompositions presented for SDP matrices (\eg,
Cholesky, eigenvalue or square root, see Corollary \ref{corSdpToVectors}), the
factor $Y$ is \textit{not necessarily} non-negative. 
It is still an open question whether checking $C^{n*}$ membership is in NP (making
the problem NP-complete, because it is NP-Hard) or not.
Checking membership in the dual cone $C^n$ is
even co-NP complete.%
\footnote{Recall a decision problem is NP (reps.~co-NP) if 
and only if
there is a polynomial-time algorithm that can verify ``yes'' (resp.~``no'') instances.
}
Under a legitimate well-accepted (but still open)
assumption co-NP$\neq$ NP, a co-NP complete problem can not belong to NP -- if that were
the case, all co-NP problems would belong to NP, which is unlikely.
 This way,
it is very likely that checking $C^n$ membership is \textit{not} even in NP.
For more details on such aspects, we refer the reader to
the article 
``On the computational complexity of membership problems for the completely
positive cone and its dual'' by 
Peter Dickinson and
Luuk Gijben.%
\footnote{Published in \textit{Computational optimization and applications}
in 2014, vol 57(2), pp. 403--415, available on-line at
\url{http://www.optimization-online.org/DB_FILE/2011/05/3041.pdf}.}

\subsubsection{Comparing with the SDP relaxation of $(QP_+)$}
We now investigate the reasons why 
replacing $C^{n*}$ with $S_n^+$
does not lead to such a strong result (reformulation).
In other words, we replace the completely positive ($A$-)constrained set $[C^{n*}_A]$ 
from \eqref{eqCWithConstraint}
with an SDP set $[SDP_A]$ using the same constraint defined by $A$.
More exactly, we investigate why the characterization of extreme solutions
from Prop.~\ref{propExtremeSolCnstar}
of $[C^{n*}_A]$ does not hold in the same way for
\begin{equation}\label{eqCWithConstraintSDP}
[SDP_A]= \left\{X\succeq \zeros:~A\sprod X = b^2\right\}.
\end{equation}

First, we can still say that the rank-1 matrices $X\succeq \zeros$ such that
$A\sprod X=b^2$ are extreme solutions of $[SDP_A]$. It is enough to check
that the arguments for point (i) from the proof of
Prop.~\ref{propExtremeSolCnstar} still hold for $[SDP_A]$. However, one should
bear in mind that a rank-1 SDP matrix $X=\y\y^\top$ might \textit{not} verify
$\y\geq 0$, and so, $\y$ might \textit{not} be a feasible solution of the initial program
$[QP_+]$ from \eqref{eqqcpplus1}-\eqref{eqqcpplus3}.

Secondly, the SDP set might have extreme rays and the proof of (iii) from 
Prop.~\ref{propExtremeSolCnstar} fails in the SDP case. This follows from the fact that even if $A\geq
\zeros$, we can still find SDP matrices $Z$ such that $A\sprod Z=0$. This means
there might well be matrices $X+tZ\in[SDP_A], \forall t>0$, based on $A\sprod
(A+tZ)=A\sprod X = b^2$. If $\exists Z\succeq \zeros$ such that $A\sprod Z=0$
and $Q\sprod Z<0$, the SDP relaxation of 
$[QP_+]$ from \eqref{eqqcpplus1}-\eqref{eqqcpplus3} is unbounded. We will
assume that this relaxation is not unbounded, \ie, $A\sprod Z=0
\implies Q \sprod Z \geq 0$.

Finally, we tackle the point (iii) of the proof of
Prop.~\ref{propExtremeSolCnstar}. We can still prove there is no
extreme solution of $[SDP_A]$ with rank higher than 1.
For this, we can still write \eqref{eqAX} with independent $\y_1$ and 
$\y_2$. However, we can no longer state $t_1,~t_2>0$. If we have 
$t_1=0$, then $\y_1\y_1^\top$ is a ray of $[SDP_A]$ and it is clear $X$ is not
an extreme point because we can add or subtract multiples of $\y_1\y_1^\top$
from the description $X=\sum_{i=1}^k \y_i\y_i^\top$ and remain in $[SDP_A]$.
The same happens if $t_2=0$.
We can hereafter assume $t_1\neq 0$ and $t_2\neq 0$. This way, we can still 
construct the family $X_\alpha$ of matrices from \eqref{eqXalpha}.
One can check that $X_\alpha$ remains in the $[SDP_A]$ for sufficiently small
values of $\alpha$, \ie, check that if 
$\alpha\in [-\varepsilon,+\varepsilon]$ with $\varepsilon<\min(|t_1|,|t_2|)$, 
then the coefficients of $\y_1\y_1^\top$ and 
$\y_2\y_2^\top$ in the sum composing $X_\alpha$ remain strictly positive.
This is enough to guarantee that $X$ is
not
an extreme solution.

%

To summarize, we have found two differences 
between the completely positive (re-){\it formulation} and the SDP {\it relaxation}.
First, the feasible area in the SDP case can have extreme rays and the
objective is unbounded if there is $Z\succeq \zeros$ such that $A\sprod Z=0$ and
$Q\sprod Z<0$. Secondly, it the objective is not unbounded, the optimal solution
has rank 1 like in the completely positive case, but it is not necessarily non-negative.

\subsection{Relaxations of the copositive formulation of the maximum stable}

\subsubsection{A second completely positive formulation of the maximum stable}
We have already 
introduced the $\alpha(G)$ formulation \eqref{eqStableViaQP} and 
proven it in Prop.~\ref{propStableCopositive}. We will show that

\begin{equation}\label{eqStableViaQP2}
{\alpha(G)}=\max
\left\{\e_n\e_n^\top \sprod X:~
\left(A^G+I_n\right)\sprod X=1,~X\in C^{n*}\right\},
\end{equation}
where 
 $A^G$ is the adjacency matrix of 
graph $G$ and recall $\e_n^\top=\underbrace{[1~1~1\dots 1]}_{\text{n positions}}$.

First, notice this program is very similar to $C(QP_+)$ from \eqref{eqqcpplus1}-
\eqref{eqqcpplus3}. In particular, the matrix $A^G+I_n$ satisfies exactly all
conditions imposed on $A$ in Prop.~\ref{propExtremeSolCnstar}. As such, an 
optimal solution of \eqref{eqStableViaQP2} can be achieved by an extreme
solution of the feasible area that has the form $X=\y\y^\top$ with $\y\geq
\zeros$ by virtue of Prop.~\ref{propExtremeSolCnstar}. It is enough to prove the
following:
\begin{equation}\label{eqStableViaQP3}
{\alpha(G)}=\max
\left\{\e_n\e_n^\top \sprod \y\y^\top:~
\left(A^G+I_n\right)\sprod \y\y^\top=1\right\},
\end{equation}
Consider a feasible solution $\x$ of \eqref{eqStableViaQP} with objective
value 
$\left(A^G+I_n\right)\sprod \x\x^\top=\frac 1t$. Let us define $\y=\sqrt{t}\x$ and
one can calculate 
$\left(A^G+I_n\right)\sprod \y\y^\top=\sqrt{t}^2 \left(A^G+I_n\right)\sprod
\x\x^\top=t\frac 1t= 1$, \ie, $\y$ is feasible in \eqref{eqStableViaQP3}. Based on $\e_n^\top \x=1$, we obtain $\e_n^\top
\y=\sqrt{t}$, and so, $\e_n\e_n^\top \sprod \y\y^\top=\sqrt{t}^2=t$. From a feasible
solution $\x$ of \eqref{eqStableViaQP} with value $\frac 1t$, we constructed
a feasible solution in \eqref{eqStableViaQP3} with objective value $t$.
The converse is also possible. Consider any feasible solution $\y$ of
\eqref{eqStableViaQP3} of objective value $(\e_n^\top \y)^2=t$ (notice any feasible
$\y$ is non-zero, and so, $t>0$). Let us define $\x=\frac {\y}{\sqrt{t}}$. From
$(\e_n^\top \y)^2=t$, we have $(\e_n^\top \x)^2=\left(\frac 1{\sqrt{t}}\e_n^\top
\y\right)^2=\frac 1{\sqrt{t}^2}t=1$, \ie, $\x$ is feasible in
\eqref{eqStableViaQP}. The objective value of $\x$ is
$\left(A^G+I_n\right)\sprod \x\x^\top=
\frac 1{\sqrt{t}^2}\left(A^G+I_n\right)\sprod \y\y^\top=\frac 1t$.

From a feasible solution of \eqref{eqStableViaQP} with objective value $\frac
1t$ we can construct a feasible solution of \eqref{eqStableViaQP3} with objective
value  $t$. Conversely, from a feasible solution of \eqref{eqStableViaQP3} with
objective value $t$ we can construct a feasible solution of
\eqref{eqStableViaQP} with objective value $\frac 1t$. This is enough to
guarantee that the optimum of \eqref{eqStableViaQP3} is 1 divided by the optimum
of \eqref{eqStableViaQP}, \ie, it is $\frac {1}{\frac 1{\alpha(G)}}=\alpha(G)$.

\subsubsection{\label{secsos}Sum-of-squares relaxations of the copositive formulation of $\alpha(G)$}

\paragraph{The copositive formulation of $\alpha(G)$}

The dual of \eqref{eqStableViaQP2} can be calculated as in the SDP case, 
see 
also
the description of primal-dual conic programs from
Section~\ref{secconWeakDual}.
We can apply the technique used in the proof of
Prop.~\ref{propTwiceDualize}, but we need the dual cone of $C^{n*}$, \ie, $C^n$.
The dual of \eqref{eqStableViaQP2} can thus be written as follows:
\begin{equation}\label{eqStableCopos}
\alpha(G) = \min \left\{t:~t\left(A^G+I_n\right)-\e_n\e_n^\top \in C^n\right\}.
\end{equation} 

It is possible to show that both \eqref{eqStableViaQP2} and
\eqref{eqStableCopos} are strictly feasible. Let us start with 
\eqref{eqStableViaQP2} and notice that if $Y\in\inter\left(C^{n*}\right)$, then
$\texttt{diag}(Y)\neq 0$ and $\left(A^G+I_n\right)\sprod Y\geq
\texttt{diag}(Y)\sprod \texttt{diag}(I_n)>0$. As such, we can define 
$X=\frac Y{\left(A^G+I_n\right)\sprod Y}$ that also belongs to
$\inter\left(C^{n*}\right)$
and is feasible in \eqref{eqStableViaQP2}. We only need to show that $C^{n*}$
has a non-empty interior, \ie, we have to generate some strictly feasible
matrix of $C^{n*}$.
For this, consider the 
set $\AA^{(0,1)}=\left\{nI_n+A\in S_n:~0<A_{ij}<1\right\}$. For any 
$nI_n+A\in \AA^{(0,1)}$ and $M\in S_n$, there is a sufficiently small
$\varepsilon>0$ such that $nI_n+A-\varepsilon M,~nI_n+A+\varepsilon M\in \AA^{(0,1)}$.
Because of this last property, 
to prove
$\AA^{(0,1)}\subsetneq \inter\left(C^{n*}\right)$ 
it is now enough to show $\AA^{(0,1)}\subsetneq C^{n*}$.
We can write 
\begin{equation}\label{eqintercnstar}\AA^{(0,1)}\ni nI_n+A=\sum\limits_{i < j} A_{ij} E^{ij}+
\sum\limits_{i\in[1..n]} \left(n+A_{ii}-\sum_{j\neq
i}A_{ij}\right)E^{ii},\end{equation}
where $E^{ij}$ is a matrix full of zeros except at positions $(i,i)$,
$(i,j)$, $(j,i)$ and $(j,j)$ where it has ones. All terms in
above \eqref{eqintercnstar} can be written under the form $\y\y^\top$ for a
non-negative $\y$. For $A_{ij}E^{ij}$, it is enough to take a $\y$ full of 
zeros except for $y_i=y_j=\sqrt{A_{ij}}$. For 
$\left(n+A_{ii}-\sum_{j\neq
i}A_{ij}\right)E^{ii}$, we take a $\y$ full of zeros, except for
$y_i=\sqrt{n+A_{ii}-\sum_{j\neq i}A_{ij}}>\sqrt{n+A_{ii}-(n-1)}>0$. This allows
one to write $nI_n+A$ in the form required by the $C^{n*}$ definition
\eqref{eqcomplpositive1}.

We now show \eqref{eqStableCopos} has strictly feasible solutions. 
It is enough to show there is a sufficiently large $t$ such that
$\left(t\left(A^G+I_n\right)-\e_n\e_n^\top\right)\sprod \y\y^\top >0$
for any non-negative $\y\geq \zeros$. Using
$\left(t\left(A^G+I_n\right)-\e_n\e_n^\top\right)\sprod \y\y^\top
\geq 
\left(t I_n-\e_n\e_n^\top\right)\sprod \y\y^\top$, it suffices 
to show that this last scalar product is strictly positive if $t$ is large
enough.  But this simply follows from the fact that the matrix
$t I_n-\e_n\e_n^\top$ becomes diagonally dominant and can have arbitrarily
large eigenvalues when $t\to \infty$, \ie, we can easily have
$t I_n-\e_n\e_n^\top \succ \zeros$.

Since both \eqref{eqStableViaQP2} and \eqref{eqStableCopos} are strictly
feasible, we can apply the strong duality Theorem~\ref{conthStrongDual3}
for linear conic programming that states
that both programs have the same optimum value and they effectively reach
this value.

\paragraph{\label{secStrengtheningNatural}A ``natural'' strengthening bounded by the Lov\'asz number}

There are several cone hierarchies and relaxations used to approximate the
copositive cone $C^n$. Let us first recall hierarchy \eqref{eqHierarchySDP}
and define a ``natural'' strengthening (more constrained restricted version)
that replaces $C^n$ with 
$S_n^+ + \NN^n$, where $\NN^n\subsetneq S_n$ is the set of non-negative symmetric matrices.
The strengthening can simply be written by modifying \eqref{eqStableCopos}:

\begin{equation}\label{eqStableCoposRelax}
\alpha^0(G) = \min \left\{t:~t\left(A^G+I_n\right)-\e_n\e_n^\top \in S_n^+ +
\NN^n\right\} ~~\left[\geq \alpha(G)\right].
\end{equation} 
Since the dual cone of $S_n^+ + \NN^n$ is $S_n^+\cap \NN^n$,\footnote{To see
this, 
simply notice $S_n^+\subsetneq S_n^+ + \NN^n$, and so, 
$\left(S_n^+ + \NN^n\right)^* \subsetneq (S_n^+)^*=S_n^+$. Similarly, 
$\left(S_n^+ + \NN^n\right)^*\subsetneq (\NN^n)^*=\NN^n$.
As such, $\left(S_n^+ + \NN^n\right)^*\subset S_n^+\cap \NN^n$. One
can check that $X\in S_n^+\cap \NN^n$ and $Y=Y_1+Y_2$ with $Y_1\in S_n^+$ and
$Y_2 \in \NN^n$ yield $X\sprod (Y_1+Y_2)=X\sprod Y_1 + X\sprod Y_2\geq 0$, using
the fact that $X$ is both SDP and non-negative.
\setcounter{testfoot8}{\value{footnote}}\label{testfootpage8}%

}
the dual can be written as follows:
\begin{equation}\label{eqStableComplPosRelax}
{\alpha^0(G)}=\max
\left\{\e_n\e_n^\top \sprod X:~
\left(A^G+I_n\right)\sprod X=1,~X\in S_n^+\cap \NN^n\right\}.
\end{equation}

 The first program
\eqref{eqStableCoposRelax} is strictly feasible by taking a sufficiently large
$t$, so that 
$t\left(A^G+I_n\right)-\e_n\e_n^\top =
\underbrace{tI_n-\e_n\e_n^\top}_{\succ \zeros}+\underbrace{tA^G}_{\geq \zeros}$
and
$t\left(A^G+I_n\right)-\e_n\e_n^\top \pm \varepsilon M=
\underbrace{tI_n-\e_n\e_n^\top\pm\varepsilon M}_{\succ
\zeros}+\underbrace{tA^G}_{\geq \zeros}$ using a sufficiently small $\varepsilon>0$
for any $M\in S_n$.
The dual \eqref{eqStableComplPosRelax} is strictly feasible taking 
$X=\frac {tI_n+\e_n\e_n^\top}{(A^G+I_n)\sprod\left({tI_n+\e_n\e_n^\top}\right)}$
for a sufficiently large $t$. 
We can apply the strong duality Theorem~\ref{conthStrongDual3}
for linear conic programming to state 
 that both programs have the same optimum
value $\alpha^0(G)$ and they effectively reach this value.

We now prove the following:
\begin{equation}\label{eqStableCoposRelaxRelax}
\alpha^0(G) = \min \Big\{t:~tI_n+\sum_{\{i,j\}\in E: i<j} t_{ij}E_{ij}-\e_n\e_n^\top \in S_n^+ + \NN^n\Big\},
\end{equation} 
where $t_{ij}$ are decision variables and $E_{ij}$ is a matrix
that contains a value of one at positions $(i,j)$ and $(j,i)$ and zeros everywhere else.
This is a relaxation of
\eqref{eqStableCoposRelax} because it lifts the constraints $t_{ij}=t$
$\forall \{i,j\}\in E$.
However, we can show that any feasible solution of
\eqref{eqStableCoposRelaxRelax} with objective value $t$ can be transformed
into a feasible
solution of \eqref{eqStableCoposRelax} with the same objective value $t$. 

Take any edge $\{i,j\} \in E$ such that $t_{ij}\neq t$. If $t_{ij}<t$, one
can simply increase $t_{ij}$ to $t$ and remain feasible because the resulting matrix 
is the old matrix
plus a non-negative increase that belongs to $\NN^n$. If $t_{ij}>t$, let 
us decrease $t_{ij}$ to the minimum value
$\overline{t}$ that keeps the resulting matrix in $S^+_n+\NN^n$. 
Let us focus on the $2\times 2$ minor 
corresponding to $i$
and $j$. If $\overline{t}>t$, this minor is not SDP, and so, a decrease
from $\overline{t}$ down towards $t$ would only represent a decrease of the
$\NN^n$ component the matrix in $S_n^++\NN^n$. 
The only possible case that can forbid any decrease below $\overline{t}$
is
$\overline{t}=t$.



As a relaxation of \eqref{eqStableCoposRelax}, the new program
\eqref{eqStableCoposRelaxRelax} remains strictly feasible. This means that the
following dual does achieve the optimum value $\alpha^0(G)$.

\begin{align}\label{eqStableComplPosRelaxRelax1}
{\alpha^0(G)}&=\max
\left\{\e_n\e_n^\top \sprod X:~ \left(A^G+I_n\right)\sprod X=1,~X_{ij}=0~\forall \{i,j\}\in E, ~X\in S_n^+\cap \NN^n\right\}\\
&=\max\left\{\e_n\e_n^\top \sprod X:~ I_n \sprod X=1,~X_{ij}=0~\forall \{i,j\}\in E, ~X\succeq\zeros, X\geq \zeros \right\}.
\label{eqStableComplPosRelaxRelax2}
\end{align}
It is clear that above \eqref{eqStableComplPosRelaxRelax2} 
\textit{without the non-negativity} constraint $X\geq 0$
 is equivalent to the $(D\vartheta_G)$ formulation from
\eqref{dtheta1}-\eqref{dtheta4}.
As such, $\alpha^0(G)\leq
OPT(D\vartheta_G)=\vartheta(G)$. This ``natural'' strengthening  of the copositive formulation
of the maximum stable is bounded by the Lov\'asz theta number. 
Since we have 
$\alpha(G)\leq \alpha^0(G)$
in \eqref{eqStableCoposRelax}, 
we can 
write $\alpha(G)\leq \alpha^0(G)\leq \vartheta(G)$.

\paragraph{\label{secsoshierarchy}The sum of squares (\sos) hierarchy}
This is the only subsection where I will use a few results without an explicit 
proof. In all cases, I will leave a comment on the margin of the document.

This subsection is devoted to the \sos~approach due to Parilo, De Klerk and Pasechnik as cited
in the ``Copositive Optimization'' article of the \textit{Optima 89} newsletter 
(see Footnote \total{testfoot4}, p.~\pageref{testfootpage4}) or
in the article 
 ``Semidefinite Bounds for the Stability Number of a Graph via Sums of Squares of Polynomials''
by Neboj\originalv{s}a Gvozdenovi\'c and Monique Laurent.%
\footnote{Published in 
\textit{Mathematical Programming} in 2007, vol 110(1), pp 145--173,
available on-line as of 2017 at
\url{http://oai.cwi.nl/oai/asset/11672/11672D.pdf}.
\setcounter{testfoot5}{\value{footnote}}\label{testfootpage5}%
}

~\\
\noindent\textbf{The general \sos setting}\\
First, notice that $M\in C^n$ (recall Def.~\eqref{eqcopositive}) $\iff
p_M(\x)=M\sprod [x_1^2~x_2^2~\dots x_n^2]^\top [x_1^2~x_2^2~\dots
x_n^2]=\sum_{i,j=1}^n M_{ij}x_i^2x_j^2\geq 0~\forall \x\in\R^n$. This means 
that the polynomial $p_M(\x)$ 
takes only non-negative values over the reals.
 We can start using results
related to non-negative polynomials.
There are at least two conditions that
guarantee that a polynomial is always non-negative: (i) all coefficients are
positive and all powers of the variables are even, (ii) it can be
written as a sum of squares (\sos).
One of the earliest work on these aspects dates back to Hilbert who
studied the classes of non-negative polynomials that can always be written as an
\sos. He proved, for instance, that non-negative univariate polynomials
($n=1$) and non-negative polynomials of degree 2 can always be decomposed into a
\sos, see also Hilbert's 17$^{\text{th}}$ problem. However, even if a
non-negative polynomial does not respect any of above conditions (i) or (ii), 
we could multiply it by another non-negative polynomial and obtain a product
 that does respect (i) or (ii).

Applying this last argument, we notice that $p_M(\x)$ is non-negative if
and only if 
\begin{equation}\label{eqpmr}p_M^{(r)}=\left(\sum_{i=1}^n x_i^2\right)^r p_M(\x)\end{equation} is 
non-negative. P{\' o}lya proved in the 1970s the following theorem.
\begin{theorem}\label{thPolya} If $p_M(\x)$ is a homogeneous polynomial (with all
terms of the same degree) that is strictly
positive over $\R^n-\{0\}$, then $p_M^{(r)}$ has only non-negative coefficients%
\marginpar{\vspace{0.5em}\fbox{\begin{minipage}{0.8\linewidth}
    \scriptsize \textbf{Missing:} a proof for the result of P\'
olya (Theorem~\ref{thPolya}).\end{minipage}}}
for a sufficiently large $r$. 
\end{theorem}

Considering 
our $p_M$ above of the form
$p_M(\x)=M\sprod [x_1^2~x_2^2~\dots x_n^2]^\top [x_1^2~x_2^2~\dots
x_n^2]$, we notice that
$p_M^{(r)}$ accepts an obvious \sos decomposition for a sufficiently large $r$,
\ie, all terms of $p_M^{(r)}$ are squares multiplied by a positive value.
We now can define the
cone:
\begin{equation}\label{eqknr}
\K_n^{(r)}=\{M\in S_n:~p_M^{(r)}\text{ can be written as an \sos}\}.
\end{equation}
Using Theorem~\ref{thPolya}, we have $\inter(C^n)=\bigcup\limits_{r\geq 1}
\K_n^{(r)}$. We also have $\K_n^{(i)}\subseteq K_n^{(i+1)}$ because
switching from $i$ to $i+1$ is equivalent to
multiplying an \sos form with $\sum_{i=1}^n x_i^2$, which develops into a sum
of squared (and thus \sos) terms.
Replacing $C^n$ with $\K_n^{(r)}$
in \eqref{eqStableCopos}, we obtain the following strengthening (restriction) of
\eqref{eqStableCopos}:
\begin{equation}\label{eqalphasos}
\alpha^{(i)}(G) = \min \left\{t:~t\left(A^G+I_n\right)-\e_n\e_n^\top \in \K_n^{(i)}\right\}
\end{equation}
and we naturally have 
$$\alpha^{(0)}(G)\geq \alpha^{(1)}(G)\geq \alpha^{(2)}(G)\dots
\geq \alpha^{(r)}(G)\to \alpha(G),$$ using $\K_n^{(i)}\subseteq K_n^{(i+1)}$, with 
the convergence following from $\inter(C^n)=\lim\limits_{r\to \infty}
\K_n^{(r)}$. As such, we easily have 
$\alpha(G)=\left\lfloor \alpha^{(r)}(G)\right\rfloor$ for a sufficiently large $r$.
A sufficient value of $r$ is $r=\alpha(G)^2$, as proved in 
\marginpar{\vspace{0.5em}\fbox{\begin{minipage}{0.8\linewidth}
    \scriptsize \textbf{Missing:} the proof for the sufficiency of
$r=\alpha(G)^2$.\end{minipage}}}
Theorem 4.1.~of 
the article
``Approximation of the stability number of a graph via copositive programming''
by 
Etienne de Klerk and Dmitri Pasechnik.%
\footnote{Published in \textit{SIAM journal on optimization} in 2002,
vol 12(4), pp.~875--892, 
available on-line as of 2017 at 
\url{https://dr.ntu.edu.sg/handle/10220/6790}.}
It is conjectured that $\alpha^{(r)}(G)=\alpha(G)$ for $r\geq \alpha(G)-1$, see 
also
Conjecture 1 in the article indicated at Footnote \total{testfoot5} (p.~\pageref{testfootpage5}).

We will show that optimizing over some $\K^{(i)}_n$ might be easier than
optimizing over $C^n$, particularly for bounded values of $i$. While checking membership in $C^n$ is NP-hard
and probably not NP (assuming NP$\neq$co-NP, see last paragraph 
of Section~\ref{secqpplus}), checking membership in $K^{(i)}_n$ requires
solving an SDP with ${{n+i+1}\choose{i+2}}\times{{n+i+1}\choose{i+2}}$
variables, as we will see later in Remark~\ref{refOnlyDegree2}.
More generally, we will show one can determine an \sos~decomposition of any polynomial of degree $2d$
by solving and SDP with at most ${{n+d}\choose{d}}\times {{n+d}\choose{d}}$ variables, using results
from the Phd thesis ``Structured Semidefinite Programs
and Semialgebraic Geometry Methods
in Robustness and Optimization'' by Pablo Parrilo.%
\footnote{Defended in 2000 at 
California Institute of Technology, available on-line as of 2017 at
\url{http://www.mit.edu/~parrilo/pubs/files/thesis.pdf}.}
~\\

\noindent \textbf{Characterizing $K^{(0)}_n$}\\
Let us start with $K^{(0)}_n$ and we will show that $\alpha^{(0)}(G)$ is actually
equal to the
natural strengthening $\alpha^0(G)$ from \eqref{eqStableCoposRelax} which
is bounded by 
the Lov\'asz theta number $\vartheta(G)$ as described in 
Subsection~\ref{secStrengtheningNatural}.
However, let us first present a small example.
\begin{example}Consider $p^{(0)}_M(\x)=p_M(\x)=x_1^4+x_2^4+3(x_1x_2)^2$ that
accepts multiple \sos~decompositions, 
\eg,  $(x_1^2+x_2^2)^2+(x_1x_2)^2$
or $(x_1^2)^2+(x_2^2)^2+(x_1x_2)^2$.
Such 
decompositions can be determined by solving an SDP with a null objective function
(SDP feasibility problem).
\end{example}
\begin{proof} We will later show (see Remark \ref{refOnlyDegree2}) that for such polynomial, the only terms that can
appear in the squared expressions are $x_1^2$, $x_2^2$ and $x_1x_2$. Let us
define $\xx^\top=[x_1^2,~x_2^2,~x_1x_2]$ and write $p_M(\x)=\MM\sprod
\xx~\xx^\top$, with $\MM\in S_3$. 

\noindent {\it Step 1}~~
This last equality 
$p_M(\x)=\MM\sprod
\xx~\xx^\top$
does not allow one to uniquely identify $\MM$.
This comes from the fact that $(x_1x_2)^2$ can be written
both as a square of $x_1x_2$ (involving a non-zero $\MM_{33}$ in this
last equality) and as a product of 
$x_1^2$ and $x_2^2$ (involving $\MM_{12}$). We can write 
$p_M(\x)=p_M(\x)+\lambda_{12}\left(-x_1^2x_2^2+ (x_1x_2)^2\right)$.
This way, $\MM$ can take the form
$$\MM_{\lambda_{12}}=
\begin{bmatrix}
1 & \frac {3 -\lambda_{12}}2 & 0\\
\frac {3 -\lambda_{12}}2 & 1 & 0\\
0             &        0   & \lambda_{12}
\end{bmatrix}.$$

\noindent {\it Step 2}~~
Assume there is an \sos decomposition $p_M(\x)= \sum_{i=1}^k (\v_i\sprod
\xx)^2= \sum_{i=1}^k \v_i\v_i^\top \sprod \xx~\xx^\top=VV^\top\sprod \xx~\xx^\top$.
This can only happen if $\MM=VV^\top$ is SDP (see also
Prop.~\ref{propranktransprod}). The above decomposition $(x_1^2+x_2^2)^2+
(x_1x_2)^2$ corresponds to $\lambda_{12}=1$. We can find a decomposition by
solving an SDP program $\MM_{\lambda_{12}}\succeq \zeros$ with null objective function,
\ie, we only want to know if there exists feasible $\MM_{\lambda_{12}}\succeq\zeros$.

Both steps can generate infinitely-many \sos decompositions.
At Step 1, there is a feasible matrix $\MM_{\lambda_{12}}$ for each
$\lambda_{12}\in R$. At Step 2, any SDP matrix 
$\MM_{\lambda_{12}}$ accepts infinitely many factorizations 
$\MM=VV^\top$, as stated in
Corollary \ref{corSdpToVectors}.
\end{proof}

Let us return to general case of the polynomial $p_M(\x)$ of (homogeneous) degree
$2d=4$, \ie, all terms have degree $d=4$. The squared expressions contain no free terms because
$p_M(\x)$ has no free term. They can neither contain terms of degree $d-1=1$
like $x_i$
because there is no way other squared expression cancel some $x_i^2$. We define
$\xx$ to contain all monomials of degree $d=2$ and we attempt to write
$p_M(\x)=\MM\sprod \xx~\xx^\top$. As already stated in the proof of above
example, an \sos decomposition $p_M(\x)= \sum_{i=1}^k (\v_i\sprod
\xx)^2= \sum_{i=1}^k \v_i\v_i^\top \sprod \xx~\xx^\top=VV^\top\sprod \xx~\xx^\top$
exists if and only there is some SDP $\MM=VV^\top$ such that $p_M(\x)=\MM\sprod \xx~\xx^\top$.
All expressions of the form $p_M(\x)=\MM\sprod \xx~\xx^\top$ can be found by 
writing 
\begin{align*}p_M(\x)&=p_M(\x)+\sum_{i<j}\lambda_{ij}\left(-(x_i^2)(x_j^2)+(x_ix_j)^2\right)\\
                 &+\sum_{\substack{j<j'\\i\neq j,i\neq j'}}\mu_{ijj'}\left((x_i x_{j})(x_i x_{j'})-(x_i^2)(x_jx_{j'})\right)\\
                 &+\sum_{\substack{i'\neq j',j\neq i',j\neq j'\\i<j,i<i',i<j'}}\mu_{ii'jj'}\left((x_ix_j)(x_{i'}x_{j'})-(x_ix_{i'})(x_jx_{j'})\right),
\end{align*}
where ``$i<j,i<i',i<j'$'' in the last sum comes from the fact that 
(i) permuting $i$ and $j$ is equivalent to permuting $i'$ and $j'$,
(ii) permuting $i$ and $i'$ is equivalent to permuting $j$ and $j'$,
and
(iii) permuting $i$ and $j'$ is equivalent to permuting $i'$ and $j$.
A feasible $\MM$ can take the following form:
\begin{equation}\label{eqMM}
\MM_{\lambdalambda,\mumu}=
\begin{bmatrix}
M_{11}             &    M_{12}-\lambda_{12} & \ldots   &  M_{1n}-\lambda_{1n}            &    *        &      *      &\ldots     &    *             \\
M_{12}-\lambda_{12}&    M_{22}              & \ldots   &  M_{2n}-\lambda_{2n}            &    *        &      *      &\ldots     &    *             \\
\vdots             &\vdots                  & \ddots   & \vdots                          &\vdots       & \vdots      &\ddots     &\vdots            \\
M_{n1}-\lambda_{n1}&  M_{n2}-\lambda_{n2}   & \ldots   &  M_{nn}                         &    *        &      *      &\ldots     &    *             \\
             *     &                 *      & \ldots   &    *                            &2\lambda_{12}&      *      &\ldots     &    *             \\
             *     &                 *      & \ldots   &    *                            &     *       &2\lambda_{13}&\ldots     &    *             \\
\vdots             &\vdots                  & \ddots   & \vdots                          &\vdots       & \vdots      &\ddots     &\vdots            \\
             *     &                 *      & \ldots   &    *                            &     *       &      *      &\ldots     &2\lambda_{(n-1),n}\\
\end{bmatrix},
\end{equation}
where the asterisks stand for null terms or linear combinations of the $\mumu$ variables. However, if 
$\MM_{\lambdalambda,\mumu}\succeq\zeros$ for some $\mumu\neq\zeros$, then 
$\MM_{\lambdalambda,\zeros}\succeq \zeros$, because any minor of 
$\MM_{\lambdalambda,\zeros}$ (replace above asterisks with zeros) can be seen as a diagonal of blocks that also appear in the 
minors of $\MM_{\lambdalambda,\mumu}$ if $\mumu\neq \zeros$. All these blocks need to 
have a non-negative determinant in both matrices. As such, we can find an \sos~decomposition if 
and only if $\MM_{\lambdalambda,\zeros}\succeq\zeros$.

We now investigate the relation between the above $\MM_{\lambdalambda,\zeros}$
and the initial matrix $M$ that defines $p_M(\x)=M\sprod [x_1^2~x_2^2~\dots x_n^2]^\top
[x_1^2~x_2^2~\dots
x_n^2]$.
We obtain relatively easily that if $\MM_{\lambdalambda,\zeros}\succeq\zeros$, then $M$ can be written as an SDP matrix (the leading principal minor
of size $n\times n$ of $\MM_{\lambdalambda,\zeros}\succeq\zeros$) plus a non-negative matrix filled with all the $\lambda_{ij}$ values. In 
other words, $M\in S_n^+ + \NN^n$, and so, $K^{(0)}_n\subseteq S_n^++\NN^n$. 

We can also show $S_n^++\NN^n\subseteq K^{(0)}_n$. This
follows from the fact that if $S_n^++\NN^n\ni
M=\left[M\right]'+[\lambdalambda]'$ with $\left[M\right]'\succeq\zeros$ and
$[\lambdalambda]'\geq\zeros$, we can also
write $M$ as the sum of an SDP matrix and a non-negative matrix
$[\lambdalambda]$ such that $\texttt{diag}([\lambdalambda])=\zeros$.
For this,
let us (re-)write $M=\underbrace{\left[M\right]'+\texttt{diag}([\lambdalambda]')}_{\left[M\right]\succeq\zeros} + \underbrace{[\lambdalambda]'-
            \texttt{diag}([\lambdalambda]')}_{[\lambdalambda]\geq
\zeros}=\left[M\right]+[\lambdalambda]$ such that $\left[M\right]\succeq\zeros$,
$[\lambdalambda]\geq\zeros $ and $[\lambdalambda]$ has zeros on the diagonal. We
can thus construct a matrix $\MM_{\lambdalambda,\zeros}\succeq\zeros $ as above from these 
two matrices $\left[M\right]$ and $[\lambdalambda]$. This means $p_M(\x)$ is an
\sos, so that
$M\in K^{(0)}_n$.
We have just proved that
$S_n^++\NN^n\subseteq K^{(0)}_n$
so that
 $K^{(0)}_n=S_n^++\NN^n$.

We obtained that $\alpha^{(0)}(G)$ 
and $\alpha^0(G)$ have the same feasible area.
 This means that $\alpha^{(0)}(G)$ is 
the natural strengthening $\alpha^0(G)$ from \eqref{eqStableCoposRelax} 
bounded by 
the theta number $\vartheta(G)$ as described in
Subsection~\ref{secStrengtheningNatural}.

~\\
\noindent \textbf{Studying $K^{(i)}_n$ and finding the \sos~decomposition of $p_M^{(i)}$}\\
Recall the definitions 
\eqref{eqpmr} and \eqref{eqknr} of $p_M^{(i)}$ and respectively~$K^{(i)}_n$. We simply
obtain that $p_M^{(i)}$ is a polynomial of degree $2d=2(i+2)$.
We can find an \sos~decomposition of a polynomial of degree $2d$ using the following approach already exemplified above.
For most general polynomials of degree $2d$, we define a column vector $\xx$ containing all monomials of degree less than or equal to $d$. We have
that $p_M^{(i)}$ is an \sos~if and only if we can 
write it $p_M^{(i)}(\x)=\MM\sprod \xx~\xx^\top$ with $\MM\succeq\zeros$.
We can write $\sum_{j=1}^k (\v_j\sprod
\xx)^2= \sum_{j=1}^k \v_j\v_j^\top \sprod \xx~\xx^\top=VV^\top\sprod \xx~\xx^\top$
if and only if there exists $\MM=VV^\top\succeq \zeros$ such that $p_M^{(i)}=\MM\sprod \xx~\xx^\top$. The number of (independent) squared expressions is equal to the rank 
of $V$ which is equal to the rank of $\MM$ (see Prop.~\ref{propranktransprod}).

Checking membership in $K^{(i)}_n$ reduces to checking SDP membership for a matrix $\MM$ of the
form \eqref{eqMM} but of much larger size and with
 more variables like $\lambdalambda$ or $\mumu$. For instance, 
a term like
$x_i^2\cdot M_{ij}(x_i)^2(x_j)^2=M_{ij}x_i^4x_j$
of 
$p_M^{(i)}(\x)$ leads to a constraint
$M_{ij}=
\sum_{ab=x_i^4x_j^2} \MM_{a,b}$, where $\MM_{a,b}$ indicates the $\MM$ term
corresponding to monomials 
$a$ and $b$ in $\MM$. 
An optimization problem over $K^{(i)}_n$ can be transformed into 
an SDP problem by replacing an initial constraint of the form $M\in K^{(i)}_n$ into
some $\MM\succeq\zeros$, but where the size of $\MM$ can be much larger than $n\times n$, 
recall $\xx$ contains all monomials of degree less than or equal to $d$.

\paragraph{The size of the SDP programs used for computing \sos~decompositions}
In the most general setting, the length of $\xx$ is given by the number of
monomials of degree less than or equal to $d$. A classical combinatorics method
known as ``stars and bars'' can determine this number as ${n+d}\choose{d}$; it
relies on encoding any monomial using a 
string of stars and bars that arises quite out of the blue (\eg, $x_1^2x^3x_4$ is
$\star\star||\star|\star$).
However, I prefer to show it using a different argument that I personally find more natural. Consider the following set 
$\{x_1,~x_2,\dots x_n,~c_1,~c_2,\dots c_d\}$ of cardinal $n+d$. 
Choosing some
$x_{i_1},~x_{i_2},\dots x_{i_k}$ amounts to building an initial expression $e=x_{i_1}x_{i_2}\dots x_{i_k}$ (or $e=1$ 
if no $x_i$ term is chosen). Then, choosing some 
$c_j$ amounts to the following: (i) if the $j^\text{th}$ factor of $e$ exists, copy (insert) it at position $j+1$
or (ii) if the $j^\text{th}$ factor of $e$ does not exist ($e$ is shorter), do not modify $e$. 
Under this interpretation, any 
monomial of degree less than or equal to $d$ corresponds to choosing $d$ elements from $\{x_1,~x_2,\dots x_n,~c_1,~c_2,\dots c_d\}$.
\begin{itemize}
    \item $x_1$ corresponds to choosing $\{x_1,~c_2,~c_3,\dots c_d\}$. We start with expression
$e=x_1$ and the copying elements $c_2,~c_3,\dots c_d$ do not modify $e$ because $e$ has no factors at positions $2,~3,\dots d$.
    \item $x_1^2x_2$ corresponds to choosing $\{x_1,~x_2,~c_1,~c_4,~c_5,\dots c_d\}$. Indeed, we first
have the expression $e=x_1x_2$; then, $c_1$ will insert a copy of $x_1$ at position 2 to obtain
$e=x_1x_1x_2$. The $c_4$ element does not change $e$, because $e$ has no factor at position $4$.
Same applies to $c_5,~c_6\dots c_d$, and so, the final $e$ is $e=x_1x_1x_2$.
    \item $x_5x_7^3x_9$ corresponds to choosing $\{x_5,~x_7,~x_9,~c_2,~c_3,~c_6,~c_7,\dots c_d\}$.
We start with $e=x_5x_7x_9$ and $c_2$ inserts at position $3$ a copy of the second element $x_7$, leading
to $e=x_5x_7x_7x_9$. Then, $c_3$ duplicates the third element $x_7$, generating 
$e=x_5x_7x_7x_7x_9$. Elements $c_6,~c_7,\dots c_d$ will perform no modification on $e$ because $e$ has
no factor at positions $6,~7,\dots d$.
    \item $x_7^{d}$ corresponds to choosing $\{x_7,~c_1,~c_2,\dots c_{d-1}\}$. Indeed, we start with 
$e=x_7$ and $c_1$ duplicates $x_7$ leading to $x_7x_7$. Then $c_2$ duplicates the second term, leading
to $x_7x_7x_7$. Applying this for all $c_1,~c_2,\dots c_{d-1}$, we obtain that $x_7$ is duplicated 
$d-1$ times, and so, the final expression is $\underbrace{x_7x_7\dots x_7}_{d\text{ times }}$.
    \item $x_7^2x_9^{d-2}$ corresponds to choosing $\{x_7,~x_9,~c_1,~c_3,~c_4,~\dots c_{d-1}\}$. We
start with $e=x_7x_9$ and $c_1$ duplicates $x_7$ leading to $e=x_7x_7x_9$. Since $c_2$ is not chosen, $x_7$
is not duplicated again. On the other hand, $x_9$ is duplicated in cascade $d-3$ times and we obtain
$e=x_7x_7\underbrace{x_9x_9x_9\dots x_9}_{d-2\text{ times }}$.
    \item $1$ corresponds to choosing $c_1,~c_2,\dots c_d$.
\end{itemize}
\begin{remark} \label{refOnlyDegree2}The polynomial $p_M^{(i)}$ of degree $2d=2(i+2)$ also has the property that it contains
only monomials of degree $2d$. In this case, we do \underline{not} need to construct $\xx$ using \underline{all} monomials of 
degree less than or equal to $d$, but we only need monomials of degree exactly $d$.
The number of such monomials is 
$${{n+d-1}\choose{d}}={{n+i+1}\choose{i+2}}.$$
\end{remark}
\begin{proof} First, since $p_M^{(i)}$ has no free member, we can \textit{not} use a squared expression with a free term, because
\textit{no} other squared expression can cancel a free term. 
We can repeat this argument by induction. Let us take any $d'<d$;
the induction basis is that there is {\it no} term of degree strictly smaller than
$d'$ in any squared expression.
Now notice we can not have a squared expression of the form $(e+x)^2$ with 
\texttt{degree(x)}=$d'$, because
we can not cancel the monomial $x^2$. Indeed, such monomial could only be canceled by some 
$\left(e+x'-x''\right)^2$ with $x^2=x'x''$ and the induction hypothesis states
there are {\it no} terms 
such as $x'$ or $x''$ of degree strictly smaller than $d'$
in any squared expression.
We proved by
induction that there are {\it no} terms of degree $d'<d$ in our squared expressions, \ie, 
$\xx$ is constructed using only monomials of degree $d$. 

Investigating in detail the combinatorial argument above this proposition, we notice that
expressions of degree exactly $d$ are only constructed when $c_d$ is not chosen
(hint: let $c_\ell$ be the first chosen among the $c_i$'s and notice there 
are at least $\ell$ elements chosen among the $x_i$'s because none of
$c_1,~c_2,\dots c_{\ell-1}, c_d$ is chosen). The number of monomials of degree exactly $d$ is thus
${n+d-1}\choose{d}$.
\end{proof}

A final remark sometimes partially overlooked is that we took somehow for
granted that we can \textit{not} have an \sos~decomposition with 
some terms of degree $\ddd>d$. However, this follows from the following argument. Take the maximum degree of a term in a squared expression
and assume this value is $\ddd>d$. 
Assume there is an \sos~decomposition $\sum\left(e_i+f_i\right)^2$, where the polynomials $f_i'$s contain all terms
of degree $\ddd$. The monomials of degree $2\ddd$ are those resulting from (summing
up) the terms  $f_i^2$. 
We need to have $\sum f_i^2=0$, which means \textit{any} $f_i$ yields $f^2_i=0$, and so, $f_i=0$. 
 
\subsection{Further characterization of the completely positive and the
copositive cones}
We first recall below the following basic hierarchy from \eqref{eqHierarchySDP} of
Section~\ref{seccopos}. 
\begin{equation}\label{eqHierarchySDPBis}
C^{n*}\subset S_n^+\cap \NN^n \subset S_n^+ \subset S_n^+ + \NN^n \subset C^n,
\end{equation}
We showed in
Footnote \total{testfoot8} (p.~\pageref{testfootpage8}) that 
$\left(S_n^+ + \NN^n\right)^* = S_n^+\cap \NN^n$.
We will now prove the converse:
$\left(S_n^+\cap \NN^n \right)^*=S_n^+ + \NN^n$.
I could not easily find a direct proof using the properties of these matrices.
However, 
it is direct consequence of the following general proposition
combined with the fact that $S_n^+ + \NN^n$ is closed (which follows 
from Prop.~\ref{propSDPcones} and the fact that $\NN^n$ is closed).
More exactly, 
 property below enables us
to derive the following implication:
${\left(S_n^+ + \NN^n\right)^*}^* = \left(S_n^+\cap \NN^n\right)^*$
$\implies$
$S_n^+ + \NN^n = \left(S_n^+\cap \NN^n\right)^*$.

\begin{proposition}\label{propDualDualClosureCone} Given any convex cone $C$, we have $\closure(C)={C^*}^*$.
\end{proposition}
\begin{proof}
We first show $\closure(C)\subseteq {C^*}^*$. Assume the opposite: there
is a sequence $(\c_i)$ with all $\c_i\in C$ and $\lim\limits_{i \to \infty}
\c_i=\c\notin {C^*}^*$. This means there is ${\c^*}\in C^*$ such that $\c\sprod {\c^*}=z<0$.
But since $\lim\limits_{i \to \infty} \c_i=\c$, there exists some $\c_i$
sufficiently close to $\c$ such that $\c_i\sprod {\c^*}\in[z-\varepsilon,z+\varepsilon]$, for
any $\varepsilon>0$. For a sufficiently small $\varepsilon$, this means $\c_i\sprod {\c^*}<0$
which is a contradiction of ${\c^*}\in C^*$ and $\c_i\in C$.
As such, $\lim\limits_{i \to \infty} \c_i=\c\notin {C^*}^*$ is false
$\implies\closure(C)\subseteq {C^*}^*$.

We now show $\b\notin \closure(C)\implies \b\notin {C^*}^*$. We apply the simple
separation Theorem~\ref{thSimpleSep} on
$\closure(C)$ and $\{\b\}$ to obtain there is a strictly separating inequality
$\c\sprod {\c^*}>\b\sprod {\c^*}~\forall \c\in \closure(C)$.
Since $\zeros \in C$, we have $\zeros\sprod {\c^*}>\b\sprod {\c^*} \implies \b\sprod {\c^*}<0$.
We now show $\c^*\in C^*$, \ie, $\c\sprod\c^*\geq 0~\forall \c\in C$.
Assume $\exists \cc\in C$ such that $\cc\sprod {\c^*} <0$. Using the cone property,
$t\cc\in C$ for any arbitrarily large $t$. This means $t\cc\sprod {\c^*}$ can be
arbitrarily low, easily less than $\b \sprod {\c^*}$, which is a contradiction. The
assumption $\exists \cc\in C$ such that $\cc\sprod {\c^*}<0$ was false. All $\c\in C$
satisfy $\c\sprod {\c^*} \geq 0$.
We produced ${\c^*} \in C^*$ such that $\b\sprod {\c^*}<0$ for any $\b$
outside $\closure(C)$.
\end{proof}

Using the results from Section \ref{secsoshierarchy}, 
the above hierarchy \eqref{eqHierarchySDPBis} is refined to the 
hierarchy below. 
Recalling P\' olya's Theorem~\ref{thPolya}, we have $\inter(C^n)=
\K_n^{(\infty)}=
\lim\limits_{i\to\infty}
\K_n^{(i)}$.

\begin{equation}\label{eqHierarchySDPTri}
C^{n*}\subset S_n^+\cap \NN^n \subset S_n^+ \subset S_n^+ +
\NN^n=\K^{(0)}_n\subset \K^{(1)}_n\subset \K^{(2)}_n\subset \dots \K^{(\infty)}_n \subset C^n,
\end{equation}

We can also build an outer approximation hierarchy for $C^n$. Consider
$\N_r^n=\left\{\z\in \N^n:~\sum_{j=1}^n z_i\leq r\right\}$ and define 
$\PP_n^{(r)}=\left\{X\in S_n:~X\sprod \z\z^\top\geq 0~\forall
\z\in\N_r^n\right\}$.%
\footnote{Some papers define $\N_r^n$ using  $\sum_{j=1}^n z_i= r$. However, as
far as I can see, this does no longer ensures 
the property
$\PP_n^{(r+1)}\subset \PP_n^{(r)}$ that we need afterwords. We would have 
$X=
\left[
\begin{smallmatrix}
    7   &    -8  &    0 \\
    -8  &    7   &    0 \\
    0   &   0    &   2
\end{smallmatrix}
\right]\in \PP_3^{(3)}$ but
$X\notin \PP_3^{(2)}$.}
Since $\PP_n^{(r)}$ contains less constraints than $C^n$,
we simply obtain $C^n\subset \PP_n^{(r)}~\forall r\in \N$. Notice also $\N_r^n\subsetneq
\N_{r+1}^n$, and so $C^n\subset\PP_n^{(r+1)}\subset \PP_n^{(r)}\forall r\in \N$.
Writing $\PP_n^{(\infty)}=\lim\limits_{r\to\infty}\PP_n^{(r)}$, we will prove
that
$$C^n=\PP_n^{(\infty)}\subset \dots\subset \PP_n^{(3)}\subset \PP_n^{(2)}\subset \PP_n^{(1)}$$

We still have to show the first equality; the
other inclusions were proved in above paragraph.
We know $C^n\subseteq \PP_n^{(\infty)}$, because $C^n$ contains more constraints
than $\PP_n^{(\infty)}$. We only need to show that $X\in \PP_n^{(\infty)}\implies
X\in C^n$. Assume the opposite: 
$\exists X\in \PP_n^{(\infty)}$ such that
$X\notin C^n$. This means there exists $\y\in \R^n_+$ such that 
$X\sprod \y\y^\top<0$. 
Now we can construct a rational $\y_q$ such that
$X\sprod \y_q\y_q^\top<0$. Such a $\y_q$ can simply
be constructed by taking sufficiently many decimals of
of $\y$, enough to make 
$X\sprod \y_q\y_q^\top$
as close as necessary to
$X\sprod \y\y^\top$.
By multiplying rational $\z_q$ with the least common denominator, we
obtain an \textit{integer} $\z_i$ such that $X\sprod \z_i\z_i^\top<0$, which
contradicts $X\in \PP_n^{(\infty)}$. The assumption $X\notin C^n$ was false,
and so, $C^n=\PP_n^{(\infty)}$.

The final remark is that all inclusions in the basic hierarchy
\eqref{eqHierarchySDPBis} are strict. 
Indeed, $S_n^+\cap \NN^n \subsetneq S_n^+$ simply because 
$\left[\begin{smallmatrix} 1 &-1\\-1&1\end{smallmatrix}\right]
\in S_2^+-S_2^+\cap \NN^2$. It is also very easy to check
$\left[\begin{smallmatrix} 0 &1\\1&0\end{smallmatrix}\right]$
belongs to $S_2^+ + \NN^2$, but not to $S_2^+$.
We can derive
$C^{n*}\subsetneq S_n^+\cap \NN^n$
from the fact that 
$C^{n*}$ and  $S_n^+\cap \NN^n$ have disjoint dual cones
$C^n$ and resp.~$S_n^+ + \NN^n$. Indeed, we can show 
$S_n^+ + \NN^n \subsetneq C^n$ using
the so-called Horn-matrix.%
\footnote{For an exact reference, check the comments on equation $(7)$ 
from the ``Copositive Optimization'' article of the \textit{Optima 89} newsletter,
see the \texttt{http} link from Footnote \total{testfoot4}, p.~\pageref{testfootpage4}.}
$$
H=
\left[
\begin{array}{rrrrr}
1 & -1 & 1  & 1  & -1 \\
-1& 1  & -1 & 1  & 1  \\
1 &-1  &1   & -1 & 1  \\
1 & 1  & -1 & 1  & -1 \\
-1& 1  & 1  & -1 & 1  
\end{array}
\right]
$$

\subsubsection{The Horn matrix is copositive}

We will show this without decomposing $H\sprod \x\x^\top$ and calculating.
We start from some $\x\geq 0$ and we investigate the evolution of function $H\sprod
\x\x^\top$ as we change variables $\x$ in a particular way described next.
The key is to notice that if we can decrease $x_i$ and $x_j$ by any
$\varepsilon\in(0,min(x_i,x_j)]$, 
the value $H\sprod \x\x^\top$ decreases, for any $i,j\in[1..5]$ with $i\neq j$.
Indeed, take any two lines $i$ and $j$ and notice they contain four times 
$\left[\begin{smallmatrix} -1 \\ 1 \end{smallmatrix}\right]$
or 
$\left[\begin{smallmatrix} 1 \\ -1 \end{smallmatrix}\right]$ and one times
$\left[\begin{smallmatrix} 1 \\ 1 \end{smallmatrix}\right]$. The contribution
to $H\sprod
\x\x^\top$ of a sub-matrix 
$\left[\begin{smallmatrix} -1 \\ 1 \end{smallmatrix}\right]$
or
$\left[\begin{smallmatrix} 1 \\ -1 \end{smallmatrix}\right]$ is a multiple of
$x_i-x_j$. Since we decrease \textit{both variables} $x_i$ and $x_j$ by $\varepsilon$,
the value $x_i-x_j$ does not change. On the other hand, the contribution
of a sub-matrix
$\left[\begin{smallmatrix} 1 \\ 1 \end{smallmatrix}\right]$ is 
$x_\ell(x_i+x_j)$ for some $\ell\notin\{i,j\}$ that has to verify $x_\ell\geq
0$. Decreasing both $x_i$ and $x_j$ can only decrease $x_\ell(x_i+x_j)$.
An analogous (transposed) phenomenon happens on the columns
$i$ and $j$.
This way, decreasing $x_i$ and $x_j$ by the same $\varepsilon>0$ can only decrease
$H\sprod \x\x^\top$. 

We can sequentially decrease pairs of non-zero variables $x_i$ and $x_j$ until we end up
with a final $\x$ that contains only one non-zero variable and whose value
$H\sprod \x{\x}^\top$ is less than or equal to the initial one. This finishes the
proof, because this final value is a multiple of a diagonal element of $H$ that is
surely non-negative.

\subsubsection{The Horn matrix does not belong to $S_5^++\NN^5$}

The easiest way to see $H$ is not SDP is 
by noticing it does not respect the structure from 
Prop.~\ref{propRankRectangular}.
Indeed, the first four elements of
the last row can not be written as a linear combination of the rows of
$[H]_{[1..4]}$, because of columns 2 and 3, where $[H]_I$ is the principal minor
of $A$ that selects rows $I$ and columns $I$. 
We still need to show $H-N$ is not SDP for any $N\geq \zeros$. 
Assume the contrary: there is $N\geq \zeros$ such that $H-N\succeq\zeros$
and we sequentially conclude:
\begin{enumerate}
\item[(i)] \begin{sloppypar}The operation $H-N$ can not 
decrease any diagonal element $H_{ii}$, because that would lead to 
$\det \left([H-N]_{\{i,j\}}\right)<0$ for $j=i+1\texttt{ mod 5}$. Indeed, the product of
the elements on the main diagonal of such
$[H-N]_{\{i,j\}}$ would become less than 1 if $N_{ii}>0$ or $N_{jj}>0$, which is strictly less that the product
of the elements of the second diagonal (remark $(H-N)_{ij}\leq H_{ij}=-1$).
We thus need to have $\texttt{diag}(N)=\zeros$.
\end{sloppypar}
\item[(ii)]\begin{sloppypar} The operation
$H-N$ can not decrease any $H_{ij}$ 
for $j=i+1\texttt{ mod 5}$, 
because this would lead to 
$\det\left([H-N]_{\{i,j\}}\right)<0$. Indeed, this would increase the product of the elements of
the second diagonal of $[H-N]_{\{i,j\}}$ while the main diagonal does not change by virtue of
$\texttt{diag}(N)=\zeros$ from above point (i). We obtain that all negative elements
of $H$ can not be decreased by the operation $H-N$. \end{sloppypar}
\item[(iii)]
Notice we have 
$H^3=[H]_{[1..3]}=
[H]_{[2..4]}=
[H]_{[3..5]}$. Since we need such matrices $H^3$ to remain SDP after subtracting $N$,
the extremal elements of the second diagonal (\ie, $H^3_{13}$ and $H^3_{31}$)
need to remain unchanged. This follows from applying
Prop.~\ref{propRankRectangular} and the fact that the other elements of any
$H^3$ are not changed by the $H-N$ operation by virtue of (i) and (ii) above. We
obtain $N_{13}=N_{24}=N_{25}=0$ and the same applies to the transposed elements.
\item[(iv)] The only elements of $N$ that can still be non-zero are $N_{14}$, $N_{25}$ and
their transposed, so that $|H-N|_{\{1,3,5\}}=H_{\{1,3,5\}}$.
Since the determinant of this $3\times 3$ matrix is $-4$, $H-N$ can not be SDP.

\end{enumerate}
We obtained there is no $N\geq \zeros$ that can lead to $H-N\succeq\zeros$, and thus,
$H\notin S_5^+ + \NN^5$.

\subsection{A final short property: the Schur complement does not apply in
$C^{n*}$}
I finish with a final remark not directly related to other property from this
chapter.
In certain proofs, I was tempted to use a ``Schur complement
property'' with completely positive matrices but this is not possible.
\begin{proposition}
The Schur complement property from Prop.~\ref{propSchurGen} does not hold for
completely positive matrices. In particular, 
if
$\left[\begin{smallmatrix}
1 & \y^\top \\
\y& Y
\end{smallmatrix}\right]
\in 
C^{(n+1)*}$, we do not necessarily have $Y-\y\y^\top \in C^{n*}$.
\end{proposition}
\begin{proof} It is enough to give an example. 
We will exhibit a $2\times 2$ matrix $Y$ with some zeros and
a vector $\y$ with no zero. It is enough to 
take the sum
$\left[1~0~a\right]^\top \left[1~0~a\right]
+\left[1~a~0\right]^\top \left[1~a~0\right]$ and divide it
by $2$ to obtain the completely positive matrix
below.
$$
\begin{bmatrix}
      1   & a/2 & a/2 \\
      a/2 & a^2/2 & 0 \\
      a/2 & 0 & a^2/2
\end{bmatrix}$$
The non-diagonal terms of $Y$ are zero. If we try to apply the Schur complement, we obtain negative entries
on the non-diagonal terms of 
$Y-\y\y^\top $, so that this matrix can not be completely
positive.
\end{proof}

\section{\label{sec4}SDP relaxations and convexifications of quadratic programs}
Let us give a short warning: an important difficulty in this section
 comes from parsing a number of 
 (Lagrangian) notations. It may be useful to print this section (twice) to more easily
jump from one formula or notation to another.
\subsection{\label{sec41}The most general quadratic program: SDP relaxation and total Lagrangian}
We presented SDP reformulations and relaxations of \textit{convex} quadratic programs
in Section \ref{sec34}.
We now introduce  a quadratic program in its most general form, not necessarily
convex. Notice this form could include 
$0-1$ quadratic programs by adding $x^2_i=x_i~\forall i\in[1..n]$ to the set of quadratic
constraints~\eqref{qp2}.

\begin{subequations}
\begin{align}[left ={(QP)  \empheqlbrace}]
\min~~&     Q\sprod \x\x^\top +\c^\top \x                             \label{qp1}\\
s.t~~ &B_i   \sprod \x\x^\top + \d_i^\top       \x = (\leq) e_i~~~~\forall i \in [1..p]   \label{qp2}\\
      &\x\in \R^n                                                     \label{qp3}
\end{align}
\end{subequations}
The basic SDP relaxation of above $(QP)$ program is the following:
\begin{subequations}
\begin{align}[left ={SDP(QP)  \empheqlbrace}]
\min~~&     Q\sprod X   +\c^\top \x                                         \label{sdpqp1}\\
s.t~~ &B_i   \sprod X  +\d_i^\top \x= (\leq) e_i~~~~\forall i \in [1..p]   \label{sdpqp2}\\
      &\begin{bmatrix} 1&\x^\top\\ \x & X\end{bmatrix}\succeq \zeros \label{sdpqp3}
\end{align}
\end{subequations}

Let us start with a rather negative result: the SDP relaxation of
$\min\{-x^2:~x=0\}=0$ is $\min\{-X_{11}:~x=0,~X\succeq xx^\top =0\}=-\infty$.
This means that the above basic SDP relaxation can be arbitrarily bad
in the worst case. However, we will overcome this with different
strengthening methods in the next subsections. 


Let us introduce the Total Lagrangian (TL) in variables $\x$ of $(QP)$ from
\eqref{qp1}-\eqref{qp3}. We can write:
\begin{subequations}
\begin{align}
\left(TL^\sx(QP)\right)=&\max_{\mumu}\LL_{TL^\sx(QP)}(\mumu)    \label{eqTotLagrObj} \\
             &\LL_{TL^\sx(QP)}(\mumu)=
                \min_{\x\in\R^n}\left(Q+\sum_{i=1}^p \mu_i B_i\right) \sprod \x\x^\top 
                + \left(\c^\top+\sum_{i=1}^p \mu_i\d^\top_i\right)\x - \mumu^\top \e
                \label{eqTotLagrBody}
\end{align}
\end{subequations}
Notice we have $\mu_i\geq 0$ if we have an inequality constraint $B_i\sprod \x\x^\top+d_i^\top x \leq e_i$ for some $i\in[1..p]$.

\begin{theorem} \label{thTotLagrDualSdp} The optimum total Lagrangian
$\left(TL^\sx(QP)\right)$ of a quadratic program $(QP)$ is equal to the optimum value 
of the \underline{dual} of the basic SDP relaxation $(SDP(QP))$ of $(QP)$ from \eqref{sdpqp1}-\eqref{sdpqp3}.
\end{theorem}
\noindent \textit{Proof.}
To compute $\LL_{TL^\sx(QP)}(\mumu)$ for a fixed $\mumu$, we can use the results for
unconstrained quadratic programs from Section \ref{secqpunconstrained}, more exactly
Proposition \ref{propCqpunconstrained} that can be applied on $\LL_{TL^\sx(QP)}(\mumu)$ as follows:

\begin{align*}
\LL_{TL^\sx(QP)}(\mumu)
=\max_{t}~~& t  - \mumu\e                                                                \\
s.t.~ &
\begin{bmatrix}
    -t            &   \frac 12 \left(\c^\top+\sum_{i=1}^p \mu_i\d^\top_i\right)    \\
\frac 12 \left(\c+\sum_{i=1}^p \mu_i\d_i\right)       & Q+\sum_{i=1}^p \mu_i B_i
\end{bmatrix}
\succeq \zeros.                                                  
\end{align*}
Replacing this in \eqref{eqTotLagrObj}, we obtain:
\begin{subequations}
\begin{align}
\left(TL^\sx(QP)\right)
=\max_{t,\mumu}~~& t  - \sum_{i=1}^p \mu_i\e_i                  \label{eqTotLagrObj2}                         \\
s.t.~ &
\begin{bmatrix}
    -t            &   \frac 12 \left(\c^\top+\sum_{i=1}^p \mu_i\d^\top_i\right)    \\
\frac 12 \left(\c+\sum_{i=1}^p \mu_i\d_i\right)       & Q+\sum_{i=1}^p \mu_i B_i
\end{bmatrix}
\succeq \zeros,                          \label{eqTotLagrConstr2}
\end{align}
\end{subequations}
\begin{sloppypar}
\noindent
which is exactly the dual of $SDP(QP)$ from \eqref{sdpqp1}-\eqref{sdpqp3}. Recall 
that an inequality constraint $B_i\sprod \x\x^\top+d_i^\top \x \leq e_i$ leads
to non-negative $\mu_i\geq 0$ in the total Lagrangian. When dualizing such $\mu_i$
of above \eqref{eqTotLagrObj2}-\eqref{eqTotLagrConstr2},
we obtain a inequality constraint 
in $SDP(QP)$, see also Prop.~\ref{propMainDualityWithNonNegatives}.
However, we
can certainly write 
$\left(TL^\sx(QP)\right)=OPT(DUAL(SDP(QP)))$. One should bear in mind that there 
might be a duality gap between $(DUAL(SDP(QP)))$ and $(SDP(QP))$. \qed
\end{sloppypar}

The fact that $\left(TL^\sx(QP)\right)\leq OPT(SDP(QP))$ can also be shown
by noticing 
\begin{equation}\label{eqTotLagrIsSdp}\left(TL^\sx(QP)\right)=\left(TL^\sX(QP)\right),\end{equation}
where
$\left(TL^\sX(QP)\right)$ is the total Lagrangian of $(SDP(QP))$, \ie, 
it is defined using the same formulas 
\eqref{eqTotLagrObj}-\eqref{eqTotLagrBody} but we replace the term
$\x\x^\top$ from \eqref{eqTotLagrBody} with $X$ and add constraint $X\succeq \x\x^\top$
to \eqref{eqTotLagrBody}.
To show $\LL_{TL^\sX(QP)}(\mumu)=\LL_{TL^\sx(QP)}(\mumu)$, 
it is enough to show there is an optimal solution 
$(X,\x)$ of $\LL_{TL^\sX(QP)}$ such that $X=\x\x^\top$.
We consider two cases:
(a)
$Q+\sum_{i=1}^p \mu_i B_i\nsucceq \zeros$
and (b)
$Q+\sum_{i=1}^p \mu_i B_i\succeq \zeros$.
In the first case (a), we take an $\x\in\R^n$ such that
$\left(Q+\sum_{i=1}^p \mu_i B_i\right)\sprod \x\x^\top<0$. By replacing
$\x$ with $t\x$ we obtain a quadratic concave function in
$t$ that converges to $-\infty$; by taking $X=\x\x^\top$
$\LL_{TL^\sX(QP)}(\mumu)$ converges to $-\infty$ as well.
In the case (b),
there is no use to consider any $X=Y+\x\x^\top$
with some non-zero $Y\succeq \zeros$:
the existence of such $Y$ can only increase the value of 
$\LL_{TL^\sX(QP)}(\mumu)$
by $\left(Q+\sum_{i=1}^p \mu_i B_i\right)\sprod Y\geq 0$.
This confirms that $\left(TL^\sx(QP)\right)=\left(TL^\sX(QP)\right)\leq OPT(SDP(QP))$.

The dual of $(SDP(QP))$ is actually constructed by relaxing all constraints
from \eqref{sdpqp1}-\eqref{sdpqp3}, \ie, the proof
of the general dualization property (Proposition~\ref{propTwiceDualize})
starts by computing 
in \eqref{eq2111}
the total SDP Lagrangian
using coefficients 
$\x'$. These coefficients $\x'$ may include $\mumu$ 
and also the dual variable ($t$) of the constraint
the puts a value of 1 in the upper right corner of
the matrix from \eqref{sdpqp3}.
This explains the very close relationship between 
$DUAL(SDP(QP))$ and 
$\left(TL^\sX(QP)\right)$.
The above argument could actually extend to a second proof of
Theorem~\ref{thTotLagrDualSdp}. Many convexification
ideas presented next have certain roots in the fact that 
$\left(TL^\sX(QP)\right)$
is very related to $DUAL(SDP(QP))$.

\subsection{\label{sec42}Partial and total Lagrangians for quadratic programs with linear 
equality constraints}

We consider a version of the general quadratic program $(QP)$ from \eqref{qp1}-\eqref{qp3}
in which we explicitly separate linear equality constraints $A\x=\b$. 
Furthermore, we consider adding $\pp$ redundant constraints
$\B_j\sprod \x\x^\top + \dd_j^\top  \x = \ees_j~\forall j\in[1..\pp]$ that
are always satisfied for all $\x\in\R^n$ such that $A\x=\b$. An example of such redundant constraint
is $A^\top_j A_j \sprod \x\x^\top = b^2_j$, where $A_j$ is row $j$ of $A$.
By summing up over all $j\in[1..\pp]$ we can also obtain
$A^\top A \sprod \x\x^\top = \sum\limits_{j=1}^{\pp} b_j^2$. Many other examples
can be found, \eg, we will later see the redundant constraint sets from
 Example \ref{exRedund1} and Example~\ref{exRedund2}.
However, for now, it is enough to say that such constraints have no impact on the initial quadratic
program, but they can be useful to convexify it (\ie, so that the factor of the quadratic term
becomes SDP)  or to strengthen its SDP relaxation.

We formulate
\begin{subequations}
\begin{alignat}{4}[left ={(QP_\segal)  \empheqlbrace}]
\min~~&     Q\sprod \x\x^\top +\c^\top \x                              &&                        \label{qpp1}\\
s.t~~ &A\x = \b                                                        &&                        \label{qpp2}\\
      &\B_j   \sprod \x\x^\top + \dd_j^\top       \x =        \ees_j   &&\forall j \in [1..\pp]  \label{qpp3}\\
      &B_i   \sprod \x\x^\top + \d_i^\top       \x = (\leq) e_i        &~~&\forall i \in [1..p]  \label{qpp4}\\
      &\x\in \R^n                                                                                \label{qpp5}
\end{alignat}
\end{subequations}
The formulation of the SDP relaxation leads to a program $SDP(QP_\segal)$ defined by writting
the SDP form \eqref{sdpqp1}-\eqref{sdpqp3} corresponding to above $(QP_\segal)$. More exactly, we obtain the following SDP relaxation 
of $\left(QP_\segal\right)$.
\begin{subequations}
\begin{alignat}{4}[left ={SDP\left(QP_\segal\right)  \empheqlbrace}]
\min~~&     Q\sprod    X      +\c^\top \x                              &&                        \label{qppsdp1}\\
s.t~~ &A\x = \b                                                        &&                        \label{qppsdp2}\\
      &\B_j   \sprod   X       + \dd_j^\top       \x =        \ees_j   &&\forall j \in [1..\pp]  \label{qppsdp3}\\
      &B_i   \sprod    X      + \d_i^\top       \x = (\leq) e_i        &~~&\forall i \in [1..p]  \label{qppsdp4}\\
      &\begin{bmatrix} 1&\x^\top\\ \x & X\end{bmatrix}\succeq \zeros   &&                        \label{qppsdp5}
\end{alignat}
\end{subequations}
 
\subsubsection{The partial Lagrangians of $(QP_{=})$ and $SDP(QP_=)$}
We introduce below the Partial Lagrangian (PL) in variables $\x$ of $(QP_\segal)$.

\begin{subequations}
\begin{align}
&\left(PL_\sx(QP_\segal)\right)=\max_{\mumu}\LL_{PL^\sx(QP_\segal)}\left(\mumu\right)   \label{eqPartLagrObj} \\
&             \LL_{PL^\sx(QP_\segal)}(\mumu)=
                 \left\{
                    \begin{array}{ll}
                        \min &\left(Q+\sum_{i=1}^p \mu_i B_i\right) \sprod \x\x^\top 
                              + \left(\c^\top+\sum_{i=1}^p \mu_i\d^\top_i\right)\x - \mumu^\top \e \\
                        s.t.~&A\x=b\\
                             &\x\in\R^n
                    \end{array}
                 \right. \label{eqPartLagrBody}
\end{align}
\end{subequations}

This partial Lagrangian does \text{not} dualize the linear equality constraints $A\x=\b$ from \eqref{qpp2}
and ignores the redundant constraints \eqref{qpp3}. In other words, we obtain the partial Lagrangian
of the initial 
program 
without redundant constraints. Notice that adding redundant constraints
to above $\LL_{PL^\sx(QP_\segal)}(\mumu)$
would not change its value, because $\LL_{PL^\sx(QP_\segal)}(\mumu)$ already imposes $A\x=b$, and so,
the redundant constraints are satisfied.

The above partial Lagrangian is equal to the following augmented partial Lagrangian that also 
dualizes the redundant constraints
\eqref{qpp3}. 

\begin{subequations}
\begin{align}
&\left(PL_\sx(QP_\segal)\right)=\max_{\mumu,\mumumu}\LL_{PL^\sx(QP_\segal)}\left(\mumu,\mumumu\right)   \label{eqPartLagrObj2} \\
&             \LL_{PL^\sx(QP_\segal)}(\mumu,\mumumu)=
                 \left\{
                    \begin{array}{ll}
                        \min &\left(Q+\sum_{i=1}^p \mu_i B_i+\sum_{j=1}^{\pp} \mumumus_j \B_j \right) \sprod \x\x^\top+ \\
                             &\left(\c^\top+\sum_{i=1}^p \mu_i\d^\top_i+\sum_{j=1}^{\pp} \mumumus_j \dd^\top_j \right)\x \\
                             &- \mumu^\top \e -\mumumu^\top \ee \\
                        s.t.~&A\x=b\\
                             &\x\in\R^n
                    \end{array}
                 \right.\label{eqPartLagrBody2}
\end{align}
\end{subequations}

Notice we slightly abused notations, because we used $\LL_{PL^\sx(QP_\segal)}$ both as 
a function with one argument $\mumu$ in \eqref{eqPartLagrBody}
or as an augmented function with two arguments $\mumu$ and $\mumumu$ in
\eqref{eqPartLagrBody2} just above.
However, it is not hard to check that the value of this partial Lagrangian only depends on
the first argument, \ie, since the 
$\mumumu$ terms are multiplied with (redundant) expressions that are equal to zero, we 
can write:
\begin{equation*}\LL_{PL^\sx(QP_\segal)}(\mumu,\mumumu)=\LL_{PL^\sx(QP_\segal)}(\mumu)~\forall
\mumu,\mumumu,\end{equation*} 

Let us now introduce the augmented partial Lagrangian in variables $X$ (such that $X\succeq\x\x^\top$) of the SDP relaxation 
$(SDP\left(QP_\segal\right))$ from \eqref{qppsdp1}-\eqref{qppsdp5}.
\begin{subequations}
\begin{align}
&\left(PL^\sX(QP_\segal)\right)=\max_{\mumu,\mumumu}\LL_{PL^\sX(QP_\segal)}\left(\mumu,\mumumu\right)   \label{eqPartLagrObj3} \\
&             \LL_{PL^\sX(QP_\segal)}(\mumu,\mumumu)=
                 \left\{
                    \begin{array}{ll}
                        \min &\left(Q+\sum_{i=1}^p \mu_i B_i+\sum_{j=1}^{\pp} \mumumus_j \B_j \right) \sprod X        + \\
                             &\left(\c^\top+\sum_{i=1}^p \mu_i\d^\top_i+\sum_{j=1}^{\pp} \mumumus_j \dd^\top_j \right)\x \\
                             &- \mumu^\top \e -\mumumu^\top \ee \\
                        s.t.~&A\x=b\\
                             &\begin{bmatrix} 1&\x^\top\\ \x & X\end{bmatrix}\succeq \zeros
                    \end{array}
                 \right.\label{eqPartLagrBody3}
\end{align}
\end{subequations}
The redundant constraints are dualized in objective function terms that are no longer redundant. In this relaxation, we have \textit{no}
guarantee that the term of $\mumumus_j$ is zero, \ie, we no longer have
$\B_j   \sprod   X       + \dd_j^\top       \x -        \ees_j=0~\forall j \in [1..\pp]$,
as the only constraint on $X$ is $X\succeq \x\x^\top$. 
It is not hard to see that the above program is: (i) a Lagrangian relaxation of
$(SDP\left(QP_\segal\right))$ 
and (ii) an SDP relaxation of 
$\LL_{PL^\sx(QP_\segal)}(\mumu,\mumumu)$.
Combining (i) and (ii), we directly obtain:
\begin{subequations}
\begin{align}
\label{eqTempus}
\left(PL^\sX(QP_\segal)\right)
    &\leq OPT(SDP\left(QP_\segal\right)),\\
\left(PL^\sX(QP_\segal)\right)
    &\leq \left(PL^\sx(QP_\segal)\right).
\label{eqPartLagrSdpLower}
\end{align}
\end{subequations}                                                                                                           


\subsubsection{\label{secTotLagrWithEqualities}The total Lagrangian using equality constraints}
The total Lagrangian can be expressed as in
\eqref{eqTotLagrObj}-\eqref{eqTotLagrBody}.
However, technically, we now separate the dual variables $\betabeta$
associated to the linear equality constraints. We also separate the dual variables $\mumumu$
associated with the redundant constraints. Thus, the total Lagrangian from
\eqref{eqTotLagrObj}-\eqref{eqTotLagrBody}  evolves to:

\begin{subequations}
\begin{align}
\Big(TL^\sx&(QP_\segal)\Big)=\max_{\mumu,\mumumu,\betabeta}\LL_{TL^\sx(QP_\segal)}(\mumu,\mumumu,\betabeta) \label{eqTotLagrObj2WithEquality} \\
             &\LL_{TL^\sx(QP_\segal)}(\mumu,\mumumu,\betabeta)=
            \left\{
              \begin{array}{ll}
                 \displaystyle \min_{\x\in\R^n}&\left(Q+\sum_{i=1}^p \mu_i B_i+\sum_{j=1}^{\pp} \mumumus_j \B_j\right) \sprod \x\x^\top\\
                 &+\left(\c^\top+\sum_{i=1}^p \mu_i\d^\top_i+\sum_{j=1}^{\pp} \mumumus_j\dd^\top_j+\betabeta^\top A\right)\x\\
                 & - \mumu^\top \e - \mumumu^\top \ee  -\betabeta^\top \b
              \end{array}
            \right.
                \label{eqTotLagrBody2}
\end{align}
\end{subequations}

\begin{sloppypar}
We can define the SDP version of the above total Lagrangian 
by
replacing $\x\x^\top$ with $X$ in \eqref{eqTotLagrBody2} 
and by adding constraint $X\succeq \x\x^\top$.
We thus construct the total Lagrangian
$\LL_{TL^\sX(QP)}(\mumu,\mumumu,\betabeta)$ in variables $X$ (and $\x$).
For any fixed $\mumu$ and $\mumumu$,
the value $\max\limits_{\betabeta}\LL_{TL^\sX(QP)}(\mumu,\mumumu,\betabeta)$ can be seen as a Lagrangian 
of $\LL_{PL^\sX(QP_\segal)}\left(\mumu,\mumumu,\right)$.
As such, the optimum SDP total Lagrangian 
$\Big(TL^\sX(QP_\segal)\Big)=\max\limits_{\mumu,\mumumu,\betabeta}\LL_{TL^\sX(QP_\segal)}(\mumu,\mumumu,\betabeta)$,
satisfies:
\begin{equation}\label{eqTLLessPLSDP}\left(TL^\sX(QP_\segal)\right)\leq \left(PL^\sX(QP_\segal)\right).
\end{equation}
\end{sloppypar}

\begin{proposition}\label{propTotPartSDPLagrEqual} 
The total Lagrangian $\Big(TL_{\mumu^*,\mumumu^*}^\sx(QP_\segal)\Big)$ for 
\underline{fixed $\mumu^*$ and $\mumumu^*$} (formally defined as the maximum of (\theeqtest.a) below) is equal to 
$OPT\left(\LL_{PL^\sX(QP_\segal)}(\mumu^*,\mumumu^*)\right)$
from \eqref{eqPartLagrBody3}. This is enough to ensure:
\begin{equation}\label{eqTLEqualPLSDP}\left(TL^\sx(QP_\segal)\right)=\left(TL^\sX(QP_\segal)\right)= \left(PL^\sX(QP_\segal)\right).
\end{equation}
\end{proposition}
\begin{proof}
Let us introduce notational shortcuts
$Q_{\mumu,\mumumu}= Q+\sum_{i=1}^p \mu_i B_i+\sum_{j=1}^{\pp} \mumumus_j \B_j$
and
$\c^\top_{\mumu,\mumumu}=\c^\top+\sum_{i=1}^p \mu_i\d^\top_i+\sum_{j=1}^{\pp} \mumumus_j\dd^\top_j$.
Starting from the right term of \eqref{eqTLEqualPLSDP}, notice
that $\left(\LL_{PL^\sX(QP_\segal)}(\mumu^*,\mumumu^*)\right)$ 
from \eqref{eqPartLagrBody3}
can be compactly written:

\begin{align}
&             \LL_{PL^\sX(QP_\segal)}(\mumu^*,\mumumu^*)=
                 \left\{
                    \begin{array}{ll}
                        \min &Q_{\mumu^*,\mumumu^*} \sprod X+
                             \c^\top_{\mumu^*,\mumumu^*} \x
                             - \mumu{^*}^\top \e -\mumumu{^*}^\top \ee \\
                        s.t.~&A\x=b\\
                             & \begin{bmatrix}1&\x^\top\\ \x & X\end{bmatrix}\succeq \zeros
                    \end{array}
                 \right.\label{eqPartLagrBody2BisSDPBefore}
\end{align}
The total Lagrangian 
\eqref{eqTotLagrObj2WithEquality}-\eqref{eqTotLagrBody2}
with fixed $\mumu^*$ and $\mumumu^*$, can be written in compact notations using \textit{only variables
$\betabeta$} associated to the linear equality constraint $A\x=\b$.

\addtocounter{equation}{1}
\setcounter{eqtest}{\value{equation}}
~~~~~~~~~~~~~~
\begin{minipage}{0.85\textwidth}
\begin{align}
\label{eqtest}
&\Big(TL_{\mumu^*,\mumumu^*}^\sx(QP_\segal)\Big)=\notag \\
&= 
       \left\{
         \begin{array}{llr}
                \max\limits_{\betabeta}&\LL_{TL^\sx(QP_\segal)}(\mumu^*,\mumumu^*,\betabeta) &~ \text{(\theequation a)}\\
                                s.t.&\LL_{TL^\sx(QP_\segal)}(\mumu^*,\mumumu^*,\betabeta)= 
                                    \left\{
                                      \begin{array}{ll}
                                         \displaystyle \min_{\x\in\R^n}&Q_{\mumu^*,\mumumu^*} \sprod \x\x^\top
                                         +\left(\c^\top_{\mumu^*,\mumumu^*} +\betabeta^\top A\right)\x\\
                                         & - \mumu{^*}^\top \e - \mumumu{^*}^\top \ee  -\betabeta^\top \b
                                      \end{array}
                                    \right.                                              &~ \text{(\theequation b)}\\
         \end{array}
       \right. \notag \\
            &= 
            \left\{
              \begin{array}{llr}
                 \displaystyle \max_{t,\betabeta}& t - \mumu{^*}^\top \e - \mumumu{^*}^\top \ee  -\betabeta^\top \b &~~~~~~~~~~~~~~~~~~~~~~~~~~~~~~~~~~~~~~~~~~~~~~~~~~~~~~ \text{(\theequation c)}\\
                        s.t.~&\begin{bmatrix}
                                    -t        & \frac {\c^\top_{\mumu^*,\mumumu^*} +\betabeta^\top A}2 \\
\frac {\c_{\mumu^*,\mumumu^*} +A^\top \betabeta}2 & Q_{\mumu^*,\mumumu^*} 
                              \end{bmatrix}\succeq \zeros                                                           &~~~~~~~~~~~~~~~~~~~~~~~~~~~~~~~~~~~~~~ \text{(\theequation d)}
              \end{array}
            \right. \notag
\end{align}
\end{minipage}

~\\

\noindent where we applied 
Proposition \ref{propCqpunconstrained} 
 from Section \ref{secqpunconstrained},
similarly to what we did for the total Lagrangian formulation \eqref{eqTotLagrObj2}-\eqref{eqTotLagrConstr2}.
We obtained 
in (\theeqtest c)-(\theeqtest d)
the dual of 
$\LL_{PL^\sX(QP_\segal)}(\mumu^*,\mumumu^*)$
from
\eqref{eqPartLagrBody2BisSDPBefore}. 
Since this last program is strictly feasible
(it has no
constraint on $X$ other than $X\succeq \x\x^\top$), 
we can apply the strong duality Theorem~\ref{thStrongDual2}
to conclude:
\begin{align*}
&OPT\left(DUAL\left(\LL_{PL^\sX(QP_\segal)}(\mumu^*,\mumumu^*)\right)\right)=
 \Big(TL_{\mumu^*,\mumumu^*}^\sx(QP_\segal)\Big) = 
OPT\left(\LL_{PL^\sX(QP_\segal)}(\mumu^*,\mumumu^*)\right)
\end{align*}

We can also notice that
$\left(\LL_{PL^\sX(QP_\segal)}(\mumu^*,\mumumu^*)\right)$ 
could be either bounded or unbounded ($-\infty$); recall these
two cases are also separated in the strong duality Theorem~\ref{thStrongDual2}.
If it is unbounded, then (\theeqtest d)
is surely infeasible meaning that the total Lagrangian (\theeqtest
a)--(\theeqtest b) has to be unbounded ($-\infty$).

Till here, we worked with an arbitrary fixed 
solution 
$(\mumu^*,\mumumu^*)$.
From now on, 
let us consider that
$(\mumu^*,\mumumu^*)$ 
is
the optimal solution
of the SDP partial Lagrangian $\left(PL^\sX(QP_\segal)\right)$
 from \eqref{eqPartLagrObj3};
we proved above that
$\left(PL^\sX(QP_\segal)\right)=
OPT(\LL_{PL^\sX(QP_\segal)}(\mumu^*,\mumumu^*))$
is effectively reached by the total Lagrangian 
$(TL^\sx(QP_\segal))$
in a point defined by above $\mumu^*$ and $\mumumu^*$ and some $\betabeta^*$.
Combining this with 
$\left(TL^\sx(QP_\segal)\right)= \left(TL^\sX(QP_\segal)\right)\leq \left(PL^\sX(QP_\segal)\right)$
from 
\eqref{eqTotLagrIsSdp}
and
\eqref{eqTLLessPLSDP}, we obtain
the sought inequality \eqref{eqTLEqualPLSDP} that we recall
below for the reader's convenience.
\end{proof}
\begin{equation}\label{eqTLEqualPLSDP_bis}
\left(TL^\sx(QP_\segal)\right)=\left(TL^\sX(QP_\segal)\right)=
\left(PL^\sX(QP_\segal)\right)
\end{equation}

Now recall Theorem~\ref{thTotLagrDualSdp} on total Lagrangians
that states 
$OPT(DUAL(SDP(P_\segal)))=\left(TL^\sx(QP_\segal)\right)$.
Combining this with above \eqref{eqTLEqualPLSDP_bis} and \eqref{eqTempus}-\eqref{eqPartLagrSdpLower}, we obtain the following
fundamental hierarchies (inequalities):

\hspace{-2em}\begin{minipage}{1.015\linewidth}
\begin{subequations}
\begin{align}
\label{eqFundLagr1}
OPT(DUAL(SDP(P_\segal)))&=\left(TL^\sx(QP_\segal)\right) = \left(TL^\sX(QP_\segal)\right) = \left(PL^\sX(QP_\segal)\right) \leq \left(PL^\sx(QP_\segal)\right)\\
\label{eqFundLagr2}
OPT(DUAL(SDP(P_\segal)))&=\left(TL^\sx(QP_\segal)\right) = \left(TL^\sX(QP_\segal)\right) = \left(PL^\sX(QP_\segal)\right) \leq OPT(SDP(P_\segal))
\end{align}
\end{subequations}
\end{minipage}

~\\

The values of $\left(PL^\sx(QP_\segal)\right)$ and $OPT(SDP(P_\segal))$ can not be ordered
in the most general setting.  We will provide Example~\ref{exempluUnu} in which $\left(PL^\sx(QP_\segal)\right)<OPT(SDP(P_\segal))$
and Example~\ref{exempluDoi} in which $\left(PL^\sx(QP_\segal)\right)>OPT(SDP(P_\segal))$.
As a side remark, recall from equations \eqref{eqTempus}-\eqref{eqPartLagrSdpLower}
that
$(PL^\sX(QP_\segal))$ is a relaxation of both
$(PL^\sx(QP_\segal))$ and $OPT(SDP(P_\segal))$.

\subsubsection{Using convexifications to 
obtain
$\left(PL^{\protect\sX}(QP_{\protect\segal})\right) = \left(PL^{\protect\sx}
(QP_{\protect\segal})\right)$}

We here show how
$(PL^\sx(QP_\segal))$ becomes equal to $(PL^\sX(QP_\segal))$ by exploiting 
the relaxation of the redundant constraints.
The following proposition is a (quite deeply) modified and adapted version of a theorem from a lecture note of Fr\'ed\'eric~Roupin,%
\footnote{See slide 57 of
\url{http://lipn.univ-paris13.fr/~roupin/docs/MPROSDPRoupin2017-partie2.pdf}.}
but the proof is personal.
\begin{proposition}\label{propConvexifier} If $\left(PL^\sx(QP_\segal)\right)=-\infty$ 
it is clear that \textbf{all} 
programs from the hierarchy \eqref{eqFundLagr1} are unbounded. 
If $\left(PL^\sx(QP_\segal)\right)\neq -\infty$, we can take an optimal solution $\mumu^*$ of $\left(PL^\sx(QP_\segal)\right)$
in \eqref{eqPartLagrObj}-\eqref{eqPartLagrBody}. If there exists $\mumumu^*$ that convexifies
the augmented formulation of $\left(PL^\sx(QP_\segal)\right)$ from
\eqref{eqPartLagrObj2}-\eqref{eqPartLagrBody2},
then $\left(PL^\sX(QP_\segal)\right) = \left(PL^\sx(QP_\segal)\right)$.
This means that the first fundamental hierarchy \eqref{eqFundLagr1} collapses.
\end{proposition}
\vspace{-0.4em}
\begin{sloppypar}
\noindent \textit{Proof.}
We have 
$\left(PL^\sX(QP_\segal)\right)\leq \left(PL^\sx(QP_\segal)\right)$
from \eqref{eqFundLagr1}.
Since we say (slightly abusing notations) that
$\left(PL^\sx(QP_\segal)\right)=\LL_{PL^\sx(QP_\segal)}(\mumu^*)=\LL_{PL^\sx(QP_\segal)}(\mumu^*,\mumumu^*)$,
we can write 
$\LL_{PL^\sX(QP_\segal)}(\mumu^*,\mumumu^*) \leq 
\left(PL^\sX(QP_\segal)\right)
\leq 
\left(PL^\sx(QP_\segal)\right)
=
\LL_{PL^\sx(QP_\segal)}(\mumu^*,\mumumu^*)
$.
It is thus enough to show that 
$\LL_{PL^\sX(QP_\segal)}(\mumu^*,\mumumu^*)
=
\LL_{PL^\sx(QP_\segal)}(\mumu^*,\mumumu^*)$.

Since the Hessian of $\LL_{PL^\sX(QP_\segal)}(\mumu^*,\mumumu^*)$ 
is equal to the Hessian of $\LL_{PL^\sx(QP_\segal)}(\mumu^*,\mumumu^*)$, both programs are convex and have an SDP Hessian.
Writing the optimal solution of $\LL_{PL^\sX(QP_\segal)}(\mumu^*,\mumumu^*)$ as
$X=\x\x^\top+Y$, we notice $Y\succeq\zeros$ can only increase the objective value of
$\LL_{PL^\sX(QP_\segal)}(\mumu^*,\mumumu^*)$, and so, an optimal $Y$ can be $Y=\zeros$. An optimal solution of 
$\LL_{PL^\sX(QP_\segal)}(\mumu^*,\mumumu^*)$ satisfies $X=\x \x^\top$, \ie, 
$\LL_{PL^\sX(QP_\segal)}(\mumu^*,\mumumu^*)=
\LL_{PL^\sx(QP_\segal)}(\mumu^*,\mumumu^*)=\left(PL^\sx(QP_\segal)\right)$.\qed
\end{sloppypar}
\vspace{0.4em}

We now provide an example in which the hierarchy \eqref{eqFundLagr1} collapses, but the last
inequality in \eqref{eqFundLagr2} is strict, \ie, there is a duality gap between
$(DUAL(SDP(P_\segal)))$ and $(SDP(P_\segal))$. The example relies on
the fact that 
$\min\left\{x_{12}:\left[\begin{smallmatrix}0&x_{12}&0\\x_{12}&x_{22}&0\\0&0&1+x_{12}\end{smallmatrix}\right]\succeq\zeros\right\}=0$,
while the dual of this program has optimum value -1.

\begin{example}\label{exempluUnu} We present the $(QP_\segal)$ example on the left and its partial Lagrangian $\left(PL^\sx(QP_\segal)\right)$ on the right.
After that, we will compare $(SDP(QP_\segal))$
and
$(DUAL(SDP(QP_\segal))$.

\vspace{0.3em}
\noindent\begin{minipage}{0.35\linewidth}
\begin{alignat*}{4}[left ={(QP_\segal)  \empheqlbrace}]
\min~~        &\left[\begin{smallmatrix}0&\frac 12&0&0\\ \frac12 &0&0&0\\ 0&0&0&0\\ 0&0&0&1\end{smallmatrix}\right]\sprod \x\x^\top &&       \\
s.t.~         &x_4=0                                                                                               &&                        \\
\mumumus_{44}:&\left[\begin{smallmatrix}0&  0     &0&0\\  0      &0&0&0\\ 0&0&0&0\\ 0&0&0&1\end{smallmatrix}\right]\sprod \x\x^\top=0 &&     \\
\mu_{23}     :&\left[\begin{smallmatrix}0&  0     &0&0\\  0      &0&1&0\\ 0&1&0&0\\ 0&0&0&0\end{smallmatrix}\right]\sprod \x\x^\top=0 &&     \\
\mu_{13}     :&\left[\begin{smallmatrix}0&  0     &1&0\\  0      &0&0&0\\ 1&0&0&0\\ 0&0&0&0\end{smallmatrix}\right]\sprod \x\x^\top=0 &&     \\
\mu_{11}     :&\left[\begin{smallmatrix}1&  0     &1&0\\  0      &0&1&0\\ 1&1&0&0\\ 0&0&0&0\end{smallmatrix}\right]\sprod \x\x^\top=0 &&     \\
\mu_{12}     :&\left[\begin{smallmatrix}0&-\frac12&0&0\\-\frac 12&0&0&0\\ 0&0&1&0\\ 0&0&0&0\end{smallmatrix}\right]\sprod \x\x^\top=1 &&     \\
      &\x\in \R^4                                                                                                                            
\end{alignat*}
The first above constraint is a linear equality constraint. The next one is a redundant constraint, \ie, 
we have  $\pp=1$ with
regards to the canonical formulation \eqref{qpp1}-\eqref{qpp5} of $(QP_\segal)$.
The other constraints are classical quadratic constraints.
\end{minipage}
~~~~~
\begin{minipage}{0.66\linewidth}
The left program has only two feasible
 solutions $\x^\top=[0~0~1~0]$ and
$\x^\top=[0~0~-1~0]$ both of value 0.\footnotemark
We now formulate
the (augmented) partial Lagrangian.
\begin{align*}
\Big(PL^\sx(QP_\segal)\Big)&=\max_{\mumu,\mumumu}\min_{\substack{\x\in\R^4\\x_4=0}}\\
&\left[\begin{smallmatrix}\mu_{11}   &\frac 12 -\frac {\mu_{12}}2      &\mu_{13}+\mu_{11}&0        \\  
        \frac 12 -\frac {\mu_{12}}2 &               0                 &\mu_{23}+\mu_{11} &0        \\  
        {\mu_{13}}+\mu_{11}         &\mu_{23}+\mu_{11}                &\mu_{12}          &0        \\  
        0                           & 0                               &     0            &1+\mumumus_{44}
     \end{smallmatrix}\right]\sprod \x\x^\top - \mu_{12}\\
     &=[\mumu]\sprod \x\x^\top -\mu_{12}
\end{align*}
The above (inner) minimization problem can be unbounded only if $[\mumu]_{3\times 3}\succeq \zeros$,
where $[\mumu]_{3\times 3}$ is the leading principal minor of size $3\times 3$ of $[\mumu]$.
The variable $\mumumus_{44}$ is redundant and does not appear
in the non-augmented
formulation \eqref{eqPartLagrObj}-\eqref{eqPartLagrBody} of $\Big(PL^\sx(QP_\segal)\Big)$.
However, notice that Prop~\ref{propConvexifier} does hold, \ie, the augmented partial
$\LL_{PL^\sx(QP_\segal)}(\mumu^*,\mumumus_{44}^*)$
from the inner minimization problem 
is
convex for any optimal solution $(\mumu^*,\mumumus_{44}^*)$ of
$\Big(PL^\sx(QP_\segal)\Big)$.
Since $[\mumu^*]_{3\times 3}\succeq \zeros$
and
$[\mumu^*]_{22}=0$, we need to have $[\mumu^*]_{12}=0$, \ie, $\mu^*_{12}=1$. The rest of the elements
can be, for instance, $\mu^*_{11}=\mu^*_{13}=\mu^*_{23}=\mumumus^*_{44}=0$. However, for any choice of these
latter elements, the objective value of the minimization problem will be $-1$ (notice $\x=\zeros$
leads to an objective value of $-\mu^*_{12}=-1$), and so, 
$\Big(PL^\sx(QP_\segal)\Big)=-1$.
\end{minipage}
\footnotetext{Constraints $\mu_{23}$ and $\mu_{13}$ enforce $x_2x_3=x_1x_3=0$;
combining this with constraint $\mu_{11}$ that states $x_1^2+2x_2x_3+2x_1x_3=0$,
we obtain $x_1=0$. Constraint $\mu_{12}$ enforces $-x_1x_2 +x_3^2=1\implies
x_3^2=1$. Finally, $x_2=0$ follows from constraint
$\mu_{23}$.}

~\\

\noindent The SDP relaxation $SDP(QP_\segal)$ can be written:
$\min\left\{X_{12}:X=\left[\begin{smallmatrix}0&X_{12}&0&0\\X_{12}&X_{22}&0&0\\0&0&1+X_{12}&0 \\ 0&0&0&0\end{smallmatrix}\right]\succeq\x\x^\top\succeq \zeros\right\}=0$.
We have $OPT(SDP(QP_\segal))=0$ and this solution is achieved by $X_{12}=0$, $X_{22}=0$,
$X_{33}=1+X_{12}=1$ and
$\x=[0~0~0~0]^\top$ (or $\x=[0~0~1~0]^\top$); notice $X_{12}$ can not be strictly less than 0 because $X_{11}=0$.
One could calculate that $OPT(DUAL(SDP(QP_\segal)))=-1$. All values in the hierarchies \eqref{eqFundLagr1}-\eqref{eqFundLagr2} are equal to $-1$ except
$OPT(SDP(QP_\segal))=0$.
\end{example}

\subsection{The case of $0-1$ quadratic programs: partial and total Lagrangians}

A part of this section aims at proving results related to the QCR method of
Alain Billionnet, Sourour Elloumi and Marie-Christine Plateau
(see link in Footnote \total{testfoot6}, p.~\pageref{testfootpage6})
and to the article of Alain Faye and Fr\'ed\'eric~Roupin
indicated in Footnote \codefootnotetri, p.~\pageref{testfootpage3}. However, 
I think the presentation style (using longer and more detailed arguments) and the order of
the theorems from
this manuscript is completely different.

\subsubsection{Main characterization}
We interpret a $0-1$ quadratic program exactly as a particular case 
of $(QP_\segal)$ from \eqref{qpp1}-\eqref{qpp5} in which the 
first $n$ constraints \eqref{qpp4} are $x^2_i=x_i$. We can thus have $p\geq n$
non-redundant quadratic constraints. All results from the previous Subsection~\ref{sec42}
do hold in this new $0-1$ context. In fact, the $0-1$ programs are in some sense 
simpler because 
we can prove 
$OPT(DUAL(SDP(QP_\segal)))= OPT(SDP(QP_\segal))$, \ie, the hierarchy \eqref{eqFundLagr2}
collapses and we can no longer have a duality gap 
$\Big(PL^\sx(QP_\segal)\Big) < OPT(SDP(QP_\segal))$ as in Example \ref{exempluUnu}.

\begin{theorem}\label{thSdpEqualsPL01}The following fundamental hierarchies hold for $0-1$ quadratic programs.

\noindent \begin{minipage}{1.01\linewidth}
\begin{subequations}
\begin{align}
OPT(DUAL(SDP(P_\segal)))&=\left(TL^\sx(QP_\segal)\right) = \left(TL^\sX(QP_\segal)\right)   =  \left(PL^\sX(QP_\segal)\right) =    OPT(SDP(P_\segal)) \label{eqFundLagr1Bis} \\
                        & \leq \left(PL^\sx(QP_\segal)\right) \label{eqFundLagr2Bis}
\end{align}
\end{subequations}
\end{minipage}
\end{theorem}
\begin{proof} The first equality \eqref{eqFundLagr1Bis} follows from the fact that we can 
show that the hierarchy
\eqref{eqFundLagr2} collapses, \ie, all inequalities in \eqref{eqFundLagr2} are equalities
in a $0-1$ context.
To show this, it is enough to prove there is no duality gap between $(DUAL(SDP(P_\segal)))$
and $(SDP(P_\segal))$. This is a consequence of the strong duality Theorem~\ref{thStrongDual1},
considering that $(DUAL(SDP(P_\segal)))$ is strictly feasible
and bounded -- the unbounded case $OPT(DUAL(SDP(P_\segal)))=\infty$ 
would lead to 
to $(OPT(SDP(P_\segal)))=\infty$ by virtue of
$OPT(DUAL(SDP(P_\segal)))\leq OPT(SDP(P_\segal))$.
However, 
the dual is always strictly feasible because the SDP constraint \eqref{eqTotLagrConstr2}
can be  written under the form:
$$
Y+Y_\mu=
Y+
\begin{bmatrix}
-t                &-\frac 12 \mu_1  &-\frac 12 \mu_2  & \ldots &-\frac 12 \mu_n \\
- \frac 12 \mu_1  & \mu_1           &        0        & \ldots & 0              \\
- \frac 12 \mu_2  &      0          &       \mu_2     & \ldots & 0              \\
\vdots            &    \vdots       & \vdots          & \ddots & 0              \\
- \frac 12 \mu_n  &        0        &      0          & \ldots & \mu_n          \\
\end{bmatrix} \succeq \zeros.
$$
The above matrix can be strictly feasible (positive definite) by taking $\mu_1=\mu_2\dots=\mu_n=M$
for a sufficiently large $M$, so that the bottom-right $n\times n$ minor 
$[Y+\texttt{diag}(\mumu)]_{n\times n}$ become positive definite.
A sufficient $M$ value can be $1$ minus the lowest eigenvalue of $Y$. Using the Sylvester criterion
(Prop.~\ref{propSylvester}) in reversed order, we can prove that the above matrix
$Y+Y_\mu$ is positive definite
by showing that the whole determinant is positive. This can always be the case by taking a sufficiently
large value of $-t$. By developing the Leibniz formula for the determinant of the matrix
$Y+Y_\mu$, 
there will be a term $-t\det\left([Y+\texttt{diag}(\mumu)]_{n\times n}\right)$ that can be arbitrarily large, 
so as to make the determinant as high as possible. Using the strong duality as mentioned above, 
we obtain $OPT(DUAL(SDP(P_\segal)))=OPT(SDP(P_\segal))$, and so,
the hierarchy \eqref{eqFundLagr2} collapses into \eqref{eqFundLagr1Bis}.
Finally, \eqref{eqFundLagr2Bis} follows from
\eqref{eqFundLagr1}.
\end{proof}

The above proof does show that $OPT(DUAL(SDP(P_\segal)))=OPT(SDP(P_\segal))$, but the strong duality
theorem only guarantees
that $(SDP(P_\segal))$ does reach the optimum value. The program $(DUAL(SDP(P_\segal)))$ does not 
necessarily effectively reach its optimum value. This will become clear in the following example.

\subsubsection{Two examples: $OPT(DUAL(SDP(P_{\protect\segal})))$
may not reach its own optimum value and
$OPT(SDP(P_{\protect\segal}))$ may
be strictly lower than
$\left(PL^{\protect\sx}(QP_{\protect\segal})\right)$
}

\begin{example} \label{exTLDoesNotReachOpt}We modify 
Example~\ref{exempluUnu} as
follows:
\begin{itemize}
\item[--]
The constraint associated to dual variable $\mu_{12}$ becomes
$$\begin{bmatrix}0&-\frac12&0&0\\-\frac 12&0&0&0\\ 0&0&-1&0\\
0&0&0&0\end{bmatrix}\sprod \x\x^\top=0$$
\item[--]
We add  four constraints
$x_i^2=x_i$ with $i\in[1..4]$ whose dual values are $\mu_1,~\mu_2,~\mu_3$ and $\mu_4$. 
These constraints imply that all variables are binary.
\end{itemize}
The $SDP(P_\segal)$ relaxation of $(QP_\segal)$ has the solution
$X=\zeros_{4\times 4}$ of objective value $0$. We will show that
$(DUAL(SDP(P_\segal)))$ converges to $0$, even if there is no
feasible solution with value 0 in this dual.
For this, we first write
$(DUAL(SDP(P_\segal)))$ as follows:

\begin{align*}[left ={(DUAL(SDP(P_\segal)))  \empheqlbrace}]
\max~~& t \\ 
s.t.~ &
 \left[\begin{smallmatrix}
-t                 & \frac {-\mu_1}2                    &  \frac {-\mu_2}2                & \frac {-\mu_3}2    &  \frac {-\mu_4}2       \\
\frac {-\mu_1}2    & \mu_{11}+\mu_1                     &\frac 12 -\frac {\mu_{12}}2      &\mu_{13}+\mu_{11}   &0                       \\  
\frac {-\mu_2}2    &        \frac 12 -\frac {\mu_{12}}2 &   \mu_{2}                       &\mu_{23}+\mu_{11}   &0                       \\  
\frac {-\mu_3}2    &        {\mu_{13}}+\mu_{11}         &\mu_{23}+\mu_{11}                &-\mu_{12}+\mu_{3}   &0                       \\  
\frac {-\mu_4}2    &        0                           & 0                               &     0              &1+\mumumus_{44}+\mu_4
     \end{smallmatrix}\right] \succeq \zeros,                          
\end{align*}
A solution of value zero of the above program would clearly set $t=0$. This
would imply that Row 1 contains only zeros, so that $\mu_1=\mu_2=\mu_3=\mu_4=0$. Furthermore, 
because $\mu_2$ stands alone on the diagonal at position (3,3), row
3 has to contain only zeros as well. This means $\mu_{12}=1$ which leads
to a negative value at position (4,4) on the diagonal. The resulting
matrix can not be SDP.

We can prove the optimal solution converges to zero.
Let us take
$\mu_1= 0       $, $\mu_2=\varepsilon$, $\mu_3=0$, $\mu_4=0$, $\mu_{11}=M$, $\mu_{12}=0$, $\mu_{13}=-\mu_{11}=-M$,
$\mu_{23}=-\mu_{11}=-M$, $\mumumus_{44}=0$. The variables $\varepsilon$ and $M$
stand, resp.~, for a very small and a very large positive value.
The above program simplifies to:
\begin{align*}
\max~~& t                                                       \\
s.t.~ &[{\mumu}_{M,\varepsilon}]=
 \left[\begin{smallmatrix}
-t                 &   0                                &  {-\varepsilon}/2                  &  0                 &  0                     \\
0                  &          M                         &     1/2                         &0                   &0                       \\  
{-\varepsilon}/2      &        1/2                         &   \varepsilon                      &0                   &0                       \\  
0                  &        0                           &     0                           &0                   &0                       \\  
0                  &        0                           & 0                               &     0              &1
     \end{smallmatrix}\right] \succeq \zeros,                          
\end{align*}
We take a sufficiently large value of $M$ ($>\frac 1{4\varepsilon}$), so 
as to be able to have $\det([{\mumu}_{M,\varepsilon}]_{3\times 3})>0$ for small values of $-t$, where
$[{\mumu}_{M,\varepsilon}]_{3\times 3}$ is the leading principal minor of size $3\times 3$.
However, the optimum value of above program is $-\frac \varepsilon 4$, associated
to $-t= \frac \varepsilon 4$. 
Recall $M$ can be as large as necessary to ensure $\det([{\mumu}_{M,\varepsilon}]_{3\times 3})>0$.
The objective value of this solution is $-\frac \varepsilon 4$ and its limit is $\lim_{\varepsilon\to0} -\frac \varepsilon 4=0$.
\end{example}

We continue with an example in which the fundamental inequality \eqref{eqFundLagr2Bis} is strict, 
\ie, $OPT(SDP(P_\segal)) < \left(PL^\sx(QP_\segal)\right)$.
\begin{example}\label{exempluDoi} We introduce a $0-1$ program on the left,
the partial Lagrangian on the right and then we will analyse the SDP relaxation.

~\\

\noindent\begin{minipage}{0.35\linewidth}
\begin{alignat*}{4}[left ={(QP_\segal)  \empheqlbrace}]
\min~~        &\left[\begin{smallmatrix}0 & - 1 \\ -1 & 0  \end{smallmatrix}\right]\sprod \x\x^\top                                   &&     \\
s.t.~         &x_1+x_2=1                                                                                                              &&    \\
\mu_1:        &\left[\begin{smallmatrix}1 &   0 \\ 0  & 0  \end{smallmatrix}\right]\sprod \x\x^\top -x_1=0                            &&     \\
\mu_2:        &\left[\begin{smallmatrix}0 &   0 \\ 0  & 1  \end{smallmatrix}\right]\sprod \x\x^\top -x_2=0                            &&     \\
              &\x\in \R^2                                                                                                                            
\end{alignat*}
Notice that the first constraint in the above program is a linear equality constraint. We have
no redundant constraints, \ie, we have  $\pp=0$ with regards to the canonical formulation \eqref{qpp1}-\eqref{qpp5} of
$(QP_\segal)$. 
\end{minipage}
~~~~~
\begin{minipage}{0.65\linewidth}
The only feasible solutions of $(QP_\segal)$ are $[x_1~x_2]=[1~0]$ and $[x_1~x_2]=[0~1]$, both of value
$OPT(QP_\segal)=0$.
We now formulate
the partial Lagrangian.
\begin{align*}
\left(PL^\sx(QP_\segal)\right)&=\max_{\mu_1,\mu_2}\min_{\substack{\x\in\R^2\\x_1+x_2=1}}
\left[\begin{smallmatrix}
        \mu_1 &  -1 \\ -1 & \mu_2 
     \end{smallmatrix}\right]\sprod \x\x^\top - \mu_{1}x_1-\mu_2x_2\\
\end{align*}
We know by the Lagrangian definition that $\left(PL^\sx(QP_\segal)\right)\leq OPT(QP_\segal)$. We will show 
that for $[\mu_1~\mu_2]=[-1~-1]$, the Lagrangian reaches $OPT(QP_\segal)=0$. Indeed, replacing these
$[\mu_1~\mu_2]$ values in above formula, we obtain
$$
\min_{\substack{\x\in\R^2\\x_1+x_2=1}}
\left[\begin{smallmatrix}
        -1 &  -1 \\ -1 & -1 
     \end{smallmatrix}\right]\sprod \x\x^\top + x_1+ x_2
=-(x_1+x_2)^2 + (x_1+x_2)= 0
$$
\end{minipage}

~\\

\noindent For now, we have $\left(PL^\sx(QP_\segal)\right)=OPT(QP_\segal)=0$ and let us turn towards $(SDP(QP_\segal))$.
We notice that $[x_1~x_2]=[0.5~0.5]$ combined with $X=
\left[\begin{smallmatrix}
        0.5 &  0.25 \\ 0.25 & 0.5 
     \end{smallmatrix}\right]
= \x\x^\top + 0.25 I_2=
\left[\begin{smallmatrix}
        0.25 &  0.25 \\ 0.25 & 0.25 
     \end{smallmatrix}\right]
+
\left[\begin{smallmatrix}
        0.25 &  0    \\  0   & 0.25 
     \end{smallmatrix}\right]
$ is a feasible solution of $(SDP(QP_\segal))$ with objective value $-0.5$. This leads to
$$OPT(SDP(QP_\segal))\leq -0.5 < 0 = \left(PL^\sx(QP_\segal)\right),$$
and so, the inequality \eqref{eqFundLagr2Bis} can be strict.
\end{example}
\subsubsection{The limit of the strongest convexification is 
$\left(PL^{\protect\sx}(QP_{\protect\segal})\right)$}

\paragraph{The strongest convexifications can 
lead to 
$\left(PL^{\protect\sX}(QP_{\protect\segal})\right) =
\left(PL^{\protect\sx}(QP_{\protect\segal})\right)$
\label{parBestConvexification}}
We recall Proposition \ref{propConvexifier}. It states that if
$\mumu^*$ is an optimal solution of the
non-augmented \eqref{eqPartLagrObj}-\eqref{eqPartLagrBody} formulation of 
 $\left(PL^\sx(QP_\segal)\right)$
and there 
there exists $\mumumu^*$ that convexifies
the augmented \eqref{eqPartLagrObj2}-\eqref{eqPartLagrBody2} formulation of $\left(PL^\sx(QP_\segal)\right)$, 
then $\left(PL^\sX(QP_\segal)\right) = \left(PL^\sx(QP_\segal)\right)$. In other words, 
if the convexification is successful (for the optimal $\mumu^*$ chosen above), both hierarchies \eqref{eqFundLagr1Bis}-\eqref{eqFundLagr2Bis}
collapse (in the binary case).

Let $Q_{\mumu^*}$ be the factor of the quadratic term $\x\x^\top$  
associated to the optimal solution $\mumu^*$
of the non-augmented
formulation \eqref{eqPartLagrObj}-\eqref{eqPartLagrBody} 
of
$\left(PL^\sx(QP_\segal)\right)$. The factor of
$\x\x^\top$ in the augmented formulation 
\eqref{eqPartLagrObj2}-\eqref{eqPartLagrBody2} 
of $\left(PL^\sx(QP_\segal)\right)$
has the form
$Q_{\mumu^*,\mumumu}=Q_{\mumu^*}+\sum_{j=1}^{\pp} \mumumus_j \B_j$ where the $\B_j's$ with $j\in[1..\pp]$ are
the quadratic factors of $\x\x^\top$ in the redundant constraints \eqref{qpp3}
of the $(QP_\segal)$ definition.

Let us first investigate
the case in which
if $Q_{\mu^*}$ is not positive over $\left\{\y\in\R^n:A\y=\zeros\right\}$, 
\ie, over the null space $\texttt{null}(A)$ of $A$ (see also \eqref{eqDefNull}).
In this case, we can show that $\left(PL^\sx(QP_\segal)\right)=-\infty$.
Indeed, if there is some $\y\in\R^n$ such that $A\y=\zeros$ and
$Q_{\mu^*}\sprod \y\y^\top = -z < 0$, we can take any feasible solution $\x\in\R^n$ 
of 
$\LL_{PL^\sx(QP_\segal)}(\mumu^*)$
from the non-augmented 
$\left(PL^\sx(QP_\segal)\right)$ formulation \eqref{eqPartLagrObj}-\eqref{eqPartLagrBody}.
This $\x$ satisfies $A\x=\b$ and notice that $\x+t\y$ satisfies $A(\x+t\y)=\b$ as well $\forall t>0$.
If we write the objective
value of $\x+t\y$ as a function of $t$, we obtain a polynomial of degree 2
with the leading term $-zt^2$. This is a concave polynomial that is unbounded from below regardless
of its non-quadratic terms.  
In this case it is not difficult to achieve
$\left(PL^\sX(QP_\segal)\right) = \left(PL^\sx(QP_\segal)\right)=-\infty$. 

We hereafter focus on the contrary case: $Q_{\mu^*}$ is positive over $\texttt{null(A)}$.
We are looking for redundant constraints that can convexify any 
quadratic factor $Q_{\mu}$ that is non-negative (positive) over $\texttt{null}(A)$.
In fact, it may be enough to convexify 
the matrix $Q_{\mu^*}$ associated to the 
optimal solution $\mu^*$ 
of $\left(PL^\sx(QP_\segal)\right)$, but
let us keep in mind a broader objective.
However, if it is possible to convexify $Q_{\mu^*}$ into $Q_{\mu^*,\mumumu^*}\succeq\zeros$,
we can apply Proposition \ref{propConvexifier}, 
and so, both the fundamental hierarchies \eqref{eqFundLagr1Bis}-\eqref{eqFundLagr2Bis}
collapse --
also recall Theorem~\ref{thSdpEqualsPL01}.
In such a case, we obtain 
that
$OPT(SDP(Q_\segal))=\left(PL^\sX(QP_\segal)\right)$ reaches the limit $\left(PL^\sx(QP_\segal)\right)$.

However, if the redundant constraints are not strong enough to convexify 
$Q_{\mumu^*}$ for the above chosen optimal $\mumu^*$, they might be able to
convexify some $Q_{\mu}$ for some other $\mumu$. If this happens, 
we are certain that $\left(PL^\sX(QP_\segal)\right)$ is bounded, although
its value may be  
$\LL_{PL^\sx(QP_\segal)}(\mumu)
<
\LL_{PL^\sx(QP_\segal)}(\mumu^*)
=\left(PL^\sx(QP_\segal)\right)$.
Finally, keep in mind that
only
$\left(PL^\sx(QP_\segal)\right)$ does not depend
on the chosen redundant constraints in the
fundamental hierarchy
\eqref{eqFundLagr1Bis}-\eqref{eqFundLagr2Bis}, \ie, one can consider that
all other terms in this hierarchy are actually indexed by a set of chosen redundant 
constraints.

\paragraph{Examples of redundant constraints of different strengths\label{parRedundantConstr}}

\begin{example}\label{exRedund1}(A unique redundant constraint)
A redundant constraint can be constructed by observing that $(A_i\x-b_i)^2=0$
for any $i\in[1..p]$, where $A_i$ is the row $i$ of $A$. We obtain 
$(A_i\x)^2-2 b_iA_i\x + b_i^2=0$, equivalent to 
\begin{equation}\label{eqConv1}
A_i^\top A_i \sprod \x\x^\top-2b_iA_i\x + b_i^2=0,
\end{equation} by virtue
of Lemma~\ref{lemmaprod}. By summing over all $i\in[1..p]$, we obtain
\begin{equation}\label{eqConv2}
A^\top A \sprod \x\x^\top-2\b^\top A\x + \b^\top \b =0
\end{equation}
As long as we consider a partial and not a total
Lagrangian, we have $2b_iA_i\x=2b_i^2$  because
$A\x=\b$ is active; thus, the constraint \eqref{eqConv1} reduces to 
$A_i^\top A_i \sprod \x\x^\top = b_i^2$ and \eqref{eqConv2} reduces to
\begin{equation}\label{eqConv3}
A^\top A \sprod \x\x^\top = \b^\top \b 
\end{equation}
\end{example}

\begin{example}\label{exRedund2}(A set of redundant constraints) We first use constraints $A\x=b$.
For each $i\in[1..p]$, we generate a set of constraints $x_jA_i\x=x_j\b_i$ for all $j\in[1..n]$.
There exists a linear combination of these constraints that generate the constraint from
previous Example \ref{exRedund1}.
\end{example}
\begin{proof} For each $i\in[1..p]$, we perform the following. We multiply by $A_{ij}$
the constraint $x_jA_i\x=x_j b_i,~\forall j\in[1..n]$. By summing up over all $j$, we obtain
$\left(\sum\limits_{j=1}^n A_{ij}x_j\right) A_i\x = \left(\sum\limits_{j=1}^n A_{ij}x_j\right) b_i$ that
can be written $A_i\x A_i\x = A_i \x b_i$, or furthermore 
$A_i^\top A_i \sprod \x\x^\top - b_i A_i\x=0$. We now multiply $-A_i\x+b_i=0$ by $b_i$
to obtain $-b_iA_i\x + b^2_i=0$ and we add this to the previous equality to obtain
$A_i^\top A_i \sprod \x\x^\top - 2b_i A_i\x+b_i^2=0$, which is exactly \eqref{eqConv1}.
\end{proof}

~\\
\noindent{\textbf{The above constraint sets are not equivalent in the general
non-binary case}}

\begin{remark}\label{remUniqueConstrWeaker}
In the general non-binary case, the first redundant constraint set (Example~\ref{exRedund1}) 
might be weaker. For instance, consider $Q_{\mu^*}=
\left[\begin{smallmatrix} 0 & 1 \\ 1 & 0\end{smallmatrix}\right]$ and a unique linear
constraint $A\x=\b$ with $A=[0~1]$ and $\b=0$. The constraint generated by the first Example~\ref{exRedund1}
is \eqref{eqConv2}, \ie,
$A^\top A \sprod \x\x^\top=0$, equivalent to
$\left[\begin{smallmatrix} 0 & 0 \\ 0 & 1\end{smallmatrix}\right]
\sprod \x\x^\top = 0$. This 
$A^\top A=\left[\begin{smallmatrix} 0 & 0 \\ 0 & 1\end{smallmatrix}\right]$
matrix can not convexify $Q_{\mu^*}$.
On the other hand, the second redundant constraint set (Example~\ref{exRedund2})
generates $x_1 A \x = x_1b = 0$ (or $x_1x_2=0$, \ie, 
$A'\sprod \x\x^\top =\left[\begin{smallmatrix} 0 & 1 \\ 1 & 0\end{smallmatrix}\right]\sprod \x\x^\top=0$)
and $x_2 A \x = x_2b = 0$
(or $x_2x_2=0$, \ie, 
$A''\sprod \x\x^\top=\left[\begin{smallmatrix} 0 & 0 \\ 0 & 1\end{smallmatrix}\right]\sprod \x\x^\top=0$).
The matrices 
$A'=\left[\begin{smallmatrix} 0 & 1 \\ 1 & 0\end{smallmatrix}\right]$
and 
$A''= \left[\begin{smallmatrix} 0 & 0 \\ 0 & 1\end{smallmatrix}\right]$
can easily convexify $Q_{\mu^*}$.
\end{remark}

\begin{remark}\label{remExemple2SuccessConvex}
The second redundant constraint set (Example~\ref{exRedund2}) can 
convexify any matrix $Q_{\mu^*}$ that is non-negative over \texttt{null}$(A)$. This way, we can always apply Prop.~\ref{propConvexifier} and 
collapse the hierarchy \eqref{eqFundLagr1}, \ie, the convexification is optimal
and it reaches the limit value $\left(PL^\sx(QP_\segal)\right)$.
The proof of this convexification is given in Appendix~\ref{appConvex}, see more exactly
Prop~\ref{propConvexSDP}. 

If $Q_{\mu^*}$ is \textit{strictly} positive 
over $\texttt{null(A)}$ (\ie, $\u^\top Q \u >0~\forall \u\in\texttt{null(A)}-\{\zeros\}$),
$Q_{\mu^*}$ can be convexified more easily. We can use any unique redundant constraint
of the form
$\x^\top A^\top S A \x = \b^\top S \b$ for some $S\succ\zeros$, \eg, for instance
\eqref{eqConv3} corresponds to choosing $S=I_n$.
In other words, we can always construct $Q_{\mu^*}+\lambda A^\top S
A\succeq\zeros$ for a sufficiently large $\lambda$, see the proof in Prop.~\ref{propConvexDP}.
\end{remark}

~\\
\noindent{\textbf{The above constraint sets are equivalent in the binary case}}
\vspace{0.7em}

\noindent For the binary case, we will prove below (Prop.~\ref{propEquivConv}) that the SDP programs integrating
the SDP versions of above redundant constraint sets (Examples~\ref{exRedund1}
and~\ref{exRedund2}) are equivalent. Recall that in the 0-1 case the convexified
total and partial Lagrangians reach $OPT(SDP(P_\segal))$ as stated in \eqref{eqFundLagr1Bis}.
Using above two statements, the two redundant 
constraint sets 
make the convexified
total and partial Lagrangians reach the same value $OPT(SDP(P_\segal))$, \ie, they are equivalent.

\begin{proposition} \label{propEquivConv}The SDP constraints associated to the redundant constraints from
Example~\ref{exRedund1} and Example~\ref{exRedund2} are equivalent. This means that
the best convexifications (of the total or partial Lagrangians) achieved 
by the two redundant constraint sets
have \underline{the same value in the $0-1$ case}, \ie, 
that of the SDP bound $OPT(SDP(P_\segal))$ expressed using either set of redundant constraints.
\end{proposition}
\begin{proof} Notice using Examples~\ref{exRedund1} and \ref{exRedund2} that the 
two SDP constraint sets are respectively:
\begin{subequations}
\begin{align}
&A^\top A \sprod X-2\b^\top A\x + \b^\top \b =0 \label{eqRed1}\\
& \sum\limits_{k=1}^n A_{ik}X_{jk}= x_jb_i,     ~\forall j\in[1..n],~i\in[1..p] \label{eqRed2}
\end{align}
\end{subequations}
The implication \eqref{eqRed2}$\implies$\eqref{eqRed1} can be constructed by applying on
constraints \eqref{eqRed2} the linear combination presented in the proof of Example~\ref{exRedund2}.
This linear combination leads to \eqref{eqRed1}.

It is more difficult to show the converse \eqref{eqRed1}$\implies$\eqref{eqRed2}. 
We write \eqref{eqRed1} as 
$0=A^\top A \sprod (X-\x\x^\top) + A^\top A \sprod \x\x^\top -2\b^\top A\x + \b^\top \b =
A^\top A \sprod (X-\x\x^\top) + 0$, where we used \eqref{eqConv2}
which is a consequence of $A\x=\b$ (this constraint does appear in both SDP
formulations even if it is not necessary as it can be inferred from
\eqref{eqRed1} or \eqref{eqRed2} respectively).
We thus obtain $A^\top A \sprod (X-\x\x^\top)=0$ and Prop.~\ref{propABprod} implies that
$A^\top A (X-\x\x^\top)=\zeros$. Taking row $r\in[1..n]$ and column $c\in[1..n]$ of this product, we obtain:
\begin{align*}
0 &= \sum_{k=1}^n (A^\top A)_{rk}(X-\x\x^\top)_{kc}\\
  &= \sum_{k=1}^n \left(\sum_{i=1}^p A_{ir}A_{ik}\right) \left (X_{kc} -x_kx_c\right)\\
  &= \sum_{k=1}^n \sum_{i=1}^p A_{ir}\left (A_{ik} X_{kc} -A_{ik} x_kx_c\right)\\
  &= \sum_{k=1}^n \sum_{i=1}^p A_{ir}A_{ik} X_{kc} - \sum_{i=1}^p A_{ir} b_i x_c \tag{\text{we used $A_i\x=b_i$}}\\
  &= \sum_{i=1}^p A_{ir}\left (\sum_{k=1}^n A_{ik} X_{ck} -x_c b_i\right)
\end{align*}
Since the last formula holds for all $r\in[1..n]$, we can reformulate it in terms of the 
rows $A_i$ of $A$, obtaining
$\sum_{i=1}^p A_{i}\left (\sum_{k=1}^n A_{ik} X_{ck} -x_c b_i\right)=\zeros$. This is a linear
combination of the rows $A_i$ of $A$. Assuming a legitimate condition $rank(A)=p$ (\ie, the 
constraints $A\x=\b$ are linearly independent), this linear combination can lead to $\zeros$ only
if $\sum_{k=1}^n A_{ik} X_{ck} -x_c b_i=0~\forall i\in[1..p]$. 
Since this holds for any $c\in[1..n]$, we have obtained \eqref{eqRed2}.
This proof is taken from
Prop.~5 of the paper ``Partial Lagrangian relaxation for General Quadratic Programming''
by Alain Faye and Fr\'ed\'eric~Roupin.\footnote{Published in 
\textit{A Quarterly Journal of Operations} in 2007,
vol 5(1), pp.~75-88.
\setcounter{testfoot3}{\value{footnote}}\label{testfootpage3}}
\end{proof}

\subsubsection{Collapsing both hierarchies by  convexification and an
associated {\tt Branch-and-bound}}

We here focus on solving {\it binary} equality-constrained quadratic programming.
The main idea is that we  can use
the best convexification constructed in Section~\ref{secTotLagrSandwich} to
determine fast lower bounds for a \texttt{Branch-and-bound} (Section~\ref{secbbound}) that solves
the initial binary equality-constrained  quadratic problem. As a side remark, certain convexification ideas below can well
apply to the non-binary problem as well.

\paragraph{\label{secTotLagrSandwich}Determining the best convexification
coefficients 
 $\mumu^*$ and $\mumumu^*$ 
by solving $DUAL(SDP(P_{\protect\segal}))$}

According to \eqref{eqFundLagr1Bis}, in the binary case,
the value of a convexified total Lagrangian reaches $OPT(SDP(P_\segal))$.
However, the quality of $OPT(SDP(P_\segal))$ is dependent on the redundant constraints it integrates.
Let us write the total Lagrangian from \eqref{eqTotLagrObj2WithEquality}-\eqref{eqTotLagrBody2}
in a more compact form:

\begin{subequations}
\begin{align}
\Big(TL^\sx&(QP_\segal)\Big)=\max_{\mumu,\mumumu,\betabeta}\LL_{TL^\sx(QP_\segal)}(\mumu,\mumumu,\betabeta) \label{eqTotLagrObj2WithEqualityBis} \\
             &\LL_{TL^\sx(QP_\segal)}(\mumu,\mumumu,\betabeta)=
            \left\{
              \begin{array}{ll}
                 \displaystyle \min_{\x\in\R^n}&Q_{\mumu,\mumumu} \sprod \x\x^\top
                 +\left(\c^\top_{\mumu,\mumumu} +\betabeta^\top A\right)\x\\
                 & - \mumu^\top \e - \mumumu^\top \ee  -\betabeta^\top \b,
              \end{array}
            \right.
                \label{eqTotLagrBody2Bis}
\end{align}
\end{subequations}
where 
$Q_{\mumu,\mumumu}= Q+\sum_{i=1}^p \mu_i B_i+\sum_{j=1}^{\pp} \mumumus_j \B_j$
and
$\c^\top_{\mumu,\mumumu}=\c^\top+\sum_{i=1}^p \mu_i\d^\top_i+\sum_{j=1}^{\pp} \mumumus_j\dd^\top_j$.

Let us first show that $\Big(TL^\sx(QP_\segal)\Big)$ is bounded.
For this, we can use the fact that in the $0-1$ case the objective function contains $n$ terms of
the form $\mu_i x_i^2-\mu_i x_i$ (for all $i\in[1..n]$) that
could always strictly convexify any matrix. 
In other words,
there always exist $\mumu^{o}$
and
$\mumumu^{o}$ such that $Q_{\mumu^{o},\mumumu^{o}}\succ\zeros$, \ie,
take a sufficiently large $\mu_i^o$ in the terms $\mu_i^o x_i^2-\mu_i^o x_i$ (for all
$i\in[1..n]$). We have
$Q_{\mumu^{o},\mumumu^{o}}$ invertible by virtue of
$\det\left(Q_{\mumu^{o},\mumumu^{o}}\right)>0$. The (transposed) gradient of the
objective function 
of 
$\LL_{TL^\sx(QP_\segal)}(\mumu^{o},\mumumu^{o},\betabeta)$ from
\eqref{eqTotLagrBody2Bis} 
is
$\nnabla^\top_{\mumu^{o},\mumumu^{o}}(\x)=2 \x^\top Q_{\mumu^{o},\mumumu^{o}}
+
\c^\top_{\mumu,\mumumu} +\betabeta^\top A$. 
Since $Q_{\mumu^{o},\mumumu^{o}}$ is invertible, there 
 exists a stationary point 
$\x^{o}$ in 
which 
$\nnabla^\top_{\mumu^{o},\mumumu^{o}}(\x^{o})=\zeros$. The stationary point of a (strictly) convex function is its minimizer,
and so, we are sure that 
$OPT(DUAL(SDP(P_\segal)))=\left(TL^\sx(QP_\segal)\right)=OPT(SDP(P_\segal))$ is bounded.

Let us study what happens if we determine $\mumu^*$ and $\mumumu^*$ 
from the optimal values of 
$DUAL(SDP(P_\segal))$ from \eqref{eqTotLagrObj2}-\eqref{eqTotLagrConstr2}, \ie, 
solve \eqref{eqTotLagrObj2}-\eqref{eqTotLagrConstr2} and
retrieve the optimal value of $\mumu$ (think of $\mumu$ and $\mumumu$
as merged into a unique $\mumu$).
He hereafter consider that the coefficients $\mumu^*$ and $\mumumu^*$ are fixed.
We will show that using these coefficients in the augmented partial Lagrangian allows it to 
 reach the same value as the total Lagrangian $\left(TL^\sx(QP_\segal)\right)$.
We recall the definition of the augmented partial Lagrangian 
from \eqref{eqPartLagrBody2} using more compact notations
for fixed $\mumu^*$ and $\mumumu^*$:
\begin{align}
&             \LL_{PL^\sx(QP_\segal)}(\mumu^*,\mumumu^*)=
                 \left\{
                    \begin{array}{ll}
                        \min &Q_{\mumu^*,\mumumu^*} \sprod \x\x^\top+
                             \c^\top_{\mumu^*,\mumumu^*} \x
                             - {\mumu^*}^\top \e -\mumumu{^*}^\top \ee \\
                        s.t.~&A\x=b\\
                             &\x\in\R^n
                    \end{array}
                 \right.\label{eqPartLagrBody2Bis}
\end{align}
For the fixed $\mumu^*$ and $\mumumu^*$ determined above, we have
$Q_{\mumu^*,\mumumu^*}\succeq\zeros$ because of the bottom-right term of
\eqref{eqTotLagrConstr2}.
As such,
the above program has the same optimum value as
its SDP version obtained by replacing $\x\x^\top$ with $X$
in \eqref{eqPartLagrBody2Bis} and adding $X\succeq \x\x^\top$.
This SDP partial Lagrangian 
$\LL_{PL^\sX(QP_\segal)}(\mumu^*,\mumumu^*)$
is formally defined by
\eqref{eqPartLagrBody3}.

\begin{sloppypar}
One can easily check that 
$\Big(TL^\sx(QP_\segal)\Big)
=\LL_{TL^\sx(QP_\segal)}(\mumu^*,\mumumu^*,\betabeta^*)
\leq \LL_{PL^\sx(QP_\segal)}(\mumu^*,\mumumu^*)
=\LL_{PL^\sX(QP_\segal)}(\mumu^*,\mumumu^*)\leq 
\left(PL^\sX(QP_\segal)\right)$. Both inequalities follow from Lagrangian relaxation relations.
But since we know $OPT(DUAL(SDP(P_\segal)))= \Big(TL^\sx(QP_\segal)\Big)=
\left(PL^\sX(QP_\segal)\right)=OPT(SDP(P_\segal))$ from the fundamental hierarchy
\eqref{eqFundLagr1Bis},
we can (re-)write: 

\end{sloppypar}
\begin{equation}\label{eqLastConvex}
\scalebox{.95}{$
OPT\left(\LL_{PL^\sx(QP_\segal)}(\mumu^*,\mumumu^*)\right)=\Big(TL^\sx(QP_\segal)\Big)=OPT(DUAL(SDP(QP_\segal)))=OPT(SDP(QP_\segal)).
$}
\end{equation}

\begin{remark}\label{remarkOptCoefs}
In the $0$--$1$ case, the coefficients $\mumu^*$ and $\mumumu^*$ determined by solving
$DUAL(SDP(P_\segal))$ as explained above
can be used in the partial Lagrangian
program $\left(\LL_{PL^\sx(QP_\segal)}(\mumu^*,\mumumu^*)\right)$ and make this
program reach the value $OPT(SDP(QP_\segal))$. 
If $(QP_\segal)$
 integrates  
the redundant constraints from Example~\ref{exRedund2},
the hierarchy \eqref{eqFundLagr1} collapses
as stated in 
Remark~\ref{remExemple2SuccessConvex},
\ie, we can write
$\left(\LL_{PL^\sx(QP_\segal)}(\mumu^*,\mumumu^*)\right)
 = \Big(TL^\sx(QP_\segal)\Big)
 =\left(PL^\sx(QP_\segal)\right)$.
Applying~\eqref{eqLastConvex},
we obtain 
$\left(\LL_{PL^\sx(QP_\segal)}(\mumu^*,\mumumu^*)\right)=OPT(SDP(QP_\segal))=\left(PL^\sx(QP_\segal)\right)$,
\ie, 
both \eqref{eqFundLagr1Bis}-\eqref{eqFundLagr2Bis} collapse;
the convexification is optimal,
reaching its limit value $\left(PL^\sx(QP_\segal)\right)$ from
\eqref{eqFundLagr2Bis}.
In fact, since we are in the $0-1$ case, we can also use the redundant constraint from
Example~\ref{exRedund1} and obtain the same value
$OPT(SDP(Q_\segal))=\left(PL^\sx(QP_\segal)\right)$ using
Prop.~\ref{propEquivConv}.
\end{remark}

Finally, we showed in Example \ref{exTLDoesNotReachOpt} that
the optimal solution of $DUAL(SDP(QP_\segal))$
might only converge towards $OPT(SDP(QP_\segal))$ without effectivelly reaching this value.
In this case, we can have a sequence of solutions ($\mumu^*_i,\mumumu^*_i$) whose objective
values converge towards the optimum. We can apply the same calculations
from any $\mumu^*_i,\mumumu^*_i$ that reach an objective value arbitrarily close to the optimum $OPT(SDP(QP_\segal))$.

\paragraph{Using the optimal convexification coefficients in a convex quadratic
\texttt{Branch-and-bound}\label{secbbound}}

Let us here consider that the only quadratic constraints of $(QP_\segal)$ are
of the form $x_i^2=x_i~\forall i\in[1..n]$ (\ie, integrality constraints). However,
$(QP_\segal)$ can integrate various redundant constraint sets, \eg, like those from 
Example~\ref{exRedund1} or
Example~\ref{exRedund2}. In fact, these two constraint sets produce the same SDP relaxation value
$OPT(SDP(QP_\segal))$ in the $0-1$ case, by virtue of Prop.~\ref{propEquivConv}. 
If we use any of these constraint sets, we obtain 
$OPT(SDP(QP_\segal))=\left(PL^\sx(QP_\segal)\right)$ as stated in
Remark~\ref{remarkOptCoefs}. This means that the convexification is optimal,
because it reaches its limit $\left(PL^\sx(QP_\segal)\right)$.

Once the optimal coefficients $\mumu^*$ and $\mumumu^*$ have been found, we focus
on the fact that the augmented partial Lagrangian $\LL_{PL^\sx(QP_\segal)}(\mumu^*,\mumumu^*)$ from \eqref{eqPartLagrBody2Bis}
is convex with linear constraints and has the optimum value $OPT(SDP(QP_\segal))$ by virtue of \eqref{eqLastConvex}.
From now on, we can apply a convex quadratic solver to optimize by \texttt{Branch-and-bound}
the binary version of 
$\LL_{PL^\sx(QP_\segal)}(\mumu^*,\mumumu^*)$ from \eqref{eqPartLagrBody2Bis}.
One can verify that the objective value of $\x$ in this program is equal to the value 
of $\x$ in $(QP_\segal)$ whenever we have $A\x=\b$ and $\x$ is binary (\ie, all
quadratic constraints are satisfied). 
By solving the binary version of 
$\LL_{PL^\sx(QP_\segal)}(\mumu^*,\mumumu^*)$ we actually 
solve $(QP_\segal)$.
Optimizing a convex quadratic
program takes usually less time than solving an SDP. We solved a single SDP program (\ie, $(SDP(QP_\segal))$ 
and its dual equal to the total Lagrangian) to obtain the best convexification coefficients $\mumu^*,\mumumu^*$; we use the
quality of the SDP bound at the root of the \texttt{Branch-and-bound} tree.
More importantly, at each \texttt{Brand-and-bound} node some of the $\x$
variables are fixed to binary values and we can still solve a reduced partial
Lagrangian only using the remaining variables. The value of this reduced partial
Lagrangian can be
used to prune the node.

Appendix~\ref{appbbound} briefly discusses
a more refined approach that further restricts
the feasible area of
$\LL_{PL^\sx(QP_\segal)}(\mumu^*,\mumumu^*)$
by imposing the additional constraint $x_i\in[0,1]$, equivalent to
$x^2_i\leq x_i~\forall i\in[1..n]$.

\section{Basic elements of several other research topics: under construction}
\subsection{Approximation algorithms using SDP programming}
We here only present a (famous) SDP $0.8785$--approximation algorithm for
Max--Cut, but one should keep in mind there are also other SDP approximation algorithms that exploit 
a similar approach.

\begin{proposition} (Goemans-Williamson SDP approximation algorithm) Consider a
weighted graph $G=([1..n], E)$ with weights $w_{ij}\geq 0~\{i,j\}\in E$ and
$w_{ij}=0~\forall \{i,j\}\notin E$.  The Max-Cut problem requires splitting
$[1..n]$ in two sub-sets so as to maximize the total weighted sum of the edges
with end vertices in different subsets. 
The optimum value $OPT(MC_w)$ satisfies
$$0.8785\cdot OPT(SDP_w)<OPT(MC_w)\leq OPT(SDP_w),$$ where $(SDP_w)$ is the SDP
program:
\begin{subequations}
\begin{align}[left ={(SDP_w)  \empheqlbrace}]
\max~~&\sum_{i=1}^n \sum_{j=i+1}^n \frac 12 w_{ij}(1-X_{ij})    \label{eqgw1}\\
s.t~~ &\texttt{diag}(X)=I_n                                   \label{eqgw2}\\
      &X\succeq\zeros                                         \label{eqgw3}
\end{align}
\end{subequations}
\end{proposition}
\begin{proof} One can formulate the Max-Cut problem using variables
$z_i\in\{-1,~1\}$ such that $z_i\neq z_j$ implies that vertices $i$ and $j$
belong to different subsets, and so, edge $\{i,j\}$ has a contribution 
$w_{ij}$ to the objective function. We can formulate Max-Cut
as 
$\max\left\{\sum\limits_{i=1}^n\sum\limits_{i=j+1}^n\frac 12
w_{ij}\left(1-z_iz_j\right):~z_i\in\{-1,1\}~\forall i\in[1..n]\right\}$. We now
apply the
following relaxation: we transform $z_i$ into a vector $\y_i$ (with
$i\in[1..n]$) of the unit sphere. The product $z_iz_j$ is generalized to $X_{ij}=\y_i\sprod
\y_j$, and so, $X$ is a Gram matrix that needs to be SDP (use
Prop.~\ref{propranktransprod}) and that satisfies $X_{ii}=\y_i\sprod \y_i=1$
(because $\y_i$ belongs to the unit sphere for all $i\in[1..n]$).
We thus obtain that the above $(SDP_w)$ from \eqref{eqgw1}-\eqref{eqgw3}
is a relaxation of the Max-Cut problem, and so, $OPT(MC_w)\leq OPT(SDP_w)$.

We now prove $0.8785 OPT(SDP_w)<OPT(MC_w)$. From a feasible solution of
$(SDP_w)$ from \eqref{eqgw1}-\eqref{eqgw3} we can generate a feasible Max-Cut
solution as follows. Take any vector $\v\in\R^n$ and set $z_i=-1$ if
$\y_i\sprod \v\leq 0$ and $z_i=1$ otherwise. For different vectors $\v$ we
obtain different Max-Cut solutions. But the probability of separating 
$\y_i$ and $\y_j$ (with $i,j\in[1..n]$) so as to give rise to an objective function
contribution of $w_{ij}$ is exactly $\dfrac{\arccos(\y_i\sprod \y_j)}{\pi}$.
The expected value resulting from taking different $\v$ vectors is
$\sum_{i=1}^n \sum_{j=i+1}^n w_{ij}\dfrac{\arccos(\y_i\sprod \y_j)}{\pi}$.
This expected value needs to be less than or equal to $OPT(MC_w)$.

We will show that the expected value is greater than $0.8785\cdot OPT(SDP_w)$.
For this, notice each term $\dfrac 12 w_{ij} (1-\y_i\sprod\y_j)$ from
\eqref{eqgw1} corresponds to $w_{ij}\dfrac{\arccos(\y_i\sprod \y_j)}{\pi}$
in the above sum (representing the expected value). We calculate
the minimum of 
$$
\dfrac 
{
~~\dfrac{\arccos(\y_i\sprod \y_j)}{\pi}~~
}
{
\frac 12 (1-\y_i\sprod\y_j)
}
=
\dfrac 2{\pi}
\dfrac
{\arccos(\y_i\sprod \y_j)}
{1-\y_i\sprod \y_j}
=
\dfrac 2{\pi}
\dfrac {\alpha}{1-\cos(\alpha)}=
f(\alpha).
$$
For $\alpha=0$ we have $\y_i\sprod\y_j=1$ and the correspondence 
$\dfrac 12 w_{ij} (1-\y_i\sprod\y_j)\to w_{ij}\dfrac{\arccos(\y_i\sprod \y_j)}{\pi}$
is equivalent to $0\to 0$, so we can ignore such terms. It is enough to prove
$f(\alpha)=\dfrac 2{\pi}
\dfrac {\alpha}{1-\cos(\alpha)}>0.8785$ for any $\alpha\in (0,\pi]$.
The derivative in $\alpha$ is 
$f'(\alpha)=\dfrac{2\left(1-\alpha\sin\left(\alpha\right)-\cos\left(\alpha\right)\right)}{{\pi}\left(\cos\left(\alpha\right)-1\right)^2}$.
The denominator of $f'(\alpha)$ is thus positive over the whole interval of interest
$\alpha\in(0,\pi]$. The derivative of the
numerator is $-2\alpha\cos(\alpha)$, which is
negative over $\alpha\in\left[0,\dfrac{\pi}2\right)$ and non-negative
over $\alpha\in\left[\dfrac{\pi}2,\pi\right]$.
As such, the numerator
 is $0$ in $\alpha=0$ and it decreases as $\alpha$ increases
up to $\dfrac{\pi}2$ and starts increasing after $\dfrac{\pi}{2}$. This
numerator continue to increase 
even after the point $\overline{\alpha}$ where $f'(\overline{\alpha})=0$. 
Consequently, the numerator is negative for $\alpha\leq \overline{\alpha}$ (and so is
$f'(\alpha)$) and positive for $\alpha \geq \overline{\alpha}$ (and so is
$f'(\alpha)$).
Thus, the stationary point $\overline{\alpha}$ reaches
the minimum of $f$. We obtain that for 
$\alpha=2.33112237$ the derivative is slightly negative and
for $\alpha=2.33112238$ it is slightly positive.
For both these values of
$\alpha$, $f$ is greater than $0.878567$. 
The figure below%
\footnote{The figure and the derivatives are obtained using
\url{http://derivative-calculator.net}.} 
plots in blue the value of $f$ close to
$2.331$.
\includegraphics[width=\linewidth]{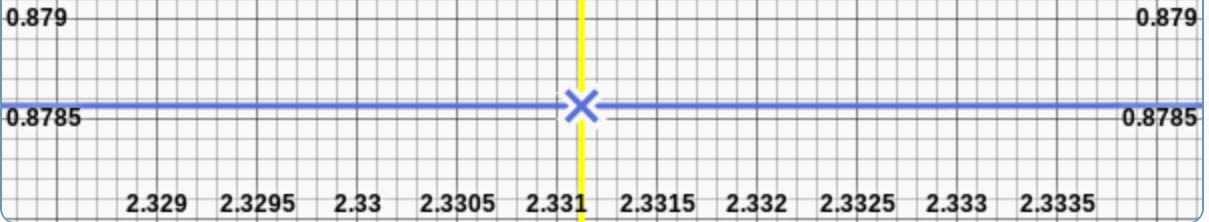}
Using numerical arguments
(see figure above), the value of $f$ is always greater than $0.8785$.
\end{proof}

As a side remark, it was
proved that the above ratio of the Goemans-Williamson algorithm is essentially
optimal if the Unique Games Conjecture holds.%
\footnote{In the article 
``Optimal inapproximability results for MAX-CUT and other 2-variable CSPs?''
by 
Subhash Khot, Guy
Kindler, Elchanan
Mossel and 
Ryan O'Donnell,
published in SIAM Journal on Computing in 2007, vol 37(1), pp 319--357, 
a preliminary version is available at
\url{https://www.cs.cmu.edu/~odonnell/papers/maxcut.pdf}.}

\subsection{Strong duality in the more general context of linear conic programming}

All SDP and linear programs presented in this work are actually particular
cases of more general linear conic programs. We here only present how the
SDP strong duality actually holds in linear conic programming.
The line of reasoning is a relatively direct generalization
of analogous results from the SDP case. As mentioned in the first paragraph of
Section~\ref{secstrongdualSDP}, the initial ideas 
are taken from a course of Anupam Gupta, also using arguments
from the lecture notes of L\' aszl\' o Lov\'asz (see
Footnote \codefootnote, p.~\pageref{testfootpage}).

\subsubsection{A preliminary conic separation result}

We need Prop.~\ref{genpropNoDPsolution} below that generalizes
Prop.~\ref{propNoDPsolution}. We first recall three basic definitions
to establish a rigorous framework.
\begin{definition}\label{defConeConvex}A convex cone $\CCC\subset\R^m$ is a set closed
under linear combinations with positive coefficients. In particular, 
if $X,Y\in \CCC$, then $tX\in \CCC~\forall t>0$ and $X+Y\in\CCC$. 

\noindent If the cone
is not convex, we only have $X\in\CCC\implies tX\in\CCC~\forall t>0$.
\end{definition}
\begin{definition}\label{defConeDual} The dual (convex) cone $\CCC^*$ of cone $\CCC$ is defined by
$$\CCC^*=\left\{Y\in\R^m:~\overline{X}\sprod Y\geq 0~\forall
\overline{X}\in\CCC\right\}.$$
\end{definition}
\begin{definition}\label{defConeInter} The interior of cone $\CCC$ is the cone
$$\inter(\CCC)=\{Z\in \CCC:~\exists \varepsilon>0\text{ s.t. }
|Z'-Z|<\varepsilon\implies Z'\in\CCC\}.$$ In other words, 
the set $\inter(\CCC)$ contains an open ball around each of its elements.

\noindent The interior is a cone because
if $\CCC$ contains a ball around $Z$,
it also contains a ball scaled by a factor of $t$ (shrinked or enlarged by $t$) around $tZ~\forall t>0$.
\end{definition}
\noindent\line(1,0){100}
\begin{proposition}\label{genpropNoDPsolution} Let $F(\x)=\sum\limits_{i=1}^n x_iA_i-B$ for any 
$\x\in\R^n$, where $B\in\R^m$ and $A_i\in \R^m~\forall i\in[1..n]$. We consider
a closed convex cone $\CCC\in\R^m$ and its dual $\CCC^*$.
The following needs to hold.
$$F(\x)\notin \inter(\CCC)~\forall \x\in\R^n \iff$$
\begin{equation}\label{geneqhyperplane}
\exists Y\in \CCC^*,~Y\neq \zeros\text{ such that }
F(\x)\sprod Y = -B\sprod Y \leq 0
,~\forall \x\in\R^n.
\end{equation}
We say that $F(\x)$ belongs to the hyperplane $\{X:~X\sprod Y=-B\sprod Y\}$.
Above equation~\eqref{geneqhyperplane} implies $A_i\sprod
Y=0~\forall i\in[1..n]$.
\end{proposition}
\begin{proof}~\\
$\Longleftarrow$\\
If $F(\x)\sprod Y\leq 0$ for some non-null $Y\in\CCC^*$, $F(\x)$ can not belong to 
the interior of $\CCC$ because any $Z\in \inter(\CCC)$
satisfies $Z\sprod Y>0$. Indeed, if $Z\in \inter(\CCC)$, then $\CCC$ contains a
ball around $Z$, and so, there exists $\varepsilon>0$ such that $Z-\varepsilon Y\in
\CCC$. If we assume $Z\sprod Y=0$, then $(Z-\varepsilon Y)\sprod Y=-\varepsilon
|Y|^2<0$, since $Y\neq \zeros$. This is a contradiction on $Y\in \CCC^*$ and
$Z-\varepsilon Y \in \CCC$. As such, any $Z\in \inter(\CCC)$ satisfies 
$Z\sprod Y>0$, and so, $F(\x)\notin \inter(\CCC)$.

\noindent $\Longrightarrow$~\\
We know that $\inter(\CCC)$ does not intersect the image of $F$. 
We can apply the hyperplane separation Theorem~\ref{thHyperGeneral}: there exists a non-zero 
$Y\in \R^{m}$ and a $c\in \R$ such that 
\begin{equation}\label{geneqStrictIneq}
F(\x)\sprod Y \leq c\leq X\sprod Y
~\forall \x\in\R^n,~\forall X\in \inter(\CCC)
\end{equation}

It is clear that we can not have $c>0$ because $X\sprod Y$ can be arbitrarily
close to $0$ by choosing $X=\varepsilon Z\in\inter(\CCC)$ for an arbitrarily small
$\varepsilon>0$ and some
$Z\in\inter(\CCC)$---recall 
$\varepsilon Z\in\inter(\CCC)$, because
$\inter(\CCC)$ is a cone (Def.~\ref{defConeInter}).
We now
prove $X\sprod Y\geq 0 ~\forall X\in \inter(\CCC)$. Let us assume the contrary:
$\exists X\in \inter(\CCC)$ such that $X\sprod Y=c'<0$. 
Since $\inter(\CCC)$ is a cone (Def~\ref{defConeInter}), we have $tX\in \inter(\CCC)~\forall t>0$. The value 
$(tX)\sprod Y=tc'$ can be arbitrarily low by choosing an arbitrarily large
$t$, and so, $(tX)\sprod Y$ can be easily less than $c$, contradiction.
This means that $X\sprod Y\geq 0~\forall X\in \inter(\CCC)$. Using $c\leq 0$,
\eqref{geneqStrictIneq} simplifies to
\begin{equation*}
F(\x)\sprod Y
\leq 0 \leq 
X\sprod Y
~\forall \x\in\R^n,~\forall X\in \inter(\CCC)
\end{equation*}

It is not hard to prove that $Y\in\CCC^*$.
For this, we will show that 
$\overline{X}\sprod Y\geq 0~\forall \overline{X}\in \CCC$, relying
on the above proved fact $\overline{X}\sprod Y\geq 0~\forall \overline{X}\in
\inter(\CCC)$.
Assume the contrary: there is some $\overline{X}\in \CCC$ such that
$\overline{X}\sprod Y<0$. 
For any $\varepsilon>0$, we have 
$\overline{X}+\varepsilon Z
\in
\inter(S_m^+)$ for any $Z\in \inter(\CCC)$---because of the cone property and of the
fact that if $\CCC$ contains a ball centered at $Z$, then
$\CCC$ contains this ball shrinked by $\varepsilon$ centered at
$\overline{X}+\varepsilon Z$.
For a small enough $\varepsilon$, 
$(\overline{X}+\varepsilon Z)\sprod Y$ remains strictly negative, which
contradicts
$\overline{X}+\varepsilon Z
\in
\inter(\CCC)$. 
We obtain $\overline{X}\sprod Y\geq 0~\forall 
\overline{X}\in\CCC$.
This means that $Y\in\CCC^*$ and $Y$ satisfies
\begin{equation}\label{geneqStrictIneqBisBis}
F(\x)\sprod Y\leq 0~\forall \x\in\R^n.
\end{equation}
We write $d=F(\zeros)\sprod Y=-B\sprod Y$. Assume for the sake of contradiction there is an $\x\in\R^n$ such
that $F(\x)=d+d'$ with $d'\neq 0$. Let us calculate 
$F\left(\frac {-d+1}{d'}\x\right)\sprod Y=
d+\frac {-d+1}{d'}d'=d-d+1=1$, which contradicts \eqref{geneqStrictIneqBisBis}. The
assumption $d'\neq 0$ was false, and so, we obtain
$$F(\x)\sprod Y = F(\zeros)\sprod Y = -B\sprod Y\leq 0,~\forall \x\in\R^n$$
where we used \eqref{geneqStrictIneqBisBis} for the last inequality.
\end{proof}
\def\CLP{\texttt{ConicLP}}
\def\DCLP{\texttt{DConicLP}}
\subsubsection{\label{secconWeakDual}Recalling the weak duality for conic programming}
Let us introduce the first conic linear program in variables $\x\in\R^n$ based
on closed convex cone $C$,
generalizing $(SDP)$ from \eqref{sdp1}-\eqref{sdp3}.
\begin{subequations}
\begin{align}[left ={(\CLP) \empheqlbrace}]
\min~~&\sum_{i=1}^n c_ix_i    \label{conic1}\\
s.t~~&\sum_{i=1}^n A_i x_i - B \in C \label{conic2}\\
    &\x\in \R^n             \label{conic3}
\end{align}
\end{subequations}

By relaxing the conic constraint \eqref{conic2} with 
multipliers $Y\in C^*$, we obtain the following Lagrangian:
$$\L(\x,Y)=\sum_{i=1}^n c_ix_i - Y\sprod (\sum A_i x_i-B)$$
Notice that if $\x$ satisfies \eqref{conic2}, we get awarded in the above Lagrangian because 
$Y\sprod (\sum A_i x_i-B)\geq 0$ when $Y\in C^*$ and $\sum A_i x_i-B\in C$.
Thus,  $\min\limits_{\x\in\R^n} \L(\x,Y)$ is a relaxation of (\CLP):%
\begin{equation}\label{coneqLagrok}\min_{\x\in\R^n} \L(\x,Y)\leq OPT(\CLP),\end{equation}
where $OPT(\CLP)$ can also be unbounded ($-\infty$).

Let us write:
$$\min_{\x\in\R^n}\L(\x,Y)=\min_{\x\in\R^n}Y\sprod B + \sum_{i=1}^n (c_i-Y\sprod A_i) x_i $$
If there is a single $i\in[1..n]$ such that $c_i-Y\sprod A_i\neq 0$, the above minimum is $-\infty$
(unbounded), by using an appropriate value of $x_i$.
To have a bounded $\min_{\x\in\R^n}\L(\x,Y)$, $Y$ needs to verify $c_i-Y\sprod A_i= 0$
for all $i\in[1..n]$. 
We are interested in finding:
$$\max_{Y\in C^*}\min_{\x\in\R^n}\L(\x,Y)$$
that can be written:
\begin{subequations}
\begin{align}[left ={(\DCLP) \empheqlbrace}]
\max~&B\sprod Y    \label{dconic1}\\
s.t ~&A_i\sprod Y = c_i~\forall i\in[1..n]\label{dconic2}\\
    ~&Y\in C^*             \label{dconic3}
\end{align}
\end{subequations}
Based on \eqref{coneqLagrok}, we obtain the weak duality:
\begin{equation}\label{coneqWeakDual}
OPT(\DCLP)\leq OPT(\CLP).
\end{equation}
In particular, if $OPT(\CLP)=-\infty$, then
$(\DCLP)$ needs to be infeasible, because otherwise the above \eqref{coneqWeakDual}
would limit $OPT(\CLP)$ from going to $-\infty$.

\subsubsection{Strong duality for the primal form}
The following theorem generalizes Theorem~\ref{thStrongDual1} of the SDP case.

\begin{theorem}\label{conthStrongDual1} If the primal $(\CLP)$ from 
\eqref{conic1}-\eqref{conic3} is bounded and has a strictly
feasible solution (Slater's interiority condition), then the primal and the dual optimal values are the same and the dual
$(\DCLP)$ from \eqref{dconic1}-\eqref{dconic3} reaches this optimum value. 
Recall (last phrase of Section~\ref{secconWeakDual}) 
that if $(\CLP)$ is unbounded, then $(\DCLP)$ is infeasible.
\end{theorem}
\begin{proof} Let $p$ be the optimal primal value. 
The system $\sum\limits_{i=1}^n c_ix_i<p$ and
$\sum_{i=1}^n A_i x_i - B\in C$ has no solution. We define
$$\A_i=
\begin{bmatrix}
-c_i       \\
A_i     
\end{bmatrix}~
\forall i\in [1..n]
\text{ and }
\B=
\begin{bmatrix}
-p       \\
B
\end{bmatrix},
$$
and the closed convex cone
$$
\CCCC=\left\{
\begin{bmatrix}
x \\      
X
\end{bmatrix}
:~
x\geq 0,~X\in C
\right\}$$
Notice that 
$\sum\limits_{i=1}^n \A_i x_i - \B\notin \inter(\CCCC)$ $\forall \x\in\R^n$
(we can \textit{not} say $\sum\limits_{i=1}^n \A_i x_i - \B\notin \CCCC$, as
the optimal solution $\x$ can cancel the value $-p$ on the first position of the
expression).
We can thus apply 
Prop.~\ref{genpropNoDPsolution} (implication ``$\Longrightarrow$'') and conclude
there is some non-zero $\Y\in {\CCCC}^*$ such that
$\A_i\sprod \Y=0$ and $-\B\sprod \Y\leq 0$. Writing 
$\Y=\left[\begin{smallmatrix}
t \\
Y\\
\end{smallmatrix}\right]$ with $t\geq 0$ and $Y\in C^*$ (necessary and
sufficient conditions to have $\Y\in {\CCCC}^*$), 
we obtain:
\begin{equation}\label{coneqeg1} tc_i=A_i\sprod Y,~\forall i\in[1..n]\end{equation}
and 
\begin{equation}\label{coneqeg2}-B\sprod Y \leq -tp.\end{equation}
Assume $t=0$. We obtain $A_i\sprod Y=0~\forall i\in[1..n]$ and $-B\sprod Y\leq
0$ with $Y\in C^*$ as discussed above. Applying again
Prop.~\ref{genpropNoDPsolution} (``$\Longleftarrow$'' implication), we conclude
$\sum\limits_{i=1}^n A_i x_i - B\notin \inter(C)~\forall \x\in
\R^n$, which contradicts the fact the primal $(\CLP)$ from~\eqref{conic1}-\eqref{conic3} is
strictly feasible. We thus need to have $t>0$.

Taking $\widehat{Y}=\frac 1t Y\in C^*$ (recall Def.~\ref{defConeConvex} and
Def.~\ref{defConeDual}), the above \eqref{coneqeg1}-\eqref{coneqeg2} become: $c_i=A_i\sprod
\widehat{Y}~\forall i\in[1..n]$ and $B\sprod\widehat{Y}\geq p$. In other words, 
$\widehat{Y}$ is a feasible solution in the
dual $(\DCLP)$ from~\eqref{dconic1}-\eqref{dconic3} and it has an objective value
$B\sprod\widehat{Y}\geq p$. Using the weak duality \eqref{coneqWeakDual}, the
feasible solution $\widehat{Y}$ of $(\DCLP)$ needs to satisfy $B\sprod\widehat{Y}\leq p$,
and so, $B\sprod\widehat{Y}=p$, \ie, the dual achieves the optimum primal value
in $\widehat{Y}$.
\end{proof}

\subsubsection{Strong duality for dual forms}
We start from the conic dual.  For the reader's convenience, we repeat the definitions of 
$(\DCLP)$ from~\eqref{dconic1}-\eqref{dconic3}:
\begin{subequations}
\begin{align}[left ={(\DCLP) \empheqlbrace}]
\max~&B\sprod Y    \label{ddconic1}\\
s.t.~&A_i\sprod Y = c_i~\forall i\in[1..n]\label{ddconic2}\\
    ~&Y\in C^*             \label{ddconic3}
\end{align}
\end{subequations}
The following is a generalization of Prop.~\ref{propEqPrimDualForms}.
\begin{proposition}\label{conpropEqPrimDualForms}
Program $(\DCLP)$ from \eqref{ddconic1}-\eqref{ddconic3} is
equivalent to $(\CLP')$ from \eqref{cconic1}-\eqref{cconic3} below
which is a program in the primal form \eqref{conic1}-\eqref{conic3}, 
provided that the system of linear
equations $A_i\sprod Y = c_i~\forall i\in[1..n]$ from \eqref{ddconic2} has at least
a feasible a solution.
\end{proposition}
\begin{proof}
We first solve the system $A_i\sprod Y = c_i~\forall i\in[1..n]$. This
system has at least a feasible solution $-B'$. The set of all solutions
is given by 
\begin{equation}\label{coneqbase}Y=-B'+\sum_{j=1}^k A'_jx'_j\text{ with }x'_j\in\R\,\forall
j\in[1..k],\end{equation}
where $A'_1,~A'_2,\dots A'_k$ are a basis of the null space of 
$\{A_i:~i\in[1..n]\}$ satisfying:
\begin{equation}\label{coneqFirstDual} 
A_i\sprod A'_j=0,~\forall i\in[1..n],~j\in[1..k].
\end{equation} 
The space spanned by (the linear combinations of) $A_i$ ($\forall i \in
[1..n]$)
and $A'_j$ ($\forall j\in [1..k]$) need to cover the whole space 
$\R^m$, by virtue of the
 rank-nullity Theorem \ref{thranknullity}.
Replacing \eqref{coneqbase} in the objective function of $(\DCLP)$ from \eqref{ddconic1}, we
obtain objective function 
$\max \left(-B'+  \sum_{j=1}^k A'_jx'_j\right)\sprod B
=
\max -B\sprod B' +\sum_{j=1}^k(B\sprod A'_j)x'_j=
-\min B\sprod B' +\sum_{j=1}^k(-B\sprod A'_j)x'_j$. The dual $(\DCLP)$ can be
written:
\begin{subequations}
\begin{align}[left ={(\CLP') \empheqlbrace}]
-\min~&B\sprod B' +\sum_{j=1}^k(-B\sprod A'_j)x'_j \label{cconic1}\\
s.t.~&\sum_{j=1}^k  A'_jx'_j - B'\in C^* \label{cconic2}\\
    &\x'\in \R^k             \label{cconic3},
\end{align}
\end{subequations}
which is a linear conic program in the primal form \eqref{conic1}-\eqref{conic3}.
\end{proof}
\begin{sloppypar}
Applying the duality $(\CLP)\to(\DCLP)$ from Section \ref{secconWeakDual}
on above $(\CLP')$, we obtain $(\DCLP')$ below. Notice we proved in 
Prop.~\ref{propDualDualClosureCone} that $\closure(C)={C^*}^*$. Since we said
from the beginning (first paragraph of Section~\ref{secconWeakDual}) that 
$C$ is closed, we can use $C={C^*}^*$.
\end{sloppypar}
\begin{subequations}
\begin{align}[left ={(\DCLP') \empheqlbrace}]
-\max~&B\sprod B' + B'\sprod Y'                          \label{dddconic1}\\
s.t.~&A'_j\sprod Y' = -B\sprod A'_j~\forall j\in[1..k]   \label{dddconic2}\\
    ~&Y'\in {C^*}^*=C                                    \label{dddconic3}
\end{align}
\end{subequations}
The system of equations \eqref{dddconic2} has at least the solution $Y'=-B$ and
this allows us to express $(\DCLP')$ in the primal form using (the
approach from) Prop~\ref{conpropEqPrimDualForms}. Any
solution of this system could be written as $-B$ plus a linear combination of 
vectors from the null space of $\{A'_j:~j\in[1..k]\}$. But we said in
Prop~\ref{conpropEqPrimDualForms} (see \eqref{coneqFirstDual} and above) 
that $\{A'_j:~j\in[1..k]\}$ are a basis of the null space of $\{A_i:~i\in[1..n]\}$.
Vectors $A'_j$ (with $j\in[1..k]$) and $A_i$ (with $i\in[1..n]$) span the whole
space $\R^m$, by virtue of the rank-nullity
Theorem \ref{thranknullity}.
The null space of $\{A'_j:~j\in[1..k]\}$ can thus be generated by 
the linear combinations of (a maximum set of independent vectors of)
$\{A_i:~i\in[1..n]\}$. This is enough to allow us to say that all solutions 
$Y'$ of \eqref{dddconic2} verify:
$$Y'=-B+\sum_{i=1}^n A_ix_i\text{ with }x_i\in\R^n\,\forall i\in[1..n].$$
Replacing this in the objective function \eqref{dddconic1} of $(\DCLP')$ above,
we obtain the objective function in variables $x_1,~x_2,\dots x_i$:
\begin{align*}
&-\max  B\sprod B'+ \left(-B+\sum_{i=1}^n A_ix_i\right)\sprod B'
=
-\max \sum_{i=1}^n (A_i\sprod B') x_i
=
-\max \sum_{i=1}^n -c_i x_i
=\min \sum_{i=1}^n c_i x_i,
\end{align*}
where we used $A_i\sprod -B'=c_i$ from the first line of the proof of
Prop.~\ref{conpropEqPrimDualForms}. This way, replacing the value of $Y'$, program $(\DCLP')$ from
\eqref{dddconic1}-\eqref{dddconic3} is completely equivalent to:
\begin{subequations}
\begin{align}[left ={(\CLP) \empheqlbrace}]
\min~~&\sum_{i=1}^n c_ix_i    \label{ccconic1}\\
s.t~~&\sum_{i=1}^n A_i x_i - B \in C \label{ccconic2}\\
    &\x\in \R^n,            \label{ccconic3}
\end{align}
\end{subequations}
which is exactly the initial program from \eqref{conic1}-\eqref{conic3}.

\begin{theorem}\label{conthStrongDual2} If the dual $(\DCLP)$ from 
\eqref{ddconic1}-\eqref{ddconic3} (or \eqref{dconic1}-\eqref{dconic3}) is bounded and has a strictly
feasible solution, then the primal and the dual optimal values are the same and the 
primal $(\CLP)$ from \eqref{ccconic1}-\eqref{ccconic3} (or
\eqref{conic1}-\eqref{conic3}) reaches this optimum value. 
\end{theorem}
\begin{proof} Since $(\DCLP)$ is feasible, we can apply
Prop.~\ref{conpropEqPrimDualForms} and obtain that $(\DCLP)$ can be reformulated
as $(\CLP')$ from \eqref{cconic1}-\eqref{cconic3}. This $(\CLP')$ needs to be
bounded and to have a strictly feasible solution as well, and so, we can apply
Theorem~\ref{conthStrongDual1} that states that the dual of $(\CLP')$ reaches 
the optimum value of $(\CLP')$. But the dual of $(\CLP')$ is $(\DCLP')$ from
\eqref{dddconic1}-\eqref{dddconic3}, which is completely equivalent to $(\CLP)$
from \eqref{ccconic1}-\eqref{ccconic3}. We obtained that $(\CLP)$ reaches the
optimum value of $(\DCLP)$.
\end{proof}
\begin{theorem}\label{conthStrongDual3} If 
both 
the primal $(\CLP)$ from \eqref{ccconic1}-\eqref{ccconic3} (or
\eqref{conic1}-\eqref{conic3}) and the dual 
$(\DCLP)$ from 
\eqref{ddconic1}-\eqref{ddconic3} (or \eqref{dconic1}-\eqref{dconic3}) are
strictly feasible, then they have the same optimum value and they both reach it.
\end{theorem}
\begin{proof} Using weak duality \eqref{coneqWeakDual}, both programs need to be
bounded. The conclusion then simply follows from combining 
Theorem~\ref{conthStrongDual1} and Theorem~\ref{conthStrongDual2}.
\end{proof}

\subsection{Polynomial Optimization}
\subsection{Algorithms for SDP optimization}
\appendix

\addtocontents{toc}{~\\ ~\\ \noindent{\bf APPENDIX} \par}


\section{\label{appBasic}On ranks, determinants and space dimensions}

This is the most elementary section from this document.
If you are an absolute beginner, you could read first App.~\ref{appBasicbasic}
before App.~\ref{appRanks}.

\subsection{\label{appRanks}The rank-nullity theorem and other interesting rank properties}
\begin{definition}\label{defrank} Given matrix $A\in \R^{n\times m}$, the rank $\texttt{rank}(A)$ has two
equivalent definitions:
\begin{itemize}\leftskip-1em
\item[(a)] the largest
order of any non-zero minor of $A$ (see Def~\ref{defMinor})
\item[(b)] the maximum number of independent rows (or
columns) of $A$, \ie, the dimension of the space generated by the rows (or columns) of $A$ using
linear combinations (see also Def.~\ref{defSpaceSize} for the notion of dimension).
\end{itemize}
\end{definition}
\begin{proof}
~\\
\noindent $(a)\implies(b)$\\
\noindent We show that the existence of a non-zero minor of order $r$. Without loss of generality, we permute rows and columns until the non-zero
minor is positioned in the upper-left corner, and, so we can consider that
$det([A]_r)\neq 0$, where $[A]_r$
is the leading principal minor of size $r\times r$, where $r=rank(A)$. Since the rows and columns of
$[A]_r$ are independent, the corresponding rows and columns in the full matrix are linearly
independent (in $A$). We have at least $r$ independent rows and columns.

We will show that any row $\a_{i}$ with $i>r$ is dependent of rows $\a_1,~\a_2,\dots \a_r$. Since
$det([A]_r)\neq 0$, there exist $x_1, x_2, \dots x_r\in \R$ such that 

\begin{equation}\label{linsol}[a_{i1}~a_{i2}~\dots
a_{ir}]=[x_1~x_2~\dots x_r] [A]_r\end{equation}Take any column $j\in[r+1..m]$ and consider the matrix obtained
by bordering $[A]_r$ with row $i$ and column $j$.
$$
\begin{bmatrix}
        &         &       &       &  a_{1j}  \\
      \multicolumn{4}{c}{\smash{\raisebox{-1\normalbaselineskip}{\scalebox{1.5}{$[A]_r$}}}} & a_{2j} \\
        &         &       &       &      \vdots \\
        &         &       &       &      a_{rj} \\
 a_{i1} & a_{i2}  & \dots & a_{ir}&  a_{ij} \\
\end{bmatrix}
$$
We subtract from the last row the first $r$ rows respectively multiplied by $x_1,~x_2,\dots x_r$. By
virtue of \eqref{linsol}, the above matrix becomes:
$$
\begin{bmatrix}
        &         &       &       &  a_{1j}  \\
      \multicolumn{4}{c}{\smash{\raisebox{-1\normalbaselineskip}{\scalebox{1.5}{$[A]_r$}}}} & a_{2j} \\
        &         &       &       &      \vdots \\
        &         &       &       &      a_{rj} \\
 0      & 0       & \dots & 0     &  a_{ij}-\sum_{k=1}^r x_k a_{kj} \\
\end{bmatrix}
$$
Since $[A]_r$ is the largest non zero minor, both above bordered matrices need to have a zero
determinant. But the determinant of the last above matrix is $\left(a_{ij}-\sum_{k=1}^r x_k
a_{kj}\right) det([A]_r)$. Since $det([A]_r)\neq 0$, we obtain that 
$a_{ij}=\sum_{k=1}^r x_k a_{kj}$. Since this needs to hold for all $j\in[r+1..m]$, we obtain
$\a_i=x_1\a_1+x_2\a_2+\dots x_r\a_r$, and so, row $i$ is dependent on the first $r$ rows. The number
of independent rows is thus exactly $r$. An analogous argument can be used to obtain that the number of
independent columns is $r$. This is the dimension of the space generated by the rows
(resp.~columns).

\noindent $(b)\implies(a)$\\
The determinant of any minor that contains rows $r+1$ rows (resp.~columns) is zero, because one
of these rows (resp.~columns) can be written as a linear combination of the other $r$.
The order of the largest non-zero minor is thus at most $r$. It is easy to see that
this largest order can not be $r-1$. If that were the case, we could apply the
above 
$(a)\implies(b)$ proof to show that all rows are dependent on some $r-1$ rows.
\end{proof}

\begin{proposition}\label{propPrincipalMinorTheorem} Given symmetric matrix $A\in \R^{n \times n}$ of rank $r$, $A$ has at least one
 non-zero \textit{principal} minor of order $r$.
\end{proposition}
\begin{proof}
The rank definition ensures the existence of a set of rows $J$ (with $|J|=r$) such that
any other row $i\in [1..n]- J$ can be written as a linear combination of  
rows $J$. This means that the matrix $A$ reduced to rows $J$ has full rank $r$.
By symmetry, each column of this latter matrix with $r$ rows can be written 
as a linear combination of the columns $J$. This means that the matrix $A$ reduced
to rows $J$ and (then) to columns $J$ has full rank $r$, \ie, the principal minor
corresponding to rows $J$ and columns $J$ is non-zero.
\end{proof}
The above proposition is sometimes referred to as the ``principal minor theorem'' and it
also holds if $A$ is skew-symmetric, \ie, if $A^\top = -A$.

\begin{theorem}\label{thranknullity}(Rank-nullity theorem) Given $A\in\R^{n\times m}$, the null space 
(kernel) of $A$ is given by
\begin{equation}\label{eqDefNull}
\texttt{null}(A)=\left\{\x\in \R^m:~A\x=\zeros_n\right\}
\end{equation}
We denote by $nullity(A)$ the dimension (maximum number of linearly independent vectors---see also
Def.~\ref{defSpaceSize}) of
$\texttt{null}(A)$. The following equation holds:
\begin{equation}
\label{eqranknullity}
m=rank(A)+nullity(A)
\end{equation}
\end{theorem}
\begin{proof}
Without loss of generality, we permute rows and columns until a non-zero
principal minor of size $r=rank(A)$ is positioned in the upper-left corner,
so that the leading principal
minor $[A]_r$ has a non-null determinant. We want to study the set of solutions
of 
\begin{equation}\label{eqSyst}
A\x=\zeros_n
\end{equation}
Using Def~\ref{defrank}, the last $n-r$ rows of $A$ can be written as a linear combination of the
first $r$ rows. It is enough to investigate only the first $r$ equations in above system
\eqref{eqSyst}. We can say that $\x$ is a solution of above system if and only if it is a solution
of the following:
\begin{equation*}
\begin{bmatrix}
                                                                                     & a_{1,r+1} & a_{1,r+2}  &\dots & a_{1,m} \\
~~~~{\smash{\raisebox{-1\normalbaselineskip}{\scalebox{2.5}{$[A]_r$}}}} & a_{2,r+1} & a_{2,r+2}  &\dots & a_{2,m} \\
                                                                                      &  \vdots   &  \vdots    &\ddots&  \vdots    \\
                                                                                      &  a_{r,r+1}&  a_{r,r+2} &\dots &  a_{r,m}\\
\end{bmatrix}
\x=\zeros_r
\end{equation*}
Using simple notational shortcuts, we write the above as follows:
\begin{equation}\label{eqSyst3}
\Big[ \left[A\right]_r~~~\left[A\right]_{m-r}\Big]
\begin{bmatrix}
\x_r\\
\x_{m-r}\\
\end{bmatrix}
=
\zeros_r,
\end{equation}
where $\left[A\right]_{m-r}$ is the $A$ minor obtained by selecting the first
$r$ rows and last $m-r$ columns; $\x_r$ selects the first $r$ components of $\x$
and $\x_{m-r}$ selects the last $m-r$. We can further re-write \eqref{eqSyst3} above as:
\begin{equation*}
[A]_r \x_r + [A]_{m-r}\x_{m-r} = \zeros_r
\end{equation*}
We can now write $\x_r$ as a function of $\x_{m-r}$, more exactly:
\begin{equation*}
 \x_r =-[A]^{-1}_r [A]_{m-r}\x_{m-r}
\end{equation*}
We can say that each of the first $r$ components of $\x$ (\ie, $\x_r$) can be
written as a linear combination of the 
last $m-r$ components (\ie, $\x_{m-r}$). The dimension of the space generated by
all above solutions $\x$ 
of \eqref{eqSyst}
reduces to the dimension
of the space generated by the last $m-r$ components $\x_{m-r}$, \ie, this
dimension is $m-r$, which confirms
\eqref{eqranknullity}.
\end{proof}

\begin{proposition} \label{proprankprod}Given $A\in \R^{n\times m}$ and $B\in \R^{m\times p}$, we have
\begin{equation}\label{eqrankprod}
rank(AB)\leq min\big(rank(A),rank(B)\big)
\end{equation}
\end{proposition}
\begin{proof} Take any $\x\in \R^n$ in the null space of $A^\top$, \ie, $A^\top \x=\zeros_m$ or $\x^\top A =
\zeros^\top_m$. We observe that $\x^\top AB=\zeros^\top_m B = \zeros_p$. This means that
$\x$ also belongs to the null space of $(AB)^\top$. This means that the null space of $(AB)^\top$ is
greater than or equal to the null space of $A^\top$. We can write
$nullity\big((AB)^\top\big)\geq
nullity(A)$. Using the rank-nullity theorem (Th~\ref{thranknullity}), we have
$rank(A^\top)+nullity(A^\top)=n$ and
$rank\big((AB)^\top\big)+nullity\big((AB)^\top\big)=n$. This means that
$rank\big((AB)^\top\big)\leq rank(A^\top)$, equivalent to 
$rank(AB)\leq rank(A)$. Analogously, we can obtain $rank(AB)\leq rank(B)$, which leads to
\eqref{eqrankprod}.
\end{proof}

\begin{proposition} \label{proprankprodInvert}Given any $A\in\R^{n\times n}$ and invertible $U$, the following holds:
$$rank(AU)=rank(A)$$
\end{proposition}
\begin{proof}
Using~\eqref{eqrankprod}, we have $rank(A)\geq rank(AU)\geq rank(AUU^{-1})=rank(A)$, which
means that all inequalities are actually equalities.
\end{proof}

\begin{proposition}\label{propRankEigen}
Any similar matrices $A,B\in \R^{n\times n}$ (\ie, such that $B = U^{-1} A
U$ for some $U\in R^{n\times n}$) have the same rank.
\end{proposition}
\begin{proof}
This simply follows from applying Proposition~\ref{proprankprodInvert} twice, once for $U$ and once
for $U^{-1}$.
\end{proof}

\begin{proposition}\label{propeigenmult} The rank of symmetric $A\in \R^{n\times n}$ is equal to $n$ minus the
multiplicity of the eigenvalue 0. We can write:
$$n = rank(A)+eigenmult(0)$$
\end{proposition}
\begin{proof}
We use the eigendecomposition (as proved in Prop~\ref{propEigenDecomp}):
$$A = U \Lambda U^\top,$$
where $U$ is invertible ($U^{-1}=U^\top$) and $\Lambda$ is a diagonal matrix with the eigenvalues of
$A$ on the diagonal.  Using Prop~\ref{propRankEigen}, $A$ and $\Lambda$ have the same rank. Since
$\Lambda$ is diagonal, its rank is equal to the number of non-zero elements on the diagonal, \ie,
$n$ minus the multiplicity of eigenvalue $0$. Additionally, remark that we proved in
Prop.~\ref{propGeoAlgeEqual} and 
Prop.~\ref{firstsym} that in symmetric matrices the algebraic multiplicity of an eigenvalue is equal
to its geometric multiplicity, hence we use the term multiplicity to refer to both. 
\end{proof}

\begin{proposition}\label{propranktransprod} Given $V\in \R^{m\times n}$, the
matrix $A=V^\top V\in \R^{n\times n}$ is
SDP. We say $A$ is the Gram matrix of the column vectors of $V$. Furthermore,
$rank(V)=rank(V^\top V)=rank(V^\top S V)$ for any positive
definite $S\in \R^{m \times m}$.
\end{proposition}
\begin{proof}~\\
1) To see $V^\top V$ is SDP, notice that for any $\x\in\R^n$, we have 
$\x^\top (V^\top V) \x
=(\x^\top V^\top)(V \x)
=( V\x)^\top(V \x)=|V\x|^2$. Also, by writing $\y=V\x$, this value simply becomes $\sum_{i=1}^n y^2_i\geq 0$.

\noindent 2)
$rank(V^\top V)=rank(V)$ follows from
$rank(V^\top S V)=rank(V)$ with $S=I_m$. We now prove the equality for an
arbitrary $S\succ\zeros$.
We will show $V^\top S V$ and $V$ have the same null space so as to apply the rank-nullity Theorem
\ref{thranknullity}. It is clear that any $\x\in\R^n$ in the null space of $V$ belongs to the null
space of $V^\top SV$, because $V\x = \zeros_m\Longrightarrow V^\top SV \x = \zeros_n$. We now prove
the converse: 
$ V^\top SV \x = \zeros_n\Longrightarrow V\x = \zeros_m$. We multiply both sides of 
$V^\top SV \x = \zeros_n$ by $\x^\top$ and we obtain $\x^\top V^\top SV \x =0$, equivalent to
$(V\x)^\top S (V \x)=0$. Using $S\succ\zeros$, this value can only be
zero if $V\x=\zeros_m$.
We obtained that the
null space of $V$ is equal to the null space of $V^\top SV$. Using the rank-nullity Theorem
\ref{thranknullity}, the two matrices need to have the same rank.
\end{proof}

\subsection{\label{appBasicbasic} Results on determinants and space dimensions}
\begin{definition}\label{defMinor} Given matrix $A$ of size $n\times m$, a minor
of $A$ is a sub-matrix obtained by selecting only some rows $J_1\subseteq [1..n]$ and some columns
$J_2\subseteq [1..m]$ of $A$. A {\it principal
minor} $[A]_J$ is a minor obtained by selecting the same rows and columns $J=J_1=J_2$.
A principal minor $[A]_J$ is a
{\it leading principal minor} if $J=[1..p]$ for some $p\leq n,m$.
We say $[A]_J$ is null (or zero) if $\det([A]_J)=0$.
The order of $[A_J]$ is $|J|$.
\end{definition}
\subsubsection{Very elementary results on matrices and determinants}
\begin{proposition}\label{lindep} Given complex matrix $A\in \C^{n\times n}$, 
$det(A)=0\iff \exists \u \in \C^n-\{\zeros\}$ such that $A\u=\zeros$.
\end{proposition}
\begin{proof}~\\
$ \Longleftarrow $ ~\\
Take $\u \neq 0$ such that $A\u=\zeros$. Without loss of generality, we assume
$u_1\neq 0$. We can write:
$$
\begin{bmatrix} a_{11} \\ a_{12} \\ \vdots \\ a_{1n} \end{bmatrix}
+
\frac{u_2}{u_1}\begin{bmatrix} a_{21} \\ a_{22} \\ \vdots \\ a_{2n} \end{bmatrix}
+
\frac{u_3}{u_1}\begin{bmatrix} a_{31} \\ a_{32} \\ \vdots \\ a_{3n} \end{bmatrix}
+
\ldots
\ldots
+
\frac{u_n}{u_1}\begin{bmatrix} a_{n1} \\ a_{n2} \\ \vdots \\ a_{nn} \end{bmatrix}
=
\begin{bmatrix} 0 \\ 0 \\ \vdots \\ 0 \end{bmatrix}
$$
We use the fact that adding a multiple of a column to another column does not
change the determinant. By performing all additions of column multiples from
above formula, we obtain only zeros in the first column. The determinant of the
resulting matrix can only be zero.\\
$\Longrightarrow$ ~\\
We proceed by induction. For $n=1$, the implication is obviously true. Assume
that it holds for $n-1$. The implication is also obvious if $A=\zeros_{n\times
n}$. We assume that $A$ has some non-zero elements and without loss of
generality we can use $a_{11}\neq 0$ (it is enough to permute lines/columns to
obtain this).

We want to find $\u \in \C^n-\{0\}$ such that $A\u=\zeros$. 
$A$ can be changed to a form in which $a_{11}$ becomes $1$: it is enough to
divide first line by the initial $a_{11}$ to obtain an equivalent system of
equations. To simplify notations, we can continue assuming $a_{11}=1$.
We perform Gaussian
elimination on first column using pivot $a_{11}=1$. We want to solve:
$$
\begin{bmatrix}
    1             &    a_{12}   &   a_{13 }    &  \ldots   & a_{1n}         \\
    0             &a_{22}-a_{12}&a_{23}-a_{13} &  \ldots   & a_{2n}-a_{1n}  \\
    0             &a_{32}-a_{12}&a_{33}-a_{13} &  \ldots   & a_{3n}-a_{1n}  \\
    \vdots        &\\
    0             &a_{n2}-a_{12}&a_{n3}-a_{13} &  \ldots   & a_{nn}-a_{1n}  \\
\end{bmatrix}
\begin{bmatrix}
    u_1 \\ u_2 \\ u_3 \\ \vdots \\ u_n \\
\end{bmatrix}
=
\begin{bmatrix}
 0    \\ 0 \\ 0 \\ \vdots \\ 0 \\
\end{bmatrix}
$$
Using compact notations, this can also be written:
\begin{equation}\label{compact}
\begin{bmatrix}
        1     &    \b^\top        \\
  \zeros_{n-1}&       B
\end{bmatrix}
\begin{bmatrix}
    u_1 \\ \v       
\end{bmatrix}
=
\begin{bmatrix}
     0 \\ \zeros_{n-1} 
\end{bmatrix}
\end{equation}
Since above row multiplications and addition did not change the determinant of
the matrix, we obtain 
$det \bigl(\begin{smallmatrix}
        1     &    \b^\top        \\
  \zeros_{n-1}&       B
\end{smallmatrix} \bigr)=0$, which means $det(B)=0$. We can use the induction
hypothesis: there exists non-zero $\v\in \C^{n-1}$ such that $B\v =
\zeros_{n-1}$. The last n-1 equations of above system are satisfied. We still
need to find $u_1$ such that $u_1+\b^\top \v = 0\implies u_1 = - \b^\top \v$.
We have just found a solution 
$\u = \big[\begin{smallmatrix}
        u_1 \\
        \v \\ 
\end{smallmatrix} \bigr]=0$
for \eqref{compact}. It is easy to check that $\u$ verifies $A\u=0$, \ie, it is
enough to reverse the Gaussian elimination since the first equation/line is
verified by $\u$.
\end{proof}
\begin{corollary} If we replace $\C$ by $\R$ in the proof of above
Proposition~\ref{lindep}, everything remains
correct. The solution of the system will be real, because $u_1 = - \b^\top \v$
is real using the induction hypothesis that $\v \in \R^{n-1}$.
\end{corollary}
\begin{proposition}The determinant of a matrix $A$ is the product of its
eigenvalues. The trace of $A$ is equal to the sum of the eigenvalues.
\label{propdetiseigenprod}
\end{proposition}
\begin{proof}
Consider the characteristic polynomial $det(xI-A)$. The eigenvalues of $A$ are
the roots of $det(xI-A)$, and so, 
\begin{equation}\label{eqdet}det(xI-A)=(x-\lambda_1)(x-\lambda_2)\dots (x-\lambda_n).\end{equation}

\noindent 1) It is enough to evaluate this equation in $x=0$ and we obtain $det(-A)=(-1)^n
\lambda_1\lambda_2\dots\lambda_n$. We also have $det(-A)=(-1)^n det(A)$, because any term in the
Leibniz formula for determinants is a product of $n$ elements of the matrix. This simply leads to
$(-1)^ndet(-A)=(-1)^n \lambda_1\lambda_2\dots\lambda_n$, and so,
$det(A)=\lambda_1\lambda_2\dots\lambda_n$\\
~\\

\noindent 2) We evaluate the term corresponding to $x^{n-1}$ of both sides of \eqref{eqdet}. In the
right side we obtain the term $-\sum_{i=1}^n \lambda_i x^{n-1}$. We will show that in the lefthand
side we obtain $-\sum_{i=1}^n \a_{ii} x^{n-1}$. Using the Leibniz formula for determinants, we
observe that the only determinant terms that make $x$ arise $n-1$ times in $det(xI-A)$ are those that use
$n-1$ diagonal terms of the form $x-a_{ii}$ with $i\in[1..n]$. Such a determinant term needs to also use the $n^\text{th}$ diagonal value as
well. To find the term of $x^{n-1}$ in $det(xI-A)$ we need to develop
$(x-a_{11})(x-a_{22})\dots(x-a_{nn})$. The $x^{n-1}$ term is $-\sum_{i=1}^n \a_{ii} x^{n-1}$.
\end{proof}

\begin{proposition}\label{propInverseComutes}Given $A,B\in \C^n$, if $AB=I_n$
then $BA=I_n$. We want a proof for the lazy, without calculating
some left $A^{-1}$ or some right $B^{-1}$.
\end{proposition}
\begin{proof} 
We first show that the columns of $B$ are linearly independent. This follows
from $\alphaalpha\in\C^n-\{\zeros\}\implies (AB)\alphaalpha\neq \zeros$. It is
clear we can not have $B\alphaalpha =\zeros$ because this would make
$AB\alphaalpha=\zeros$. 
The columns of $B$ need to be linearly independent, and so, their linear combinations cover 
a space of dimension $n$, \ie, $\C^n$ (see also Prop.~\ref{defSpaceSize}). As
such, for any column $\x_i$ of $I_n=[\x_1~\x_2\dots \x_n]$, there exist a linear combination of the columns of $B$ that
is equal to $\x_i$, \ie, $\exists\y_i$ s.~t.~$\x_i=B\y_i$. Joining toghether all
columns $\x_i$ of $I_n$, there exists $Y=[\y_1~\y_2\dots\y_n]$ s.~t.~$I_n=BY$.
We finish
by $Y=(AB)Y=A(BY)=A$ which proves $BA=I_n$.
\end{proof}

\subsubsection{The dimension of a (sub-)space}
\begin{definition}\label{defaffinelyindep} The vectors $\x_1,~\x_2,\dots \x_k\in\R^n$ are affinely indepedent if there is no
$\lambda_1,~\lambda_2,\dots \lambda_k$ with $\sum_{i=1}^k\lambda_i=0$ such that $\sum_{i=1}^k
\lambda_i\x_i=\zeros_n$. This is equivalent to the fact that $\x_1-\x_k,~\x_2-\x_k,\dots
\x_{k-1}-\x_k$ are linearly independent.
\end{definition}
\begin{proof} It is enough to show that 
$$\x_1,~\x_2,\dots \x_k \text{ affinely dependent}
\Longleftrightarrow
\x_1-\x_k,~\x_2-\x_k,\dots \x_{k-1}-\x_k \text { linearly dependent}
$$
\noindent $\Longrightarrow$\\
Given $\lambda_1,~\lambda_2,\dots \lambda_k$ with $\sum_{i=1}^k\lambda_i=0$ such that
$\sum_{i=1}^k \lambda_i\x_i=\zeros_n$, we obtain
$$\sum_{i=1}^{k-1} \lambda_i(\x_i-\x_k)\underbrace{=
  \sum_{i=1}^{k}   \lambda_i(\x_i-\x_k)}_{\textnormal{we added $+\lambda_k(\x_k-\x_k)$}}=
  \sum_{i=1}^{k}   \lambda_i\x_i -
  \sum_{i=1}^{k}   \lambda_i\x_k= \zeros_n - 0 \x_k=\zeros_n$$

\noindent $\Longleftarrow$\\
We consider there is $\lambda_1,~\lambda_2,\dots \lambda_{k-1}$ such that
$\sum_{i=1}^{k-1}\lambda_i(\x_i-\x_k)=0$,
equivalent to $\sum_{i=1}^{k-1} \lambda_i\x_i - \sum_{i=1}^{k-1}
\lambda_i\x_k=0$.
By taking $\lambda_k=-\sum_{i=1}^{k-1}\lambda_i$,
this reduces to $\sum_{i=1}^{k-1} \lambda_i\x_i+\lambda_k\x_k=0$.
\end{proof}
\begin{definition} \label{defSpaceSize}A subspace $S\subseteq \R^n$ has dimension $k$ if $k+1$ is the maximum number of
affinely independent vectors $\x_0, \x_1,~\x_2\dots x_{k}\in S$. This is equivalent to the existence of
maximum $k$ linearly independent vectors $\x_1-\x_{0},~\x_2-\x_0,\dots \x_{k}-\x_0$ by virtue of Def
\ref{defaffinelyindep}. We say that $\x_0$ is
an orgin and $\x_1-\x_0,~\x_2-\x_0,\dots \x_k-\x_0$ are a basis of $S$.

If $S$ contains $\zeros_n$, we can take $\x_0=\zeros_n$ and the definition is equivalent to the
existence of maximum $k$ linearly independent vectors
$\x_1,~\x_2,\dots \x_k$.
\end{definition}

\subsubsection{Every eigenvalue belongs to a Gershgorin disc}
\begin{theorem}\label{thGershgorin} (Gershgorin circle theorem)
Given complex matrix $A\in \C^{n\times n}$, every eigenvalue $\lambda$ belongs
to a Gershgorin disk of the following form for some $i\in[1..n]$.
$$\left\{z\in \C:~|z-A_{ii}|\leq \sum_{\substack{j=1\\j\neq i}}^n |A_{ij}|\right\}$$
\end{theorem}
\begin{proof}Consider eigenvalue $\lambda$ with an eigenvector
$\v\in\C^n-\{\zeros\}$.
Take an index $i\in[1..n]$ such that $|v_i|\geq |v_j|~\forall j\in[1..n]$.
This means that $\left|\dfrac{v_j}{v_i}\right|\leq 1~\forall j\in[1..n]$. We now develop position $i$ of 
$A\v=\lambda \v$ and we obtain $\lambda v_i=\sum\limits_{j=1}^n A_{ij}v_j$. Dividing
this by $v_i$ leads to $\lambda = A_{ii} + \sum\limits_{\substack{j=1\\j\neq
i}}^n A_{ij}\dfrac{v_j}{v_i}$. We can further develop:
\begin{align*}
\left|\lambda - A_{ii}\right|&=\left|\sum_{\substack{j=1\\j\neq i}}A_{ij}\dfrac{v_j}{v_i}\right|\\
                  &\leq \sum_{\substack{j=1\\j\neq i}}\left|A_{ij}\dfrac{v_j}{v_i}\right|
                    \tag{\text{we used the triangle inequality}}\\
                  &\leq \sum_{\substack{j=1\\j\neq i}}^n \left|A_{ij}\right|,
                    \tag{\text{we used }$\left|\dfrac{v_j}{v_i}\right|\leq 1$}
\end{align*}
which proves that $\lambda$ belongs to the Gershgorin disk associated to $i$.
\end{proof}

\section{\label{appa}Three decompositions: eigenvalue, QR and square root}

\subsection{Preliminaries on eigen-values/vectors and similar matrices}
\begin{proposition}\label{eigenexists} Given matrix $A\in \R^{n\times n}$, one can always
find some $\lambda \in \C$ and $\u\in \C^n$ such that $A\u=\lambda \u$, \ie,
there is at least one eigenvalue associated to an eigenvector.
\label{realeigenexists}
\end{proposition}
\begin{proof}
We consider the characteristic polynomial $det(xI-A)=0$. Using the fundamental
theorem of algebra, this polynomial has at least one
complex root $\lambda$. We
have $det(A-\lambda I)=0$. Using Proposition~\ref{lindep}, there exists some
$\u \in \C^n$ such that
$(A-\lambda I) \u=\zeros\implies A\u = \lambda\u$.
\end{proof}

\begin{proposition}\label{propEigenReal} All eigenvalues and proper eigenvectors of a real symmetric matrix 
are real (non complex).
\label{lambdareal}
\end{proposition}
\begin{proof}
Take symmetric matrix $A\in \R^{n\times n}$, as well as $\lambda$ and $\u$ such that $\lambda \u = A \u$. Let
$\u^*$ be the conjugate transpose of $\u$, \ie, transpose $\u$ and negate all
terms that contain $i$. We have:
$$\u^*\lambda \u = \u^* A \u $$
We take the conjugate transpose of both sides (transpose and then negate imaginary
terms) and obtain:
$$\u^* \lambda^* \u = \u^* A \u.$$
We used the fact that the conjugate of an expression can be obtained by
conjugating each of the expression's terms,\footnote{If you are unfamiliar with
complex numbers, take the product of two complex numbers: $ a_1 a_2 - b_1 b_2 +
(a_1b_1+a_2b_2) i = (a_1 + b_1 i)(a_2+b_2 i)$. By conjugating each term in the
right-hand side, we obtain the conjugate of the left-hand side, \ie,
$ a_1 a_2 - b_1 b_2 - (a_1b_1+a_2b_2) i$. For additions, the property is even easier to
verify.} as well as 
$\left(\u^*\right)^*=\u$ and
$\left(A^*\right)^*=A$.
The two above expressions have the same right-hand side and need to be equal:
$\u^*\lambda \u= \u^*\lambda^* \u\implies 
\lambda (\u^* \u) = \lambda^* (\u^* \u)$. 
Since $(a + bi)(a-b i) = a^2 + b^2$, it is easy to check 
that $\sum_i^n u^*_i u_i$ is a strictly positive real (unless $\u=\zeros$). As
such, $\lambda = \lambda^*$, \ie, $\lambda$ is real.

We now show that a proper eigenvector $\u$ is real. Suppose $\u = \u_a + \u_bi$, with $\u_a,\u_b \in \R^n$. Since $(A - \lambda
I)(\u_a + \u_bi) = \zeros$, we obtain:
$(A - \lambda I)\u_a=\zeros$ and
$(A - \lambda I)\u_b=\zeros$. This means that $\u$ is merely a combination of
other eigenvectors $\u_a$ and $\u_b$ of $\lambda$. We consider that the proper eigenvectors are $\u_a$ and $\u_b$. One
can always multiply them by complex numbers and combine them to obtain 
eigenvectors like $\u$.
\end{proof}

\begin{definition} (algebraic and geometric multiplicity) The algebraic multiplicity of an eigenvalue $\lambda$ of
matrix $A\in \C^{n\times n}$ is the multiplicity of root $\lambda$ in the
characteristic polynomial $det(xI - A)$. The geometric multiplicity of $\lambda$
is the dimension of the eigen space $\{ \u \in \C^n: A\u=\lambda \u\}$, \ie, the
maximum number of linearly independent eigenvectors of $\lambda$ (this is how we calculate the
dimension of any space that includes $\zeros$, recall Def.~\ref{defSpaceSize}). The two
multiplicities are not necessarily equal.
\end{definition}
\begin{proof}
We give an example in which the two multiplicities are not equal. Take
$A=
\bigl[\begin{smallmatrix}
        1     &        2                   \\
        0     &        1                   \\
\end{smallmatrix} \bigr]$. The characteristic polynomial is $(x-1)^2=0$, and so,
the multiplicity of eigenvalue $\lambda = 1$ is 2. To compute the geometric
multiplicity, we determine the solutions of $u_1 = u_1 + 2 u_2$ and $u_2 = 0u_1+u_2$.
We obtain $u_2=0$ and $u_1$ can take any value. The eigen space of $\lambda$
contains all vectors $\u$ with $u_2=0$; the dimension of this space is $1$.
\end{proof}

\begin{proposition}\label{simMatSameEigenVals}
Any similar matrices $A,B\in \C^{n\times n}$ (\ie, such that $B = U^{-1} A
U$ for some non-singular $U\in C^{n\times n}$) have the same characteristic polynomial.
Consequently, $A$ and $B$ have the same eigenvalues with the same algebraic
multiplicities. We also say that $B$ is the representation of $A$ in the basis
determined 
by the columns of $U$; the change of basis matrix from this basis to the
canonical basis is exactly $U$.
\end{proposition}
\begin{proof}
\begin{align*}
det(xI - B)    &= det(xI - U^{-1} A U) \\
               &= det(xU^{-1} U - U^{-1} A U) \\
               &= det(U^{-1} (xI- A) U) \\
               &= det(U^{-1})\cdot det(xI- A)\cdot det(U) \\
               &= det(xI-A)
\end{align*}
\end{proof}
\begin{proposition}\label{propSameGeoMultForSimMat}
Any two similar matrices $A,B\in \C^{n\times n}$ have the same geometric
multiplicity for any common eigenvalue $\lambda$.
\end{proposition}
\begin{proof}
Take any fixed $\lambda$ and $\u_B\in \C^n$ such that $B\u_B = \lambda \u_B$. We can write
$\lambda \u_B = U^{-1} A U \u_B$, which is equivalent to 
$\lambda U \u_B = A U \u_B$. As such, $A$ has eigenvector $\u_A=U\u_B$ with eigenvalue
$\lambda$. There is a bijection between the eigenvectors of $B$ and the
eigenvectors of $A$, given by above transformation $\u_B\to \u_A=U\u_b$.
To check the bijectivity, notice the injectivity follows from
 $U\u_B=U\u'_B\implies U(\u_B-\u'_B)=\zeros
\implies \u_B=\u'_B$ based on the non-singularity of $U$.
The surjectivity follows from $\forall \u_A\in \C^n,~\exists \u_B=U^{-1}\u_A\in\C^n$ such that
$\u_A=U\u_B$.
This bijection shows that the geometric multiplicities of $\lambda$ are
the same.
\end{proof}
\subsection{\label{apEigenDecomp}The eigenvalue decomposition}
\begin{proposition}(Eigendecomposition)\label{propEigenDecomp}
Any symmetric matrix $A\in \R^{n\times n}$ can be decomposed as follows:
\begin{align}
A &= U \Lambda U^\top \label{EigenDecomp}\\
  &= \sum_{i=1}^n \lambda_i \u_{\times,i}\u^\top_{\times,i},\label{EigenDecompBis}
\end{align}
where $\Lambda=\texttt{diag}(\lambda_1, \lambda_2, \dots \lambda_n)$ contains
the
eigenvalues of $A$ and $U$ contains columns $\u_{\times,1},~\u_{\times,2},\dots \u_{\times,n}$ 
that represent the
orthonormal unit eigenvectors of $A$.
\end{proposition}
\begin{proof} 
We provide two proofs 
(Appendix~\ref{appGeoAlgeEqual}
and
Appendix~\ref{appSchurTriang})
for showing the key fact \eqref{EigenDecomp}:
\begin{enumerate}
\leftskip -1em
    \item[--] 
The first proof was actually given in Section~\ref{sec11} and we here only
repeat it in greater detail and a bit generalized.
It relies on 
the equality between the geometric and algebraic multiplicities 
of each eigenvalue.
    \item[--] Use the Schur decomposition of complex matrices 
and above \eqref{EigenDecomp} becomes a 
simple 
re-writing of
\eqref{schurDiag}
in Prop.~\ref{firstsym}.
Reading this second proof is useful
to develop a general culture (on complex or asymmetric matrices).
\end{enumerate}

\noindent Both above proofs
also show that $U$ contains $n$ orthonormal unit eigenvectors of $A$
and that this way we have $U^\top=U^{-1}$.
After showing \eqref{EigenDecomp} with either proof, \eqref{EigenDecompBis}
simply follows from developing $A=U \Lambda U^\top$:
\begin{align*}A &= 
[\u_{\times,1},\u_{\times,2}, \dots,\u_{\times,n}]
\Lambda
\begin{bmatrix}
\u^\top_{\times,1} \\
\u^\top_{\times,2} \\
\vdots             \\
\u^\top_{\times,n}
\end{bmatrix}
=
[\u_{\times,1},\u_{\times,2}, \dots,\u_{\times,n}]
\begin{bmatrix}
\lambda_1\u^\top_{\times,1} \\
\lambda_2\u^\top_{\times,2} \\
\vdots             \\
\lambda_n\u^\top_{\times,n}
\end{bmatrix}
=
\sum_{i=1}^n \lambda_i \u_{\times,i}\u^\top_{\times,i}
\end{align*}
Writing eigenvector $\v_i=\u_{\times,i}\in\R^n$, the above formula can be
expressed in
a very compact form:
\begin{equation}\label{eigendecompnice}
A=\sum_{i=1}^n \lambda_i \v_i\v^\top_i,
\end{equation}
where $\v_1,~\v_2,\dots \v_n$ are orthonormal eigenvectors (pairwise orthogonal
and of unitary norm). 
\end{proof}

\subsubsection{\label{appGeoAlgeEqual}Proof using the equality of the geometric and algebraic multiplicities}

\begin{proposition}\label{propGeoAlgeMultResultsEigendecomp} Consider 
(possibly non-symmetric) matrix $A\in\R^{n\times n}$ such that each eigenvalue
$\lambda_i$  has the same geometric and
algebraic multiplicity $k_i$. This means root $\lambda_i$ arises $k_i$ times in
the characteristic polynomial and there are $k_i$ linearly independent eigenvectors of
$\lambda_i$. Matrix $A$ has the following eigen-decomposition
(diagonalization): 
\begin{equation}\label{EigenDecompGeoEqAlgeGen}
A=U\texttt{diag}(\lambda_1,~\lambda_2,\dots \lambda_n) U^{-1}\end{equation} 
where $U$ contains $n$ orthonormal eigenvectors of $A$ and 
$\texttt{diag}(\lambda_1,~\lambda_2,\dots \lambda_n)$ is a diagonal matrix with
the (possibly complex) eigenvalues on the diagonal. 

\noindent If $A$ is symmetric, we have $\lambda_1,~\lambda_2,\dots \lambda_n\in\R$
(Prop.~\ref{propEigenReal}). We also prove $U^{-1}=U^\top$, and so, \eqref{EigenDecompGeoEqAlgeGen} becomes:
\begin{equation}\label{EigenDecompGeoEqAlge}
A=U\texttt{diag}(\lambda_1,~\lambda_2,\dots \lambda_n) U^\top \end{equation} 
\end{proposition}
\begin{proof}
Since each eigenvalue $\lambda_i$ (with $i\in[1..n]$) has the same geometric and algebraic
multiplicity $k_i$, we can say each repetition of $\lambda_i$ as root of the
characteristic polynomial can be associated to a different eigenvector.
The eigenspace of $\lambda_i$ has dimension $k_i$ and we can surely
find an orthonormal basis of this space to represent the $k_i$ eigenvectors
associated to the $k_i$ repetitions of root $\lambda_i$. The sums of the algebraic multiplicities is $n$
because the characteristic polynomial has degree $n$, and so, we can fill
an $n\times n$ matrix $U$ with the eigenvectors of $\lambda_1,~\lambda_2,\dots
\lambda_n$.

Since
by multiplying $A$ with
any eigenvector (column) of $U$ we obtain the eigenvector multiplied by its
eigenvalue,
we can write
$U\texttt{diag}(\lambda_1,~\lambda_2,\dots \lambda_n)=AU$. By multiplying at right 
with $U^{-1}$, we obtain \eqref{EigenDecompGeoEqAlgeGen}.

If $A$ is symmetric, we can show $U^{-1}=U^\top$ using the fact that the eigenvectors are orthonormal. We already said above that the eigenvectors
corresponding to the same eigenvalue can be chosen to be orthonormal. Any
eigenvectors $\v_i$ and $\v_j$ corresponding to distinct eigenvalues
$\lambda_i\neq \lambda_j$ need to be orthogonal. 
To see this, notice
$\v_i^\top A\v_j=\v_i^\top\lambda_j\v_j=\lambda_j\v_i^\top \v_j$ and also
$\v_i^\top A\v_j=\lambda_i\v_i^\top \v_j$ based on $\v_i^\top A={(\v_i^\top
A)^\top}^\top=(A^\top \v_i)^\top =(A\v_i)^\top=\lambda_i\v_i^\top$.
This leads to $\lambda_j\v_i^\top \v_j=\lambda_i\v_i^\top \v_j$, and, using $\lambda_i\neq \lambda_j$, 
we obtain $\v_i^\top \v_j=0~\forall i,j\in[1..n],~i\neq j$. 
This shows $U^\top U = I_n$, and so, $U^{-1}=U^\top$. 
By simply replacing $U^{-1}$ with $U^\top$ in
\eqref{EigenDecompGeoEqAlgeGen}, we obtain \eqref{EigenDecompGeoEqAlge} as
needed.
\end{proof}
\begin{proposition}\label{propGeoAlgeEqual} Any eigenvalue $\lambda$ of real symmetric matrix $A$ has
the same algebraic and geometric multiplicity.
\end{proposition}
\begin{proof}We assume the characteristic polynomial of $A$ has a factor
$(x-\lambda)^s$, \ie, the algebraic multiplicity of $\lambda$ is $s$. Let $t$
be the dimension of the eigenspace of $\lambda$, \ie, the geometric
multiplicity is $t$. We will show $s=t$. Consider $t$ orthonormal
eigenvectors of $\lambda$, generated by taking an orthonormal basis of the
eigenspace of $\lambda$. We construct an unitary matrix $V$ by putting these
eigenvectors on the first $t$ columns and by filling the remaining $n-t$ columns
with other vectors that generate an orthonormal basis of $\R^n$ together with 
$\v_1,~\v_2,\dots \v_{t}$. We thus have $V^{-1}=V^\top$ and we can compute:
$$
V^\top A V=V^\top 
\left[ \lambda\v_1~\lambda\v_2~\dots \lambda\v_{t}~B\right]
=
\begin{bmatrix}
\lambda I_{t}           &  C    \\
\zeros_{n-t,t} &      D
\end{bmatrix},
$$
where the zeros on the first $t$ columns are due to the fact that each $\v_i$
with $i\in[1..t]$ is orthogonal to all the other column vectors of $V$.
By transposing above formula, we obtain the same matrix because $\left(V^\top A
V\right)^\top = V^\top A^\top V = V^\top A V$. This means that 
$
\left[
\begin{smallmatrix}
\lambda I_{t}           &  C    \\
\zeros_{n-t,t} &      D
\end{smallmatrix}
\right]
=
\left[
\begin{smallmatrix}
\lambda I_{t}           &  C    \\
\zeros_{n-t,t} &      D
\end{smallmatrix}
\right]^\top
$, and so, $C$ must be zero and $D$ must be diagonal. We can write:
$$
V^{-1}A V=
V^\top A V
=
\begin{bmatrix}
\lambda I_{t}           &  \zeros_{t,n-t}    \\
\zeros_{n-t,t} &      D
\end{bmatrix}.
$$

This means that matrices $A$ 
and
$
\left[
\begin{smallmatrix}
\lambda I_{t}           & \zeros_{t,n-t}     \\
\zeros_{n-t,t} &      D
\end{smallmatrix}
\right]
$
are similar, and so, they have the same characteristic polynomial by virtue of
Prop.~\ref{simMatSameEigenVals}. 
The characteristic polynomial of the second matrix is $(x-\lambda)^t
\det(xI-D)$.
This directly shows that we can \textit{not}
have $t>s$. This would be equivalent to the existence of a term $(x-\lambda)^t$ with $t>s$
in the characteristic polynomial of $A$, which is impossible because the
algebraic multiplicity of $\lambda$ is $s$.

We now prove by contradiction that
$t=s$. Assuming the contrary, the only remaining case is $t<s$. This means that
$\det(xI-D)$ has to
contain a term $(x-\lambda)^{s-t}$
because $\det(xI-A)$ contains a term $(x-\lambda)^s$.
This way, $\lambda$ is an eigenvalue of $D$
for which there exists at least an eigenvector $\d\in\R^{n-t}$.
We will show that $\lambda$ has a geometric multiplicity higher than $t$ in
$
\left[
\begin{smallmatrix}
\lambda I_{t}           &  \zeros_{t,n-t}    \\
\zeros_{n-t,t} &      D
\end{smallmatrix}
\right]
$, which contradicts Prop.~\ref{propSameGeoMultForSimMat}, \ie, similar matrices
must have the same eigenvalue multiplicities for a common eigenvalue $\lambda$.
It is not hard to check that $\lambda$ has at least the following $t+1$ eigenvectors in
$
\left[
\begin{smallmatrix}
\lambda I_{t}           &  \zeros_{t,n-t}    \\
\zeros_{n-t,t} &      D
\end{smallmatrix}
\right]
$:
\newcommand\vvspace[1]{%
\begin{minipage}{0em} \vspace{#1} \end{minipage}
}
\newcommand\vertbracket[2]{%
\left.  \begin{minipage}{0em} \vspace{#1} \end{minipage}\hspace{-0.4em}\right\}\scalebox{0.85}{#2}
}
$$
\begin{bmatrix}
1\\
0\\
0\\
0\\
0\\
\vdots\\
0
\end{bmatrix}
\begin{matrix}
\vvspace{0.5em}\\
\vertbracket{7.8em}{n-1}\\
\end{matrix}
~,~~
\begin{bmatrix}
0\\
1\\
0\\
0\\
0\\
\vdots\\
0
\end{bmatrix}
\begin{matrix}
\vvspace{2.3em}\\
\vertbracket{6.9em}{n-2}\\
\end{matrix}
~,~~
\begin{bmatrix}
0\\
0\\
1\\
0\\
0\\
\vdots\\
0
\end{bmatrix}
\begin{matrix}
\vvspace{3.5em}\\
\vertbracket{5.6em}{n-3}\\
\end{matrix}
\dots
~~~
\begin{bmatrix}
0\\
\vdots \\
0\\
1\\
0\\
\vdots\\
0
\end{bmatrix}
\begin{matrix}
\vertbracket{3.9em}{t-1}\\
\vvspace{1.6em}\\
\vertbracket{4em}{n-t}\\
\end{matrix}
\text{and }
\begin{bmatrix}
0\\
\vdots \\
0\\
0\\
\vvspace{1.7em}
\\
\d\\
\vvspace{1.7em}
\end{bmatrix}
\begin{matrix}
\vertbracket{5em}{t}\\
\vvspace{4.6em}\\
\end{matrix}.
$$
The last vector is an eigenvector, because it is enough to check that the top first $t$ positions of the product
with 
$
\left[
\begin{smallmatrix}
\lambda I_{t}           &  \zeros_{t,n-t}    \\
\zeros_{n-t,t} &      D
\end{smallmatrix}
\right]
$
 are
$\lambda I_t \zeros_t+\zeros_{t,n-t}\d=\zeros_t$ and the bottom $n-t$ positions are $\zeros_{n-t,t}\zeros_t+D\d=\lambda \d$ since
$\d$ is an eigenvector of $D$. We obtained that $\lambda$ has geometric multiplicity at least $t+1$ in 
$\left[
\begin{smallmatrix}
I_{t}           &  \zeros_{t,n-t}    \\
\zeros_{n-t,t} &      D
\end{smallmatrix}
\right]
$ and $t$ in $A$,
contradicting Prop.~\ref{propSameGeoMultForSimMat} as stated
above. This ensures the only possible case is $t=s$.
\end{proof}
\subsubsection{\label{appSchurTriang}Proof using the Schur triangulation of general complex square matrices}
\begin{theorem} \label{thschur}(Schur decomposition) Given any
$A\in \C^{n\times n}$, there exists unitary matrix $U\in \C^{n\times n}$ (\ie, such
that the conjugate transpose satisfies $U^*U=I$) for which:
\begin{equation}\label{schur} U^*AU=T,\end{equation}
where $T$ is an upper triangular matrix. The diagonal elements of $T$ are the
eigenvalues of $A$. The number of times eigenvalue $\lambda$ appears on the
diagonal of $T$ is the algebraic multiplicity of $\lambda$ in $A$.
\end{theorem}
\begin{proof}
We proceed by induction. For $n=1$, the theorem is clearly true. Assume it holds
for $n-1$ and we prove it also holds for $n$.

Consider an eigenvalue $\lambda$ and an eigenvector $\u$ (they exists as proved by
Proposition~\ref{eigenexists}). Without loss of generality we assume $\u$ is
unitary, \ie, $\u^*\u=1$ (from a non-unitary eigenvector we can easily obtain 
an unitary one by dividing each of its term by a non-negative real number).

We construct an unitary matrix $\overline U\in \C^{n\times n}$ with $\u$ on the
first column. We write
$\overline U = [ \u,V]$ with $V\in C^{n\times (n-1)}$. The construction of $V$ can be done
column by column as follows. The first unitary column $\v_{\times,1}$ of $V$ is chosen by
solving $\u^* \v_{\times,1}=0$ in variables $v_{11}, v_{21}, \dots v_{n1}$. The
existence of a solution for this equation can follow from the more general
Proposition~\ref{lindep}; and once a solution is found, we easily make 
$\v_{\times,1}$ unitary by dividing all terms by the initial norm $|\v_{\times,1}|$.
The $i^\texttt{th}$ unitary column $\v_{\times,i}$ (for $i\leq
n-1$) is chosen by solving the following system in variables $v_{1i}, v_{2i}, \dots
v_{ni}$: (a) $\u^* \v_{\times,i}=0$ and (b) $\v_{\times,j}^* \v_{\times,i}=0$ for all
$j\leq i-1$. There are at most $n-1$ equations for $n$ variables, and so,
a solution has to exist (for the skeptical, the coefficients of the $n-1$
equations can be put in a $n\times n$ matrix filled with zeros on the last row,
so as to obtain a null determinant and apply Proposition~\ref{lindep}). We obtain:

$$\overline{U}^* A \overline U = 
\begin{bmatrix}
\u^* \\
V^*
\end{bmatrix}
A
[\u,V]=
\begin{bmatrix}
\u^* \\
V^*
\end{bmatrix}
[\lambda \u,AV]=
\begin{bmatrix}
\lambda          &   \u^* AV   \\
\lambda V^*\u    &   V^* A V \\
\end{bmatrix}
=
\begin{bmatrix}
\lambda          &   \u^* AV   \\
\zeros_{n-1}     &   V^* A V \\
\end{bmatrix}
$$

As a side remark, remark that if $A$ is hermitian ($A=A^*$) or real symmetric, the top-right term is
also zero: $\u^* A V=(A\u)^* V=\lambda \u^* V=\zeros_{n-1}^\top$. In fact, it is
possible to particularize the above line of proof to directly prove 
that real symmetric matrices are diagonalizable
(and produce an eigen-decomposition), but it may be useful now to stay a
bit more on the general case.

Let us use more compact notations for above equation:
\begin{equation}\label{compactschur}
\overline{U}^* A \overline U = 
\begin{bmatrix}\
\lambda          &   \b^\top   \\
\zeros_{n-1}     &     B       \\
\end{bmatrix}
\end{equation}

Using the induction hypothesis, there is some unitary matrix
$U_B\in\C^{(n-1)\times(n-1)}$ such that $U^*_BBU_B=T_B$ is upper triangular.
We construct unitary matrix
$\widehat U=
\bigl[\begin{smallmatrix}
        1     &    \zeros_{n-1}^\top       \\
  \zeros_{n-1}&        U_B                 \\
\end{smallmatrix} \bigr]$ and we obtain:
\begin{align}
\widehat{U}^* \overline{U}^* A \overline U \widehat U
&=
\begin{bmatrix}
        1     &    \zeros_{n-1}^\top       \\
  \zeros_{n-1}&        U^*_B               \\
\end{bmatrix}
\begin{bmatrix}
\lambda          &   \b^\top   \\
\zeros_{n-1}     &     B       \\
\end{bmatrix}
\begin{bmatrix}
        1     &    \zeros_{n-1}^\top       \\
  \zeros_{n-1}&        U_B               \\
\end{bmatrix}
=
\begin{bmatrix}
        1     &    \zeros_{n-1}^\top       \\
  \zeros_{n-1}&        U^*_B               \\
\end{bmatrix}
\begin{bmatrix}
   \lambda    &    \b^\top U_B          \\
  \zeros_{n-1}&       BU_B               \\
\end{bmatrix}
\nonumber \\
&=
\begin{bmatrix}
   \lambda    &    \b^\top U_B          \\
  \zeros_{n-1}&  U^*_B B U_B            \\
\end{bmatrix}
=
\begin{bmatrix}
   \lambda    &    \b^\top U_B          \\
  \zeros_{n-1}&          T_B            \\
\end{bmatrix}
=
T
\end{align}
Since $T_B$ is upper triangular by the induction hypothesis, $T$ is also
upper triangular. Noting $U = \overline{U}\widehat{U}$, we obtain \eqref{schur}.
It is not hard to check that $U$ is unitary: 
$U^*U
=\widehat{U}^* \overline{U}^*\overline U \widehat U
=\widehat{U}^* \widehat U
= I$.

We still need to prove that the diagonal elements of $T$ are the eigenvalues of
$A$. Applying Proposition \ref{simMatSameEigenVals}, similar matrices $A$ and $T$ have the same
characteristic polynomial. But the characteristic polynomial of upper triangular
matrix $T$ is $(x-t_{11})(x-t_{22})\dots (x-t_{nn})$. The diagonal elements of
$T$ coincide thus with the roots of the characteristic polynomial of $A$ and
$T$.
\end{proof}

$U$ does not necessarily contain the eigenvectors of $A$ as columns, even if the construction 
 starts from an eigenvector of $A$. Consider, for instance, the matrix 
$\left[\begin{smallmatrix}
1 & -1 \\
1 & -1 
\end{smallmatrix}\right]$. 
The characteristic polynomial is $x^2$ so that root $\lambda=0$ has algebraic multiplicity 
2. However, the matrix has only one eigenvector 
$\u=\left[\begin{smallmatrix}
\frac 1{\sqrt{2}}\\
\frac 1{\sqrt{2}}
\end{smallmatrix}\right]$,
because it has rank 1; the geometric multiplicity of eigenvalue $\lambda=0$ is 1. The Schur decomposition constructs
$U$ (equal to $\overline U$ because $U_B=1$ in the proof) by putting $\u$ on the first column and by filling the other columns of $U$ so as to make
 $U$ orthonormal. We obtain
$U=\left[\begin{smallmatrix}
\frac 1{\sqrt2} & \frac 1{\sqrt2} \\
\frac 1{\sqrt2} & \frac {-1}{\sqrt2}
\end{smallmatrix}\right]$ and $U^*=U^\top=U$.
The decomposition is
$\left[\begin{smallmatrix}
\frac 1{\sqrt2} & \frac 1{\sqrt2} \\
\frac 1{\sqrt2} & \frac {-1}{\sqrt2}
\end{smallmatrix}\right]
\begin{bmatrix}
1 & -1 \\
1 & -1 
\end{bmatrix}
\left[\begin{smallmatrix}
\frac 1{\sqrt2} & \frac 1{\sqrt2} \\
\frac 1{\sqrt2} & \frac {-1}{\sqrt2}
\end{smallmatrix}\right]=
\begin{bmatrix}
0 & 2 \\
0 & 0 
\end{bmatrix}$.
$U$ does not contain an eigenvector on the second column.


\begin{proposition}\label{firstsym} If $A\in \R^{n\times n}$ is a symmetric matrix, the Schur
decomposition computes a diagonalization of $A$. There exists an unitary (orthonormal) matrix $U\in
\R^{n\times n}$ such that
\begin{equation}\label{schurDiag} T = U^*AU = U^\top A U\end{equation}
is a diagonal matrix with the eigenvalues of $A$ on the diagonal. 
The columns of $U$ are the eigenvectors of $A$ 
which needs to be real (Prop.~\ref{propEigenReal}),
and $U\in\R^{n\times n}$ so
that $U^*=U^\top$. 
\end{proposition}
\begin{proof}
Theorem \ref{thschur} shows that there is a decomposition \eqref{schurDiag} that
generates an upper triangular matrix $T$.
We apply the conjugate transpose on both sides of \eqref{schurDiag}:
$$ T^* = (U^* A U)^* = U^* A^* (U^*)^* = U^* A U = T$$
The equality $T^*=T$ can only hold if $T$ is a diagonal matrix with real
elements on the diagonal. Recalling that Theorem \ref{thschur} shows that the
diagonal elements of $T$ are the eigenvalues of $A$ (each taken with its
algebraic multiplicity), we can write
$T=\texttt{diag}(\lambda_1, \lambda_2, \dots \lambda_n)$. Multiplying
\eqref{schurDiag} by $U$ at left (and using that $U$ is unitary, \ie, 
$UU^*=U^*U=I_n$, as stated by Theorem~\ref{thschur}), we obtain:
$ U \texttt{diag}(\lambda_1, \lambda_2, \dots \lambda_n) = A U.$
The column $i$ (with $i\in[1..n]$) on both sides can be written:
$\lambda_i u_{\times,i} = A u_{\times,i}$.
This shows that column $i$ of $U$ is an eigenvector associated to $\lambda_i$.
Using Proposition~\ref{propEigenReal}, this column $u_{\times,i}$ contains only
real elements, and so, $U$ is real.
\end{proof}

\subsection{\label{secQR}The QR decomposition of real matrices}
\begin{minipage}{0.65\linewidth}
\begin{proposition}\label{propQR} Any matrix $A\in\R^{n\times n}$ can be decomposed as
$$A=QR,$$
\end{proposition}
\end{minipage}
\scalebox{0.75}{
\begin{minipage}{0.48\linewidth}
\vspace{-0.5em}
$$n\left\{
\begin{bmatrix}
~ &~ &~ & ~ & ~ \\
~ &~ &~ & ~ & ~ \\

\multicolumn{5}{c}{\smash{\raisebox{-0.1\normalbaselineskip}{\scalebox{1.8}{$A$}}}}\\
~ &~ &~ & ~ & ~ \\
~ &~ &~ & ~ & ~ \\
\end{bmatrix}
\right.
=
\underbrace{
\begin{bmatrix}
~ & ~ & ~ \\
~ & ~ & ~ \\
\multicolumn{3}{c}{\smash{\raisebox{-0.1\normalbaselineskip}{\scalebox{1.8}{$Q$}}}}\\
~ & ~ & ~ \\
~ & ~ & ~ \\
\end{bmatrix}}_{p}
\begin{bmatrix}
{\tikz\coordinate(lambda1);} ~ &~ &~ & ~ & ~ \\
\multicolumn{5}{c}{\smash{\raisebox{-0.1\normalbaselineskip}{\scalebox{1.8}{$R$}}}}\\
{\tikz\coordinate(lambda0);} ~
&\hspace{-0.8em}\raisebox{0.3\normalbaselineskip}{\scalebox{1.3}{$\zeros$}}
&
\hspace{0.4em}
{\tikz\coordinate(lambda2);} 
\hspace{0.4em}
& ~ & ~ \\
\end{bmatrix}
$$

\begin{tikzpicture}[overlay]
\draw[-,darkgray] (lambda0.south west) to (lambda1.south west);
\draw[-,darkgray] (lambda0.south west) to (lambda2.south west);
\draw[-,darkgray] (lambda1.south west) to (lambda2.south west);
\end{tikzpicture}
\vspace{-4em}
\end{minipage}
}
{\it 
\noindent where $Q\in\R^{n\times p}$ is an orthonormal matrix  (its columns are 
orthogonal\\
\noindent unit vectors meaning
$Q^\top Q =I_p$) and $R\in\R^{p\times n}$ is upper triangular.}

\begin{proof} 
In preamble, let us first (try to) capture the ``spirit'' of the factorization
we want to prove. $Q$ has to contain
a number of $p$ unit vectors that actually represent 
an orthonormal basis of the space spanned by them. Then, $A=QR$ means that each column $\a_k$
of $A$ can be written as a linear combination of the first $k$ elements of this basis (the
first $k$ columns of $Q$). It is easy to obtain this linear combination when $\a_k$ belongs to the space spanned
by these first $k$ elements of the basis. We need to construct an increasingly larger 
orthonormal basis that first covers only $\a_1$, then $\a_1$ and $\a_2$, then
$\a_1,~\a_2$ and $\a_3$, etc. Geometrically, we can first simply take the unit vector
$\e_1$ collinear with $\a_1$; at step 2, we take a vector $\e_2$ perpendicular to $\a_1$ or $\e_1$ such that
$\e_1$ and $\e_2$ determine the same 2D hyperplane as $\a_1$ and $\a_2$; at the third step
we take a vector $\e_3$ perpendicular on this 2D hyperplane so as to determine the same 3D (sub-)space as
$\a_1,~\a_2,~\a_3$. This is the goal of the Gram-Schmidt process presented next.

We now formally present the Gram-Schmidt process on the columns of $A=[\a_1,~\a_2,\dots \a_n]$.
This process takes the column vectors $\a_1,~\a_2,\dots \a_n$ and generates an orthogonal set of vectors that spans 
the same subspace as $\a_1,~\a_2,\dots \a_n$.
Let us define the \textit{normalized vector} $\e\in\R^n$ of $\u\in\R_n$ as $\e=\frac{\u}{|\u|}$ if $\u\neq
\zeros$, or $\e=\zeros$ if $\u=\zeros$ (degenerate normalized vector). Notice a non-degenerate 
normalized vector $\e$ has norm 1 because $\e\sprod \e = \frac {\u\sprod \u}{|\u|^2}=1$.
We define the projection operator $\texttt{proj}$ by setting:
\begin{equation}\label{eqProjection}
\texttt{proj}_\e \a=(\e\sprod \a)\e=\texttt{proj}_\u \a = \left(\frac{\u}{|\u|}\sprod
\a\right)\frac{\u}{|\u|}= \frac{\u\sprod \a}{\u\sprod \u}\u,
\end{equation}
where $\e$ is the normalized vector of $\u$. If $\u=\e=0$, we define $\texttt{proj}_\zeros\a=\zeros$.
This represents the projection of $\a$ on $\e$ or $\u$. The Gram-Schmidt process constructs the following
sequence:
\begin{align*}
\u_1 &= \a_1 \\
\u_2 &= \a_2 - \texttt{proj}_{\u_1}\a_2\\
\u_3 &= \a_3 - \texttt{proj}_{\u_1}\a_3-\texttt{proj}_{\u_2}\a_3\\
\vdots&\\
\u_n &= \a_n - \sum_{i=1}^{n-1} \texttt{proj}_{\u_i}\a_n
\end{align*}

Geometrically, this construction work as follows: to compute $\u_k$, 
it projects $\a_k$ orthogonally onto the subspace $U$ generated by $\u_1,~\u_2,\dots \u_{k-1}$, which is the same as the 
subspace generated by $\a_1,~\a_2,\dots \a_{k-1}$. The vector $\u_k$ is then defined to be the
difference between $\a_k$ and this projection, guaranteed to be orthogonal to all of the vectors in
the subspace $U$.

We now show formally that $\u_k$ is orthogonal to all $\u_1,~\u_2,\dots \u_{k-1}$.
Assume by induction $\u_{k'}\sprod \u_j=0~~\forall j,k'\in[1..n],~j< k'\leq k-1$.
To show $\u_k\sprod \u_j=0~\forall j<k$, we first observe $\u_k\sprod \u_j=0$ if $\u_j=\zeros$. 
If $\u_j\neq\zeros$, 
we calculate:
\begin{align*}
\u_k\sprod \u_j&= \left(\a_k - \sum_{i=1}^{k-1} \texttt{proj}_{\u_i}\a_k\right)\sprod \u_j\\
               &=
                \a_k\sprod \u_j 
                -
                \sum_{\substack{i\in[1..k-1]\\ \u_i\neq \zeros}} 
                    \frac{\u_i\sprod \a_k}{\u_i\sprod \u_i}\u_i\sprod \u_j
                    \tag{we developped $\texttt{proj}_{\u_i}\a_k$ using \eqref{eqProjection} for $\u_i\neq\zeros$}\\
               &=\a_k\sprod \u_j
                -
                 \frac{\u_j\sprod \a_k}{\u_j\sprod \u_j}\u_j\sprod \u_j 
                    \tag{we used $\u_i\sprod \u_j=0$ for $i\neq j$ and $i,j\in[1..k-1]$}\\
               &=\a_k\sprod \u_j-\u_j\sprod \a_k=0
\end{align*}
Notice we can have $\u_k=\zeros$ for certain $k\in[1..n]$. However, we can also 
calculate 
$$\u_k\sprod \u_k =\left(\a_k - \sum_{i=1}^{k-1} \texttt{proj}_{\u_i}\a_k\right)\sprod
\u_k=\a_k\sprod \u_k -                 
\sum_{\substack{i\in[1..k-1]\\ \u_i\neq \zeros}} 
                    \frac{\u_i\sprod \a_k}{\u_i\sprod \u_i}\u_i\sprod \u_k=\a_k\sprod \u_k,$$
meaning that $\u_k=\texttt{proj}_{\u_k}\u_k=\texttt{proj}_{\u_k}\a_k$. 
This allows us to write the equation at step
$k$ of above Gram-Schmidt process as:
\begin{align*}
\a_k &= \u_k + \sum_{i=1}^{k-1} \texttt{proj}_{\u_i}\a_k = \texttt{proj}_{\u_k}a_k + \sum_{i=1}^{k-1} \texttt{proj}_{\u_i}\a_k \\
       &=\sum_{i=1}^{k} \texttt{proj}_{\u_i}\a_k                 
       =\sum_{i=1}^{k} \texttt{proj}_{\e_i}\a_k            \tag{see \eqref{eqProjection}}     \\
       &=\sum_{i=1}^{k} (\e_i \sprod \a_k)\e_i,
\end{align*}                
where $\e_i$ is the normalized vector of $\u_i$. Since this holds for any $k\in[1..n]$, we can write
it in matrix form:
$$
A=[\a_1~\a_2\dots \a_n]=[\e_1~\e_2\dots \e_n]
\begin{bmatrix}
\e_1\sprod \a_1  &  \e_1\sprod \a_2  &  \e_1\sprod \a_3 & \ldots    & \e_1\sprod \a_n  \\
0                &  \e_2\sprod \a_2  &  \e_2\sprod \a_3 & \ldots    & \e_2\sprod \a_n  \\
0                &0                  &  \e_3\sprod \a_3 & \ldots    & \e_3\sprod \a_n  \\
\vdots           &\vdots             &  \vdots          & \ddots    & \vdots           \\
0                &0                  &   0              & \ldots    & \e_n\sprod \a_n
\end{bmatrix}
=\overline{Q}~\overline{R}
$$
Notice that some columns of $\overline{Q}=[\e_1~\e_2\dots \e_n]$ can be zero. We can transform the
$n\times n$ matrix $\overline{Q}$ into a $n\times p$ matrix $Q$ by removing $n-p$ zero columns. At the
same time, we need to remove the corresponding rows of $\overline{R}$ and we obtain matrix
$R\in\R^{p\times n}$ that remains upper triangular, \ie, the elements below the diagonal are zero,
leading to $A=QR$ as needed.
\end{proof}
\subsection{\label{appSquareRoot}An SDP matrix has a unique SDP square root factor}

We introduced the square root decompositions in Section~\ref{secsquareroot}. 
Let's examine the square root matrices $K\in\R^{n\times n}$ such that
$KK=A\succeq \zeros$.
Using the eigendecomposition
\eqref{firstEqDecomp},
we have $KK=A=U\Lambda U^\top$
where $\Lambda$ is diagonal and
$U$ satisfies $UU^\top =I_n$.
We can  thus
write $KK= KUU^\top K =U\Lambda U^\top \implies
U^\top K U U^\top K U = \Lambda$.
Using notation $D=U^\top K U$, we obtain
$DD=\Lambda$.
In other words, $K$ need to have the form
$K=UDU^\top$ for some $D\in\R^{n\times n}$ such that $DD=\Lambda$.

\begin{proposition} Given SDP matrix $A\in\R^{n\times n}$, there exists a unique matrix
$K\in\R^{n\times n}$ such that $KK=A$ and $K\succeq \zeros$. This $K$ is called the principal square root of $A$.
\end{proposition}
\begin{proof}
As described in above paragraph, any $K$ such that $KK=A$ satisfies $K=UDU^\top$; the columns of $U$ are the unitary orthonormal 
eigenvectors of $A$, so that $UU^\top=I$. We notice $\x^\top D \x= \x^\top U^\top U D U^\top U \x
= (U\x)^\top K (U\x)\geq 0~\forall \x\in\R^n$, and so, $D$ is also SDP. To prove the unicity of SDP matrix $K$, it is enough 
to show there is a unique SDP matrix $D\in\R^{n\times n}$ such that
$DD=\texttt{diag}(\lambda_1,\lambda_2,\dots,\lambda_n)$, where $\lambda_1,\lambda_2,\dots,\lambda_n$
are the eigenvalues of $A$. We apply the eigendecomposition on SDP matrix $D$: we obtain
$D=VEV^\top$, where $E=\texttt{diag}(e_1,e_2,\dots e_n)\geq \zeros$ and $VV^\top=V^\top V =I_n$.

Using $DD=\Lambda$, we have $VEV^\top VEV^\top=\Lambda$, and so, $VE^2V^\top=\Lambda$, or
$VE^2=\Lambda V$. Taking any $i,j\in [1..n]$, we have $(VE^2)_{ij}=(\Lambda V)_{ij}$, equivalent
to $V_{ij}e_j^2=\lambda_i V_{ij}$. We observe the following property: 
\begin{equation}
e_j^2\neq \lambda_i \implies V_{ij}=0 \tag{*}
\end{equation}

To prove that $K=UDU^\top$ is unique, it is enough to show that $D$ is unique.
We will exactly determine the value of $D$ by showing
$D=VEV^\top=\sqrt{\Lambda}$, where
$\sqrt{\Lambda}=\texttt{diag}\left(\sqrt{\lambda_1},~\sqrt{\lambda_2},\dots,
\sqrt{\lambda_n}\right)$. For this, it is enough
to prove $VE=\sqrt{\Lambda}V$, which is somehow a consequence of 
above $VE^2=\Lambda V$. More exactly, 
 we need to show that we find the same value at position $(i,j)$
of both sides of $VE=\sqrt{\Lambda}V$ 
for any $i,j\in[1..n]$. For this, we have to show $V_{ij}e_j = \sqrt{\lambda_i}V_{ij}$.
If $e_j=\sqrt{\lambda_i}$, this is clearly true. If $e_j\neq \sqrt{\lambda_i}$, we 
have $e^2_j\neq \lambda_i$ (recall both $e_j$ and $\lambda_i$ are non-negative eigenvalues of SDP
matrices), and so, $(*)$ states that $V_{ij}=0$, showing $V_{ij}e_j = \sqrt{\lambda_i}V_{ij}=0$.
\end{proof}

\section{\label{appRelated}Useful related facts}
We provide two classical results, a proposition related to the
completely positive cone, finishing with an example of a convex function
with an asymmetric non-SDP Hessian.

\subsection{Optimality conditions for linearly-constrained quadratic programs}
\begin{proposition} Consider the following linearly-constrained quadratic
 optimization problem, based on (not necessarily SDP) symmetric matrix $Q\in\R^{n\times n}$,
full-rank matrix $A\in \R^{p\times n}$ with $p\leq n$ and
$\b\in\R^p$.
\begin{subequations}
\begin{alignat}{4}[left ={(QP_\segal)  \empheqlbrace}]
\min~~        &p(\x)=\x^\top Q\x+\c^\top \x    \label{eqqegal1}                &&     \\
s.t.~         &A\x=\b                    \label{eqqegal2}                      &&    \\
              &\x\in \R^n                \label{eqqegal3}                                                                                            
\end{alignat}
\end{subequations}
The solution $\x^*$ is optimal if and only if the following conditions are
satisfied:
\begin{subequations}
\begin{align}
   &2Q\x^*+\c = A^\top \mumu  \text{ for some } \mumu\in \R^p \label{eqFirstOrder}  \\  
   &\z^\top Q \z\geq 0~\forall \z\in\texttt{null}(A)=\left\{\z\in \R^n:~A\z=\zeros\right\}
                                                              \label{eqSecOrder}
\end{align}
\end{subequations}
\end{proposition}
\noindent It can be a useful exercise to give three proofs using different techniques.

\vspace{1.5em}
\noindent \textit{Proof 1.} We solve by force the system $A\x=\b$. Since $A$ has
full rank $p$, the null space $\texttt{null}(A)$ of $A$ has dimension $n-p$ by
virtue of the  rank-nullity Theorem~\ref{thranknullity}. Let $\z_1,~\z_2,\dots
\z_{n-p}$ be a basis of the null space. Any solution $\x$ of the system
considered above has the form $\x=\v+\sum_{i=1}^{n-p} \z_iy_i$ where $A\v=\b$.
Constructing matrix $Z=[\z_1~\z_2~\dots \z_{n-p}]$, we can write the solutions
of the system under the form $\v+Z\y$ with $\y\in\R^{n-p}$.

We can write $(QP_\segal)$ from \eqref{eqqegal1}-\eqref{eqqegal3} 
under the form:
\begin{subequations}
\begin{alignat}{2}[left ={(QP_\segal)  \empheqlbrace}]
\min~~        &(\v+Z\y)^\top Q(\v+Z\y)+\c^\top (\v+Z\y)    \label{eqquncon1}                  \\
s.t.~         &\y\in\R^{n-p}                              \label{eqquncon2}                  
\end{alignat}
\end{subequations}
 We re-write above \eqref{eqquncon1} as:
$$\y^\top Z^\top Q Z \y +
2\v^\top Q Z\y
+\v^\top Q \v
+\c^\top \v
+\c^\top Z\y,
$$
where we used $\v^\top Q Z\y=(\v^\top Q Z\y)^\top = \y Z^\top Q \v$.
Using Prop.~\ref{propConvexReachesOpt}, this unconstrained quadratic program is
bounded from below if and only if it is convex and the gradient vanishes at the
optimal solution $\y^*$. The necessary and sufficient conditions for the
optimality of $\x^*=\v+\y^*$ are:
\begin{itemize}
    \item[(a)] $Z^\top Q Z\succeq \zeros$, which is equivalent to the fact that 
    $Q$ is positive over $\texttt{null}(A)$, \ie, we obtain the second
    order condition \eqref{eqSecOrder}.
    \item[(b)] \begin{sloppypar} The optimal solution $\y^*$ needs to cancel the 
    (column vector) gradient $2 Z^\top Q Z\y^* + 2Z^\top Q \v + Z^\top \c=
             Z^\top\left(2Q(\v + Z\y^*)+\c\right)=
             Z^\top\left(2Q\x^*+\c\right)$. Since the rows of $Z^\top$ are
    a transposed basis for $\texttt{null}(A)$, the above gradient can only
    cancel if $2Q\x^*+\c$ belongs to the transposed row image of $A$, \ie, there
    is some $\mumu\in \R^p$ such that $2Q\x^*+\c=A^\top \mumu$, which is 
    exactly the first order condition \eqref{eqFirstOrder}.
    \end{sloppypar}
\end{itemize}
Finally, there is a degenerate case $p=n$ in which above $Z$ and $\y$ have
dimension 0 and they disappear. In this case, \eqref{eqFirstOrder} holds because
$A^\top \mumu$ with $\mumu\in\R^n$ can cover the whole space $\R^n$ since
 $A$ is a square full rank matrix. The second order condition
\eqref{eqSecOrder} holds because it reduces to nothing using
$\texttt{null}(A)=\{\zeros\}$. The system has thus only one feasible solution
$\x^*=A^{-1}\b$
that satisfies both conditions above.\qed

\vspace{1.5em}
\noindent \textit{Proof 2.}\\
\noindent $\Longrightarrow$ The first order condition \eqref{eqFirstOrder} follows by
applying the method of Lagrange multipliers, obtaining a particular case 
of the KKT conditions.%
\footnote{See also my document ``Trying to demystify the Karush-Kuhn-Tucker
conditions'', available on-line at
\url{http://cedric.cnam.fr/~porumbed/papers/kkt.pdf}.}
However, in our case the argument reduces to the following.
The gradient $\nnabla p(\x^*)$ is perpendicular in $\x^*$ to a surface (level
set) on which
$p$ takes the constant value $p(\x^*)$. Why does not $p$ increase or decrease by moving backward or forward
from $\x^*$ along some $\z\in\texttt{null}(A)$ ? Because the gradient in $\x^*$ 
is perpendicular to $\z$. Indeed, using the chain rule, the derivative in $t=0$ 
of $f(t)=p(\x^*+t\z)$ is equal to $f'(0)=\nnabla p(\x^*)\sprod \z=0$. If this were not zero,
one could move backward or forward from $\x^*$ along $\z$ by some $\varepsilon>0$ to decrease
$p$. We obtain that the gradient $\nnabla p(\x^*)$ is perpendicular to
$\texttt{null}(A)$, and so, it belongs to the transposed row image of $A$, 
\ie, $\nnabla p(\x^*)=A^\top \mumu$ for some $\mumu\in \R^p$, which 
is exactly \eqref{eqFirstOrder}.

We prove $\z^\top Q \z\geq 0\forall \z\in\texttt{null}(A)$ by assuming
the opposite: $\exists \z\in \texttt{null}(A)$ such that $\z^\top Q
\z=-\varepsilon<0$. All $\x+t\z$ are feasible since they satisfy $A(\x+t\z)=A\x=\b$.
The function $f(t)=p(\x+t\z)$ has degree 2 and the quadratic factor is
$\z^\top Q \z t^2=-\varepsilon t^2$. This is a concave function that goes to
$-\infty$ in both directions. This means $\x^*$ is not minimal. We obtained a
contradiction from $\z^\top Q \z=-\varepsilon<0$. The second order condition
\eqref{eqSecOrder} needs to hold.

\noindent $\Longleftarrow$ We suppose both conditions
\eqref{eqFirstOrder}-\eqref{eqSecOrder} are satisfied by some 
$\x^*$ such that $A\x^*=\b$. We will prove that 
$p(\x)\geq p(\x^*)$ for any feasible $\x=\x^*+\z$, with $\z\in\texttt{null}(A)$.
Consider the function $f(t)=p(\x^*+t\z)$. Using the chain rule and
\eqref{eqFirstOrder}, we
obtain $f'(0)=\nnabla p(\x^*)\sprod \z=\nnabla p(\x^*)^\top \z$.
Replacing the gradient with the right hand side of \eqref{eqFirstOrder},
this is further equal to $\mumu^\top A \z$ which is equal to $0$
because $\z\in\texttt{null}(A)$. We thus obtained $f'(0)=0$.
We will show $f$ is convex. The only quadratic factor in $t$ of 
$p(\x^*+t\z)$ is $\z^\top Q\z t^2$ and its second derivative 
is $2\z^\top Q \z$ which is non-negative by virtue of the second order condition
\eqref{eqSecOrder}. This means that $f$ is convex and reaches its 
minimum at $t=0$, \ie, $p(\x^*)\leq p(\x)$ for all feasible $\x$.
\qed

\vspace{1.5em}
\noindent \textit{Proof 3.}\\
\noindent $\Longrightarrow$ We show the second order condition
\eqref{eqSecOrder} using the second proof above. We can prove the first order
condition \eqref{eqFirstOrder} using convexification results. 
We first apply Prop.~\ref{propConvexSDP} that states 
there exists $\lambdalambda\in \R^{n\times p}$ such that 
$p_{\lambdalambda}(\x)=p(\x)+\sum 
\lambda_{ji}(x_jA_i\x-x_j\b_i)$ is convex, where $A_i$ is the row $i$
of $A$. Notice $p_{\lambdalambda}(\x)=p(\x)$ whenever $A\x=\b$.
We can write the following program
\begin{equation}\label{eqplambda}
p(\x^*)=\min_{\x} \left\{p_{\lambdalambda}(\x):~{A\x=\b}\right\}
\end{equation}
as an SDP program because it is convex. Using the results from
Section~\ref{secTotLagrWithEqualities}, this SDP program 
takes the form of $\left(PL^\sX(QP_\segal)\right)$
from 
\eqref{eqPartLagrObj3}-\eqref{eqPartLagrBody3} with $\mumu=\zeros$ and
$\mumumu=\zeros$.
We can write 
$p(\x^*)=OPT\eqref{eqplambda}=\left(PL^\sx(QP_\segal)\right)=
\left(PL^\sX(QP_\segal)\right)$ because the Hessian of
all these programs except the first one is SDP; this means that
the first essential hierarchy \eqref{eqFundLagr1} collapses (like
in Remark~\ref{remExemple2SuccessConvex}).
As such, the total Lagrangian 
of \eqref{eqplambda} reaches $p(\x^*)$. The total Lagrangian for optimal 
Lagrangian multipliers $\betabeta\in\R^p$ is
\begin{equation}\label{eqTotalLagr}
\min_{\x\in\R^n} p_{\lambdalambda}(\x)+\betabeta^\top A\x - \betabeta^\top \b.
\end{equation}
Evaluating this total Lagrangian in $\x^*$, we obtain the value $p(\x^*)$.
Since the total Lagrangian reaches $p(\x^*)$, the minimum of \eqref{eqTotalLagr}
needs to be $p(\x^*)$. Using Prop.~\ref{propConvexReachesOpt}, this can only be
the case if the gradient in $\x^*$ of \eqref{eqTotalLagr} is zero.

Let us calculate the gradient of $p(\x)+\sum 
\lambda_{ji}(x_jA_i\x-x_j\b_i)+\betabeta^\top A\x$
in $\x^*$.
The gradient of each term $x_jA_i\x-x_j\b_i$ is 
computed as follows. For any $k\neq j$, the partial derivative on $x_k$
is $x_jA_{ik}$. The partial derivative on $x_j$ is 
$2A_{ij}x_j+\sum_{k\neq j} A_{ik}x_k-b_i=
x_jA_{ij} + A_i\x-b_i$. The second term vanishes in feasible $\x^*$. 
We obtain that the gradient of $x_jA_i\x-x_j\b_i$ in
$\x^*$ is $x^*_j A_i^\top$.
We obtain 
$\nnabla\left(p(\x)+\sum \lambda_{ji}(x_jA_i\x-x_j\b_i) +\betabeta^\top
A\x\right)_{\x^*}=
\nnabla p(\x^*)+\sum\lambda_{ji} x^*_jA_i^\top+ A^\top \betabeta$. The last two 
terms belong to the transposed row image of $A$. Since the gradient needs to be
zero, the first term $\nnabla p(\x^*)$ has to belong to the transposed row image of $A$ as well, 
\ie, $\nnabla p(\x^*)=A^\top \mumu$ for some $\mumu\in \R^p$, which 
is exactly \eqref{eqFirstOrder}.

\noindent $\Longleftarrow$ Using the fact that $Q$ is non-negative over
$\texttt{null}(A)$, we can use a relatively similar (reversed) argument 
as in the above ``$\Longrightarrow$'' proof. First, 
as above, we apply
Prop.~\ref{propConvexSDP} that states 
there exists $\lambdalambda\in \R^{n\times p}$ such that 
$p_{\lambdalambda}(\x)=p(\x)+\sum 
\lambda_{ji}(x_jA_i\x-x_j\b_i)$ is convex.
Notice $p_{\lambdalambda}(\x)=p(\x)$ whenever $A\x=\b$.
The total Lagrangian of 
$\min \left\{p_{\lambdalambda}(\x):~{A\x=\b}\right\}$ 
is $p_{\lambdalambda}(\x)+\betabeta^\top A\x - \betabeta^\top \b$ which is a
convex polynomial for any value of $\betabeta\in\R^p$. 
The value of this total Lagrangian in $\x^*$ is $p(\x^*)$.
It is enough to show
that the gradient of the total Lagrangian in $\x^*$ is zero for an appropriate
$\betabeta\in \R^p$. Using the calculations from the above ``$\Longrightarrow$''
proof, the gradient in $\x^*$ is
$\nnabla p(\x^*)+\sum\lambda_{ji} x^*_jA_i^\top+ A^\top \betabeta=
A^\top \mumu + \sum\lambda_{ji} x^*_jA_i^\top+ A^\top \betabeta$. The first two
terms belong to the transposed row image of $A$, and so, they can be canceled 
by an appropriate $\betabeta\in \R^p$. We obtain that the total Lagrangian
reaches the value $p(\x^*)$, and so, $\x^*$ needs to be a minimizer of $p$ (the
Lagrangian is always less than or equal to the constrained optimum of $p$ or
$p_{\lambdalambda}$).
\qed

\subsection{More insight and detail into the convexifications
from Section~\ref{sec4}}
\subsubsection{\label{appConvex}Constraints that can be used to convexify any matrix non-negative over
$\texttt{null}(A)$}
We consider a full-rank matrix $A\in \R^{p\times n}$ associated to linear constraint
$A\x=\b$.
Based on these constraints, one can generate various redundant quadratic constraints
that are surely satisfied when $A\x=\b$, see Section~\ref{parRedundantConstr} for examples.
We showed in Section \ref{parBestConvexification} that a partial Lagrangian subject
to $A\x=\b$ can reach the optimum value only when the Lagrangian multipliers
construct a matrix $Q$ that is non-negative over the null space 
$\texttt{null}(A)$ of $A$ (see \eqref{eqDefNull} for the null space definition);
if this does not happen, the partial Lagrangian converges to $-\infty$.
We here discuss convexifications that can make $Q$ non-negative over the whole $\R^n$
(\ie, SDP) using the Lagrangian multipliers associated to the redundant quadratic constraints.

The first paragraph
of Section 2.3.~from ``Partial Lagrangian relaxation for General Quadratic Programming''
(see Footnote \codefootnotetri, p.~\pageref{testfootpage3})
states the following result as already known in the literature. 
If $Q$ is strictly positive over $\texttt{null}(A)$, there exists $V\in \R^{n\times n}$ such
that $Q+A^\top V A\succeq\zeros$.
We prove below a generalization of this result.
This proof is not taken from existing work and we think it is original;
it uses the Bolzano-Weierstrass theorem 
\ref{thBolzanoWeierstrass} (any bounded sequence contains a convergent
sub-sequence).

\begin{proposition}\label{propConvexDP} 
Consider any $Q\in\R^{n\times n}$ strictly positive over $\texttt{null(A)}$, 
\ie, $\u^\top Q \u >0~\forall \u\in\texttt{null(A)}-\{\zeros\}$.
If $B\succeq \zeros$ satisfies
\begin{equation}\label{eqBNull}B\sprod \u\u^\top =0 \iff \u\in \texttt{null}(A),\end{equation}
then there exists $\lambda>0$ such that $Q + \lambda B\succeq \zeros$. In other words,
we can convexify $Q$ using
the Lagrangian multiplier of
a (redundant) quadratic constraint with quadratic factor $B$. The matrix $B$ can be for
instance $B=A^\top S A$ for any $S\succ \zeros$, generalizing the redundant constraint
from Example~\ref{exRedund1}.
\end{proposition}

\begin{proof} We define a set of particular interest: 
$\widetilde{X}=\{\x\in\R^n:~|\x|=1,~Q\sprod \x\x^\top<0\}$. We need to make the
elements $\x$ of this set verify $(Q+\lambda B)\sprod \x\x^\top\geq 0$.
Consider the function $f:\widetilde{X}\to \R$ defined by 
$$f(\x)=\frac {B\sprod \x\x^\top}{|Q\sprod \x\x^\top|}$$
Based on \eqref{eqBNull}, $B\succeq\zeros$ is strictly positive
over 
all $\x$ outside the null space of $A$. 
Since $\widetilde{X}\subseteq
\R^n-\texttt{null}(A)$, we easily obtain $f(\x)>0~\x\in\widetilde{X}$.

We will prove by contradiction that 
$\inf f(\x)>0$. Assuming the contrary (\ie, 
$\inf f(\x)=0$), 
there exists a sequence $(\x_i)$ with
$\x_i\in\widetilde{X}$ such that
$\lim\limits_{i\to \infty} f(\x_i)=0$. Using the
Bolzano-Weierstrass Theorem \ref{thBolzanoWeierstrass},
there exists a sub-sequence $\{\x_{n_i}\}$ such that
$\lim\limits_{i\to\infty} \x_{n_i} = \widetilde{\x}$.
We will show we can obtain
a contradiction for each of the following three cases 
that cover $\widetilde{\x} \in \R^n$:
\begin{description}
    \item[(i)]   $\widetilde{\x}\in\widetilde{X}$
    \item[(ii)]  $\widetilde{\x}\in\texttt{null}(A)$
    \item[(iii)] $\widetilde{\x}\in\R^n - \texttt{null}(A)-\widetilde{X}$
\end{description}

For case (i), it is enough to notice that the sequence $f(\x_{n_i})$ can
be arbitrarily close to $f(\widetilde{\x})>0$, which contradicts that $f(\x_i)$
converges to zero. 

For case (ii), we obtain by hypothesis that
$Q\sprod \widetilde{\x}~\widetilde{\x}^\top >0$.
This leads to
$Q\sprod \widetilde{\x_\varepsilon}~\widetilde{\x_\varepsilon}^\top >0$ for any 
$\widetilde{\x_\varepsilon}$ such that $|\widetilde{\x_\varepsilon}-\widetilde{\x}|<\varepsilon$ 
for a sufficiently small $\varepsilon$. But this would mean that $f$ is undefined in a sufficiently
small ball around $\widetilde{\x}$, and so, 
we can \textit{not} have
$\lim\limits_{i\to\infty} \x_{n_i} = \widetilde{\x}$ with
$\x_{n_i}\in\widetilde{X}~\forall i\in \N$,
contradiction.

For case (iii), 
we can use \eqref{eqBNull} and $\widetilde{\x}\notin \texttt{null}(A)$
to obtain
$B\sprod \widetilde{\x}\widetilde{\x}^\top=z>0$. When $i\to\infty$, the value
$B\sprod \x_{n_i}\x_{n_i}^\top$ can become arbitrarily close to $z$ and 
$Q\sprod \x_{n_i}\x_{n_i}^\top$ can become arbitrarily close to $0$. 
This latter fact ($\lim_{i\to \infty}Q\sprod \x_{n_i}\x_{n_i}^\top=0$)
follows from 
$Q\sprod \x_{n_i}\x_{n_i}<0~\forall i\in \N$ (because
$\x_{n_i}\in \widetilde{X}$)
and 
$Q\sprod \widetilde{\x}\widetilde{\x}^\top\geq 0$ (because
$\widetilde{\x}\notin\widetilde{X}$).
Combining the above convergence properties of
$B\sprod \x_{n_i}\x_{n_i}^\top$ and
$Q\sprod \x_{n_i}\x_{n_i}^\top$, we obtain
that
$f(\x_{n_i})
=
\frac {B\sprod \x_{n_i}\x_{n_i}^\top}
{-Q\sprod \x_{n_i}\x_{n_i}^\top}$ 
can become arbitrarily large, 
contradicting that $f(\x_i)$ 
converges to zero. 

This means there exists a (possibly large) $\lambda>0$ such that 
$f(\x)=\frac {B\sprod \x\x^\top}{|Q\sprod \x\x^\top|}>\frac 1{\lambda}~\forall \x\in \widetilde{X}$.
We obtain $\lambda {B\sprod \x\x^\top}+ Q\sprod \x\x^\top>0$ 
for all $\x$ such that $Q\sprod \x\x^\top<0$ (we developed the $\widetilde{X}$
definition, forgetting that all its elements are unitary).
For the remaining cases, \ie, for all $\x\in\R^n$ such that $Q\sprod \x\x^\top\geq 0$, 
we also obtain
$\left(Q+\lambda B\right)\sprod \x\x^\top\geq 0$ simply because $B\succeq
\zeros$ and $\lambda>0$. Combining both cases above, we get
$\left(Q+\lambda B\right)\sprod \x\x^\top\geq 0~\forall \x\in\R^n$.
\end{proof}

The next proofs are modified versions of the proofs from (Section 2.3 of)
  ``Partial Lagrangian relaxation for General Quadratic Programming''
by Alain Faye and Fr\'ed\'eric~Roupin 
(see Footnote \codefootnotetri, p.~\pageref{testfootpage3}).

\begin{proposition}\label{propConvexSDP} 
We are given a full rank matrix  $A\in\R^{p\times n}$ associated to constraints $A\x=\b$.
Consider any $Q\in\R^{n\times n}$ non-negative over $\texttt{null(A)}$, 
\ie, $\u^\top Q \u \geq 0~\forall \u\in\texttt{null(A)}-\{\zeros\}$.
There exists a linear combination of the redundant 
constraints 
$x_jA_i\x-x_j\b_i=0$ (where $A_i$ is the row $i$ of $A$, with $j\in[1..n]$ and $i\in[1..p]$) from Example~\ref{exRedund2}
that can be added to $Q\sprod \x\x^\top$ to transform $Q$ into an SDP matrix.
Equivalently, if $\A_{j\swarrow i}$ is the $n\times n$ matrix with only a non-zero row $j$ that contains $A_i$,
then there always exist $\lambda_{ji}\in \R$ (for all $j\in[1..n]$ and $i\in [1..p]$) such that 
$Q+\sum_{j,i}\lambda_{ji} \left(\A_{j\swarrow i}^\top +\A_{j\swarrow i}\right)\succeq \zeros$.

\end{proposition}
\begin{proof} We will show in Prop.~\ref{propSuperConvexifier} there exists $\W\in\R^{n\times p}$
such that $Q +A^\top \W^\top +\W A\succeq \zeros$. We can write
$\W=\sum W_{ji}E_{ji}$
where the $E_{ji}$'s represent the canonical base indexed by $j\in[1..n]$ and $i\in[1..p]$, \ie,
$E_{ji}\in\R^{n\times p}$ has a value of one at position $(j,i)$ and
only zeros at all other positions. We have 
\begin{align*}
A^\top \W^\top +\W A &= \sum_{j,i} W_{ji}\left(A^\top E^\top_{ji}+E_{ji}A\right)\\
                     &=\sum_{j,i}  W_{ji}\left(\A_{j\swarrow i}^\top +\A_{j\swarrow i}\right),
\end{align*}
which concludes the proof, with the values $\lambda_{ji}=W_{ji}$.

\end{proof}

\begin{proposition}
\label{propSuperConvexifier}
We are given a full rank matrix  $A\in\R^{p\times n}$ associated to constraints $A\x=\b$.
Consider any $Q\in\R^{n\times n}$ non-negative over $\texttt{null}(A)$, 
\ie, $\u^\top Q \u \geq 0~\forall \u\in\texttt{null(A)}-\{\zeros\}$. There exists $\W\in \R^{n\times
p}$ such that $Q +A^\top \W^\top +\W A\succeq \zeros$.
\end{proposition}
\begin{proof} We recall the QR decomposition from Prop.~\ref{propQR} and the 
Gram-Schmidt orthogonalization process described in the proof of this Prop.~\ref{propQR}.
We apply this process up to the last column $p$ of $A^\top$, so as to factorize
$A^\top = U^\top R$, where $U$ has the size of $A$, $UU^\top =I_p$ and $R\in \R^{p\times p}$ is upper triangular.
One can see this as a $QR$ decomposition restricted to the first $p$ columns, \ie, 
it can be extended to a full $QR$ decomposition by adding $n-p$ zero columns to $A^\top$ and 
$R$. However, we can equivalently factorize $A=R^\top U$.
Notice that the (first $i\leq p$) rows of $U$ span the same subspace as the
(first $i\leq
p$) rows of $A$.
Using
Prop.~\ref{proprankprod}, we have $p=rank(A)\leq rank(R^\top), rank(U)$, and so, 
$U$ and $R^\top$ are full rank and $R^\top$ is invertible. If we find $W\in\R^{n\times p}$ such 
that $S+U^\top W^\top + WU\succeq \zeros$, we can use equation below and finish the proof:
\begin{equation}\label{eqWU} WU=W \left(R^\top\right)^{-1} R^\top U=
\underbrace {W \left(R^\top\right)^{-1}}_{\W} A
\text{ and }
U^\top W^\top = A^\top \W^\top.
\end{equation}

It is enough to show there exists $W\in \R^{n\times p}$ such that $Q +U^\top W^\top +W U\succeq
\zeros$,
where the (first $i\leq p$) rows of $U$ are an orthonormal basis spanning the
same subspace as the (first $i\leq p$) rows of $A$.
Let $B$ be a matrix whose columns are an orthonormal basis of $\texttt{null}(A)=\texttt{null}(U)$. 
Using the rank-nullity Theorem~\ref{thranknullity}, we have $B\in\R^{n, n-p}$.
The fact that $Q$
is non-negative over $\texttt{null(A)}$ is equivalent to $B^\top Q B\succeq
\zeros$; this follows from
\begin{equation}
\x^\top B^\top Q B \x\geq 0~\forall \x\in\R^{n-p}\iff
(B\x)^\top Q (B \x)\geq 0~\forall \x\in\R^{n-p}\iff
\y^\top Q \y \geq 0~\forall \y\in\texttt{null}(A)
\label{eqSdpPartial}
\end{equation}

Let $L_i$ be the sub-space spanned by $\u_{i+1},~\u_{i+2},\dots \u_{p}$ and (the columns of) $B$,
where $\u_i$ is the row $i$ of $U$ written as a column vector. This sub-space has dimension $n-i$
 and is perpendicular on $\u_1,~\u_2,\dots \u_i$ (recall $UB=0$). We can also write:
\begin{equation}\label{eqLi}
L_i= \texttt{img}\left(\u_{i+1},~\u_{i+2},\dots \u_{p},B\right)=
    \left\{\y\in\R^n:~\u_j\y=0~\forall j\in[1..i]\right\},
\end{equation}
where \texttt{img}$(...)$ is the sub-space generated by the column vectors of the
matrices given as arguments. In particular, we have $L_0=\R^n$ and $L_p=\texttt{img}(B)=\texttt{null}(A)$.
We will construct a matrix $Q_i\in\R^{n\times n}$ that is non-negative
on $L_i$, by induction on $i$ from
$i=p$ (with $Q_p=Q$) down to $i=0$ (when $i=0$, $Q_0$ non-negative over $\R^n$ is equivalent to
$Q_0\succeq\zeros$).
Using an argument as the one from \eqref{eqSdpPartial}, a matrix $Q_i\in\R^{n\times
n}$ is non-negative over $L_i$ if and only if 
$M^i=\left[\u_{i+1},~\u_{i+2},\dots \u_{p},B\right] ^\top Q_i \left[\u_{i+1},~\u_{i+2},\dots
\u_{p},B\right]\succeq \zeros$. We will prove we can 
use $Q_{i}$ associated to $M^{i}\succeq\zeros$ to
construct $Q_{i-1}$ such that $M^{i-1}\succeq\zeros$.
We will thus iteratively construct $Q_i=Q+\sum\limits_{j=i+1}^p \u_j\w_j^\top +\w_j \u_j^\top$ with $i$ from $p$
down to $0$. At the last iteration we will obtain:
\begin{equation}Q_0= Q+ \sum\limits_{j=1}^p \u_j\w_j^\top +\w_j
\u_j^\top= Q+ U^\top W^\top  +   WU\label{eqQ0}\end{equation} as needed, where $W=[\w_1~\w_2~\w_3\dots \w_p]$.
We will use several times the following property:
\begin{equation}\label{eqPerpend}\u_a^\top \left(\u_j\w_j^\top +\w_j \u_j^\top\right)\u_b=0~\forall \u_a,\u_b\in
\left\{\u_1,~\u_2,\dots ,\u_{j-1},\u_{j+1},~\u_{j+2},\dots ,\u_p,~B\right\},\end{equation}
where one can read $B$ as an enumeration of column vectors (slightly abusing notations).
This simply follows from $\u_a^\top \u_j=\u_j^\top \u_b=0$ with $\u_a$ and $\u_b$ from the above set.

Now we present the induction step. We can assume $Q_i$ is already constructed and we need
to determine 
$Q_{i-1}=Q_i+\u_i\w_i^\top +\w_i \u_i^\top$. More exactly, the goal is to find
some $\w_i\in\R^n$ such that
$M^{i-1}=\left[\u_i,~\u_{i+1},\dots \u_{p},B\right] ^\top Q_{i-1} \left[\u_{i},~\u_{i+1},\dots
\u_{p},B\right]\succeq \zeros$, as argued above.
At each transition $i\to i-1$, we can use the induction hypothesis that 
$Q_i$ is associated to $M^i\succeq\zeros$ (recall this is surely true for $Q_p=Q$ and $M^p=B^\top Q B$ by hypothesis).
Let us develop the formula of $M^{i-1}$ that we will construct to be SDP.
\begin{align}
M^{i-1}&=
        \scalebox{0.92}{$
        \begin{bmatrix}\u_i^\top \\ \u_{i+1}^\top \\ \vdots\\ \u_{p}^\top \\ B^\top \end{bmatrix}
            \left(Q+\sum\limits_{j=i}^p \u_j\w_j^\top +\w_j \u_j^\top\right)
        \left[\u_{i},~\u_{i+1},\dots \u_{p},B\right] 
        $}
        =
        \scalebox{0.92}{$
            \left[
            \begin{array}{c|c}
                    M^{i-1}_{1,1} &  M^{i-1}_{1,[2..n-i+1]} \\
                    \hline
                M^{i-1}_{[2..n-i+1],1} & M^{i-1}_{[2..n-i+1],[2..n-i+1]}
            \end{array}
            \right]$}
\notag
\\
       &=
\left[
\begin{array}{c|c}
       \u_i^\top \left(Q+\sum\limits_{j=i}^p \u_j\w_j^\top +\w_j \u_j^\top\right) \u_i &
         \text{\scalebox{0.98}{$\u_i^\top \left(Q+\sum\limits_{j=i}^p \u_j\w_j^\top
+\w_j \u_j^\top\right) \left[\u_{i+1},\dots \u_{p},B\right]$}}\\
        \hline
        \begin{bmatrix}\u_{i+1}^\top \\ \vdots\\ \u_{p}^\top \\ B^\top \end{bmatrix}
            \left(Q+\sum\limits_{j=i}^p \u_j\w_j^\top +\w_j \u_j^\top\right) \u_i& 
        {\scalebox{1.5}{$M^i$}}\\
\end{array}
\right],
\label{eqBigM}
\end{align}
where we used \eqref{eqPerpend} with $j=i$ and $\u_a,\u_b
\in
\left\{\u_{i+1},~\u_{i+2},\dots ,\u_p,~B\right\}$ to remain with the $M^{i}$ term in the bottom-right cell. Let us develop the 
\textit{first line} of $M^{i-1}$:
\begin{itemize} \leftskip-1em
    \item [--] the first position is 
$M^{i-1}_{1,1}=\u_i^\top \left(Q+\sum\limits_{j=i}^p \u_j\w_j^\top +\w_j \u_j^\top\right) \u_i=
\u_i^\top \left(Q+\u_i\w_i^\top +\w_i \u_i^\top\right) \u_i$ by virtue of \eqref{eqPerpend}. We can
further develop this into  $M^{i-1}_{1,1}=\u_i^\top Q \u_i + \w_i^\top \u_i + \u_i^\top \w_i$.
    \item [--] on the remaining $n-i$ positions, we have:
         \begin{align*}
            M^{i-1}_{1,[2..n-i+1]}
            &=\u_i^\top \left(Q+\sum\limits_{j=i}^p \u_j\w_j^\top +\w_j \u_j^\top\right) \left[\u_{i+1},\dots \u_{p},B\right]\\
            &= 
         \u_i^\top Q\left[\u_{i+1},\dots \u_{p},B\right]
            +
         \u_i^\top \u_i\w_i^\top \left[\u_{i+1},\dots \u_{p},B\right]\tag{we used \eqref{eqPerpend}}\\
            &~~~+
         \u_i^\top \left(\sum\limits_{j=i}^p \w_j \u_j^\top\right) \left[\u_{i+1},\dots \u_{p},B\right]\\
         &=\u_i^\top Q\left[\u_{i+1},\dots \u_{p},B\right]
            +
         \w_i^\top \left[\u_{i+1},\dots \u_{p},B\right]\\
            &~~~+
         \u_i^\top \left[\w_{i+1},\dots \w_{p},\zeros\right]\tag{we used \eqref{eqPerpend}}
         \end{align*}
Using the orthonormality properties of $\u_{i+1},\dots \u_{p}$ and $B$, the last term
can be written 
        \begin{align*}
            \u_i^\top \left[\w_{i+1},\dots \w_{p},\zeros\right]&=
            \u_i^\top \underbrace{\left[\w_{i+1},\dots \w_{p},\zeros\right]
            \left[\u_{i+1},\dots \u_{p},B\right]^\top}_{P_i} \left[\u_{i+1},\dots \u_{p},B\right]\\
            &=\u_i^\top P_i \left[\u_{i+1},\dots \u_{p},B\right],
        \end{align*}
where $P_i\in\R^{n\times n}$ does not depend on the vector $\w_i$ we need to determine. We will need the 
following:
\begin{equation}\label{eqpi}
P_i\u_i=\zeros_{n\times 1} \text{ and } \u_i^\top P_i^\top = \zeros_{1\times n}
\end{equation}
\end{itemize}
Finally, by simplifying above formulas, the first row of $M^{i-1}$ can be written:
$$M^{i-1}_1=
\left[\u_i^\top Q \u_i + 2\u_i^\top \w_i,~
\left(\u_i^\top Q +\w_i^\top + \u_i^\top P_i\right)\left[\u_{i+1},\dots \u_{p},B\right]\right]$$

We need to determine $\w_i$ such that $M^{i-1}\succeq \zeros$, given that we can
rely on the induction hypothesis $M^i\succeq \zeros$.
The simplest way to construct an SDP matrix $M^{i-1}$ is to generate only zeros on the first row $M^{i-1}_1$. First, we would like
$\w_i^\top$ to cancel the terms $\u_i^\top Q + \u_i^\top P_i$ in above $M^{i-1}_1$ formula.
As such, $\w_i$ integrates a first term  $-\left(\u_i^\top Q + \u_i^\top P_i\right)^\top$. 
A second term of
 $\w_i$ is
 $z\u_i$; 
this second term does not change the canceled positions of the first row (see
point (b) below), 
but it
can make $M^{i-1}_{1,1}\geq 0$ for a sufficiently large $z$.
Thus, we set $\w_i=-\left(Q+P_i^\top\right)\u_i + z\u_i$, where the value of $z$ will
be determined at point (a) below. However, this $\w_i$ vector
leads to the following values on the first row of $M^{i-1}$.
\begin{sloppypar}
\begin{itemize}
\item [(a)] 
$M^{i-1}_{1,1}=\u_i^\top Q \u_i + 2\u_i^\top \Big(-\left(Q+P_i^\top\right)\u_i + z\u_i\Big)=
 \u_i^\top Q \u_i -2 \u_i^\top Q \u_i -2\u_i^\top P_i^\top \u_i+ 2\u_i ^\top z\u_i=
-\u_i^\top Q \u_i + 2z\u_i^\top \u_i=
-\u_i^\top Q \u_i + 2z$, where we used \eqref{eqpi} to cancel the $P_i^\top$ term.
We set $z=\frac 12 \left(\u_i^\top Q \u_i\right)$ to make $M^{i-1}_{1,1}=0$,
but larger values can also be chosen.
\item[(b)] $M^{i-1}_{1,[2..n-i+1]}=\left(\u_i^\top Q +\left(-\left(Q+P_i^\top\right)\u_i + z\u_i\right)^\top + \u_i^\top P_i\right)\left[\u_{i+1},\dots \u_{p},B\right]
=
z\u_i^\top \left[\u_{i+1},\dots \u_{p},B\right]={\zeros_{1\times (n-i)}}.$
\end{itemize}
\end{sloppypar}
The resulting first row of $M^{i-1}$ is filled with zeros, and so needs to be the first column by symmetry. 
Recalling \eqref{eqBigM}, we have $M^{i-1}=\left[\begin{smallmatrix} 0 & \zeros \\ \zeros & M^i\end{smallmatrix}\right]$.
Using the induction hypothesis
$M^i\succeq\zeros$, we obtain $M^{i-1}\succeq \zeros$, which finishes the induction step. At the last iteration, we obtain $M^0\succeq \zeros$, which
is enough to guarantee $Q_0\succeq \zeros$ using arguments discussed above. Recalling \eqref{eqQ0} and \eqref{eqWU}, this finishes the proof.
\end{proof}

\subsubsection{Refining the
\texttt{Branch-and-bound}
for 
equality-constrained binary quadratic
programming from Section~\ref{secbbound}\label{appbbound}}

An approach like in Section~\ref{secbbound}
 can be found in the article
``Improving the performance of standard solvers for quadratic 0-1 programs by a tight convex reformulation: The QCR method''
by Alain Billionnet, Sourour Elloumi and Marie-Christine Plateau.%
\footnote{Published in \textit{Discrete Applied Mathematics} in 2009, vol 157 (6), pp. 1185-1197, 
a draft is available at \url{http://cedric.cnam.fr/fichiers/RC1120.pdf}.
\setcounter{testfoot6}{\value{footnote}}\label{testfootpage6}%
}
They use the redundant constraints from Example~\ref{exRedund2} which yield the
optimum value 
$OPT(SDP(QP_\segal))=\left(PL^\sx(QP_\segal)\right)$ as stated in 
Remark~\ref{remarkOptCoefs}.
However, instead of using $\LL_{PL^\sx(QP_\segal)}(\mumu^*,\mumumu^*)$, they work with a restricted
version 
$\LL^{[0-1]}_{PL^\sx(QP_\segal)}(\mumu^*,\mumumu^*)$
of this program in which they also impose $x_i\in[0,1]$, equivalent to $x^2_i\leq x_i~\forall i\in[1..n]$.
However, they show that this program has the same objective value $OPT(SDP(QP_\segal))$
using an argument based on the Slater's interiority condition, stating that the
total dual Lagrangian of
$\LL^{[0-1]}_{PL^\sx(QP_\segal)}(\mumu^*,\mumumu^*)$ has no duality gap---see the
implication $(D1)\to(D2)$. However, we can here use a different argument.

\begin{sloppypar} If we construct 
$\LL^{[0-1]}_{PL^\sx(QP_\segal)}(\mumu^*,\mumumu^*)$,
$\LL^{[0-1]}_{PL^\sX(QP_\segal)}(\mumu^*,\mumumu^*)$ and
$SDP^{[0-1]}(P_\segal)$ 
from resp.~$\LL_{PL^\sx(QP_\segal)}(\mumu^*,\mumumu^*)$,
$\LL_{PL^\sX(QP_\segal)}(\mumu^*,\mumumu^*)$ and
$SDP(P_\segal)$ 
by adding constraints $x_i^2\leq x_i$ (or resp.~$X_{ii}\leq x_i$) $\forall i\in[1..n]$ in
resp.~\eqref{eqPartLagrBody2Bis}, \eqref{eqPartLagrBody3} and \eqref{qppsdp1}-\eqref{qppsdp5}, then
the following hierarchy holds and it naturally collapses:
\end{sloppypar}
\begin{subequations}
\begin{align}
OPT\left(SDP(P_\segal)\right)&= OPT\left(\LL_{PL^\sx(QP_\segal)}(\mumu^*,\mumumu^*)\right) \label{eqZero}\\
&\leq OPT\left(\LL^{[0-1]}_{PL^\sx(QP_\segal)}(\mumu^*,\mumumu^*)\right)\label{eqZerobis}\\
&= OPT\left(\LL^{[0-1]}_{PL^\sX(QP_\segal)}(\mumu^*,\mumumu^*)\right)\label{eqUnu}\\
&\leq OPT\left(SDP^{[0-1]}(P_\segal)\right)                                           \label{eqDoi}\\
&=OPT\left(SDP(P_\segal)\right)                                           \label{eqTrei}
\end{align}
\end{subequations}
\begin{proof} 
The first equality \eqref{eqZero} is taken from \eqref{eqLastConvex}. The inequality \eqref{eqZerobis} follows 
from the fact that the ``$[0-1]$'' version of the partial Lagrangian (a
minimization program) is more constrained.
The equality \eqref{eqUnu} is due to the fact that the quadratic factor 
$Q_{\mumu^*,\mumumu^*}$ of both programs (in variables $X$ or $\x$) is SDP and that $X=\x\x^\top$
respects all constraints of the SDP program in variables $X$. The inequality \eqref{eqDoi}
follows from the fact that $\left(\LL^{[0-1]}_{PL^\sX(QP_\segal)}(\mumu^*,\mumumu^*)\right)$ is a Lagrangian of
$(SDP^{[0-1]}(P_\segal))$. Finally, \eqref{eqTrei} holds because $(SDP(P_\segal))$ already 
integrates binary constraints $X_{ii}=x_{i}~\forall i\in[1..n]$.
\end{proof}

More advanced convexifications can be found in the work of A.~Billionnet,
S.~Elloumi and 
A.~Lambert, \eg, see papers
``Extending the QCR method to general mixed integer programs.''
and
``Exact quadratic convex reformulations of mixed-integer quadratically constrained problems''.%
\footnote{Both published in \textit{Mathematical Programming}, resp~.in
2012 (vol. 131(1), pp. 381-401) and~2016 (vol 158(1), pp 235-266).}
However, for the moment, such methods lie outside the scope of this non-research
document;
further progress is unessential for now.


\subsection{A convex function with an asymmetric Hessian}
We next provide an example of a convex function with an asymmetric Hessian. This
shows that a statement like ``A twice differentiable function is convex if
and only if its Hessian is SDP'' is technically not complete, because a convex
function can have an asymmetric non-SDP Hessian.
This case is omitted from certain textbooks 
(see a reference in the first paragraph of Section~\ref{secSDPHessian})
but we addressed it in our work by requiring the Hessian to
be symmetric in Prop.~\ref{propHessian}.

\begin{example}\label{exNonSymHessian}
The following function $f$ is convex for any $\mu\geq 9$ and has a
non-symmetric Hessian in $\zeros$, \ie, $\nabla^2 f(\zeros)$ is not symmetric.
\begin{equation}\label{eqNonSym}
f(x,y)=
  \begin{cases} 
   \dfrac{x^3y}{x^2+y^2}+\mu x^2+\mu y^2  & \text{if } (x,y)\neq (0,0) \\
   0         &                                      \text{if } (x,y)=0 \\
  \end{cases}
\end{equation}

\end{example}
\begin{proof} The gradient of $f$ at
$(0,0)$ can not be computed algebraically. We can however obtain
$\dfrac{\partial f}{\partial x}(0,0)=
\lim\limits_{\varepsilon\to 0} \dfrac{f(\varepsilon,0)-f(0)}{\varepsilon}
=
\lim\limits_{\varepsilon\to 0} \dfrac{\mu \varepsilon^2}{\varepsilon}
=0
$ and similarly 
$\dfrac{\partial f}{\partial y}(0,0)=0$.
The gradient of $f$ is thus:
\begin{equation*}\label{eqNonSym2}
\nnabla f(x,y)=
  \begin{cases} 
   \left(\dfrac{x^2y\left(x^2+3y^2\right)}{\left(x^2+y^2\right)^2}+2\mu x,
    \dfrac{x^3\left(x^2-y^2\right)}{\left(x^2+y^2\right)^2}      +2\mu y
    \right)
         & \text{if } (x,y)\neq (0,0) \\
   (0,0)         &                                      \text{if } (x,y)=0 \\
  \end{cases}
\end{equation*}

Let us now calculate the Hessian $\nnabla^2 f(\x)$. As above, we apply the
derivative formula to calculate $\nnabla^2 f(0,0)$:
\begin{align*}
\dfrac{\partial^2 f}{\partial x \partial x}\Big(0,0\Big)
   &=
    \lim_{\varepsilon\to 0}
    \dfrac{ 
           \dfrac{\partial f}{\partial x}\Big(\varepsilon,0\Big) 
           -\dfrac{\partial f}{\partial x}\Big(0,0\Big)
          }
          {\varepsilon}
    =
    \lim_{\varepsilon\to 0}
    \dfrac{ 
          2\mu\varepsilon
          }
          {\varepsilon}
    =2\mu\\
\dfrac{\partial^2 f}{\partial y \partial x}\Big(0,0\Big)
   &=
    \lim_{\varepsilon\to 0}
    \dfrac{ 
           \dfrac{\partial f}{\partial x}\Big(0,\varepsilon\Big) 
           -\dfrac{\partial f}{\partial x}\Big(0,0\Big)
          }
          {\varepsilon}
    =
    \lim_{\varepsilon\to 0}
    \dfrac{ 
        0
          }
          {\varepsilon}
    =0\\
\dfrac{\partial^2 f}{\partial x \partial y}\Big(0,0\Big)
   &=
    \lim_{\varepsilon\to 0}
    \dfrac{ 
           \dfrac{\partial f}{\partial y}\Big(\varepsilon,0\Big) 
           -\dfrac{\partial f}{\partial y}\Big(0,0\Big)
          }
          {\varepsilon}
    =
    \lim_{\varepsilon\to 0}
    \dfrac{ 
    ~~\dfrac{\varepsilon^3(\varepsilon^2-0^2)}{(\varepsilon^2+0^2)^2}~~
          }
          {\varepsilon}
    =\lim_{\varepsilon \to 0}
        \dfrac{\varepsilon}{\varepsilon}
    =1\\
\dfrac{\partial^2 f}{\partial y \partial y}\Big(0,0\Big)
   &=
    \lim_{\varepsilon\to 0}
    \dfrac{ 
           \dfrac{\partial f}{\partial y}\Big(0,\varepsilon\Big) 
           -\dfrac{\partial f}{\partial y}\Big(0,0\Big)
          }
          {\varepsilon}
    =
    \lim_{\varepsilon\to 0}
    \dfrac{ 
          2\mu\varepsilon
          }
          {\varepsilon}
    =2\mu\\
\end{align*}

The Hessian of $f$ is thus:

\begin{equation*}\label{eqNonSym3}
\nnabla^2 f(x,y)=
  \begin{cases} 
\begin{bmatrix}
\dfrac{2xy^3\left(3y^2-x^2\right)}{\left(x^2+y^2\right)^3} + 2\mu
&
\dfrac{x^2\left(x^4+6y^2x^2-3y^4\right)}{\left(x^2+y^2\right)^3}
\\
\dfrac{x^2\left(x^4+6y^2x^2-3y^4\right)}{\left(x^2+y^2\right)^3}
&
\dfrac {2x^3y\left(y^2-3x^2\right)}{\left(x^2+y^2\right)^3}+2\mu
\end{bmatrix}
         & \text{if } (x,y)\neq (0,0) \\
\begin{bmatrix}
2\mu & 1 \\
0    & 2\mu
\end{bmatrix}
        &                                      \text{if } (x,y)=0 \\
  \end{cases}
\end{equation*}

All fractions in the Hessian have the form $\dfrac{p_1(x,y)}{p(x,y)}$, where
$p_1$ and $p$ are both 
homogeneous polynomials of degree
6. We easily obtain $\dfrac{p_1(tx,ty)}{p(tx,ty)}=\dfrac{t^6p_1(x,y)}{t^6p(x,y)}
=\dfrac{p_1(x,y)}{p(x,y)}$. The graph of such fraction can be seen as a surface
that has the same height (value) on each ray starting from (but not touching)
the origin. The image of $\dfrac{p_1(x,y)}{p(x,y)}$ over $\R^2-\{0\}$ is equal to the 
image of $\dfrac{p_1(x,y)}{p(x,y)}$ over unit circle $x^2+y^2=1$, which needs to
be bounded. 

We now calculate the bounds of the fractions. First, notice
$\dfrac{p_1(x,y)}{p(x,y)}=p_1(x,y)$ over the unit circle, since
$p(x,y)=(x^2+y^2)^6$. Each monomial of degree 6 of $p_1(x,y)$ belongs
to the interval $[-1,1]$. Using this, we obtain, for instance, that $\dfrac{\partial^2
f}{\partial x \partial x}-2\mu\in 
\big[ 
    \min\left(2xy^3\left(3y^2-x^2\right)\right),
    \max\left(2xy^3\left(3y^2-x^2\right)\right)
\big]
\subset [-8,8]$ when $x^2+y^2=1$. By applying the same approach on all fractions, we
obtain:
$$
-
\begin{bmatrix}
  8&  10\\
10 &  8
\end{bmatrix}
\leq 
\nnabla^2 f(x,y)
-
\begin{bmatrix}
2\mu  &  0\\
0     & 2\mu
\end{bmatrix}
\leq 
\begin{bmatrix}
  8&  10\\
10 &  8
\end{bmatrix}
$$

By taking any $\mu\geq 9$, we obtain 
$\nnabla^2 f(x,y)\succeq \zeros$ for $(x,y)\neq (0,0)$. On the other hand, we 
can \textit{never} state $\nnabla^2(0,0)\succeq \zeros$ because $\nnabla^2(0,0)$
is not symmetric. We thus need to use
$\nnabla^2(0,0)+\nnabla^2(0,0)^\top\succeq \zeros$.

The idea is taken from the last pages of the article 
 ``On second derivatives of convex functions''
by Richard Dudley,%
\footnote{Published in 
\textit{Mathematica Scandinavica} in 1978, vol 41, pp 159--174,
available on-line as of 2017 at
\url{http://www.mscand.dk/article/download/11710/9726},
see also the discussion on the on-line math forum
\url{https://math.stackexchange.com/questions/1181713/convex-function-with-non-symmetric-hessian}.
} 
but we are the first to calculate an explicit minimum value of $\mu$.
The derivatives were calculated at \url{http://www.derivative-calculator.net}.
\end{proof}

\subsection{\label{appHyperSep}The separating hyperplane theorem}
\subsubsection{General theorems and their reduction to a particular case}
\begin{theorem}\label{thHyperGeneral} (Hyperplane separation theorem) Given two
disjoint convex sets
$X,Y\subset\R^n$, there exist a non-zero $\v\in\R^n$ and a real number $c$ such
that 
\begin{equation}\label{eqxyc}
\v\sprod \x \geq c \geq \v\sprod \y,\end{equation}
for any $\x\in X$ and $\y\in Y$.  The hyperplane $\{\u\in\R^n:~\v\sprod \u=c\}$
separates $X$ and $Y$.
\end{theorem}
\begin{proof}
We will show that the general theorem reduces to a simpler theorem version
in which
$Y=\{\zeros\}$, \ie, the hyperplane $\zeros$-separation Theorem
\ref{thHyperZero}.

Consider the set $Z=X-Y=\{\x-\y:~\x\in X,~\y\in Y\}$. The set $Z$ is convex:
take $\z_a=\x_a-\y_a$, $\z_b=\x_b-\y_b$ and any $\alpha\in[0,1]$ and observe $\alpha
\z_a+(1-\alpha)\z_b=\alpha(\x_a-\y_a)+(1-\alpha)(\x_b-\y_b)
                   =\alpha \x_a+(1-\alpha)\x_b-
                    (\alpha \y_a+(1-\alpha)\y_b)$. Since $X$ and $Y$ are convex, 
$\x_{\alpha}=\alpha \x_a+(1-\alpha)\x_b\in X$ and $\y_{\alpha}=\alpha \y_a+(1-\alpha)\y_b\in Y$, and
so, $\z_a+(1-\alpha)\z_b=\x_{\alpha}-\y_{\alpha}\in Z$, \ie, $Z$ is convex.

We apply the hyperplane $\zeros$-separation Theorem \ref{thHyperZero} on $Z$
and $\zeros$ (observe $\zeros\notin Z$ because $X$ and $Y$ are disjoint) and
obtain there is a non-zero $\v\in \R^n$ such that $\v\sprod \z\geq 0~\forall
\z\in Z$. This means that 
\begin{equation}\label{eqxyv}
\v\sprod \x\geq \v\sprod \y,~\forall \x\in X\text{ and }\forall\y\in
Y.\end{equation}
We obtain $\inf_{\x\in X}\v\sprod \x\geq \sup_{\y\in Y}\v\sprod \y$,
because otherwise 
\eqref{eqxyv} would be violated by some
$\x$ and $\y$ such that 
$\v\sprod \x$ is close enough to
$\inf_{\x\in X}\v\sprod \x$ and $\v\sprod \y$ is close enough to
$\sup_{\y\in Y}\v\sprod \y$. Taking $c=\frac{\inf_{\x\in X}\v\sprod
\x+\sup_{\y\in Y}\v\sprod \y}2$, \eqref{eqxyv} can be written in the form \eqref{eqxyc}.
\end{proof}

\begin{theorem}\label{thHyperZero}
(Hyperplane $\zeros$-separation theorem) Given convex set
$X\subset\R^n$ that does not contain $\zeros$, there exist a non-zero
$\v\in\R^n$ such that
$$\v\sprod \x \geq 0,$$ 
for any $\x\in X$. 
\end{theorem}
\noindent We give two proofs. The first one is based on induction and it takes
a bit more than one page. It is a personal proof; I 
doubt it can also be found in classical textbooks. The second
proof takes 2.5 pages and it follows well-established arguments (some of them using convergent sequences) that I could found on the
Internet.\footnote{I used the Wikipedia
article
\url{en.wikipedia.org/wiki/Hyperplane_separation_theorem} and the course
of Peter Norman \url{www.unc.edu/~normanp/890part4.pdf}.}
The first proof essentially relies on
Theorem~\ref{thOpenRay} from Appendix~\ref{sepperso};
the second one 
relies on
Theorem~\ref{thSimpleSep} and Theorem~\ref{thSupportHyper} 
from Appendix \ref{secschoolbook}.\\
~\\
\noindent\textit{Proof 1}
For any $\u\in\R^n$, 
let $f(\u)$ be the largest $t$ such that $t\u\in X$, or $-\infty$ if no $t\u$
belongs to $X$. If $f(\u)>0$ for all non-zero $\u\in \R^n$, then $\zeros\in X$: it is enough
to take $\u$ and $-\u$ and observe that the segment joining $f(\u)\u$ and
$f(-\u)-\u$ contains $\zeros$. Since $X$ does not contain $\zeros$, there exist
some non-zero $\u\in \R^n$ such that $f(\u)\leq 0$. The theorem then follows
from applying Theorem~\ref{thOpenRay}.
\qed\\

\noindent\textit{Proof 2}
We first prove that the closure $\overline{X}$ of $X$ (\ie, the set $X$ along with
all its limit points) is convex. Take any $\x,~\y\in\overline{X}$ and consider
two sequences $\{\x_i\}$ and $\{\y_i\}$ in $X$ that converge to $\x$ and resp.~$\y$.
Such sequences always exist because, by definition, $\overline{X}$ 
is the set of the limit points of all sequences 
of $X$.
Take any $\alpha\in[0,1]$;
it is enough to
show $\alpha \x+(1-\alpha)\y\in\overline{X}$ to prove that $\overline{X}$ is convex.
We have $\z_i=\alpha\x_i+(1-\alpha)\y_i\in X$, because $X$ is convex. We next observe that
$z=\lim\limits_{i\to\infty} z_i\in\overline{X}$, because $\overline{X}$ contains all
limit points of $X$. But $\alpha \x+(1-\alpha)\y=\alpha \lim\limits_{i\to\infty}
\x_i+(1-\alpha)\lim\limits_{i\to\infty} \y_i=\lim\limits_{i\to\infty}
\z_i=\z\in\overline{X}$, and so, 
$\overline{X}$ is convex.

If $\zeros$ does not belong to the closure $\overline{X}$, then the conclusion follows from the Simple Separation
Theorem~\ref{thSimpleSep} applied on $\zeros$ and $\overline{X}$. 
If $\zeros$ belongs to the closure of $X$, the conclusion follows from the
Simple Supporting Hyperplane Theorem~\ref{thSupportHyper} applied on
$\overline{X}$ and $\zeros$ as a boundary point of $\overline{X}$ ($\zeros$ does
not belong to the interior, because it does not belong to $X$).
\qed\\

The following variant can be generally useful, but we do not need it in this
document.
\begin{theorem}\label{thHyperGeneralOpen} (Separation theorem for open set $X$) Given two
disjoint convex sets
$X,Y\subset\R^n$ such that $X$ is open, there exist a non-zero $\v\in\R^n$ and a real number $c$ such
that 
\begin{equation}\label{eqxycopen} \v\sprod \x > c \geq \v\sprod \y, \forall \x\in X,~\y\in Y \end{equation}
The closure $\overline{X}$ of $X$ satisfies
\begin{equation}\label{eqxycopenegal} \v\sprod \overline{\x} \geq c \geq \v\sprod \y, \forall \overline{\x}\in \overline{X},~\y\in Y \end{equation}
\end{theorem}
\begin{proof} Using the standard hyperplane separation Theorem
\ref{thHyperGeneral}, there is a non-zero $\v\in\R^n$ and some $c\in\R$ such that:
\begin{equation}\label{eqxycopen2} \v\sprod \x \geq c \geq \v\sprod \y, \forall \x\in X,~\y\in Y \end{equation}
For the sake of contradiction, assume  there exists some $\x\in X$ such that $\v\sprod \x = c$. Since $X$ is
open, $X$ contains an open ball around $\x$, and so, for a sufficiently small
$\epsilon>0$, we have $\x-\epsilon \v\in X$. But $\v\sprod(\x-\epsilon\v)=c
-\epsilon |\v|^2<c$ which contradicts \eqref{eqxycopen2}. The assumption
$\v\sprod \x=c$ was false, and so, \eqref{eqxycopen2} becomes \eqref{eqxycopen}.

We still have to prove \eqref{eqxycopenegal}. Assume there is some
$\overline{\x}$ in $\overline{X}$ such that $\v\sprod \overline{\x}<c$. Since
$\overline{\x}$ has to be the limit point of some sequence $\{\x_i\}$ with
elements $\x_i\in X~\forall i\in \N^*$, we deduce that 
$\lim\limits_{i\to \infty}\v\sprod \x_i=\v\sprod\overline{\x}$. For any
$\epsilon>0$ there exists some $m\in\N^*$ such that $|\v\sprod
\x_i-\v\sprod\overline{\x}|<\epsilon~\forall i\geq m$. Taking any
$\epsilon<c-\v\sprod \overline{\x}$, we have 
$\v\sprod \x_m<c$, which contradicts \eqref{eqxycopen}. The
assumption $\v\sprod \overline{\x}<c$ was false, which proves
\eqref{eqxycopenegal}.
\end{proof}

\subsubsection{\label{sepperso}Proving the theorem using personal arguments}
\begin{theorem}\label{thOpenRay} (Hyperplane $\zeros$-separation theorem in presence of open
rays) Consider convex set $X\in\R^n$ 
(that may contain $\zeros$ or not)
and let $f(\u)$ be the largest $t$ such
that $t\u\in X$ for any $\u\in\R^n$. If $f(\u)\leq 0$, we say that
ray $\u$ is open, because there is no $\epsilon>0$ such that $\epsilon \u\in
X$. If there is at least an open ray, then there exist some
non-zero $\v\in\R^n$ such that 
$$\v\sprod \x\geq 0,~\forall \x\in X.$$
\end{theorem}

\begin{proof}
We proceed by induction. We first prove it for $n=2$ using the notion of angle.

\begin{lemma} The theorem holds for $n=2$.\end{lemma}
\begin{proof}
Without loss of generality, we consider the open ray $\u\in\R^2$ is unitary.
We associate $\u$ with an angle $\theta=0$. 
Using a slight notation abuse, let $f(\theta)=f(\u_\theta)$, where $\u_\theta$
is an unitary vector of $\R^2$ that makes an angle of $\theta$ with $\u$
measured clockwise. Technically, $\u_\theta$ 
satisfies $\u\sprod\u_\theta=\cos(\theta)$ and it 
is the first unitary vector with this property found by moving
clockwise from $\u$.
Let $\Theta$ be the set of angles
$\theta$ for which $f(\theta)>0$.
$\Theta$ belongs to segment $(0,2\pi)$ because it does not contain $\theta
=0$ and it might be open. However, $\Theta$ needs to have 
an infimum $\inf(\Theta)$ and a supremum $\sup(\Theta)$, 
see also Prop.~\ref{propSupExists}.

Assume for the sake of contradiction that
$\sup(\Theta)-\inf(\Theta)>\pi$. This means there are two angles
$\theta_M,\theta_m\in \Theta$ close enough
to $\sup(\Theta)$ and resp.~$\inf(\Theta)$ so that $\theta_M-\theta_m>\pi$.
By convexity, the segment that joins
$f(\theta_M)\u_{\theta_M}$ and
$f(\theta_m)\u_{\theta_m}$ is included in $X$ and it also
intersects the segment $[\zeros,\u]$ in some point $\epsilon \u$
with $\epsilon>0$. This contradicts the fact that $\u$ is an open ray.


We can now consider $\sup(\Theta)-\inf(\Theta)\leq \pi$. The line
$\left\{t\u^*:~t\in \R,\u^*\in \R^2-\{\zeros\}\right\}$ that goes
through $\zeros$ and makes an angle of $\inf(\Theta)$ with $\u$ (clockwise) has the whole
$X$ on one side; all points in $X$ make an angle in
$[\inf(\Theta),\sup(\Theta)]$ with $\u$ (clockwise). We can take $\v$ one of the two vectors perpendicular to $\u^*$ in $\zeros$
and obtain $\v\sprod \x \geq 0$ for all $\x\in X$.
\end{proof}

Now consider $n>2$. Take a 2-dimensional sub-space $S_2$ that contains the unitary
vector $\u$, \ie, $S_2=\left\{t\u+t'\u':~t,t'\in \R\right\}$ for some unitary $\u'$ such that
$\u\sprod\u'=0$. The intersection of two convex sets is convex, and so,
$X_2=S_2\cap X$ is convex. Using above lemma for $n=2$, there exists some
unitary $\v_2\in S_2$ such that $\v_2\sprod \x_2\geq 0,~\forall \x_2\in X_2$.
Notice
that $f(-\v_2)\leq 0$ because all $-t\v_2$ with $t>0$ do not belong to $X$ or
$X_2$, since $\v_2\sprod (-t\v_2)<0$.

Take any \textit{unitary} $\u_2\in S_2$ such that 
\begin{equation}\label{eqPlane}\v_2\sprod \u_2=0.\end{equation}
Consider the $(n-1)$-dimensional sub-space $S_{n-1}\subsetneq \R^n$ perpendicular
on $\u_2$, \ie, 
$S_{n-1}=\{\x\in X:~\u_2\sprod \x=0\}$.
Observe $\v_2$ and $-\v_2$ belong to $S_{n-1}$, using \eqref{eqPlane}.
We now project the whole space $X$ on $S_{n-1}$, \ie, we obtain the set
$X_{n-1}=\{\x-\u_2(\u_2\sprod \x):~\x\in X\}$. One can easily check
that all elements of $X_{n-1}$ satisfy $\u_2\sprod(\x-\u_2(\u_2\sprod \x))=
\u_2\sprod \x - \u_2\sprod \u_2(\u_2\sprod \x)=\u_2\sprod \x -\u_2\sprod\x=0$.

We now define function $f_{n-1}:S_{n-1}\to \R\cup\{-\infty\}$ in the same style
as $f$, \ie, $f_{n-1}(\s_{n-1})$
is the smallest $t$ for which $t\s_{n-1} \in X_{n-1}$, for any $\s_{n-1}\in
S_{n-1}$. We showed above that $f(-\v_2)\leq 0$. We can
also prove $f_{n-1}(-\v_2)\leq 0$. Recall we have 
$\v_2\sprod \x_2\geq 0,~\forall \x_2\in X_2$. Consider now the projection 
$\x_{n-1}=\x_2-\u_2(\u_2\sprod \x_2)$ and notice that $\v_2\sprod \x_{n-1}=
\v_2\sprod (\x_2-\u_2(\u_2\sprod \x_2))=\v_2\sprod \x_2
-\v_2\sprod\u_2(\u_2\sprod\x_2)=\v_2\sprod \x_2\geq 0$ (we used
\eqref{eqPlane} for the last equality). The elements $-t\v_2$
with $t>0$ can not belong to the projection of $X_2$, because
$\v_2\sprod(-t\v_2)<0$ and the elements $\x_{n-1}$ of the projection
verify $\v_2\sprod\x_{n-1}\geq 0$. Finally, remark we do not lose generality by
restricting the argument to the projections of $X_2$: all elements of $X$ that
could project on $-t\v_2$ could only belong to $X_2$, \ie, the space generated
by $\v_2$ and $\u_2$.

We can easily check that $X_{n-1}$ is convex. Consider $\x_{n-1}\in X_{n-1}$ as the projection of 
$\x_{n-1}+a\u_2 \in X$ and $\y_{n-1}\in X_{n-1}$ as the projection of
$\y_{n-1}+b\u_2\in X$. Using the
convexity of $X$, the following holds for any $\alpha\in [0,1]$: $\alpha
(\x_{n-1}+a\u_2)+(1-\alpha) (\y_{n-1}+b\u_2)\in X$. We can re-write this as:
$\alpha \x_{n-1}+(1-\alpha) \y_{n-1} +(\alpha a +(1-\alpha)b)\u_2\in X$, and so,
$\alpha \x_{n-1}+(1-\alpha) \y_{n-1}\in X_{n-1}$ (one can easily verify that the
scalar product of this with $\u_2$ is zero).

We now apply the induction hypothesis on set $X_{n-1}$ with $f_{n-1}(-\v_2)\leq
0$ in the sub-space $S_{n-1}$ (notice this is a full $(n-1)$-dimensional space
where all elements can be written as a linear combination of a canonical basis
perpendicular to $\u_2$). We obtain there is non-zero $\v_{n-1}\in S_{n-1}$ such that $\v_{n-1}\sprod
\x_{n-1}\geq 0$ for all $\x_{n-1}\in X_{n-1}$. This means that $\v_{n-1}\sprod
(\x_{n-1}+a\u_2)\geq 0$ for any $\x_{n-1}\in X_{n-1}$ and any $a\in \R$, because
$\v_{n-1}$ is perpendicular on $\u_2$. It is
easy to check that $X\subseteq\{\x_{n-1}+a\u_2:~\x_{n-1}\in X_{n-1},~a\in \R\}$.
This is enough to conclude that $\v_{n-1}\x\geq 0$ for any $\x\in X$.
\end{proof}

\subsubsection{\label{secschoolbook}Proving the theorem using well-established textbook arguments}
\begin{theorem}\label{thSimpleSep} (Simple separation theorem)
Given convex closed set $X\subset\R^n$ and some $\y\in \R^n$ such that
$\y\notin X$, there exist a non-zero $\v\in\R^n$ such that 
\begin{equation}\label{eqxycsimple}
\v\sprod \x > \v\sprod \y,~\forall \x\in X.\end{equation}
\end{theorem}
\begin{proof}
We need the following lemma.
\begin{lemma} Given closed convex set $X$, there exist a unique $\x\in X$ such that
$|\x|=\inf\{|\x'|:~\x'\in X\}$, where $|\cdot|$ is the norm (length), \eg, 
$|\x|=\sqrt{\x\sprod \x}$.
\end{lemma}
\begin{proof}
The Wikipedia proof is kind of magical for my taste, with a few tricks arising
rather out of the blue. I provide a more natural and even simpler (without
Cauchy sequences) proof.

Let $\delta=\inf\{|\x'|:~\x'\in X\}$. We first show  $X$ contains
a sequence $\{\x_i\}$ such that $\lim\limits_{i\to \infty} |\x_i|=\delta$. For
instance, we can consider $\epsilon_1=1$, $\epsilon_2=\frac 12$,
$\epsilon_3=\frac 13$,
$\dots$. For each $\epsilon_i$ ($i\in \N^*$), $X$ needs to contain some $\x_i$
such that $|\x_i|< \delta+\epsilon_i$, because otherwise we would have
 $\delta+\epsilon_i\leq \inf \{|\x'|:~\x'\in X\}$, impossible since
$\delta=\inf\{|\x'|:~\x'\in X\}$. The sequence $\{\x_i\}$ constructed this way
satisfies $\lim\limits_{i\to \infty} |\x_i|=\delta$.

This sequence $\{\x_i\}$ of elements of $\R^n$ needs to contain a convergent
subsequence using the Bolzano-Weierstrass Theorem \ref{thBolzanoWeierstrass},
\ie, there exists a sub-sequence $\{\x_{n_i}\}$ such that
$\lim\limits_{i\to\infty} \x_{n_i} = \x$. Since $X$ is closed, it contains all
limit points, and so, $\x\in X$. It is not hard now to check that the sub-sequence
$\{|\x_{n_i}|\}$ converges to $\delta$. Since for any $\epsilon$ there exists
$m\in \N^*$ such that $|\x_j|<\delta+\epsilon~\forall j\geq m$, there must be
some $n_m\geq m$ (the sub-sequence is infinite) such that
$|\x_{n_j}|<\delta+\epsilon~\forall n_j\geq n_m$. This confirms $\delta=
\lim\limits_{i\to\infty} |\x_{n_i}|=|\x|$.

We still need to show that $\x$ is the unique element of minimum norm. Suppose
there exists $\y\in X-\{\x\}$ such that $|\y|=\delta$. By convexity, $\frac
{\x+\y}2\in X$. We can calculate $|\frac {\x+\y}2|^2=\frac{\x\sprod \x +
\y\sprod \y+2\x\sprod \y}4=\frac {\delta^2+\x\sprod \y}2$. We will show
$\x\sprod \y <\delta^2$. For this, it is enough to observe that $0<|\x-\y|^2=
\x\sprod \x + \y\sprod \y -2\x\sprod \y=2\delta^2 -2\x\sprod \y$, \ie,
$\x\sprod \y<\delta^2$. We obtained that  $|\frac {\x+\y}2|^2<\delta^2$,
contradiction. There is no $\y\neq \x$ in $X$ such that $|\y|=\delta$.
\end{proof}

We will first prove the theorem for $\y=\zeros$ and then we will use a simple
translation argument to extend it for an arbitrary $\y$. Using the lemma, let us
take $\v\in X$ of minimum norm $\delta>0$ (because $\y=\zeros\notin X$). We will
prove the following: 
\begin{equation}\label{eqSeparator}
    \v\sprod \x > 0,~\forall \x \in X
\end{equation}

Consider any $\x \in X$ and write $\Delta=\x-\v$. 
Keeping in mind the goal of showing $\v\sprod \Delta\geq 0$,
consider
a function $f:[0,1]\to \R$ defined by
$f(t)=\left(\v+t \Delta\right)\sprod \left(\v+t \Delta\right)$.
By convexity, we simply have 
$\v+t\Delta\in X~\forall t\in[0,1]$, and, 
using the above lemma,
we also obtain $f(t)>f(0)=\gamma~\forall t>0$.
This means the derivative in 0 can not be negative;
we need to have $f'(t)\geq 0$. 
Since $f'(t)=2 \v\sprod \Delta+t^2\Delta\sprod \Delta$, 
this means 
$\v\sprod \Delta\geq 0$, 
enough to show $\v\sprod \x=\v\sprod\left(\v+\Delta\right)=\gamma+\v\sprod\Delta\geq
\gamma>0$.

Finally, if $\y\neq \zeros$, it is enough to consider set $X'=\{\x-\y:~\x\in X\}$ which
is convex and does not contain $\zeros$. We can apply the theorem for $\zeros$
and $X'$ and we obtain there is $\v\in\R^n$ such that $\v\sprod (\x-\y)>0$,
$\forall \x\in X$, which is equivalent to \eqref{eqxycsimple} as needed.
\end{proof}
\begin{theorem}\label{thSupportHyper} (Simple supporting hyperplane theorem)
Given convex closed set $X\subset\R^n$ such that $\zeros$ is a boundary point of
$X$, there exist a non-zero $\v\in\R^n$ such that 
\begin{equation}\label{eqxycsimplesupport}
\v\sprod \x \geq 0,~\forall \x\in X.
\end{equation}

\end{theorem}
\begin{proof}
Since $\zeros$ is a boundary point, there exists a sequence $\{\x_i\}$ of exterior
points (\ie, $\x_i\notin X,~\forall i\in \N^*$) that converges to $\zeros$.
Using the simple separation Theorem~\ref{thSimpleSep}, for each $\x_i$ there
exists a non-zero unitary $\v_i\in \R^n$ such that 
\begin{equation}\label{eqconverger}\v_i\sprod \x >\v_i\sprod \x_i~\forall \x\in
X.\end{equation} Without loss of generality, we can consider all $\v_i$ are unitary, \ie,
$|\v_i|=1~\forall i\in\N^*$. As such, the sequence $\{\v_i\}$ is bounded, and
so, we can apply the
Bolzano-Weierstrass Theorem~\ref{thBolzanoWeierstrass} to conclude that
$\{\v_i\}$ contains a
convergent sub-sequence $\{\v_{n_i}\}$ such that $\lim\limits_{i\to \infty}
\v_{n_i}=\v$.

Assume there is some $\x\in X$ such that $\v\sprod \x=-a<0$. 
We derive a contradiction using a limiting argument.
We take an $\epsilon<|a|$;
since $\lim\limits_{i\to \infty} \v_{n_i}\sprod \x =\v\sprod \x=-a$,
there exits some $m\in \N^*$ such that $\v_{n_i}\sprod \x\in
[-a-\epsilon,-a+\epsilon]$ for all $i\geq m$. 
Applying \eqref{eqconverger}, we obtain that 
all these $i\geq m$ satisfy $\v_{n_i}\sprod  \x_{n_i}<-a+\epsilon<0$.
This contradicts the fact that 
$\lim\limits_{i\to \infty} \x_{n_i}=0$. Indeed, if $\v_{n_i}$ can become
arbitrarily close to $\v$ while at the same time $\x_{n_i}$ becomes
arbitrarily close to $\zeros$, then the product $\v_{n_i}\sprod \x_{n_i}$ can
also become arbitrarily close to $0$.
\end{proof}

\paragraph{Convergence theorems on sequences}
We need several convergence results on sequences for the classical proof of the
hyperplane separation theorem, \ie, in particular for Theorems
\ref{thSimpleSep} and \ref{thSupportHyper}.

\begin{proposition}\label{propSupExists}Any bounded set $S\subsetneq \R$ has a
unique finite  least upper bound (or supremum) $\sup(S)$. Equivalently,
a unique $\inf(S)$ also exists and is finite.
\end{proposition}
\begin{proof}
To avoid unessential complication, we will assume that $S$ contains at least a
positive number. Any set $S'$ can be transformed to this form by applying 
a simple translation $S=\{s'-s'_0:~s\in S'\}$ for some fixed $s'_0\in S'$. If $\sup(S)$
exists, then $\sup(S')=\sup(S)+s'_0$.

We will determine $\sup(S)$ as a real number $a$ written in the decimal
expansion as $a=a_0.a_1a_2a_3\dots$ with potentially infinite number of digits. 
We can suppose that $a_0.a_1a_2\dots a_n99\dots 9$ with $a_n\neq 9$ is equal to 
$a_0.a_1a_2\dots a_{n-1}(a_n+1)$.
We choose
the decimals as follows.
\begin{itemize}
    \item[--] $a_0$ is the greatest integer that is not a strict upper bound for
$S$ (\ie, that is not strictly greater than all elements of $S$). A finite $a_0$ value must
exists because $S$ is bounded.
    \item[--] $a_1$ is the greatest digit such that $a_0.a_1$ is not a strict upper bound of $S$
    \item[--] $a_2$ is the greatest digit such that $a_0.a_1a_2$ is not a strict upper bound of $S$
    \item[$\vdots$]
    \item[--] $a_n$ is the greatest digit such that $a_0.a_1a_2\dots a_n$ is not a strict upper bound of $S$
    \item[$\vdots$]
\end{itemize}
We now prove that $a=a_0.a_1a_2a_3\dots$ is an upper bound of $S$. Assume there
exists an $s\in S$ such that $s>a$. This means there exists an index $n$ such that
$s$ can be written $a_0.a_1a_2\dots a_{n-1}s_ns_{n+1}s_{n+2}\dots $ where
$s_n> a_n$. Since $a_n$ is the greatest digit such that $a_0.a_1a_2\dots
a_n$ is not a strict upper bound for $S$, we obtain that $a_0.a_1a_2\dots a_{n-1}s_n$ is
a strict upper bound for $S$. This implies that $s\geq a_0.a_1a_2\dots a_{n-1}s_n$ is
also a strict upper bound of $S$, which contradicts $s\in S$. There can be no $s\in S$
such that $s>a$.

To prove the uniqueness, 
we still have to show that $a$ is the minimum upper bound, \ie, 
there is no other upper bound $a'<a$. 
For the sake of contradiction, 
assume there exists such an upper bound 
$a'<a$; it can be written $a'=a_0.a_1a_2\dots a_{n-1}a'_na'_{n+1}a'_{n+2}\dots$ such
that $a'_n< a_n$ for some $n$ and $a'_{n+1}, a'_{n+2}\dots$ are not all $9$---such a
number would reduce to $a_0.a_1a_2\dots a_{n-1}(a'_n+1)$. We thus
obtain that $a'<a_0.a_1a_2\dots a_{n-1}a_n$. Since $a'$ is an upper bound, 
$a_0.a_1a_2\dots a_{n-1}a_n$ needs to be a strict upper bound. This contradicts
the choice of $a_n$ as the greatest digit such that $a_0.a_1a_2\dots a_{n-1}a_n$
is \textit{not} a strict upper bound of $S$.

By combining the two above paragraphs, we obtain that $a=a_0.a_1a_2a_3\dots$ is
the least upper bound $\sup(S)$. Recall we can suppose that
$a_0.a_1a_2\dots a_n99\dots 9$ with $a_n\neq 9$ is equal to $a_0.a_1a_2\dots
a_{n-1}(a_n+1)$.
\end{proof}

\begin{proposition}\label{propMonotone} Any bounded monotone sequence $\{a_i\}$ of real numbers is
convergent.
\end{proposition}
\begin{proof}
Since the sequence is monotone, we can 
consider it is non-decreasing, the non-increasing case being completely
analogous.
Let $A=\left\{a_i:~i\in\{1,2,\dots \infty\}\right\}$.
Since $A$ is bounded, the least upper bound property (Prop.~\ref{propSupExists})
states that $a=\sup(A)$ exists and is finite.
For any $\epsilon>0$, there needs to exist some positive integer $n$ such
that $a_n>a-\epsilon$, because otherwise $a-\epsilon$ would be an upper bound 
lower than $\sup(A)$, impossible. Since $\{a_i\}$ is non-decreasing, all integers
$n'>n$ verify $a_{n'}\geq a_{n}>a-\epsilon$. This is exactly the definition of
the fact that $\lim\limits_{i\to \infty}a_i=a$.
\end{proof}
\begin{theorem}\label{thBolzanoWeierstrass} (Bolzano--Weierstrass theorem)
Any bounded sequence $\{\x_i\}$ of $\R^n$ contains a convergent subsequence.
\end{theorem}
\begin{proof}
We first show the theorem for $n=1$.
\begin{lemma}
Any bounded sequence $\{x_i\}$ of $\R$ contains a convergent subsequence.
\end{lemma}
\begin{proof}
We consider the set of maxima $N=\left\{i \in
\left\{1,2,\dots\right\}:~x_j<x_i,\forall j>i\right\}$. 
We distinguish three cases depending on the cardinal of $N$:
\begin{enumerate}
\item If $N=\emptyset$, then
for any index $n_i$, there exists a position $n_{i+1}$ such that
$x_{n_{i+1}}\geq x_{n_i}$.
The subsequence $x_{n_1},~x_{n_2},~x_{n_3},\dots$ is monotone non-decreasing, and so,
convergent using Prop.~\ref{propMonotone}.
\item If $N$ is not finite, then $N$ contains an infinite sequence of indices
$n_{1}<n_{2}<n_{3}<\dots $ such that $x_{n_1}>x_{n_2}>x_{n_3}\dots$. The
subsequence $x_{n_1},~x_{n_2},~x_{n_3}\dots$ is monotone decreasing, and using
Prop.~\ref{propMonotone}, it is convergent.
\item If $|N|=t$ with $t\in\N-\{0\}$, then 
$N$ contains a finite sequence of indices
$n_{1}<n_{2}<n_{3}<\dots n_t $ such that $x_{n_1}>x_{n_2}>x_{n_3}\dots >x_{n_t}$.
The set $N_{>n_t}=\left\{i \in
\left\{n_t+1,n_t+2,n_t+3,\dots\right\}:~x_j<x_i,\forall j>i\right\}$ is empty.
We can thus apply the argument of case 1 and obtain that the infinite sequence
$x_{n_t+1}$, $x_{n_t+2}$, $x_{n_t+3}$, $\dots$ contains a convergent sub-sequence.
\end{enumerate}
\vspace{-2em}
\end{proof}
We now generalize the result for any $n>1$. Considering the first position of
the sequence $\{\x_i\}$, the above lemma shows that $\{\x_i\}$ contains a
sub-sequence $\{\x^1_i\}$ whose first position converges to some $y_1$. We now
consider the second position of $\{\x^1_i\}$. Using the lemma again, we obtain 
that $\{\x^1_i\}$ contains a sub-sequence $\{\x^2_i\}$ whose second position
converges to some $y_2$. We observe that the first position of $\{\x^2_i\}$
converges to $y_1$ and the second to $y_2$. The argument can be repeated to find
sub-sequences $\{\x^1_i\}$, $\{\x^2_i\}$, $\dots $, $\{\x^n_i\}$ such that
$\{\x^n_i\}$ converges to $[y_1~y_2~\dots y_n]^\top$.
\end{proof}

\noindent \line(1,0){100}
\section*{References}
References are provided throughout the document as footnote citations. This is because I wanted
to make each reference readily available to the reader. 
I acknowledge again that
I mentioned throughout the document the work of the following people 
(lecture notes and papers only, 
excluding web-sites and responses on mathematical on-line forums), in the
order of apparition:
Maur\' icio de Oliveira,
Christoph Helmberg,
Robert Freund,
David Williamson,
Roger Horn,
Charles Johnson
H.~Ikramov,
Stephen Boyd,
Lieven Vandenberghe,
Anupam Gupta,
L\' aszl\' o Lov\' asz,
Michael Overton,
Henry Wolkowicz,
Michel Goemans,
Neboj\originalv{s}a Gvozdenovi\'c,
Donald Knuth,
Monique Laurent,
Immanuel Bomze, 
Mirjam D{\"u}r,
Chung-Piaw Teo,
Peter Dickinson, 
Luuk Gijben,
Etienne de Klerk,
Dmitri Pasechnik,
Pablo Parrilo,
Fr\'ed\'eric~Roupin,
Alain Billionnet, 
Sourour Elloumi,
Marie-Christine Plateau,
Alain Faye, 
Am\' elie Lambert,
Peter Norman,
Subhash Khot, Guy
Kindler, Elchanan
Mossel,
Ryan O'Donnell
and
Richard Dudley.
\end{document}